\documentclass[11pt]{amsart}
\usepackage{fullpage}
\usepackage[utf8]{inputenc}
\usepackage{microtype}
\usepackage{graphics}
\usepackage{amsmath,amssymb,stmaryrd,array,multirow,dsfont,mathrsfs}
\usepackage{amsthm}
\usepackage{dsfont}
\usepackage{enumerate}
\usepackage{xcolor}
\usepackage{color}
\usepackage{hyperref}
\usepackage{indentfirst}
\usepackage{pifont}
\usepackage[T1]{fontenc}
\usepackage{microtype}
\usepackage{graphics}
\usepackage{amsthm}
\usepackage{xcolor}
\usepackage{color}
\usepackage{hyperref}
\graphicspath{{Images/}}
\xdefinecolor{darkgreen}{rgb}{0,0.4,0}
\newcommand{\normmm}[1]{{\left\vert\kern-0.25ex\left\vert\kern-0.25ex\left\vert #1 
   \right\vert\kern-0.25ex\right\vert\kern-0.25ex\right\vert}}
\allowdisplaybreaks[4]
\newcommand{\beq}{\begin{equation}}
\newcommand{\eeq}{\end{equation}}
\newcommand{\beqs}{\begin{equation*}}
\newcommand{\eeqs}{\end{equation*}}
\newcommand{\ben}{\begin{eqnarray}}
\newcommand{\een}{\end{eqnarray}}
\newcommand{\beno}{\begin{eqnarray*}}
\newcommand{\eeno}{\end{eqnarray*}}

\renewcommand{\div}{{\rm div}}

\newcommand{\Lip}{{\rm Lip}\,}
\newcommand{\Id}{{\rm Id}\,}
\newcommand{\Supp}{{\rm Supp}\,}

\newcommand{\Rmnum}[1]{\uppercase\expandafter{\romannumeral #1} }
 \numberwithin{equation}{section}

\newtheorem{defi}{Definition}[section]
\newtheorem{thm}{Theorem}[section]
\newtheorem{lem}[thm]{Lemma}
\newtheorem{prop}[thm]{Proposition}
\newtheorem{rmk}[thm]{Remark}
\newtheorem{cor}[thm]{Corollary}

\def\curl{\mathop{\rm curl}\nolimits}
\def\il{|\!|\!|}

\def \d {\mathrm {d}}
\def\D {\mathrm {D}}
\def\cA{{\mathcal A}}

\def\cC{{\mathcal C}}

\def\cE{{\mathcal E}}
\def\cF{{\mathcal F}}
\def\cG{{\mathcal G}}
\def\cH{{\mathcal H}}

\def\cJ{{\mathcal J}}
\def\cK{{\mathcal K}}
\def\cL{{\mathcal L}}

\def\cN{{\mathcal N}}
\def\cO{{\mathcal O}}

\def\cQ{{\mathcal Q}}
\def\cR{{\mathcal R}}

\def\cT{{\mathcal T}}

\def\cV{{\mathcal V}}

\def\cZ{{\mathcal Z}}

\def\ud{\underline}

\let\f=\frac
\def \p {\partial}
\def\mR {\mathbb{R}}
\def\mS {\mathcal{S}}
\def\ep{\varepsilon}

\def \ltx {L_t^2L^2}

\def \pt {\partial_{t}}
\def \vr {\varrho}
\def \vp {\varphi}
\def\bN   {\mathbf{N}}
\def \bbp {\mathbb{P}_t}
\def \bbq {\mathbb{Q}_t}

\def\si {\sigma}
\def\na{\nabla}

\def \lesim {\lesssim}

\def \tz {\tilde{\zeta}}

\def \bn {\textbf{n}}
\def \hco {L_t^2H_{co}}
\def \htlde  {L_t^2\tilde{H}}
\def\lae {\Lambda\big(\f{1}{c_0},
\cN_{m,T}\big)}

\def \lca {\Lambda\big(\f{1}{c_0},\cA_{m,t}\big)}

\def\undertilde#1{\mathord{\vtop{\ialign{##\crcr
$\hfil\displaystyle{#1}\hfil$\crcr\noalign{\kern1.5pt\nointerlineskip}
$\hfil\tilde{}\hfil$\crcr\noalign{\kern1.5pt}}}}}

\title{ Incompressible limit for the free surface Navier-Stokes system.}
\author{Nader Masmoudi, Fr\'ed\'eric Rousset, Changzhen Sun}
\address{NYUAD Research Institute, New York University Abu Dhabi, PO Box 129188, Abu Dhabi,   United Arab Emirates.
 Courant Institute of Mathematical Sciences, New York University, 251 Mercer Street, New York, NY 10012, USA.}
\email{masmoudi@cims.nyu.edu}
\address{Universit\'e Paris-Saclay,  CNRS, Laboratoire de Math\'ematiques d'Orsay (UMR 8628),  91405 Orsay Cedex, France}
\email{frederic.rousset@universite-paris-saclay.fr. }

\address{Institut de Math\'ematiques de Toulouse, UMR5219, Universit\'e de Toulouse, CNRS, INSA, F-31077 Toulouse, France}
\email{changzhen.sun@math.univ-toulouse.fr}
\begin{document}
\maketitle
\begin{abstract}
    We establish uniform regularity  estimates with respect to the Mach number
    for the three-dimensional free surface compressible Navier-Stokes system in the case of slightly well-prepared initial data
     in the sense that the acoustic components like the divergence of the velocity field are 
      of size $\sqrt{\varepsilon}$, $\varepsilon$ being the Mach number. These estimates allow us  to  justify the convergence 
      towards the free surface incompressible Navier-Stokes system  in the low Mach number limit. %{\color{red}
      One of the main difficulties is the control of the regularity of the surface in presence of  boundary layers with fast oscillations.
\end{abstract}
%\tableofcontents
\section{Introduction}
We consider the motion of a slightly compressible viscous fluid with a free surface.
It takes the following form:
\beq \label{FCNS}
 \left\{
\begin{array}{l}
 \displaystyle\pt \rho^{\varepsilon} +\div( \rho^{\varepsilon} w^\varepsilon)=0,\\
 \displaystyle\pt ( \rho^{\varepsilon} w^{\varepsilon})+\div(\rho^{\varepsilon}w^{\varepsilon}\otimes w^{\varepsilon} )-\div\mathcal{L}w^{\varepsilon}+
\f{\nabla P(\rho^{\varepsilon})}{\ep^2}=0,   \\
\displaystyle \rho^{\ep}|_{t=0}=\rho_0^{\ep}, \ w^{\ep}|_{t=0}=w_0^{\ep},
\end{array} \qquad \text{$(t,x)\in \mathbb{R}_{+}\times \Omega_t^{\ep} $},
\right.\\
\eeq
where $\rho^{\ep}>0,w^{\ep}\in \mR^3$ are the density and the velocity of the fluid, $P(\rho^{\ep}),$ a smooth function of $\rho^{\ep},$ stands for the pressure. The viscous tensor $\mathcal{L}w^{\ep}$ takes the form: $$\mathcal{L}w^{\ep}=
2\mu S w^{\ep}+\lambda \div w^{\ep} \text{Id}, \quad S w^{\ep}=\f{1}{2}(\nabla w^{\ep}+\nabla^{t} w^{\ep}).$$ 
Here,  
$\mu,\lambda$ are the  viscosity parameters that are assumed to be constant and to satisfy the conditions: $\mu>0, 2\mu+3\lambda>0.$ The parameter $\ep$ is the scaled Mach number which  is  assumed  small, that is $\ep\in(0,1]$. 
We focus on a  fluid domain  given by:
$$\Omega_t^{\ep}=\{x=(y,z)|\ y\in\mathbb{R}^2,-1<z<h^{\ep}(t,y)\},$$
where the upper surface is free and the bottom is fixed.
Here $h^{\ep}(t,y)$, the surface of the fluid domain, is unknown and needs to be solved together with $(\rho^{\ep},w^{\ep}).$ Since the fluid particles do not cross the surface,
$h^{\ep}$ solves %is governed by the following equation:
\beq\label{surface equation}
\p_t h^{\ep}-w^{\ep}(t,y,h^{\ep}(t,y))\cdot \textbf{N}^{\ep}=0, \quad h^{\ep}(0,y)=h_0^{\ep}(y) \quad y\in\mathbb{R}^2
\eeq
where $\textbf{N}^{\ep}=(-\p_1h^{\ep},-\p_2 h^{\ep},1)^{t}$ denotes the outward normal vector to the surface $\Sigma_t^{\ep}=\{x=(y,z),z=h^{\ep}(t,y)\}.$
We supplement the system \eqref{FCNS} and \eqref{surface equation} with the following physical conditions. At the upper boundary, the continuity of the stress tensor reads:
\beq\label{upbdry}
\mathcal{L}u^{\ep} \textbf{N}^{\ep}=\frac{1}{\ep^2}\big(P(\rho^{\ep})-P(\bar{\rho})\big)\textbf{N}^{\ep} \quad \text{on} \quad \Sigma_t^{\ep}
\eeq
where $\bar{\rho}>0$ is a reference constant density.
At the bottom, we prescribe a slip boundary condition:
\beq\label{bebdry}
w_3^{\ep}=0, \quad \mu\p_3 w_{j}^{\ep}=a w_j^{\ep}\quad(j=1,2), \quad \text{on} \quad \{z=-1\},
\eeq
where $a$ is a constant that quantifies the effects of the friction at the boundary (this can be easily generalized
 to a smooth function $a$, see \cite{masmoudi2021uniform}). The case of the Dirichlet boundary condition at the bottom raises other difficulties even without the presence of a free surface and is left for future work. Note that we could also consider the case of a strip with infinite depth, see Section \ref{rmkhalfspace}.

%Involving one term that depends on the small parameter $\varepsilon,$
 The system \eqref{FCNS} can be  obtained from a suitable  scaling of the original physical variables. Indeed, we get \eqref{FCNS}, \eqref{surface equation} by performing the following scaling:
 $$\tilde{\rho}(t,x)=\rho^{\ep}(\ep t,x),\,\tilde{w}(t,x)=\ep w^{\ep}(\ep t,x),\, \tilde{h}=h^{\ep}(\ep t,x),\,\tilde{\mu}=\ep \mu,\,\tilde{\lambda}=\ep\lambda,$$ where
 $\tilde{\rho},\tilde{u},\tilde{h}$ satisfy:
 \beq
 \left\{
\begin{array}{l}
 \displaystyle\pt \tilde{\rho} +\div(\tilde{\rho} \tilde{w})=0,\\
 \displaystyle\pt (\tilde{\rho} \tilde{w})+\div(\tilde{\rho}\tilde{w} \otimes \tilde{w} )-\div\tilde{\mathcal{L}}\tilde{w}+
\nabla P(\tilde{\rho})=0,   \\
\displaystyle \p_t \tilde{h}+\tilde{w}(t,y,\tilde{h}(t,y))\cdot \tilde{\textbf{N}}=0, \\ 
\end{array} 
\right.\\
\eeq
where $\tilde{\mathcal{L}}\tilde{w}=2\tilde{\mu}S\tilde{w}+\tilde{\lambda}\div \tilde{w}.$

The aim of this paper is to study the low Mach number limit problem, that is to
study the behavior of (strong) solutions to \eqref{FCNS} when $\varepsilon$ tends to 0.
Formally, due to the singular term $\frac{\nabla P(\rho^{\varepsilon})}{\varepsilon^2},$ the pressure (and hence the density $\rho^{\varepsilon}$) is expected to tend to a constant state in some suitable space, one thus expect that the limit of the solutions to \eqref{FCNS}, if it exists in a sufficiently strong sense,  will be the solution to the following incompressible free surface Navier-Stokes system:
  \beq \label{FINS}
 \left\{
\begin{array}{l}
 \displaystyle \bar{\rho}(\pt w^0+w^{0}\cdot\nabla w^{0} )-2\mu\div\, S w^{0}+
\nabla \pi^0=0,  \\
 \displaystyle \div\, w^0=0, 
 \qquad\qquad\qquad\qquad\qquad \text{$(t,x)\in \mathbb{R}_{+}\times \Omega_t^0 $}, \\
 %\displaystyle \p_t h^0-w^0(t,y,h^0(t,y))\cdot \textbf{N}^0=0,\quad (t,y) \in \mathbb{R}_{+}\times \mR^2,\\
 w^0|_{t=0}=w_0^0, \, h^0|_{t=0}=h_0^0,
\end{array}
\right.\\
\eeq
supplemented with the boundary conditions:
\begin{align*}
\p_t h^0-w^0(t,y,h^0(t,y))\cdot \textbf{N}^0=0,\quad (t,y) \in \mathbb{R}_{+}\times \mR^2,\\
    {S}w^0 \textbf{N}^0=\pi^0 \textbf{N}^0 \quad \text{on} \quad \{z=h^0(t,y)\},\\
    w_3^{0}=0, \quad \p_3 w_{j}^{0}=a w_{j}^{0} \, (j=1,2) \quad \text{on} \quad \{z=-1\},
\end{align*}
where $\textbf{N}^{0}=(-\p_1h^{0},-\p_2 h^{0},1)^{t}.$

The rigorous justification of the low Mach number limit has been %drawn much research attention since the last four decades. It has been 
studied extensively in  different contexts depending on the generality of the system (isentropic or non-isentropic), the type of the system (Navier-Stokes or Euler), the type of solutions (strong solutions or weak solutions), the properties of the domain  (without boundaries, with fixed or free boundaries), as well as the type of the initial data considered (well-prepared or ill-prepared). %There have been enormous works dealing with this issue when the fluid is occupied in the domain without boundaries or with  fixed boundaries. 
The mathematical justification of the low Mach number limit was initiated by Ebin \cite{MR431261}, Klainerman-Majda \cite{MR615627,MR668409} for \textit{local strong solutions} of compressible fluids (Euler or Navier-Stokes), in the whole space with well-prepared data ($\div\, u_0^{\varepsilon}=\cO(\varepsilon), \nabla P_0^{\varepsilon}=\cO(\varepsilon^2)$) and later, by Ukai \cite{MR849223} for ill-prepared data ($\div\, u_0^{\varepsilon}=\cO(1),\nabla P_0^{\varepsilon}=\cO(\varepsilon)$).
These works are then extended by several authors in different settings.
One can refer for instance to \cite{MR2211706, MR1917042,Mr1834114,MR1946548} 
 for the study of the non-isentropic (Euler or Navier-Stokes) equations %in various physical domain (whole space, torus) 
 under ill-prepared initial data whenever the domain is the whole space or the torus, and also
\cite{MR834481,MR2812710} for bounded domains with well-prepared initial data.  
 There are also many other related works, one can see for example \cite{MR2106119,MR1308856, MR1886005,MR3563240,MR1702718,MR2575476,MR4011035,MR918838,MR4293725,MR1628173,MR1710123}. For more exhaustive information, one can
refer for example to the well-written survey papers by 
 Alazard \cite{MR2425022}, Danchin \cite{MR2157145}, Feireisl \cite{MR3916821}, Gallagher \cite{MR2167201},
 Jiang-Masmoudi \cite{MR3916820},
 Schochet \cite{MR3929616}.

The analysis of the   low Mach number limit problem for the isentropic compressible Navier-Stokes (CNS) system in domains with fixed boundaries, which is more  related to the interest of the current paper, 
has been done in two different directions. Roughly speaking, for (CNS) in fixed bounded domains, one can either justify the limit process directly from global weak solutions, or prove that local strong solutions exist on a time interval independent of the Mach number and use compactness arguments to pass to the limit. For the first case, Lions and Masmoudi \cite{MR1628173} investigated the convergence of \textit{weak solutions} to (CNS) in bounded domains with various boundary conditions. Later on, 
for the same problem in bounded domains with Dirichlet boundary conditions, the authors in \cite{MR1697038, MR3281954} noticed that under some geometric assumption on the domain,
the acoustic waves are damped in a boundary layer so that local in time  strong convergence ($L_{t,x}^2$) holds.  One can also refer to \cite{MR2575476} for the justification of the convergence towards a solution of the incompressible Navier-Stokes system in unbounded domains by using the local energy decay for the acoustic system. All these results hold true for ill-prepared initial data. 
Concerning the local strong solutions, uniform high order energy estimates are established in \cite{MR2812710} with Dirichlet boundary conditions  and in \cite{MR3240080} with Navier-slip boundary conditions by assuming the initial data to be well-prepared.  Recently, we established in \cite{masmoudi2021uniform}  uniform high  regularity estimates in bounded domains with Navier-slip boundary conditions and {ill-prepared} initial data. To match the boundary layer effects due to the fast oscillations and the ill-prepared initial data assumption, we proved uniform estimates in an anisotropic functional framework with only one normal derivative close to the boundary.
 
  There are only a few works dealing with  the low Mach number limit problem for systems in the presence of free boundaries.
   They deal with inviscid systems. In  \cite{MR3812074}, Lindblad-Luo prove uniform a-priori estimates for the  free boundary compressible Euler equations in the case of a bounded reference domain. More recently, this result is extended by Luo
\cite{MR3887218} for unbounded reference domains and by Disconzi-Luo \cite{MR4097326}
for a bounded reference domain but with surface tension. All these results are based on the assumption that the initial datum is sufficiently well-prepared in the sense that the time derivatives up to at least order two are bounded initially, an assumption which
is stronger than the usual well-prepared data assumption which requires one time derivative to be bounded initially. Regarding viscous fluids, the author in \cite{ou-free} considered the 1d  compressible Navier-Stokes system with free boundaries and established uniform estimates with respect to the Mach number and the Froude number for both well-prepared and ill-prepared initial data. 
Nevertheless, within our knowledge, there is no related work for  multidimensional viscous systems. Indeed, in the multidimensional case, there are several difficulties that do not appear in the 1d case, 
% For instance, the uniform estimates of high order spatial derivatives cannot be obtained by those of (weighted) time derivatives directly. {\color{red}Indeed, one can recover the compressible part easily by using the equations and the a-priori control of time derivatives, but one will encounter many difficulties to control the vorticity, mainly due to the lack of the information of its trace on the boundary.
%Moreover, 
as will be explained later, a boundary layer appears in the multidimensional case which will preclude the uniform control of higher order ($\geq 2$) normal derivatives of the solution.
The aim of the current work is thus to 
investigate the low Mach number limit problem for 3d viscous fluids
solving  \eqref{FCNS}-\eqref{bebdry}.
For the simplicity of presentation (compared to the case of general bounded domains) 
%and without losing generality (compared to the half-space case), 
we choose a channel with finite depth as the reference domain. Nevertheless, one can extend easily our analysis to the cases where the reference domain is the half space or a bounded domain,
 we shall explain more about this aspect in Section \ref{rmkhalfspace}.

The core of the analysis in 
this paper is to establish some uniform high regularity estimates  in order to get the existence of a local strong solution on a time interval independent of $\ep$. Due to the presence of the diffusion term as well as the singular linear term, a boundary layer correction to the  highly oscillating acoustic waves appears and creates unbounded  high order normal derivatives of the velocity. 
%, since the vorticity interacts with the high speed oscillating acoustic waves through the boundary, a boundary layer is highly expected to be created, which excludes the uniform control of the high order normal derivatives. Moreover, at the scale $\tau=t/{\ep}$ of the system \eqref{FCNS2} behaves very like the small viscosity approximation of the free-surface Euler equation.
Therefore, we  need to work in a functional framework based on conormal Sobolev spaces that minimizes the use of normal derivatives near the boundary in the spirit of \cite{MR3786770,MR3590375, MR3741102}.
Note that in the current situation, we have to handle simultaneously  fast oscillations in time and a boundary layer effect
 so that the difficulties and the analysis will be very different from the ones in \cite{MR3741102}, where   compressible slightly viscous fluids are considered. Indeed, the energy estimates for conormal derivatives 
  cannot be directly  obtained  since tangential vector fields do not commute with the singular part of the system. Moreover, to include only slightly well-prepared data (we will explain later what it means), it will be impossible for us to get uniform estimates for time derivatives. In \cite{masmoudi2021uniform}, we  could  establish uniform estimates 
  for the  isentropic compressible Navier-Stokes system  with Navier boundary condition in smooth fixed domains and ill-prepared initial data. For free surface fluids, there are extra difficulties essentially related to the control of the regularity  of the free surface. Indeed, because of the occurrence of the singular terms, the compressible part of the system behaves at time scale $\tau=t/\ep$ like a small viscosity approximation of the acoustic system, we thus cannot obtain uniform extra regularity for the surface from the diffusion term. 
   This is the main reason for which some kind of well-prepared assumption will be needed.
    We could nevertheless impose an assumption that we call slightly well-prepared which is weaker than the usual
     well-prepared assumption that requires one time derivative to be of order  $\cO(1)$ and thus much weaker
     than the assumption made for the free surface Euler system, for example in  \cite{MR3812074}, where two derivatives
      of the solution  are assumed to be $\cO(1)$  initially. We only require  
 the first time derivative of the solution  to be of  order $\ep^{-\f{1}{2}},$ this is thus intermediate between ill-prepared $\cO(\ep^{-1})$ and well-prepared $\cO(1),$ see also Remark \ref{rmk-data}. The main heuristics  is that despite the extra difficulties  arising from the boundary layer effects (note that the presence of a boundary layer is a feature of the viscous problem and is absent in the inviscid case), the presence  of the diffusion term can 
  help us to gain some regularity of the surface (not necessarily uniform). It thus allows us to include more general data compared to the corresponding works on inviscid systems \cite{MR3812074,MR3887218,MR4097326}.
We shall explain more precisely below after the reformulation of the system and the statement of the main results.
\subsection{Reformulation of the system in a fixed domain}%Preliminaries: reformulation of the system and the change of the variable}
 Let us set  $$\varrho^{\ep}=\f{P(\rho)-P(\bar{\rho})}{\ep},$$
the system \eqref{FCNS} can be rewritten into the following symmetric form: 
  \beq \label{FCNS1}
 \left\{
\begin{array}{l}
 \displaystyle g_1(\ep\varrho^{\ep})\big(\pt \varrho^{\varepsilon}+ w^{\ep}\cdot \na\varrho^{\ep}\big)+\f{\div w^\varepsilon}{\ep}=0,\\[3pt]
 \displaystyle g_2\big(\ep \varrho^{\ep})(\pt   w^{\varepsilon}+w^{\varepsilon}\cdot\na w^{\varepsilon} \big)-\div \mathcal{L} w^{\varepsilon}+
\f{\nabla\varrho^{\varepsilon}}{\ep}=0,  \qquad \text{$(t,x)\in \mathbb{R}_{+}\times \Omega_t^{\ep} $}, \\[3pt]
 \displaystyle w^\ep|_{t=0} =w_0^{\varepsilon},\quad \vr^{\ep}|_{t=0}=\vr_0^{\varepsilon}\\
 \end{array}
 \right.
\eeq
where the scalar functions $g_1, g_2$ are defined by: 
\beq\label{defofg12}
g_2(s)=\rho^\ep=P^{-1}(P(\bar{\rho})+
s),\quad g_1(s)=(\ln g_2)' (s);\quad s>-\bar{P}=-P(\bar{\rho}).
\eeq
%We assume without loss of generality that $\bar{\rho}=P(\bar{\rho})=1.$
Moreover, the boundary condition \eqref{upbdry} is transformed into
\beq\label{upbdry1}
\mathcal{L}u^{\ep} \textbf{N}^{\ep}=\frac{\varrho^{\ep}}{\ep}\textbf{N}^{\ep} \quad \text{on} \quad \Sigma_t^{\ep}.
\eeq
In the following, we shall work on the system \eqref{FCNS1}, \eqref{surface equation} with boundary conditions \eqref{bebdry}, \eqref{upbdry1}.

We then  choose an appropriate change of coordinates to reduce the free-surface domain to a fixed one. One natural possibility is to use Lagrangian coordinates, nevertheless, since we shall consider the problem in the conormal Sobolev setting, the Lagrangian transformation would be also only bounded in the conormal setting, this would raise additional difficulties.
Therefore, instead of using Lagrangian coordinates, we shall use the following smoothing diffeomorphism \cite{MR3060183}, where the map will enjoy the usual Sobolev regularity. Let us set $\mathcal{S}=\mathbb{R}^2\times [-1,0],$ and consider the map
\beq\label{changeofvariable}
\begin{aligned}
\Phi_t^{\ep}:\mathcal{S}&\rightarrow\Omega_t^{\ep}\\
(y,z)&\mapsto \Phi^{\ep}(t,y,z)
=(y,\vp^{\ep}(t,y,z))^t
\end{aligned}
\eeq
where 
\beq\label{defvpep}
\vp^{\ep}(t,y,z)=z+\eta^{\ep}(t,y,z)(1+z).
\eeq
Here $\eta^{\ep}$ is given by a smoothing extension
\beq\label{defeta}
(\mathcal{F}\eta^{\ep})(t,\xi,z)=e^{-\delta_0(1+|\xi|^2)z^2}(\mathcal{F}h^{\ep})(t,\xi)
\eeq
where $\mathcal{F}$ stands for the Fourier transform with respect to the horizontal variable $y\in \mR^2,$ $\delta_0$ is a small parameter such that $\det(\D\Phi_0^{\ep})>0,$ which ensures that $\Phi_0^{\ep}$
is a diffeomorphism.
Note that $$\det(\D\Phi_0^{\ep})=\p_z\varphi^{\ep}(0,x)=1+h^{\ep}(0,x)+(\eta^{\ep}-h^{\ep})(0,x)+\p_z\eta^{\ep}(0,x)(1+z)>2c_0>0$$
as long as 
\beq\label{diffcdtion0}
1+h^{\ep}(0,x)\geq 3c_0>0, \,\forall x\in \mS,
\eeq
\beq\label{diffcdtion}
\|(\eta^{\ep}-h^{\ep})(0)\|_{L^{\infty}(\mS)}+\|\p_z\eta^{\ep}
(0)\|_{L^{\infty}(\mS)}<c_0,
%\f{\inf_{x\in\mS}(1+h^{\ep}(t,x))}{2}
\eeq
where  $c_0>0$ is a fixed constant. Let us notice that $\eqref{diffcdtion}$ holds %Therefore, $\Phi^{\ep}$ is a diffeomorphism when \eqref{diffcdtion0}-\eqref{diffcdtion} are satisfied. Note that \eqref{diffcdtion} holds 
 if $\|h^{\ep}(0)\|_{H^{s}(\mR^2)}<+\infty$, for some $s>2$ and $\delta_0$ is chosen sufficiently small.
Moreover, we have that
$$\|\nabla \varphi^{\ep}(t)\|_{L^2(\mathcal{S})}\lesssim |h^{\ep}(t)|_{H^{\f{1}{2}}(\mR^2)},$$
which means that we gain one half derivative.

Let  us now set
$$u^{\ep}(t,y,z)=w^{\ep}(t,y,\Phi^{\ep}(t,y,z)), \quad \sigma^{\ep}=\varrho^{\ep}(t,y,\Phi^{\ep}(t,y,z))$$
where $u^\ep$ and $\sigma^\ep$ are defined in $\mathcal{S}$.
 Then we set,
$\p_j^{\varphi^{\ep}}u^{\ep}=(\p_j w^{\ep})\circ \Phi^{\ep}, \p_j^{\varphi^{\ep}}\sigma^{\ep}=(\p_j \varphi^{\ep})\circ \Phi^{\ep},$ where $j=0,1,2,3$ with $\p_0=\p_t,\p_3=\p_z$ which yields 
\beq\label{newder}
\begin{aligned}
&\p_i^{\varphi^{\ep}}=\p_i-\f{\p_i\varphi^{\ep}}{\p_z\varphi^{\ep}}\p_z, \quad i=0,1,2,\quad \p_z^{\varphi^{\ep}}=\f{1}{\p_z\varphi^{\ep}}\p_z. 
\end{aligned}
\eeq
The equations \eqref{FCNS1}, \eqref{surface equation} and the boundary conditions \eqref{upbdry1}, \eqref{bebdry} are reformulated into the following systems:
\beq \label{FCNS2}
 \left\{
\begin{array}{l}
 \displaystyle g_1(\ep\si^{\ep})\big(\pt^{\varphi^{\ep}} \si^{\varepsilon}+ u^{\ep}\cdot \na^{\vp^{\ep}}\si^{\ep}\big)+\f{\div^{\vp^{\ep}} u^\varepsilon}{\ep}=0,\\
 \displaystyle g_2\big(\ep \si^{\ep})(\pt^{\vp^{\ep}}   u^{\varepsilon}+u^{\varepsilon}\cdot\na^{\vp^{\ep}} u^{\varepsilon} \big)-\div^{\vp^{\ep}} \mathcal{L}^{\vp^{\ep}} u^{\varepsilon}+
\f{\nabla^{\vp^{\ep}}\sigma^{\varepsilon}}{\ep}=0,  \qquad \text{$(t,x)\in \mathbb{R}_{+}\times \mS $} \\[5pt]
 \displaystyle u^\ep|_{t=0} =w_0^{\varepsilon}(\Phi_0^{\ep}(x)):=u_0^{\ep},\quad \si^{\ep}|_{t=0}=\vr_0^{\varepsilon}(\Phi_0^{\ep}(x)):=\sigma_0^{\ep},\\
 \end{array}
 \right.
\eeq
\beq\label{surfaceeq2}
\pt h^{\ep}-u^{\ep}(t,y,h^{\ep}(t,y))\cdot \bN^{\ep}=0,
\eeq
\beq\label{upbdry2}
\mathcal{L}^{\vp^{\ep}}u^{\ep} \textbf{N}^{\ep}=\frac{\sigma^{\ep}}{\ep}\textbf{N}^{\ep} \quad \text{on} \quad \{z=0\},
\eeq
\beq\label{bebdry2}
u_3^{\ep}=0, \quad \mu\p_z^{\vp^{\ep}} u_{j}^{\ep}=a u_j^{\ep}\quad(j=1,2), \quad \text{on} \quad \{z=-1\}.
\eeq
\subsection{Conormal spaces and notations.} %definitions of norms}
Before stating our results, we need to 
introduce some notations.
We define the conormal vector fields:
$$Z_0=\ep\p_t,\, Z_1=\p_{y_1},\,Z_2=\p_{y_2},\,Z_3=\phi(z)\p_z$$
where the weight function is $\phi(z)=z(1+z)/(2-z)^2.$
We then introduce the space-time conormal space as follows, for $p=2,+\infty,$
$$L_{t}^p H_{co}^{m}\big( \mS\big)=\{f | \ 
Z^{\alpha}f\in L^p([0,t];L^2( \mS)),|\alpha|\leq m\},$$
%$$L_t^{\infty}W_{co}^{m,\infty}( \mathbb{R}_{+}^3)=\{f\in L^{\infty}([0,t],L^\infty(\Omega)),Z^{\alpha}f\in L^\infty([0,t]\times\Omega),|\alpha|\leq m\},$$
%Throughout this paper, we shall only use the spaces where $p,q=2,+\infty.$
%Moreover, we shall denote for simplicity
with the corresponding norms:
\beq\label{defconormal1}
%\|f\|_{m,t}=
\|f\|_{L_t^pH_{co}^{m}}=\sum_{|\alpha|\leq m}\|Z^{\alpha} f\|_{L^p([0,t],L^2(\mS))},\quad 
\eeq
where $\alpha=(\alpha_0,\alpha')=(\alpha_0,\alpha_1,\alpha_2,\alpha_3)\in \mathbb{N}^4.$ Moreover, we shall also use the $L_{t,x}^{\infty}$ type norm defined by:
\beq\label{inftynorm}
\il f\il_{k,\infty,t}=\sum_{|\alpha|\leq k}\|Z^{\alpha} f\|_{L^\infty([0,t]\times\mS)}. 
\eeq
To distinguish the number  of time and space derivatives, we introduce also the norm:
\beq
\|f\|_{L_t^p\mathcal{H}^{j,l}}=\sum_{\alpha_0\leq j,|\alpha'|\leq l}\|Z^{\alpha}f\|_{L^p([0,t],L^2(\mS))},  %\quad \|f\|_{L_t^2\mathcal{H}^{m}}= \|f\|_{L_t^2\mathcal{H}^{m,0}}
\eeq
and to simplify, we use 
\beq\label{chj}
\mathcal{H}^{j}=\cH^{j,0}.
\eeq
To measure the regularity along the boundary, we use:%the the following norms along the boundary will also be used:
\beq\label{bdynorm}
  |f|_{L_t^p\tilde{H}^s}=\sum_{j=0}^{[s]}|(\varepsilon\partial_t)^j f|_{L_t^p{H}^{s-j}(\mR^2)},  \quad
|f|_{k,\infty,t}=\sum_{|\alpha|\leq k,\alpha_3=0}|Z^{\alpha}f|_{L^\infty([0,t]\times\mR^2)}.
  \eeq
Finally, to measure pointwise regularity at a given time $t$ (in particular also with $t=0$), we shall use the semi-norms:
\beq\label{seminorm-bdry}
|f(t)|_{\tilde{H}^{s}}=\sum_{j=0}^{[s]}|(\varepsilon\partial_t)^j f(t)|_{{H}^{s-j}(\mR^2)},
\eeq
\begin{align}\label{seminorm}
&\|f(t)\|_{H_{co}^m}:=\sum_{|\alpha|\leq m}\|(Z^{\alpha}f)(t)\|_{L^2(\mS)},\quad \|f(t)\|_{\mathcal{H}^{j,l}}:=\sum_{\alpha_0\leq j,|\alpha'|\leq l}\|(Z^{\alpha}f)(t)\|_{L^2(\mS)},
\end{align}
\beq
\|f(t)\|_{k,\infty,\mS}:=\sum_{|\alpha|\leq k} \|(Z^{\alpha} f)(t)\|_{L^{\infty}(\mS)}.
\eeq
\subsection{Main results}
Before stating our main result, we first introduce the definition of the compatibility conditions which are
necessary to obtain  smooth enough solutions for the initial-boundary value problem of parabolic systems. 
\begin{defi}[Compatibility condition]
We say that $(\sigma_0^{\ep},u_0^{\ep})$ satisfy the compatibility condition up to order $m$ if for $j=0,1\cdots m-1,$
\beq\label{Compatibility condition}
\begin{aligned}
&\qquad\qquad\qquad (\varepsilon\partial_t)^{j}\big(\mathcal{L}^{\vp^{\ep}}u^{\ep}\bn^{\ep}\big)|_{t=0}= (\varepsilon\partial_t)^{j}(\sigma^{\ep}/\ep)\big|_{t=0}, \quad\text{on} \quad \{z=0\},\\
&\qquad\qquad\qquad\ep^j\pt^{j+1}h^{\ep}|_{t=0}=(\ep\pt)^j(u^{\ep}\cdot\bN^{\ep})|_{t=0} \quad\text{on} \quad \{z=0\},\\
&\qquad \big((\varepsilon\partial_t)^{j}u_3^{\ep}\big)\big|_{t=0}
=0,
\quad \big((\ep\pt)^j\p_z^{\vp^{\ep}}u_{j}^{\ep}\big)|_{t=0}=\f{a}{\mu}(\ep\pt)^j u^{\ep}_{j}|_{t=0} \ (j=1,2) \quad\text{on} \quad \{z=-1\}.\\
\end{aligned}
\eeq
\end{defi}
Note that the restriction of time derivatives of the solution at the initial time is defined inductively by using the equations. For instance:
\beqs
(\pt h^{\ep})|_{t=0}=u_0^{\ep}|_{z=0}\cdot (-\nabla_y h_0^{\ep},1)^{t}
\eeqs
\beqs
(\varepsilon\partial_t u^{\varepsilon})|_{t=0}=\frac{1}{g_2(\ep\sigma_0^{\ep})}(-\varepsilon \ud{u}_0^{\ep}\cdot\nabla u_0^{\ep}+\ep \div^{\vp_0^{\ep}} \mathcal{L}^{\vp_0^{\ep}}u_{0}^{\ep}-\nabla^{\vp_0^{\ep}}\sigma_0^{\ep}),
\eeqs
where $$\ud{u}_0^{\ep}=\big(u_{0,1}^{\ep},u_{0,2}^{\ep},(u_0^{\ep}\cdot\bN_0^{\ep}-(\pt\vp^{\ep})|_{t=0})/\p_z\vp_0^{\ep}\big)^{t},\quad \vp_0^{\ep}(\cdot)=\vp^{\ep}(0,\cdot)=z+\eta^{\ep}(0,\cdot)(1+z).$$
We remark that 
$\pt \vp^{\ep}|_{t=0}, \p_z \vp_0^{\ep}$ are determined by $(\pt h)^{\ep}|_{t=0}$ and $h_0^{\ep}$ respectively through \eqref{defvpep} and \eqref{defeta}.

We now define the space for the initial data:
\beq
Y_{m}^{\ep}=\bigg\{(\sigma^{\ep}_0,u_0^{\ep}, h_{0}^\ep)\in H^3(\Omega_0)^4 \times H^{m-\f{1}{2}}(\mathbb{R}^2)\bigg|
\begin{aligned}
& Y_{m}^{\ep}(0)<+\infty, 
 (\sigma_{0}^{\ep},u_0^{\ep},h_0^{\ep})    \text{ satisfy }\\ &\text{compatibility  condition up to order $m$}
\end{aligned}\bigg\},
\eeq
where
\beq\label{initialnorm}
\begin{aligned}
 &Y_{m}^{\ep}(0)=:|h_0^{\ep}|_{H^{m-\f{1}{2}}}+\ep^{\f{1}{2}}|h_0^{\ep}|_{H^{m+\f{1}{2}}} +\ep^{-\f{1}{2}}
 \|\nabla \si^{\ep}(0)\|_{m-5,\infty,\mS}+\|\nabla u^{\ep}(0)\|_{1,\infty,\mS}\\
 &+\ep^{\f{1}{2}}\|(\sigma_0^{\ep},u_0^{\ep})\|_{H^3(\mS)}
 +\ep^{\f{1}{2}}(\|(\sigma^{\ep},u^{\ep})(0)\|_{H_{co}^m(\mS)}
 +\|\nabla(\sigma^{\ep},u^{\ep})(0)\|_{H_{co}^{m-1}})\\
 &+\|(\sigma^{\ep},u^{\ep})(0)\|_{H_{co}^{m-1}}+\|\nabla(\sigma^{\ep},u^{\ep})(0)\|_{H_{co}^{m-2}}+\ep^{\f{1}{2}}\|\p_t(\sigma^{\ep},u^{\ep})(0)\|_{H_{co}^{m-1}(\mS)}+\ep^{\f{1}{2}}\|\pt \omega^{\ep}(0)\|_{H_{co}^{m-4}(\mS)}
.
\end{aligned}
\eeq
In the above, expression, $\omega^{\ep}=\curl^{\vp^{\ep}}u^{\ep}$  stands for  the vorticity. 

 To prove Theorem \ref{thm-localexis}, we introduce the following quantities:
\beq\label{defcN}
\cN_{m,T}^{\ep}=\cE_{m,T}^{\ep}+\cA_{m,T}^{\ep}=\colon \cE^{\ep}_{low,T}+\cE_{high,m,T}^{\ep}+\cA^{\ep}_{m,T}.
\eeq
Here, $\cE^{\ep}_{m,T}$ is composed of the  low order energy norms
$\cE^{\ep}_{low,T}$
and the high order energy norms ${\cE}^{\ep}_{high, m,T}$ :
\beqs
\begin{aligned}
\cE^{\ep}_{low,T}&=\|\ep^{\f{1}{2}}\pt(\sigma^{\ep},u^{\ep})\|_{L_T^{\infty}L^2}
+\ep^{\f{1}{2}}\|(\sigma^{\ep},u^{\ep})\|_{L_t^{\infty}H^3}
+\ep^{\f{3}{2}} \|\nabla^4 u^{\ep}\|_{\ltx},
\end{aligned}
\eeqs
\beq\label{defenergy-high}
\begin{aligned}
\cE_{high,m,T}^{\ep}=&\ep^{\f{1}{2}}|h^{\ep}|_{L_t^{\infty}\tilde{H}^{m+\f{1}{2}}}+\varepsilon^{\f{1}{2}} \|(\si^{\ep},u^{\ep})\|_{L_{T}^{\infty}H_{co}^{m}}\\
&+\ep^{\f{1}{2}} \|\nabla u^{\ep}\|_{L_T^{\infty}H_{co}^{m-1}\cap L_T^{2}H_{co}^m}+\ep^{\f{1}{2}}\|\nabla^2 u^{\ep}\|_{L_T^{\infty}H_{co}^{m-2}\cap L_T^{2}\cH^{m-1}}\\
&+|h^{\ep}|_{L_t^{\infty}\tilde{H}^{m-\f{1}{2}}}+
\varepsilon^{\f{1}{2}}\|\pt(\sigma^{\ep},u^{\ep})\|_{L_T^{\infty}\cH^{m-1}}+\ep^{\f{1}{2}}\|\pt\nabla
u^{\ep}\|_{L_T^2\cH^{m-1}\cap L_T^2H_{co}^{m-2}\cap L_T^{\infty}H_{co}^{m-4}}\\
&+\ep^{-\f{1}{2}}\|(\nabla^{\vp^{\ep}}\sigma^{\ep},\div^{\vp^{\ep}} u^{\ep})\|_{L_T^{\infty}H_{co}^{m-2}\cap L_T^2H_{co}^{m-1}}+\|(\sigma^{\ep},u^{\ep})\|_{L_t^{\infty}H_{co}^{m-1}}
+\|\nabla u^{\ep}\|_{L_T^{\infty}H_{co}^{m-4}\cap L_T^{2}H_{co}^{m-1}}\\
\end{aligned}
\eeq
whereas $\cA_{m,T}^{\ep}$ contains the 
$L_{t,x}^{\infty}$ norms:
\beq\label{defcA-free}
 \begin{aligned}
  \mathcal{A}^{\ep}_{m,T}&=\il
  \nabla u^{\ep}\il_{1,\infty,T}
  +\il(\ep^{\f{1}{2}}\pt (\sigma^{\ep},u^{\ep}),\ep^{-\f{1}{2}}(\nabla ^{\vp^{\ep}}\sigma,\div^{\vp^{\ep}} u)\il_{m-5,\infty,T}+\il(\text{Id},\ep\pt)(\sigma^{\ep},u^{\ep})\il_{m-4,\infty,T}\\
  &\qquad +
  \varepsilon^{\frac{1}{2}}\il\nabla u^{\ep}\il _{m-3,\infty,T}+\varepsilon^{\frac{1}{2}}\il(\sigma^{\ep},u^{\ep})\il_{m-2,\infty,T}+|h^{\ep}|_{m-2,\infty,T}.
 \end{aligned}
\eeq

Our main result is the following:
\begin{thm}[Uniform estimates]\label{thm-localexis}
 Define $0< c_0<\f{1}{2}$ such that 
\beqs
 \sup_{s\in [-3c_1\bar{P}, 3\bar{P}/{c_1}]} |(g_1,g_2)(s)|\in [{c_0},{1}/{c_0}]
\eeqs
where $0<c_1<\f{1}{4}$ is a fixed constant.
Given $m\geq 7$ an integer, suppose that the initial data belongs to $Y_m^{\ep}$, is such that
\beqs
1+h_0^{\ep}(x)\geq 3c_0>0, \quad
\sup_{\ep\in (0,1]}Y_{m}^{\ep}(0)< +\infty,
\eeqs
\beqs
-c_1 \bar{P}\leq \ep\sigma_0^{\ep}(x)\leq {\bar{P}}/{{c_1}},\quad \forall x\in\mS,\quad\forall \ep\in(0,1],
\eeqs
and $\delta_0$ (the parameter appearing in \eqref{defeta}) is chosen such that 
\eqref{diffcdtion} holds for $t=0$ so that
\beqs
\p_z\vp_0^{\ep}\geq 2c_0, %\quad |(\nabla\vp_0^{\ep},\nabla^2\vp_0^{\ep})(x)|\leq \f{1}{2c_0}, 
\quad \forall x\in\mS,\quad\forall \ep\in(0,1].
\eeqs
Moreover,  (taking $c_0$ smaller if necessary), we can also assume that 
\beqs
|(\nabla\vp_0^{\ep},\nabla^2\vp_0^{\ep})(x)|\leq \f{1}{2c_0}, \quad \forall x\in\mS,\quad\forall \ep\in(0,1].
\eeqs
Then there exist $T_0>0,$ $ \ep_0 \in (0, 1],$ such that for any $\ep \in (0,  \ep_0],$
the system \eqref{FCNS2}-\eqref{bebdry2} has a unique solution
which satisfies:
$\cN_{m,T_0}^{\ep}(\sigma^{\ep},u^{\ep})<+\infty.$
In particular, we have the uniform estimate
\beqs
\begin{aligned}
&\sup_{0<\ep\leq \ep_0} \big(\| (\sigma^{\ep},u^{\ep})\|_{L_{T_0}^2H_{co}^{m}(\mS)\cap L_{T_0}^{\infty}H_{co}^{m-1}(\mS)}%+\|(\sigma^{\ep},u^{\ep})\|_{L_{T_0}^{\infty}H_{co}^{m-2}(\mS)}
+\ep^{-\f{1}{2}}\|(\div^{\vp^{\ep}} u^{\ep},\nabla\sigma^{\ep})\|_{L_{T_0}^{\infty}H_{co}^{m-2}(\mS)\cap L_{T_0}^{2}H_{co}^{m-1}}\\
&\qquad\qquad\qquad\qquad\qquad+\il\nabla u^{\ep}\il_{1,\infty,T_0}+\ep^{-\f{1}{2}}\il(\nabla\sigma^{\ep}, \div^{\vp^{\ep}} u^{\ep})\il_{m-5,\infty,T_0}\big)<+\infty.\end{aligned}
\eeqs
Moreover, the following properties hold: for any $(t,x)\in [0,T_0]\times\mS, \ep\in(0,\ep_0]$,
\beq\label{asstobeshown}
\p_z\vp^{\ep}(t,x)\geq c_0, \quad 
|(\nabla\vp^{\ep},\nabla^2\vp^{\ep})(t,x)|\leq {1}/{c_0}, \quad
-2c_1 \bar{P}\leq \ep\sigma^{\ep}(t,x)\leq {2\bar{P}}/{{c_1}}.
\eeq
\end{thm}
\begin{rmk}\label{rmk-data}
In view of the definition of $Y_{m}^{\ep},$ we have assumed that %the restriction of 
the first time derivative of the solution is of size of order $\ep^{-\f{1}{2}},$ which is better
than the usual  well-prepared data case (where $\pt(\sigma^{\ep},u^{\ep})|_{t=0}$ is assumed to be order $1$). 
This assumption is crucial in our analysis to control the  regularity on the surface. We shall give more details in Subsection \ref{remarkondata}.
 Note that our assumption is thus  much weaker than the one in \cite{MR3812074,MR3887218,MR4097326} for the inviscid system
 where two time derivatives are assumed to be bounded initially.
\end{rmk}
%\begin{rmk}
%Compared to the recent works \cite{MR3812074,MR3887218,MR4097326}
%on the free surface Euler system, where the uniform estimates are established by assuming the acoustic waves to be order $\cO(\ep^2)$ in some suitable spaces, 
%we can allow here more general initial data (the acoustic waves to be order $\cO({\ep^{\f{1}{2}}})$).
%This is essentially due to the presence of the dissipation term
%which is helpful in some sense to the control of the regularity of the surface. 
%%which can provide us some smoothing effects and help us to control the regulairity of the surface.
%\end{rmk}

\begin{rmk}
It is also possible to 
 prove the uniform estimates by imposing an alternative assumption on the size of the acoustic waves, we can assume  them to be of order $\ep$ in a low regularity $H_{co}^1$ space and of  order $1$ in a higher regularity $H_{co}^m$ norm.
\end{rmk}
\begin{rmk}
In view of the definition \eqref{defenergy-high}, one has three kinds of bounds for the solution. The first two lines of
\eqref{defenergy-high} only imply  that the highest order norm with pointwise estimates in time  $L_t^{\infty}H_{co}^{m}$ of $(\sigma^{\ep},u^{\ep})$ can be unbounded and has a size $\cO(\ep^{-\f{1}{2}}).$
Nevertheless, in the  two last  lines of \eqref{defenergy-high}, we are able to get that
the  $L_t^2$ type norm with maximal number of derivatives,  $L_t^2H_{co}^m$ of $(\sigma^{\ep},u^{\ep})$ 
 and  the  $L_t^{\infty}H_{co}^{m-1}$ norm (so with one less derivative) are uniformly bounded.
%consist in the uniform estimates for $(\sigma^{\ep},u^{\ep})$ to the second highest order in $L_t^{\infty}$ type norm $L_t^{\infty}H_{co}^{m-1}$ and to the highest order in $L_t^2$ type norm $L_t^2H_{co}^m$.} %while the last two lines of \eqref{defenergy-high} 
Moreover, the first term in the fourth line of \eqref{defenergy-high} shows that the compressible part of the remains of   size  $\cO(\ep^{\f{1}{2}})$ in  $L_T^{\infty}H_{co}^{m-2}\cap L_T^2H_{co}^{m-1}$.
\end{rmk}
\begin{thm}[Convergence]\label{thm-convergence}
Assuming that $(u_0^{\ep},h_0^{\ep})$ tends to $(u_0^0,h_0^0)$ in $H^1(\mS)\times 
L^2(\mR^2)$ and the assumptions made in Theorem \ref{thm-localexis} hold. Let 
$(\sigma^{\ep},u^{\ep},h^{\ep})$ be
the solution to \eqref{FCNS2}-\eqref{bebdry2}.
Then 
$(P(\bar{\rho})+\ep\sigma^{\ep},u^{\ep},h^{\ep})$  converge in $C^{\gamma}([0,T_{0}]\times \mS)\times C([0,T_{0}],L_{loc}^2(\mS)) \times C([0,T_{0}], H_{loc}^{s}(\mathbb{R}^2))$
to $(P(\bar{\rho}),u^{0},h^0)$ where $0\leq \gamma<\f{1}{2}$
 and $0<s<m-{1}/{2}.$
 Moreover, $u^0$ has the additional regularity:
\beq\label{addi-regu}
u^{0}\in C\big([0,T_0],\cH^{0,m-2}\big),\quad
\nabla u^0\in 
L^2([0,T_0],\cH^{0,m-1})\cap L^{\infty}([0,T_0]\times\mS)
\eeq
and one can find $\pi^0 \in L^2([0,T_0],\cH^{0,m-1})$ such that $(u^0,\pi^0, h^0)$ solves uniquely the following incompressible free-surface Navier-Stokes system:
\beq \label{FINS1}
 \left\{
\begin{array}{l}
 \displaystyle\bar{\rho}(\pt^{\vp^0} u^0+ u^{0}\cdot\nabla^{\vp^0} u^{0})-\div^{\vp^0} S^{\vp^0} u^{0}+
\nabla^{\vp^0}\pi^0=0,  \\
 \displaystyle \div^{\vp^0} u^0=0, \qquad\qquad\qquad\qquad\qquad \text{$(t,x)\in [0,T_0]\times \mS,$} \\
  \displaystyle u^0|_{t=0}=u_0^0, h^0|_{t=0}=h_0^0\\
\end{array}
\right.\\
\eeq
with boundary conditions:
\begin{align}
\displaystyle \p_t h^0+u^0(t,y,0)\cdot {\bN}^0=0,\\
 {S}^{\vp^0}u^0 {\bN}^0=\pi^0 {\bN}^0 \quad \text{ on } \{z=0\},\\
 \displaystyle u_3^{0}=0, \quad \f{\mu}{\p_z\vp_0}\p_z u_{j}^{0}=a u_{j}^{0} \quad (j=1,2) \quad \text{ on }  \{z=-1\}.\label{FCNS1-bc}
\end{align}
Here $\vp^0$ is defined in 
\eqref{defvpep} (replacing $h^{\ep}$ by $h^0$), $\bN^0=(-\p_1h^0,-\p_2h^0,1)^{t}$.
\end{thm}
\begin{rmk}
Due to the absence of  estimate for the second order normal derivatives of the velocity $u^0$ (and thus for the strong trace of the normal derivative), the solution to \eqref{FINS1}-\eqref{FCNS1-bc}, must be interpreted in the following sense: $\div^{\vp^0}u^0=0$ holds
in $L^2([0,T_0]\times\mS)$ and
for any vector field $\psi =(\psi_1,\psi_2,\psi_3)^t\in \big[C_{c}^{\infty}\big(\overline{Q_{T_0}}\big)\big]^3$ with $\psi_3|_{z=-1}=0,$
 the following identity holds:  for any $0<t\leq T_0,$ 
\beq\label{def-weak}
\begin{aligned}
&\bar{\rho}\int_{\mS} u^0\cdot \psi(t,\cdot) \,\d\cV_t^0 + 2\mu \int_0^t\int_{\mS} S^{\vp^0}u^0 \cdot \nabla^{\vp^0}\psi\,\d\cV_s^0\d s + \bar{\rho}\int_0^t\int_{\mS}(u^0\cdot\nabla^{\vp^0} u^0)\cdot \psi\,\d\cV_s^0\d s\\
&=\bar{\rho}\int_{\mS} u^0\cdot \psi(0,\cdot) \,\d\cV_0^0+\bar{\rho}\int_0^t\int_{\mS}u^0\cdot\pt^{\vp^0}\psi\,\d\cV_s^0\d s+\int_0^t\int_{\mS}{\pi}^{0}\div^{\vp^{0}}\psi\,\d\cV_s^{0}\d s\\
&\quad+\int_0^t\int_{z=0} (u^0\cdot \bN^0) u^0\cdot\psi \,\d y\d s+ a
\int_0^t\int_{z=-1}(u_1^0\cdot\psi_1+u_2^0\cdot\psi_2)\,\d y\d s
\end{aligned}
\eeq
where $\d\cV_t^0=\f{1}{\p_z \vp^0}(t,\cdot)\,\d y \d z.$
\end{rmk}
\begin{rmk}
Note that we do not end up in  the classical space of existence and uniqueness for  the free boundary incompressible Navier-Stokes system, nevertheless, the uniqueness of the solution in our functional spaces can be obtained 
by taking benefits of the control of the Lipschitz norm of the solution. One can refer to subsection \ref{subsection-uniqueness} for the proof.
 \end{rmk}

\subsection{Main difficulties, general strategies.}
Due to the simultaneous presence of the singular term in the equation as well as the viscous term and boundaries, we are confronted with 
both difficulties resulting from boundary layer effects and fast time oscillations. These two phenomena are well understood when they occur separately, but some new difficulties occur when they occur at the same time. Indeed,  on the one hand, regarding the vanishing viscosity limit problem (see for instance \cite{MR3590375,MR3741102}),
one can estimate   the high order tangential derivatives by direct energy estimates, and then use the vorticity to control the normal derivatives. 
Nevertheless, for the system with  low Mach number,  the tangential derivatives ($\p_y$)  are not easy to control uniformly, since they do not commute with $\nabla^{\vp^{\ep}},\div^{\vp^{\ep}}$ and thus create singular
commutators.
 Without the  a priori knowledge of the tangential derivatives, the estimate of  the vorticity cannot be performed. 
On the other hand, for the compressible free boundary Euler system with a low Mach number, uniform estimates are established
for example  in 
\cite{MR4097326, MR3812074, MR3887218}. %If we understand correctly, 
Besides the difficulties arising from the Taylor sign condition and the regularity of the surface, 
the idea behind getting uniform estimates is to control first weighted time derivatives $(\ep \partial_{t})^k$ and then to recover space derivatives 
by using the equations and by direct energy estimates for the vorticity. %which solves a transport equation.
Here, in the case of viscous fluids, the vorticity is not easy to estimate due to the lack of information on its trace  on the boundaries. We shall explain more precisely in the following. For the sake of notational convenience, we will drop the $\ep-$dependence of the solution.

Indeed, the vorticity $\omega=\curl^\varphi u$ solves a  transport-diffusion equation
with  Dirichlet boundary condition (see \eqref{norpz}, \eqref{tanpz}) under the form
\beq\label{omegabdry}
\omega|_{\p\mS}\approx \p_y u+\div^{\vp}u|_{\p \mS}.
\eeq
Let us consider the simplest case, the heat equation with zero source  and initial data but with nonhomogeneous Dirichlet condition in a half space :
\beq\label{heateq}
\bar{\rho}\p_t f-\mu \Delta f=0, \quad f|_{t=0}=0,\quad f|_{z=0}=f^{b,1}, \quad (t,x)\in [0,T]\times\mathbb{R}_{-}^3,
\eeq
%This equation admits the explicit formulae:
%which can be solved explicitly:
%\beq\label{heatfor}
%f(t,x)=2\tilde{\mu}\int_0^t \int_{\mathbb{R}^2} \p_z E_3(t-s,y-y',z)f^{b,1}(s,y')\,\d s\d y
%\eeq
%where $$E_3(t,y,z)=\f{1}{4\pi\tilde{\mu} (t-s)^{{3}/{2}}}e^{-\f{|y|^2+|z|^2}{4\tilde{\mu} t}},\qquad (\tilde{\mu}=\mu/{\bar{\rho}})$$ is the fundamental solution of the three dimensional heat equation.
%By Young's inequality, one has that
By using the heat kernel, we obtain
\beqs
\|f\|_{L_t^2H_{co}^{m-1}}\lesssim T^{\f{1}{4}}|f^{b,1}|_{L_t^2\tilde{H}^{m-1}}.
\eeqs
By applying this estimate to $\omega,$ we see that the boundary contribution when estimating $\|\omega\|_{L_t^2H_{co}^{m-\f{1}{2}}}$
is more or less $|(\p_y u,\div u)|_{L_t^2\tilde{H}^{m-1}},$ which requires the foreknowledge of the tangential derivatives and
which indicates the loss of half derivative. One could also use the (tangential) smoothing effects of the heat equation to overcome this loss of derivative. Nevertheless, in this way, it seems impossible to extract the extra $\ep$ or $T$ which are essential to close the estimate. More precisely, by using maximal regularity, one gets that 
\beqs
\begin{aligned}
\|\omega\|_{L_t^2H_{co}^{m-1}}&\leq C |(\p_y u,\div^{\vp} u)|_{L_t^2\tilde{H}^{m-\f{3}{2}}}+ \text{other terms}\\
&\leq C(\|\nabla u\|_{L_t^2H_{co}^{m-1}}+\|\nabla\div^{\vp} u\|_{L_t^2H_{co}^{m-2}})+\text{other terms}
\end{aligned}
\eeqs
which does not gain anything. Note that the constant $C$ is independent of $T$ and $\ep.$ 
%We will encounter a similar problem when trying to control  $\|\omega\|_{L_t^{\infty}H_{co}^{m-3}}.$

%Indeed, by computations, one has roughly $\omega|_{\p \mS}\approx (\p_y u+\div u)|_{\p \mS},$ which indicates the loss of one derivative when dealing with the boundary terms.
%The same problem also occurs in the inviscid limit problem. In \cite{MR3590375}, the first two authors overcome this difficulty in the following way: 
 %and then reduce the control of normal derivative to that of $\div u$ and $\omega\times n.$ More precisely, one uses the following (rough) inequality:
%\beq\|\nabla f(t)\|_{H_{co}^k}\lesssim \|\omega\times n(t)\|_{H_{co}^k}+\|\div u(t)\|_{H_{co}^k}+\|\p_y u(t)\||_{H_{co}^k}.\eeq
%Now $\omega\times n$ possess better boundary conditions than that of $\omega$ in the sense that $\omega\times n|_{\p\mS}\approx \p_y u.$ 
%Another difficulty lies in the control of high order tangential derivatives, namely $\|\p_y^{m-1}u\|_{L_t^{\infty}L^2}.$Indeed, as $\p_y$ does not commute with $\nabla^{\vp},\div^{\vp},$ taking $\p_y$ derivatives on the equation would lead to the oscillating commutators. Therefore, it seems impossible to obtain the a-priori estimate for tangential derivative by direct energy estimates. 

To overcome these problems, we %adopt the similar strategy proposed in \cite{lowmachbounded} where the similar problem is considered in a fixed domain. More precisely, we 
split the velocity $u$ into a compressible part $\nabla^{\vp}\Psi$ and an incompressible part $v$ %where $\nabla^{\vp}\Psi, v$ are the compressible and incompressible part of the velocity 
(see definition \eqref{defofQ}, \eqref{defofP}). On the one hand, the compressible part is governed by the elliptic equation $\Delta^{\vp} \Psi=\div^{\vp}u$ with mixed boundary conditions (with homogeneous Dirichlet boundary condition on the upper boundary and homogeneous Neumann boundary condition on the bottom).  Hence the estimate for its gradient $\nabla^2\Psi$ can be deduced from the estimate of $\div^{\vp}u.$ 
We then use induction arguments and the equations to establish  high-order estimates of $\div^{\vp} u.$ On the other hand, the incompressible part $v,$ solves, up to the control of non-local commutators, a transport-diffusion equation and hence one can use direct energy estimates to get some suitable estimates (say $\|\p_y^{m-1}v\|_{L_t^{\infty}L^2}$ and $\|\nabla v\|_{L_t^2H_{co}^{m-1}}$), which together with the estimates on $\div^{\vp}u,$ lead to the uniform control of $\|\p_y^{m-1} u\|_{L_t^{\infty}L^2(\mS)}$ and 
$\|\nabla u\|_{L_t^2H_{co}^{m-1}}.$
The final task is to 
estimate $\|\nabla v\|_{L_t^{\infty}H_{co}^{m-4}}$ which 
stems from a careful study on $\omega\times \bn.$ We remark that this strategy has been employed 
by the authors in \cite{masmoudi2021uniform} where  uniform in low Mach number estimates
are established in the case of smooth  fixed bounded domain with Navier boundary condition and ill-prepared initial data. However, %in the free surface setting, 
as will be explained in next subsection, 
there are various  extra difficulties for the free boundary problem arising from the control of the regularity of the surface.
\subsection{Remarks on the slightly well-prepared data assumption.}\label{remarkondata}
%As mentioned before, compared to the fixed domain case, there would be some extra difficulties for the free boundary problem, arising from the regularity of the surface. 
In the free surface setting, a very sensitive part of the analysis is  the control of the regularity of the surface. 
This is the reason why we have to allow the initial data to be slightly well-prepared. Indeed, since the incompressible part $v^{\ep}$ satisfies the boundary condition (see \eqref{eqofv}, \eqref{defpi-intro})
\beqs
(2\mu S^{\vp}v-\pi \text{Id})\bN|_{z=0}=2\mu(\div^{\vp}u\text{Id}-\nabla^{\vp}\nabla^{\vp}\Psi)\bN |_{z=0},
\eeqs
 in order to perform energy estimates for $v$ at order $m-1,$ it requires  information on $\|\nabla^3\Psi\|_{L_t^2H_{co}^{m-3}},$ 
which, by elliptic estimates, can be controlled by $\|\nabla\div^{\vp}u\|_{L_t^2H_{co}^{m-2}}$
and $|h|_{L_t^2\tilde{H}^{m+\f{1}{2}}}.$
Nevertheless, due to the fast oscillations, we cannot expect  $|h|_{L_t^2\tilde{H}^{m+\f{1}{2}}}$(or alternatively $\|\nabla u\|_{\hco^{m}}$) to be uniformly bounded. A similar problem occurs when one recovers the $L_t^2H_{co}^{m-1}$ norm of $\nabla^2\Psi$ from the one of $\div^{\vp}u$ by elliptic estimates. To overcome this problem, we assume the data to be slightly well-prepared so that $\|\div^{\vp} u\|_{L_t^{\infty}H_{co}^1}$ can be proved to be of order $\ep^{\vartheta}, $ ($0<\vartheta<1$ to be chosen). This can make an extra $\ep^{\vartheta}$ appear in front of $|h|_{L_t^2\tilde{H}^{m+\f{1}{2}}}$ in the process of the elliptic estimates (one can  refer to Step 3 of the following subsection for more details).
In turn, to control uniformly the term $\ep^{\vartheta}|h|_{L_t^2\tilde{H}^{m+\f{1}{2}}},$ which reduces to the estimate of $\ep^{\vartheta}\|\nabla u\|_{L_t^2H_{co}^{m}},$ we must assume that the compressible part 
$(\div^{\vp}u,\nabla\sigma)$ has the size of  
$\cO(\ep^{1-\vartheta})$ in $L_t^2H_{co}^{m-1}.$
Indeed, when performing the highest-order energy estimates, we need to be careful with the singular term 
\beq\label{singular-showcase}
\ep^{2\vartheta-1}\int_0^t\int_{\mS} Z^{\alpha}\sigma \underbrace{[Z^{\alpha},\div^{\vp}]u}_{[Z^{\alpha},\f{\bN}{\p_z\vp}\cdot\p_z] u}
+Z^{\alpha}u\cdot\underbrace{[Z^{\alpha},\nabla^{\vp}]\sigma}_{[Z^{\alpha},\f{\bN}{\p_z\vp}\p_z] \sigma}\ \d\cV_s \d s, \quad\,|\alpha|=m.
\eeq
By direct computations, these terms can be bounded by (up to other good terms and upon the foreknowledge of $|\ep^{\vartheta}h|_{L_t^2\tilde{H}^{m+\f{1}{2}}}$ )
$$\ep^{\vartheta-1}|\ep^{\vartheta}h|_{L_t^2\tilde{H}^{m+\f{1}{2}}}\big(\|Z\sigma\|_{L_t^2H_{co}^{m-1}}\il\nabla u\il_{0,\infty,t}+\|u\|_{\hco^m}\il\nabla\sigma\il_{0,\infty,t}\big)\Lambda\big(\f{1}{c_0},|h|_{m-2,\infty,t}\big),$$
which can be uniformly bounded if 
$$\|Z\sigma\|_{L_t^2H_{co}^{m-1}}=\cO(\ep^{1-\vartheta}),\quad  \il\nabla\sigma\il_{0,\infty,t}=\cO(\ep^{1-\vartheta}).$$
By optimizing,  $\vartheta=1-\vartheta$, we shall thus prove 
 the uniform estimates by assuming that $(\nabla\sigma, \div^{\vp} u)|_{t=0}=\cO(\ep^{\f{1}{2}})$. 
 By using the same ideas, it would be also possible  to establish  uniform estimates by assuming that  the compressible part is  of size at $\cO(\ep^\vartheta)$ 
($\f{1}{2}<\vartheta\leq 1$) in a low regularity space (say $H_{co}^{1}$) and 
 $\cO(\ep^{1-\vartheta})$ in a higher regularity space (say $H_{co}^{m-1}$).

One may wonder whether the introduction of the Alinhac good unknown which is used frequently in  free boundary problems can help us to avoid to lose  derivatives on the surface and to get uniform estimates without any size assumption on the data. However, this quantity does not seem useful here. Indeed, the use of the Alinhac good unknown would require the validity of the Taylor sign condition ($\p_{\bn}^{\vp}\sigma|_{z=0}>0$), which seems out of reach for ill-prepared data
since $\sigma$ solves a transport equation with a source term of size of $\cO(\ep^{-1}).$

%We will make it clear in the next subsection. 
%The extra difficulties due to the presence of the free surface will be explained during the sketch of the proof in the next subsection.
%{\color{red} Insist on there will be extra difficulty arsing from the surface and there will be some difficulty from the norms we choose}
\subsection{Sketch of the proof}\label{secsketchproof}
Let us explain the main steps for the proof of Theorem \ref{thm-localexis}. The uniform energy estimates will  be  established in the following steps:

\textbf{Step 1: $\ep-$dependent high-order energy estimates and $\ep-$independent high-order time derivative estimates.}

In this step, we aim to obtain two kinds of energy estimates. The first one is
the estimate of $\ep^{\f{1}{2}}\|(\sigma,u)\|_{L_t^{\infty}H_{co}^m}$
and $\|\ep^{\f{1}{2}}\pt(\sigma,u)\|_{L_t^{\infty}\mathcal{H}^{m-1}}.$ Since the spatial conormal vector fields $Z_1,Z_2,Z_3$ do not commute with $\nabla^{\vp}$ and $\div^{\vp},$ it seems hard to get the uniform estimate of $\|(\sigma,u)\|_{L_t^{\infty}H_{co}^{m}}$ by direct energy estimates. Nevertheless, it is easy to get an $\ep-$dependent estimate involving the control of  
$\|\nabla(\sigma,u)\|_{L_t^2H_{co}^{m-1}}.$
This can be done by applying $ Z^{\alpha} (|\alpha|\leq m)$ to the system \eqref{FCNS2} and then by 
performing standard energy estimates making use of the symmetric structure. We remark that at this stage we do not lose regularity on the surface. Indeed, 
besides the term listed in \eqref{singular-showcase} (setting $\vartheta=\f{1}{2}$),
the possible most problematic commutator term is 
\beqs
\ep \int_0^t\int_{\mS}Z^{\alpha}\textbf{N}\cdot\p_z\mathcal{L}^{\vp}u Z^{\alpha}u \ \d \cV_s\d s, \ \d\cV_s=\f{1}{\p_z\vp} \ \d y\d z
\eeqs
which can be bounded by:
$\ep^{\f{1}{2}}|h|_{L_t^2\tilde{H}^{m+\f{1}{2}}}\|u\|_{L_t^2H_{co}^m}\il\ep^{\f{1}{2}} \p_z\cL^{\vp}u\il_{\infty,t}.$ Note that the estimate of $\ep^{\f{1}{2}}|h|_{L_t^2\tilde{H}^{m+\f{1}{2}}}$ 
is available
owing to the control of 
$\ep^{\f{1}{2}}\|\nabla u\|_{L_t^2H_{co}^m}$ and $\il\ep^{\f{1}{2}} \p_z\cL^{\vp}u\il_{\infty,t}$
 by the terms appearing in $\cA_{m,t},$ using the equation of the velocity.
%Therefore, at our disposal, it is not necessary to use the so-called 'Alihnac good unknowns' \cite{MR976971}, which is frequently used in many free boundary problems \cite{MR3260858,MR2646814,MR2138139,MR3590375,MR3419883} to avoid the loss of regularity on the surface. Since it is only a small improvement by using the good unknown (we can hope to get $\ep^{\f{1}{2}}\|(\sigma,u)\|_{L_t^{\infty}H_{co}^{m}}$ by using it), we decide not to use it as the computations would be much involved.

The estimate of
$\|\ep^{\f{1}{2}}\pt(\sigma,u)\|_{L_t^{\infty}\mathcal{H}^{m-1}}$ can also be derived by straightforward energy estimates. 
The main observation is that:
although the weighted time derivatives $\ep^{\f{1}{2}}(\ep\pt)^k\pt$  
do not commute with $\nabla^{\vp},$ their commutator can be uniformly controlled even for the singular term. Indeed, direct computation shows that for $k\leq m-1,$
\beqs
\ep^{\f{1}{2}}\f{[(\ep\p_t)^{k}\pt,\div^{\vp}]u}{\ep}=\ep^{k-\f{1}{2}}\big[\p_t^{k+1},\f{\textbf{N}}{\p_z\vp}\big]\cdot\p_z u
\eeqs
whose $L_{t}^2L^2(\mS)$ norm is uniformly controlled as long as $k\geq 1$ thanks to the boundedness of $|\ep^{\f{1}{2}}\pt^2 h|_{L_t^2\tilde{H}^{m-\f{3}{2}}}$
(see \eqref{surface3}).
%$\|\p_t\textbf{N}\|_{L_t^2\cH^{m-1}}\il\p_z f\il_{\infty,t}+\il\p_t \textbf{N} \il_{\infty,t}\|\p_z f\|_{L_t^2\cH^{k-1}}.$
We remark that in view of the definition \eqref{defeta},
 the boundedness of $ \textbf{N}$ can be derived from that of $h.$  The case $k=0$ needs to be treated differently and is explained in the next step.

The second kind of estimate is for the terms $\ep^{\f{1}{2}} \|(\nabla^{\vp}\sigma,\div^{\vp} u)\|_{L_t^{\infty}H_{co}^{m-1}},$ $\ep^{\f{1}{2}} \|\nabla^{\vp}\div^{\vp} u\|_{\hco^{m-1}},$
which follows again from direct energy estimates, we thus do not detail more here.
%Their remaining proof is a bit involved than the former one,  we do not detail here since the difficulty comes mainly from the computational complexity.

\textbf{Step 2. Uniform lower order energy estimates.} 
In this step, we aim to show the boundedness of 
 $\|\ep^{\f{1}{2}}\pt(\sigma,u)\|_{L_t^{\infty}L^2}.$ We remark that a naive energy estimate fails due to bad commutators with the singular term. Actually, the $L_{t}^2L^2(\mS)$ norm of the term $\ep^{-\f{1}{2}}[\pt,\div^{\vp}]u=\ep^{-\f{1}{2}}\p_t(\bN/\p_z\vp)\cdot\p_z u$ is out of control. The trick to avoid this problem is to multiply
$\p_t^{\vp}\eqref{FCNS2}_1$ by $\ep\p_t\sigma$ and multiply $\p_t\eqref{FCNS2}_2$ by $\ep\p_t^{\vp} u.$ In this way, the singular term can be dealt with as:
\beq\label{trick1}
\begin{aligned}
&\quad%\ep^{-\f{1}{2}}
\int_0^t \int_{\mS}\pt\nabla^{\vp}\sigma\p_t^{\vp}u+\pt^{\vp}\div^{\vp}u\pt\sigma \ \d\cV_s\d s\\
&=\int_0^t\int_{z=0}\pt^{\vp}u\cdot \bN \pt\sigma \ \d y\d s+\int_0^t \int_{\mS}\p_t^{\vp}u [\pt,\nabla^{\vp}]\sigma \ \d \cV_s\d s,\\
\end{aligned}
\eeq
where $\d \cV_s=\p_z\vp \ \d y\d z.$
The first boundary term  combined with another boundary term which comes from the integration by parts of the viscous term, result in a good term that can be controlled. Namely
\beqs
\ep\int_0^t \int_{z=0}\pt\big[-\cL^{\vp} u+\f{\sigma}{\ep}\text{Id}\big]\bN \cdot\pt^{\vp}u \ \d y \d s=-\ep\int_0^t \int_{z=0}(-\cL^{\vp} u+\f{\sigma}{\ep}\text{Id})\pt \bN \cdot\pt^{\vp}u \ \d y \d s.
\eeqs
Note that the trace of $\f{\sigma}{\ep}$ 
on the upper boundary 
can be expressed as the spatial tangential derivatives of the velocity (see \eqref{sigmabdry}) %$\eqref{sigmaelliptic}_2$, 
which can be easily treated by the trace inequality. 
The second term in \eqref{trick1}
is also manageable since $\ep^{-\f{1}{2}}\|[\pt,\nabla^{\vp}]\sigma\|_{L_t^2L^2(\mS)}$ can be roughly bounded by $\|\ep^{-\f{1}{2}}\nabla^{\vp}\sigma\|_{L_t^2L^2(\mS)}.$

  It should be mentioned that the above strategy does not apply for 
  the control of $\ep^{\f{1}{2}}\|(\p_y,Z_3)\pt(\sigma,u)\|_{L_t^{\infty}L^2}$ due to the bad commutator terms. We thus use the strategy of the splitting 
  mentioned before to deal with them in the following steps.

\textbf{Step 3. Recovering high order spatial derivatives of $(\nabla \sigma,\nabla\nabla^{\vp}\Psi)$ by induction.}
Denote by $\nabla^{\vp}\Psi$  
the compressible part of the velocity which is defined by the unique solution to the elliptic equation with mixed boundary conditions:
%where time-dependent projectors $\mathbb{P}_t,\mathbb{Q}_t$ are  defined in \eqref{defofP},\eqref{defofQ}. 
 \beq\label{def-compressible}
 \left\{
\begin{array}{l}
 -\div^{\vp}\nabla^{\vp}\Psi=-\div^{\vp}u,\\
 \Psi|_{z=0}=0,\\
 \p_{\bn}\Psi|_{z=-1}=0.
 \end{array}
 \right.
 \eeq
In this step, we aim to control the $L_t^2H_{co}^{m-1}$ norm of %the high order norm of $\nabla^{\vp}(\sigma,\nabla^{\vp} \Psi):$
$\nabla^{\vp}(\sigma,\nabla^{\vp} \Psi),$ which can be reduced to the control of $\ep^{-\f{1}{2}}\|(\nabla^{\vp}\sigma,\div^{\vp}u)\|_{L_t^2H_{co}^{m-1}}.$ We will use the equation and induction arguments to recover the latter. %high order spatial derivatives of $(\nabla^{\vp}\sigma,\div^{\vp}u).$
Indeed, let us rewrite the system 
\eqref{FCNS2} as follows:
\beq\label{rewrite-comp}
\left\{
\begin{array}{l}
   - \div^{\vp}u=
   g_1 \ep\pt\sigma +\ep g_1\underline{u}\cdot \nabla\sigma,   \\[6pt]
    -\mu \ep \curl^{\vp}\omega- \nabla^{\vp}\big(\sigma-(2\mu+\lambda)\ep\div^{\vp}u\big)=g_2\ep\p_t u+\ep g_2 \underline{u}\cdot\nabla u.
\end{array}
\right.
\eeq
where $$\underline{u}=(u_1,u_2,u_z)=:(u_1,u_2,\f{u\cdot \bN-\pt\vp}{\p_z\vp}).$$ In view of \eqref{rewrite-comp}, one wants to show that for $j+l\leq m-1,$
 \beq\label{inductionpre1}
 \ep^{-\f{1}{2}}\|\div^{\vp}u\|_{L_t^2\cH^{j,l}}\lesssim \|\ep^{\f{1}{2}}\pt\sigma\|_{L_t^2\cH^{j,l}}+\mathcal{O}(\ep^{\f{1}{2}})\lesssim  \ep^{-\f{1}{2}}\|\nabla^{\vp}\sigma\|_{L_t^2\cH^{j+1,l-1}}+\mathcal{O}(1),%(\ep^{\f{1}{2}}),
 \eeq
\beq\label{inductionpre2}
\begin{aligned}
\ep^{-\f{1}{2}}\|\nabla^{\vp}\sigma\|_{L_t^2\cH^{j,l}}&\lesssim\|\ep^{\f{1}{2}}\pt \nabla^{\vp}\Psi\|_{L_t^2\cH^{j+1,l}}+\mathcal{X}_{m,t}+\mathcal{O}(\ep^{\f{1}{2}})\\
&\lesssim \ep^{-\f{1}{2}}
\|\div^{\vp}u\|_{L_t^2\cH^{j+1,l-1}}%+\ep^{\f{1}{2}}|h|_{L_t^2\tilde{H}^{m+\f{1}{2}}}
+\mathcal{X}_{m,t}+\mathcal{O}(1),%\big((T+\ep)^{\f{1}{2}}\big),
 \end{aligned}
\eeq
where 
$$\mathcal{X}_{m,t}\approx \ep^{\f{1}{2}}\|\nabla^{\vp}\div^{\vp}u\|_{L_t^2H_{co}^{m-1}}+\ep^{\f{1}{2}}\|\nabla^{\vp}u\|_{L_t^2H_{co}^m}$$%\Lambda\big(\f{1}{c_0}, |h|_{m-2,\infty,t}\big)$$ 
which has been controlled in the first step. These two inequalities in hand, we can conclude by induction arguments.
Note that the inequality \eqref{inductionpre1} results from the equality $\eqref{rewrite-comp}_1$ and the product estimate \eqref{crudepro}. %established in . 
To obtain \eqref{inductionpre2}, we take $\div^{\vp}$ of the equation $\eqref{rewrite-comp}_2$ and use the boundary condition \eqref{upbdry2}
to get the following elliptic equation:
\beq
\left\{
\begin{array}{l}
 \Delta^{\vp}(\ep\theta)=\div^{\vp}\big[\bar{\rho}\ep\p_t\nabla^{\vp}\Psi+\ep(\f{g_2-\bar{\rho}}{\ep}\ep\pt + g_2 \underline{u}\cdot\nabla)u\big]=\colon \div^{\vp}\tilde{G}\\[5pt]
 \ep\theta|_{z=0}=-2\ep\mu(\p_1u_1+\p_2u_2)+\ep(\omega\times \bN)_3|_{z=0}\\[5pt]
 \p_{\bn}\theta|_{z=-1}=\tilde{G}\cdot \bn+\mu\ep\curl^{\vp}\omega\times \bn|_{z=-1}.
\end{array}
\right.
\eeq
where $\theta=\sigma/\ep-(2\mu+\lambda)\div^{\vp}u.$ Inequality \eqref{inductionpre2} is thus the consequence of the elliptic estimates in the conormal setting (see Section 5). We remark that the trace of $\omega\times\bN$ involves only tangential derivatives of the velocity  on the boundary  (see \eqref{omegatimesn}).
 
 Now that $\div^{\vp} u$ has been bounded, we can control the compressible part of the velocity
 $\nabla^{2}\Psi$ by again elliptic estimates. Nevertheless, there will be a loss of one derivative on the surface if no smallness condition is made on the compressible part.
 Indeed, as $\nabla^{\vp}\Psi$ solves equation \eqref{def-compressible}, we have by the elliptic estimates that
\beq
\|\nabla^{2}\Psi\|_{L_t^{2}\cH^{0,m-1}}\lesssim (| h|_{L_t^2H^{m+\f{1}{2}}}+\|\div^{\vp}u\|_{L_t^2\cH^{0,m-1}})\lca%mbda(\|\div^{\vp} u\|_{L_t^{\infty}H_{tan}^1}+|h|_{2,\infty,t}),
\eeq
where $\Lambda$ denotes a polynomial.
This estimate involves more regularity of the surface than that we can afford  since we have only the control of $|h|_{L_t^2H^{m-\f{1}{2}}}$.
Nevertheless, %we are aware of that{\color{red} I don't get the sentence} 
  %performing the variational arguments, proving the 
   checking the proof of the elliptic estimates for $\nabla^2\Psi,$ we find that
the main problematic term is indeed $\nabla\Psi Z^{\alpha}\nabla\bN$ $ (|\alpha|=m-1,\alpha_0=0),$ whose $L_t^2L^2(\mS)$ norm can be bounded by $$\|\nabla\Psi\|_{L_{t,x}^{\infty}}|h|_{L_t^2H^{m+\f{1}{2}}}\lesssim \Lambda\big(\f{1}{c_0},|h|_{3,\infty,t}\big) \|\div^{\vp}u\|_{L_t^{\infty}H_{tan}^1}|h|_{L_t^2H^{m+\f{1}{2}}}.$$
The right hand side can be controlled if $\|\div^{\vp}u\|_{L_t^2H_{tan}^1}=\mathcal{O}(\ep^{\f{1}{2}}).$ Hopefully, once assuming $\ep^{\f{1}{2}}(\pt\sigma,\pt u)(0)$ to be bounded uniformly in $H_{co}^1(\mS),$ we can show that $\|(\nabla^{\vp}\sigma,\div^{\vp}u)\|_{L_t^{\infty}H_{co}^1}= \mathcal{O}(\ep^{\f{1}{2}}).$ This is one reason that we need the initial data to be slightly well-prepared.

\textbf{Step 4. Uniform energy estimate of the incompressible part of the velocity.}
Set $v=u-\nabla^{\vp}\Psi$ the incompressible part of the velocity. By the computations in Section 5, we find that $v$ solves the following system: 
\beq\label{eqofv}
\left\{
\begin{array}{l}
   \bar{\rho} \p_t^{\vp}v-\mu\Delta^{\vp}v+\nabla^{\vp}\pi=-(f+\nabla^{\vp}q+\bar{\rho}[\mathbb{P}_t,\p_t^{\vp}]u) ,\\
    \div^{\vp}v=0,\\
   (2\mu S^{\vp}v-\pi \text{Id})\bN|_{z=0}=2\mu(\div^{\vp}u\text{Id}-(\nabla^{\vp})^2\Psi)\bN |_{z=0},\\
   v_3|_{z=-1}=0,\, \mu\p_z^{\vp} v_{j}=a u_{j}|_{z=-1}, \,j=1,2,
\end{array}
 \right.
\eeq
where $\bbq, \bbp $ are time-dependent projectors  defined in \eqref{defofQ} \eqref{defofP} and 
\beq\label{defpi-intro}
f=(g_2u\cdot\nabla^{\vp}u+\f{g_2-\bar{\rho}}{\ep}\ep\pt^{\vp}u),\nabla^{\vp}q=-\mathbb{Q}_t(f-\mu\Delta^{\vp}v),\nabla^{\vp}\pi=\mathbb{P}_t[\nabla^{\vp}(\f{\sigma}{\ep}-(2\mu+\lambda)\div^{\vp}u)].
\eeq
Note that $\nabla^{\vp}\pi$ does not vanish identically since $\mathbb{Q}_t\nabla^{\vp}\neq \nabla^{\vp}$
and  that $\nabla^{\vp}\pi$ is actually not singular  though it seems to involve $\sigma/\ep$.
Indeed  by the definition of $\bbq$ and  the boundary conditions \eqref{sigmabdry} \eqref{omegatimesn}, $\pi$ solves the elliptic equation:%the control of this term can be reduced to 
%$\pi$ is harmonic with vanishing Neumann boundary condition on the lower boundary and with bounded Dirichlet boundary condition on the upper boundary:
\beqs
\left\{
\begin{array}{l}
     \Delta^{\vp}\pi=0,  \\
      \pi|_{z=0}=-2\mu(\p_1u_1+\p_2u_2)-2\mu(\Pi(\p_1u\cdot\bN,\p_2u\cdot\bN,0)^{t})_3,\\
      \p_z^{\vp}\pi|_{z=-1}=0.
\end{array}
\right.
\eeqs
The key point is that the  trace of $\pi$ on the upper boundary can be uniformly bounded.

In view of \eqref{eqofv}, we expect to perform energy estimates to get  a priori control of $\|v\|_{L_t^{\infty}H_{co}^{m-1}}, \|\ep^{\f{1}{2}}\pt v\|_{L_t^{\infty}H_{co}^{m-2}}$ and  $\|\nabla^{\vp}v\|_{L_t^2H_{co}^{m-1}}, \|\ep^{\f{1}{2}}\pt\nabla v\|_{L_t^{\infty}H_{co}^{m-2}}.$ 
Of course, due to the interaction with the compressible part through the boundary, their control rely also on the information for the compressible part $\nabla^{\vp}\Psi$ and we cannot get higher order estimates.

\textbf{Step 5. Control of the normal derivative of the velocity.}
We have obtained the estimates of $\|\nabla^{\vp}u\|_{L_t^2H_{co}^{m-1}}$ 
in Step 3 and Step 4. 
It remains to control 
$\ep^{-\f{1}{2}}\|(\nabla^{\vp}\sigma, \div^{\vp}u)\|_{L_t^{\infty}H_{co}^{m-2}}$ and $\|(\nabla v,\ep^{\f{1}{2}}\pt\nabla v)\|_{L_t^{\infty}H_{co}^{m-4}},$ which is useful to control the $L_{t,x}^{\infty}$ norm of the solution. The former quantity can be obtained again by induction arguments while the latter quantity can be deduced from that of $\omega\times \bn.$
Indeed, we have roughly the estimate:
\beqs
\|(\nabla v,\ep^{\f{1}{2}}\pt\nabla v)\|_{L_t^{\infty}H_{co}^{m-4}}\lesssim \|(\text{Id}, \ep^{\f{1}{2}}\pt)(\omega\times \textbf{n})\|_{L_t^{\infty}H_{co}^{m-4}}+\|(v,\ep^{\f{1}{2}}\pt v)\|_{L_t^{\infty}H_{co}^{m-3}}+|(h,\ep^{\f{1}{2}}\pt h)|_{L_t^{\infty}\tilde{H}^{m-2}}.
\eeqs
Let us explain the estimate of $\|(\text{Id}, \ep^{\f{1}{2}}\pt)(\omega\times \textbf{n})\|_{L_t^{\infty}H_{co}^{m-4}}.$ 
Direct computations show that:
\beq
\omega\times \textbf{n}|_{\p\mS}=-2\Pi(\p_1 u\cdot \textbf{n},\p_2 u\cdot \textbf{n},0)^{t}
\eeq
where $\Pi=\Id_{3\times 3}-\bn\otimes \bn.$
We define the modified vorticity
$\omega_{\textbf{n}}=\omega\times \bn+2\Pi(\p_1 v\cdot \textbf{n},\p_2 v\cdot \textbf{n},0),$ so that:
$$\omega_{\bn}|_{\p\mS}=-2\Pi(\p_1 \nabla^{\vp}\Psi\cdot \textbf{n},\p_2 \nabla^{\vp}\Psi\cdot \textbf{n},0)^{t}.$$ The advantage of working on $\omega_{\textbf{n}}$ rather than $\omega\times \textbf{n}$
is that the former one only involves the compressible part of velocity on the boundary, whose estimates have been established in Step 3.
To estimate $\omega_{\bn},$ we shall thus instead use a lifting of the boundary conditions by using the Green's function of the solution to the heat equation with non-homogenous 
 boundary conditions and control the remainder by energy estimates.
  %we split the system for $\omega_{\bn}$ into two systems, one of which carries on all the nonlinear terms and initial data but with trivial Dirichlet boundary condition so that the direct energy estimate is feasible, the other one is just a free heat equation with zero initial data and nontrivial Dirichlet boundary condition for which we can use the explicit formulae.
 More precisely, let $\omega^h$ solves the heat equation \eqref{heateq} with boundary condition $\omega^{b,1}|_{z=0}=\omega_{\bn}|_{z=0},$
 we use \eqref{heateq} to get roughly that:
\beqs
\begin{aligned}
\|(\text{Id}, \ep^{\f{1}{2}}\pt)\omega_{\textbf{n}}^h\|_{L_t^{\infty}H_{co}^{m-4}}&\lesssim T^{\f{1}{4}}\big(|(\text{Id}, \ep^{\f{1}{2}}\pt)\nabla\Psi|_{L_t^{\infty}\tilde{H}^{m-3}}+|(\text{Id}, \ep^{\f{1}{2}}\pt)h|_{L_t^{\infty}\tilde{H}^{m-3}}\big)\\
&\lesssim T^{\f{1}{4}}(\|(\text{Id}, \ep^{\f{1}{2}}\pt)\div^{\vp} u\|_{L_t^{\infty}H_{co}^{m-3}}+|(\text{Id}, \ep^{\f{1}{2}}\pt)h|_{L_t^{\infty}\tilde{H}^{m-3}}).
\end{aligned}
\eeqs
The remainder $\omega_{\bn}-\omega_{\bn}^h$ can then be controlled by direct energy estimates.

\textbf{Step 6. $L^{\infty}_{t,x}$ estimates.}
This final step is dedicated to the estimates of the $L_{t,x}^{\infty}$ type norms defined in $\cA_{m,T}.$ 
Most of them 
can be controlled thanks to the Sobolev embedding and the quantities appearing in $\cE_{m,T}.$ 
The estimate of the remaining terms $\ep^{-\f{1}{2}}\il\nabla\sigma\il_{m-5,\infty,T}$ and $\il \nabla u\il_{1,\infty,t}$ are obtained from
the maximum principle of the damped  transport equation satisfied by $\nabla\sigma$ and the estimate for the heat equation satisfied by  $\omega$.

\textbf{Structure of the paper:}
We state the uniform a-priori estimates 
in Section 2, which are shown in the following sections. Some preliminaries (useful lemmas, identities, %reformulation of the boundary conditions,
projections, and elliptic estimates) are first shown in Sections 3-5. The control of the energy norm $\cE_{m,T}$ is achieved in Sections
6-Section 11. The $L_{t,x}^{\infty}$ type estimates are established in Section 12. 
 Theorem \ref{thm-localexis}
and Theorem \ref{thm-convergence} are then proved in Section 13 and Section 14 respectively. In Section 15, we explain how our results can be extended to the case when the reference domain is changed into a channel with infinite depth. Finally, one technical product estimate is presented in the appendix.

\textbf{Further notations}
 
$\bullet$ We denote $\Lambda(\cdot,\cdot)$
a polynomial that may differ from line to line but independent of $\ep\in(0,1].$

$\bullet$ 
The traces on the upper boundary \{z=0\} and lower boundary \{z=-1\} for a function $f\in H^1(\mS)$ are denoted by $f^{b,1}$ and $f^{b,2}$ respectively.

$\bullet$ We use the notation  $\lesssim$ for
$\leq C(1/c_0)$ for some number
$C(1/c_0)$ that depends only on $1/c_0.$

$\bullet$ We use the notation $L_t^2L^2=L^2([0,t]\times\mS).$

$\bullet$ We denote
$\|f\|_{E^k,t}=\|f\|_{L_t^2H_{co}^k}+\|\nabla f\|_{L_t^2H_{co}^{k-1}}.$

\section{Uniform a-priori estimates}
Our main a priori estimate is the following:
\begin{thm}\label{thm-apriori}
Let $c_0 \in (0, \f{1}{2}]$ such that:
\beq\label{preassumption1}
 \sup_{s\in [-3c_1\bar{P}, 3\bar{P}/{c_1}]} |(g_1,g_2)(s)|\in [{c_0},{1}/{c_0}]
\eeq
where $0< c_1<\f{1}{4}$ is a fixed constant. Suppose that for some $0<T\leq 1,$ for all $(t,x)\in [0,T]\times\mS,\ep\in[0,1],$ it holds that:
\beq\label{preassumption}
\p_z\vp^{\ep}(t,x)\geq c_0,\quad |(\nabla\vp^{\ep},\nabla^2\vp^{\ep})(t,x)|\leq {1}/{c_0},\quad -3c_1\bar{P}\leq \ep\sigma^{\ep}(t,x)\leq 3\bar{P}/{c_1}.
%\quad \forall (t,x)\in [0,T]\times\mS,\ep\in[0,1].
\eeq
Then there exist two  continuous  functions $P_1, P_2: \mathbb{R}_{+}\times \mathbb{R}_{+}\rightarrow \mathbb{R}_{+},$ and $\vartheta>0$ which are independent of $\ep,$ 
%and  which is also independent of $\ep,$ 
such that the following estimate holds:
\beq\label{apriori}
\cN_{m,T}^{\ep}\leq P_1\big(\f{1}{c_0}, Y^{\ep}_m(0) \big)+ (T+\ep)^{\vartheta}P_2\big(\f{1}{c_0},Y^{\ep}_m(0)+\cN^{\ep}_{m,T}\big)
\eeq
where $\cN_{m,T}^{\ep}$ is defined in \eqref{defcN}. 
\end{thm}
 This theorem is a direct consequence of the following two propositions.
\begin{prop}\label{prop-energy}
Under the assumption of Theorem \ref{thm-apriori}, 
there exist two  $\ep-$independent continuous functions
$P_3, P_4: \mathbb{R}_{+}\times \mathbb{R}_{+}\rightarrow \mathbb{R}_{+},$ such that: %the following estimate for energy norm:
\beq\label{apriori-energy}
\mathcal{E}^{\ep}_{m,T}\leq P_3\big(\f{1}{c_0},Y^{\ep}_m(0)\big)+(T+\ep)^{\vartheta_1}P_4\big(\f{1}{c_0},Y^{\ep}_m(0)+\cN^{\ep}_{m,T}\big).
\eeq
\end{prop}
\begin{proof}
 This proposition is obtained by energy estimates, we split it into several sections (Section 6-11).
 By Lemma \ref{lemsurface} for the estimate of the surface, Lemmas
 \ref{lemhighest}, \ref{lemnablasigma}, \ref{intermediate-epnablau} for $\ep-$dependent estimates to the highest order, %\ref{intermediate-epnablau}, 
 Lemmas
 \ref{sigmainduction}, \ref{lemhigh-v},  \ref{sigmainduction1}, \ref{lemnablau-LinftyL2},  \ref{lemsec-normal-u} for the uniform estimates, 
 we can find two polynomials $\Lambda_5,\Lambda_6$ whose coefficients are independent of $\ep,$ such that:
 \beq\label{tildeEmT}
 ({\cE}^{\ep}_{high,m,T})^2\leq \Lambda_5\big(\f{1}{c_0}, |h^{\ep}|_{L_T^{\infty}\tilde{H}^{m-\f{1}{2}}}^2+Y^{\ep}_{m}(0)^2\big) Y^{\ep}_{m}(0)^2+(T+\ep)^{\f{1}{4}}\Lambda_6\big(\f{1}{c_0}, \cN^{\ep}_{m,T}\big).
 \eeq
By Lemma \ref{lemlow-full}, there exist  polynomials $\Lambda_7,\Lambda_8$
whose coefficients are independent of $\ep,$ such that:
%\beqs
%\leq \Lambda_7(1/c_0)  (Y^{\ep}_m)^2(0)+(T+\ep)^{\f{1}{2}}\Lambda_8\big(\f{1}{c_0}, \cN^{\ep}_{m,T}\big)+\tilde{\cE}^2_{m,T},\eeqs
\beqs
(\tilde{\cE}_{low,T})^2\lesssim \Lambda_7\big(\f{1}{c_0}, |h^{\ep}|_{3,\infty,T}^2
\big)
(Y_m^{\ep}(0)^2+({\cE}_{high,m,T}^{\ep})^2)+ (T+\ep)^{\f{1}{2}}\Lambda_8\big(\f{1}{c_0}, \cN_{m,T}^{\ep}\big).
\eeqs
By the Sobolev embedding $H^{\f{3}{2}}(\mR^2)\hookrightarrow L^{\infty}(\mR^2),$
\beqs
|h^{\ep}|_{3,\infty,T}^2\lesssim |h^{\ep}|_{L_T^{\infty}\tilde{H}^{m-\f{1}{2}}}^2,
\eeqs
we thus find two polynomials $\Lambda_9$ and $\Lambda_{10}$ such that:
\beq\label{sec2:eq1}
 ({\cE}_{m,T}^{\ep})^2\leq \Lambda_9\big(\f{1}{c_0},|h^{\ep}|_{L_T^{\infty}\tilde{H}^{m-\f{1}{2}}}^2+Y^{\ep}_{m}(0)^2 \big) Y^{\ep}_{m}(0)^2+(T+\ep)^{\f{1}{4}}\Lambda_{10}\big(\f{1}{c_0},\cN^{\ep}_{m,T}\big).
\eeq
By \eqref{surface1}, there exists a polynomial $\Lambda_{11},$ such that: 
\beqs|h^{\ep}|_{L_t^{\infty}\tilde{H}^{m-\f{1}{2}}}^2\leq Y^{\ep}_m(0)^2+T^{\f{1}{2}}%\bigg(
\Lambda_{11}
\big(\f{1}{c_0}, \cN_{m,T}^{\ep}\big).
\eeqs
Plugging  this inequality
into \eqref{tildeEmT}, one finds two other polynomials $\Lambda_{12},\Lambda_{13},$
and a constant $\vartheta_2>0,$ such that:
\beqs
({\cE}_{high,m,T}^{\ep})^2\lesssim \Lambda_{12}(\f{1}{c_0}, Y^{\ep}_m(0)^2 )+(T+\ep)^{\vartheta_2}\Lambda_{13}(\f{1}{c_0}, Y_m^{\ep}(0)+\cN_{m,T}^{\ep}).
\eeqs
We thus finish the proof by inserting the above inequality into \eqref{sec2:eq1}.
\end{proof}
\begin{prop}\label{prop-Linfty}
Assume that \eqref{preassumption} holds, we have the a-priori estimate for the $L_t^{\infty}L^{\infty}(\mS)$ norms, 
\begin{equation}
\cA_{m,T}^{\ep}\leq 
%Y_{m}^{\ep}(0)+Y_{m}^{\ep}(0)^4+
\Lambda\big(\f{1}{c_0},Y_m(0)\big) +
    \Lambda\big(\f{1}{c_0},|h^{\ep}|_{3,\infty,t}\big)
    \tilde{\cE}_{m,T}^{\ep}+(\tilde{\cE}_{m,T}^{\ep})^4 
+(T+\ep^{\f{1}{4}})
\Lambda_{14}\big(\f{1}{c_0}, \cN_{m,T}^{\ep}\big).
\end{equation}
where $\Lambda_{14}$ is a polynomial with $\ep-$independent coefficients.
\end{prop}

\begin{proof}
Its proof is presented in Section 12.
\end{proof}

\section{Preliminaries I: Useful lemmas.}
In this section, we list some elementary lemmas which will be often used throughout this paper. 

\subsection{Product and commutator estimates.}
We begin with the following product and commutator estimates in $\mathbb{R}^2.$
\begin{lem}\label{lempro-com}
Let $f,g: \mathbb{R}^2\rightarrow \mathbb{R}$ belong to the spaces appearing in below. For any $s\geq 1,$
\beq\label{classical-product-1}
|\Lambda^s(fg)|_{L^2(\mathbb{R}^2)}\lesssim |f|_{H^s(\mR^2)}|g|_{L^{\infty}(\mR^2)}+|g|_{H^s(\mR^2)}|f|_{L^{\infty}(\mR^2)}
\eeq
\beq\label{classical-product}
|[\Lambda^s, f]g|_{L^2(\mathbb{R}^2)}\lesssim |f|_{H^{s-1}(\mR^2)}|g|_{L^{\infty}(\mR^2)}+|f|_{W^{1,\infty}(\mR^2)}|g|_{H^{s-1}(\mR^2)}
\eeq
For any $-1<s\leq 1,$
\beq\label{commutator-R2}
|[\Lambda^s,g]f|_{L^2(\mathbb{R}^2)}\lesssim |f|_{H^{s-1}(\mR^2)}|g|_{H^{2^{+}}(\mR^2)},
\eeq
\beq\label{product-R2}
|fg|_{H^s(\mathbb{R}^2)}\lesssim |f|_{H^s(\mR^2)}\min\{|g|_{H^{1^{+}}(\mR^2)}, |g|_{W^{1,\infty}(\mR^2)}\}.\qquad
\eeq
 where $(\Lambda^s f)(y)=\cF_{\xi\rightarrow y}^{-1}\big((1+|\xi|^2)^{\f{s}{2}}\hat{f}(\xi)\big),$ %is the multiplier with   
 $a^{+}$ denotes a real number that is larger but arbitrary close to $a.$ 
\end{lem}
The product estimate \eqref{classical-product-1} and the commutator estimate \eqref{classical-product} can be found in \cite{MR1340046} for example, \eqref{commutator-R2} is indeed a restatement of 
(A.6) in \cite{MR2354691}. The proof of \eqref{product-R2} is presented in the appendix.
%\end{proof}
\begin{cor}
Let $k\geq 2$ be an integer, one has the following estimates:
\beq\label{rough-product-bdry}
|(fg)(t)|_{\tilde{H}^{k+\f{1}{2}}}\lesssim |f(t)|_{\tilde{H}^{[\f{k}{2}]^{+}}}|g(t)|_{\tilde{H}^{k+\f{1}{2}}}+|g(t)|_{\tilde{H}^{[\f{k+1}{2}]+1^{+}}}|f(t)|_{\tilde{H}^{k+\f{1}{2}}},
\eeq
\beq\label{rough-com-bdry}
|[Z^{\alpha},f]g(t)|_{{H}^{\f{1}{2}}}\lesssim 
|f(t)|_{\tilde{H}^{[\f{k}{2}]^{+}}}|g(t)|_{\tilde{H}^{k-\f{1}{2}}}+|g(t)|_{\tilde{H}^{[\f{k+1}{2}]+1^{+}}}|f(t)|_{\tilde{H}^{k+\f{1}{2}}}, |\alpha|= k,
\eeq
where $\tilde{H}^s$ is defined in \eqref{seminorm-bdry}, and commutator $[Z^{\alpha}, f]g=Z^{\alpha}(fg)-f Z^{\alpha}g.$
\end{cor}
\begin{proof}
For any $|\alpha|\leq k,$ we write
\beq\label{usefulidentity}
Z^{\alpha}(fg)(t)
=\bigg(\sum_{|\beta|\leq [\f{|\alpha|}{2}]-1}+\sum_{|\alpha-\beta|\leq [\f{|\alpha|+1}{2}]}\bigg)Z^{\beta}f(t) Z^{\alpha-\beta}g(t)
\eeq
Inequality \eqref{rough-product-bdry} can then be derived from product estimate \eqref{product-R2}. The proof of  \eqref{rough-com-bdry} follows in the same way. 
\end{proof}
 The following (crude) product estimates in $L_t^{\infty}\cH^{j,l}$ will be useful for instance in the elliptic estimates. 
\begin{lem}\label{lemcrudepro}
Let $Z^{\alpha}=(\ep\pt)^j \cZ^{\alpha'}$ with $\cZ=(Z_1,Z_2,Z_3), |\alpha'|\leq l=k-j, k\geq 2.$ One has the crude estimates: for any integer $n \in [0 , k-1]$
\begin{equation}\label{crudepro}
\|(fg)(t)\|_{\cH^{j,l}}\leq \|f(t)\|_{\cH^{j,l}} \il g\il_{n,\infty,t}+\|g(t)\|_{\cH^{j,l}}\il f\il_{k-n-1,\infty,t},\\ %[\f{k}{2}]-1\\
\end{equation}
\begin{equation}\label{crudecom}
\begin{aligned}
\|[Z^{\alpha}, f]g(t)\|_{L^2(\mS)}
&\lesssim \bigg(\sum_{\tiny
\begin{array}{l}
    j'\leq j,l'\leq l,  \\
   j'+l'\leq k-n
\end{array}
}\|f(t)\|_{\cH^{j',l'}}\bigg)\il g\il_{n,\infty,t}\\
&\,+\big(\|g(t)\|_{\cH^{j-1,l}}+\|g(t)\|_{\cH^{j,l-1}}\big)
\il f\il_{k-n-1,\infty,t}.
\end{aligned}
\end{equation}
\end{lem}
We also have the following composition estimates:
 \begin{cor}
 Suppose that $\psi\in C^0(Q_t)\cap L_t^2H_{co}^m$ with 
 $$A_1\leq  \psi(t,x)\leq A_2, \quad \forall(t,x)\in Q_t.$$
 Let
 $F(\cdot):[A_1,A_2]\rightarrow \mathbb{R}$ be a smooth function satisfying 
 $$\sup_{s\in [A_1,A_2],j\leq m}|F^{(j)}|(s)\leq B.$$
 Then we have the composition estimate:
 \beq\label{composition}
 \|F(\psi(\cdot,\cdot))-F(0)\|_{L_t^pH_{co}^m}\leq  \Lambda(B,\il \psi\il_{[\f{m}{2}],\infty,t})\|\psi\|_{L_t^pH_{co}^m},
 \eeq
\end{cor}

 \begin{cor}\label{corg12}
 Let $g_1(\ep\sigma),g_2(\ep\sigma)$ defined in \eqref{defofg12} and assume Property \eqref{preassumption1} and Assumption \eqref{preassumption}  hold.
 Then one has the following 
 estimates: for $j=1,2$
  \beq\label{esofg12-3}
 \|g_j(\ep\sigma)-g_j(0)\|_{L_t^pH_{co}^m}\lesssim \ep \Lambda\big (\f{1}{c_0}, \il\sigma\il_{[\f{m}{2}],\infty,t}\big)\|\sigma\|
 _{L_t^pH_{co}^m}.
  \eeq
 \beq\label{esofg12-1}
 \|Z g_j\|_{L_t^p\cH^{m-1}}\leq \ep \Lambda\big(\f{1}{c_0},\il \sigma\il_{[\f{m}{2}],\infty,t}\big)\|(\sigma,Z\sigma)\|_{L_t^p\cH^{m-1}}, 
 \eeq
 \beq\label{esofg12-2}
 \|Z g_j\|_{L_t^pH_{co}^{m-1}}\leq \ep \Lambda\big(\f{1}{c_0},\il \sigma\il_{[\f{m}{2}],\infty,t}\big)\|\sigma\|_{L_t^pH_{co}^m},
 \eeq

 \end{cor}
\begin{proof}
Inequality \eqref{esofg12-3} is a direct consequence of the composition estimate \eqref{composition}. To get \eqref{esofg12-1}, \eqref{esofg12-2},
one can apply \eqref{crudepro} for 
$n=[\f{m-1}{2}]-1$ and use again \eqref{composition}.
\end{proof}
The next lemma states the generalized 
product estimate and commutator estimate \cite{MR1070840}.
\begin{lem}\label{prodcomu}
For $|\alpha|\leq m, \alpha_0=0$
 we have the product estimate  and commutator estimates:
 \ben\label{product}
 \|Z^{\alpha}(fg)\|_{\ltx}
 \lesim \|f\|_{L_t^2\cH^{0,m}} \il g\il _{0,\infty,t}+
 \|g\|_{L_t^2\cH^{0,m}}\il f\il _{0,\infty,t},
 \een
 \ben\label{commutator}\|[Z^{\alpha},f]g\|_{\ltx}\lesim\|  f\|_{L_t^2\cH^{0,m}}\il g\il _{0,\infty,t}+\|g\|_{L_t^2\cH^{0,m-1}}\il   f\il_{1,\infty,t}.
 \een
\end{lem}
We finally state the following Sobolev embedding and trace inequalities whose proofs can be found in Proposition 2.2 of \cite{MR3590375}.
\begin{lem}
For each $t\in [0,T],$ we have:
\begin{equation}\label{soblev-embed}
    \| f(t)\|_{L^{\infty}(\mS)}\lesssim \|(f,\nabla f)(t)\|_{H_{tan}^{s_1}(\mS)}^{\f{1}{2}}\|f(t)\|_{H_{tan}^{s_2}(\mS)}^{\f{1}{2}},\quad  s_1+s_2>2, s_1,s_2\geq 0,
\end{equation}
\beq\label{trace}
\begin{aligned}
&|f(t,\cdot,0)|_{H^s(\mR^2)}+|f(t,\cdot,-1)|_{H^s(\mR^2)}\\
&\lesssim \|\p_z f(t)\|_{H_{tan}^{s-{1}/{2}}(\mS)}^{\f{1}{2}} \|f(t)\|_{H_{tan}^{s+{1}/{2}}(\mS)}^{{1}/{2}}+\| f(t)\|_{H_{tan}^{s+{1}/{2}}(\mS)},\quad s\geq\f{1}{2}.
\end{aligned}
\eeq
where we have used the notation
$\|f(t)\|_{H_{tan}^s(\mS)}=\|\Lambda^s f(t)\|_{L^2(\mS)}.$ %where $\Lambda^s$ is the Fourier multiplier with symbol $(1+|\xi|^2)^{\f{s}{2}}.$
\end{lem}
\subsection{Regularity of the extension and some further commutator estimates.}
We first show that the 
diffeomorphism $\Phi$ has the same regularity as $u$ in $\mS,$
which stems from the fact that
the extension function $\vp$ gains
half a space derivative with respect to $h.$ Before stating the main estimates, let us recall that $\vp$ and $\eta$ are defined in \eqref{defvpep},
\eqref{defeta}.
\begin{lem}\label{exth}
For any integers $ j, k\geq 0,$ we have the following estimates:
\begin{equation}\label{gradexth-1}
 \|[(\ep\pt)^j\nabla\vp](t)\|_{H^{k}(\mS)}\lesssim |[(\ep\pt)^j h](t,\cdot)|_{H^{k+\f{1}{2}}(\mathbb{R}^2)},
\end{equation}
\begin{equation}\label{gradexth}
 \|\nabla\vp\|_{L_t^2\cH^{j,k}(\mS)}\lesssim | h|_{L_t^2\tilde{H}^{k+j+\f{1}{2}}(\mathbb{R}^2)}. 
\end{equation}
Moreover, we have the $L_{t,x}^{\infty}$ estimates for $\eta:$ %which is defined in \eqref{defeta}
\begin{equation}\label{exthLinfty}
    \|[(\ep\pt)^j\eta](t)\|_{W^{k,\infty}(\mS)}\lesssim |[(\ep\pt)^j h](t)|_{W^{k,\infty}(\mathbb{R}^2)}\lesssim |h|_{k+j,\infty,t}.
\end{equation}
%\begin{equation}
 %   \il \eta\il_{k+j,\infty,t}\lesssim |h|_{k+j,\infty,t}
%\end{equation}
\begin{proof}
These estimates can be deduced from Young's inequality and the following estimates:
\beqs
\int_{-1}^0 e^{-2\delta_0 z^2\langle \xi\rangle^2}\d z\lesssim \delta_0^{-\f{1}{2}}\langle  \xi\rangle^{-1};\quad \|\cF^{-1}(e^{-\delta_0 z^2\langle \xi\rangle^2})\|_{L_z^{\infty}L_y^1}\lesssim 1.
\eeqs
One can refer to Proposition 3.1 of \cite{MR3590375} for the detail of the case $j=0.$ The case for $j>0$ follows from the observation that time derivatives commute with the actions $\vp(h)$ and $\eta(h).$
\end{proof}
\end{lem}

\begin{lem}\label{lemfpzphi}
 Suppose that:
 $\p_z\vp(t,x)\geq c_0$ for $(t,x)\in [0,T]\times \mS.$ Then for any $k\in\bN,$
 \beq\label{fpzphi1}
 \big\|\f{f}{\p_z\vp}\big\|_{L_t^pH_{co}^k}\lesssim \Lambda\big(\f{1}{c_0},|h|_{[\f{k}{2}]+1,\infty,t}+\il f\il_{[\f{k}{2}],\infty,t}\big)\big(\|f\|_{L_t^pH_{co}^k}+|h|_{L_t^p\tilde{H}^{k+\f{1}{2}}}\big), \quad p=2,+\infty.
 \eeq
\end{lem}
\begin{proof}
Let us write: 
\beqs
\f{f}{\p_z\vp}=\f{f}{1+\eta+\p_z\eta(1+z)}=f-f\f{\eta+\p_z\eta(1+z)}{1+\eta+\p_z\eta(1+z)}.
\eeqs
Therefore, one obtains \eqref{fpzphi1} by applying the product estimate \eqref{crudepro} for $n=[\f{k}{2}]$ and  composition estimate \eqref{composition} for $F(x)=\f{x}{1+x}\, (0<x<1).$
\end{proof}
\begin{rmk}
Similar to \eqref{fpzphi1}, under the same assumption as in Lemma 
\ref{lemfpzphi}, the following estimate also holds true, 
\beq\label{fpzphi2}
\big\|\f{f}{\p_z\vp}\big\|_{L_t^p\cH^{0,k}}\lesssim \Lambda\big(\f{1}{c_0},|h|_{1,\infty,t}+\il f\il_{0,\infty,t}\big)(\|f\|_{L_t^p\cH^{0,k}}+|h|_{L_t^p\tilde{H}^{k+\f{1}{2}}}), \, p=2,+\infty.
\eeq
%\beq\label{fpzphi2}
%\ep^{\f{1}{2}} \big\|\f{f}{\p_z\vp}\big\|_{L_t^pH_{co}^k}\lesssim \Lambda\big(\f{1}{c_0},|h|_{[\f{k}{2}]+1,\infty,t}+\ep^{\f{1}{2}}\il f\il_{[\f{k}{2}],\infty,t}\big)(\ep^{\f{1}{2}}\|f\|_{L_t^pH_{co}^k}+|h|_{L_t^p\tilde{H}^{k+\f{1}{2}}}), \quad p=2,+\infty.\eeq
\end{rmk}
The next lemma contains useful commutator estimates.
\begin{lem}
Under the assumption \eqref{preassumption}, the following commutator estimates hold, for 
$j=1,2,3, |\alpha|\leq k$
\beq
\begin{aligned}
\|[Z^{\alpha},\p_j^{\vp}]f\|_{\ltx}
&\lesssim \Lambda\big(\f{1}{c_0},|h|_{[\f{k}{2}]+1,\infty,t}\big)|h|_{L_t^2\tilde{H}^{k-n+\f{1}{2}}}\il\nabla f\il_{n,\infty,t}\\
&+\Lambda\big(\f{1}{c_0},|h|_{k-n,\infty,t})\|\nabla f\|_{L_t^2H_{co}^{k-1}} \,\quad (0\leq n\leq k-1). \label{comgrad}
\end{aligned}
\eeq
If $\alpha_0=0,$  we have that:
\begin{align}  &\|[Z^{\alpha},\p_j^{\vp}]f\|_{\ltx}\lesssim\Lambda\big(\f{1}{c_0},|h|_{1,\infty,t}\big)\|\nabla f\|_{L_t^2H_{co}^{k-1}}+\Lambda\big(\f{1}{c_0},\il\nabla f\il_{0,\infty,t}\big)|h|_{L_t^2\tilde{H}^{k+\f{1}{2}}}.\label{comgrad1}\end{align}
Moreover,  for $k\geq 3,$
\beq\label{comgradT}
\begin{aligned}
\|[Z_0^k\pt,\p_j^{\vp}]f\|_{\ltx}&\lesssim  \Lambda\big(\f{1}{c_0},\il (\p_z f, \ep\pt \p_z f)\il_{0,\infty,t}+|(h,\pt h)|_{k-2,\infty,t}+(\int_0^t|\ep\pt^2 h(s)|^2_{k-2,\infty}\d s)^{\f{1}{2}} \big)\\
&\quad\cdot\big(\ep\sum_{l\leq k-1}|Z_0^l\p_t^2 h|_{L_t^2H^{\f{1}{2}}}+\|Z_0 \p_z f\|_{L_t^2\cH^{k-1}}+\|Z_0\p_z f\|_{L_t^{\infty}\cH^{1}}\big).
\end{aligned}
\eeq
%where we denote $\il f \il_{*,l,\infty,t}=\sum_{j\leq l}\il Z_0^j \p_z f \il_{0,\infty,t}.$
\end{lem}
\begin{proof}
By the definition \eqref{newder} for $\nabla^{\vp},$
\begin{equation}\label{comidentity}
\begin{aligned}
&[Z^{\alpha},\p_j^{\vp}]f=[Z^{\alpha}, {\bN_j}/{\p_z\vp}]\p_z f+({\bN_j}/{\p_z\vp})[Z^{\alpha},\p_z ]f.
\end{aligned}
\end{equation}
Moreover,  there exist smooth functions $C_{\phi,\beta,\alpha},C_{\phi,\gamma,\alpha}$ which depend on derivatives of $\phi$
%$\phi$ and its derivatives,
such that:
\beq\label{identity-com-nor}
[Z^{\alpha},\p_z]=\sum_{|\beta|\leq |\alpha|-1}C_{\phi,\beta,\alpha}Z^{\beta}\p_z=\sum_{|\gamma|\leq  |\beta|-1}C_{\phi,\gamma,\alpha}\p_z Z^{\gamma}.
\eeq
Therefore, we get \eqref{comgrad} by \eqref{crudecom}, \eqref{fpzphi1}. and get \eqref{comgrad1} by \eqref{commutator}, \eqref{fpzphi2}.

Next, for \eqref{comgradT}, we use the following direct expansion
\beq\label{expansion-T}
[Z_0^k\pt, g]w=\bigg(\sum_{0\leq l\leq 1 }+\sum_{0\leq k-l\leq k-3}\bigg) \big(C_{k,l}Z_0^{k-l} \pt g \, Z_0^{l} w\big)+C_{k,2}Z_0^{k-2}\pt g Z_0^2w.
\eeq
to obtain:
\beq\label{com-temp}
\begin{aligned}
\|[Z_0^k\pt, g]w\|_{L^2L^2}&\lesssim \|Z_0\pt g\|_{L_t^2\cH^{k-1}}\il w\il_{1,\infty,t}+\|Z_0 w\|_{L_t^2\cH^{k-1}}\il\pt g\il_{k-3,\infty,t}\\
&+\|Z_0 w\|_{L_t^{\infty}\cH^1}\|Z_0^{k-2}\pt g\|_{L_t^2L^{\infty}}
\end{aligned}
\eeq
Applying \eqref{com-temp} with $g=\f{\bN_j}{\p_z\vp}, w=\p_z f$, and using
\eqref{gradexth-1}, we get \eqref{comgradT}.
\end{proof}
\subsection{Energy identities and Korn inequality}
We now present some  identities which will be often used in the energy estimates:
\begin{lem}\label{lemipp}
It holds that:
\beq\label{ipp-1}
\int_{\mS} g_1(\p_t^{\vp^{\ep}}+u\cdot\nabla^{\vp^{\ep}})\sigma(t) \cdot \sigma(t)\,\d \cV_t=\f{1}{2}\p_t\int_{\mS}g_1|\sigma|^2(t)\d \cV_t-\f{1}{2}\int_{\mS}
(\p_t^{\vp^{\ep}}g_1+\div^{\vp^{\ep}}(g_1u))|\sigma|^2(t)\ \d \cV_t,
\eeq
\beq\label{ipp-2}
\int_{\mS} g_2(\p_t^{\vp^{\ep}}+u\cdot\nabla^{\vp^{\ep}})u(t) \cdot u(t)\ \d \cV_t=\f{1}{2}\p_t\int_{\mS}g_2|u|^2(t)\ \d \cV_t, 
\eeq
\beq\label{ipp-3}
\begin{aligned}
&\qquad\int_{\mS}(-\div^{\vp^{\ep}}\mathcal{L}^{\vp^{\ep}}u+{\nabla^{\vp^{\ep}}\sigma}/{\ep})\cdot u(t)\ \d \cV_t\\
&=\int_{\mS} 2\mu|
{S}^{\vp^{\ep}}u(t)|^2+\lambda|\div^{\vp^{\ep}}u(t)|^2\ \d\cV_t-\int_{\mS}\sigma\div^{\vp^{\ep}}u(t)\ \d \cV_t+a\int_{z=-1}|u_{\tau}|^2\ \d y.
\end{aligned}
\eeq
 where $u_{\tau}=(u_1,u_2,0)^t$ denotes the tangential components of $u,$ $\d\cV_t=\p_z\vp\ \d y\d z$ is the measure in $\mS$ coming from the change of variable \eqref{changeofvariable}. 
 \end{lem}
 \begin{proof}
% These identities can be derived from the following identities:
By direct computations, one can obtain the following identities:
 \beqs
\int_{\mS}\p_j^{\vp^{\ep}}f(t) g(t)\,\d \mathcal{V}_t =-\int_{\mS} f(t)\p_j^{\vp^{\ep}} g(t)\,\d\cV_t+\int_{\p\mS}f(t) g(t) \mathbf{N}_j\,\d y,\quad j=1,2,3
\eeqs
\beqs
\int_{\mS}\pt^{\vp^{\ep}}f(t) g(t)\,\d \mathcal{V}_t =\pt \int_{\mS} fg(t) \,\d \cV_t-\int_{\mS} f(t)\pt^{\vp^{\ep}} g(t)\,\d\cV_t+\int_{z=0}f(t) g(t) \p_t h\,\d y,
\eeqs
which, along with the equation \eqref{surfaceeq2}-%\eqref{upbdry2},
\eqref{bebdry2} lead to \eqref{ipp-1}-\eqref{ipp-3}. Note that in the derivation of \eqref{ipp-2}, we have used the fact that $\p_t^{\vp^{\ep}}g_2+\div^{\vp^{\ep}}(g_2u)=0$ in $[0,t]\times\mS.$
 \end{proof}
 
 The next lemma shows that %the dissipation give the control of $\|\nabla^{\vp} u\|_{L_t^2L^2(\mS)}:$
 one can control the gradient of the velocity by $S^{\vp}u.$
 \begin{lem}[\label{lemkorn}{Korn's inequality}]
 Suppose that  \eqref{preassumption} is true, then there exists $\Lambda_0(\f{1}{c_0}),\Lambda_1(\f{1}{c_0})$ such that:
 \beq\label{korn1}
 \begin{aligned}
 \int_{\mS} |\nabla u|^2(t)\,\d\cV_t&\leq \Lambda_0\big(\f{1}{c_0}\big) \int_{\mS} |\nabla^{\vp^{\ep}} u|^2(t)\,\d\cV_t\leq \Lambda_1\big(\f{1}{c_0}\big)\int_{\mS} (|{S}^{\vp^{\ep}} u|^2+|u|^2)\,\d\cV_t.
 \end{aligned}
 \eeq
 As a consequence, we have also:
 \beq\label{korn}
 \begin{aligned}
 \int_0^t\int_{\mS} |\nabla u|^2\,\d\cV_s\d s
 \leq \Lambda_1\big(\f{1}{c_0}\big)\int_0^t\int_{\mS} (|
 {S}^{\vp^{\ep}} u|^2+|u|^2)\,\d\cV_s\d s.
 \end{aligned}
 \eeq
 \end{lem}
These two inequalities can be shown similarly as in the proof of Proposition 2.9 in \cite{MR3590375}.
 \section{Preliminaries II: Reformulations of the boundary conditions}
 For notational convenience, from now on, we will skip the $\ep$-dependence of the solution.
 \begin{prop}
 The following boundary condition on 
 $\{z=0\}$ hold: 
\beq\label{sigmabdry}
\f{\sigma}{\ep}=(2\mu+\lambda)\div^{\vp}u-2\mu(\p_1u_1+\p_2u_2)+\mu(\omega\times \bN)_3,%|_{z=0}.
\eeq
\beq\label{omegatimesn}
\omega\times\bn=-2\Pi(\p_1u\cdot\bn,\p_2u\cdot\bn,0)^{t},
\eeq
\beq\label{tan-nor}
 \Pi(\p_{\bn}^{\vp}u)=-\Pi(\p_1u\cdot\bn,\p_2u\cdot\bn,0)^{t},
\eeq
\beq\label{nor-nor}
\begin{aligned}
\p_{\bn}^{\vp}u\cdot\bn&=|\bN|^2\p_z^{\vp}u\cdot\bn-(\bn_1\p_1 u\cdot\bn+\bn_2\p_2u\cdot\bn)\\
&=|\bN|
(\div^{\vp}u-\p_1u_1-\p_2u_2)-(\bn_1\p_1 u\cdot\bn+\bn_2\p_2 u\cdot\bn)
\end{aligned}
\eeq
where $\omega=\nabla^{\vp}\times u, \Pi=\text{Id}_{3}-\bn\otimes\bn,$ here $Id_3$ denotes the identity matrix of order 3.
\end{prop}
\begin{proof}
The first identity can be deduced from the boundary condition \eqref{upbdry2}. Indeed, by taking the third component of \eqref{upbdry2}, one gets that on the upper boundary
$\{z=0\},$
\beqs
\begin{aligned}
\f{\sigma}{\ep}&=\lambda\div^{\vp}u+2\mu\p_z^{\vp}u\cdot \bN +\mu\big[(\nabla^{\vp}u-(\nabla^{\vp}u)^{t})\cdot\bN\big]_{3}\\
&=(2\mu+\lambda)\div^{\vp}u-2\mu(\p_1u_1+\p_2u_2)+\mu(\omega\times \bN)_3.
\end{aligned}
\eeqs
%The third inequality follows from the 
Note that we have used the identity
\beq\label{norpz}
\p_z^{\vp}u\cdot\bN=\div^{\vp}u-\p_1u_1-\p_2u_2
\eeq
which holds indeed in the whole domain $\mS.$
For the second identity \eqref{omegatimesn}, we have that on the upper boundary:
\beq
\begin{aligned}
\mu \omega\times \bN=\mu\Pi(\omega\times \bN)&=2\mu\Pi\big(-(\nabla^{\vp}u)^{t}\bN+ S^{\vp}u\bN\big)\\
&=\Pi\big(-2\mu(\nabla^{\vp}u)^{t}\bN+
(\sigma/{\ep}-\lambda\div^{\vp}u)\bN\big)\\
&=-2\mu\Pi(\p_1u\cdot\bN,\p_2u\cdot\bN,0)^{t}.
\end{aligned}
\eeq
Note that $(\nabla^{\vp}u)^{t}\cdot\bN=(\p_1u\cdot\bN,\p_2u\cdot\bN,0)^{t}+(\p_z^{\vp}u\cdot\bN)\bN.$
The inequality \eqref{tan-nor} can be derived in a similar way:
\beq
\mu\Pi(\p_{\bn}^{\vp}u)=\mu\Pi(2 S^{\vp}u \bn-(\nabla^{\vp}u)^{t}\cdot\bn)=-\mu \Pi \big((\nabla^{\vp}u)^{t}\cdot\bn\big).
\eeq
The inequality \eqref{nor-nor} follows from direct computations and identity \eqref{norpz}.

\begin{rmk}
By the identity: $|\bN|\p_z^{\vp} u=\p_{\bn}^{\vp}u-\bn_1\p_1u-\bn_2\p_2u$, we have also:
\beq\label{tanpz}
|\bN|\Pi\p_z^{\vp} u=\Pi (\p_1u\cdot\bn,\p_2u\cdot\bn,0)^{t}-\Pi(\bn_1\p_1u+\bn_2\p_2 u).
\eeq
\end{rmk}
\begin{rmk}
In view of \eqref{norpz}, \eqref{tanpz}, we have that
$\p_z^{\vp}u\approx \div^{\vp}u+\p_y u$ on $\{z=0\},$
so that:
\begin{equation}\label{grad-upbdry}
\begin{aligned}
|(\nabla^{\vp}u)^{b,1}|_{\htlde^k}&\lesssim \Lambda\big(\f{1}{c_0}, \il\div^{\vp}u\il_{0,\infty,t}+\il u \il_{1,\infty,t}+|h|_{1,\infty,t}\big)\\%_{1,\infty,t}
&\qquad\qquad\big(|(\div^{\vp}u)^{b,1}|_{\htlde^k}+
|u^{b,1}|_{\htlde^{k+1}}+|h|_{\htlde^{k+1}}\big).
\end{aligned}
\end{equation}
Recall that we denote for any $f,$ $f^{b,1}=f|_{z=0}.$
\end{rmk}
\end{proof}

 \section{Preliminaries III: Projection operators.}\label{sec-projection}
 \subsection{Definition of the projection.} %and reformulation the of equations.}
We define the projection operator
$\mathbb{Q}_t$:
\begin{equation}
   \begin{aligned}
   \mathbb{Q}_t: \quad &
L^2(\mS,\,\d \cV_t)^3\rightarrow L^2(\mS,\,\d \cV_t)^3\\
&\qquad\qquad f \rightarrow \mathbb{Q}_t f=\nabla^{\vp}\vr
   \end{aligned} 
\end{equation}
where $\vr$ satisfies the elliptic equation with mixed boundary condition:
\beq\label{defofQ}
\left\{
\begin{array}{l}
     -\Delta^{\vp}\vr=-\div^{\vp}f \quad \text{ in } \mS \\[5pt]
     \vr|_{z=0}=0\\[5pt]
     \p_{z}^{\vp}\vr|_{z=-1}=f\cdot e_3
\end{array}
\right.
\eeq
where $e_3=(0,0,1)^t.$
We define also the projection 
\beq\label{defofP}
\mathbb{P}_t=\text{Id}-\mathbb{Q}_t.
\eeq
Let us notice that $\bbp,\bbq$ depends actually on $\vp(t,\cdot),$ but we used a lighten notation.
\begin{rmk}
Let us notice that the definition of the projection $\bbq$ is not the same as the standard Leary projection where only the Neumann boundary condition is involved. Nevertheless, the definition \eqref{defofQ} is classical in free boundary problems, one can refer for example to \cite{MR611750}.
%Indeed, since we split the velocity into the compressible part $(\nabla^{\vp}\Psi)$ and incompressible part ($v$) by these projections, using Neumann boundary condition ($\nabla^{\vp} \Psi\cdot\bN=u\cdot\bN$) on the upper boundary leads to $v\cdot\bn|_{z=0}=0,$ this contradict with the our expectation that the limiting equation of the surface would be governed by  $\p_t h^0+v^0\cdot\bn^0=0.$ There are also some advantages imposing vanishing Dirichlet boundary condition on the upper boundary. For instance, one can control with the aid of Poincar\'e inequality $\nabla^{\vp}\Psi$ by only $\div^{\vp}u,$ so that one can prove that $\il\nabla^{\vp}\Psi\il_{0,\infty,t}=\cO(\ep)$ which is necessary in our argument. 
%In the definition of %$\bbq,$ compressible part \eqref{defofQ}, we use  Dirichlet boundary condition rather than Neumann boundary condition on the upper boundary mainly due to the following reasons:\\
 
\end{rmk}
\begin{rmk}
We remark that these two projectors are time-dependent since $\vp$ depends on $t.$
One also notes that in general, $\bbp \nabla^{\vp}\neq 0, \bbq \nabla^{\vp}\neq \nabla^{\vp}.$
These facts will lead to some extra commutators when we act the projection to the equations $\eqref{FCNS2}_2.$ 
\end{rmk}

Let us set $v=\bbp u, \nabla^{\vp}\Psi=\bbq u.$ Applying the $\bbp$ projection on the velocity equation $\eqref{FCNS2}_2,$ one gets:
\beqs
\bar{\rho}\pt^{\vp}v%-\mu\Delta^{\vp} v
+\bbp\nabla^{\vp}(\sigma/{\ep}-2(\mu+\lambda)\div^{\vp}u)=-\bbp(f-\mu \Delta^{\vp} v)-\bar{\rho}[\bbp,\pt^{\vp}]u
\eeqs
where $$f=\f{g_2-\bar{\rho}}{\ep}\ep\pt^{\vp}u+g_2u\cdot\nabla^{\vp}u.$$
By definition $\bbp\nabla^{\vp}$ can be expressed as a gradient, we thus denote $$\nabla^{\vp}\pi=\bbp\nabla^{\vp}(\sigma/{\ep}-2(\mu+\lambda)\div^{\vp}u).$$
%For the convenience of the energy estimates,
To shorten the notation, we denote further $$\nabla^{\vp}q=-\bbq (f-\mu\Delta^{\vp} v).$$ Therefore, the above equations read:
\beq\label{eqofv1}
\bar{\rho}\pt^{\vp}v-\mu\Delta^{\vp} v+\nabla^{\vp}\pi=-(f+\nabla^{\vp}q+\bar{\rho}[\bbp,\pt^{\vp}]u).
\eeq
We are now in position to compute the boundary values of $v.$ On the bottom, in light of \eqref{bebdry2} and the fact $\p_z^{\vp}\Psi=u_3,$ we get that
\beq\label{v-bdry-bot}
v_3|_{z=-1}=0, \quad \p_z^{\vp}v_{\tau}|_{z=-1}=\p_z^{\vp}u_{\tau}|_{z=-1}-\nabla_{\tau}^{\vp}\p_z^{\vp}\Psi|_{z=-1}=\f{a}{\mu} u_{\tau}|_{z=-1}.
\eeq
where $\nabla_{\tau}^{\vp}=(\p_1^{\vp},\p_2^{\vp},0)^{t}, f_{\tau}=(f_1,f_2,0)^t.$
Note that $\nabla_{\tau}^{\vp}=(\p_1,\p_2,0)^t$ on the boundary $\{z=-1\}$ since $\p_{\tau}\vp|_{z=-1}=0.$

On the upper boundary, one first notices that by definition, $\pi|_{z=0}=\sigma/{\ep}-2(\mu+\lambda)\div^{\vp}u.$ Therefore, with the aid of the condition \eqref{upbdry2}, we find that:
\beq\label{v-bdry-up}
(2\mu S^{\vp}v-\pi \text{Id}_3)\bN|_{z=0}=2\mu(\div^{\vp}u \text{Id}_3-(\nabla^{\vp})^2\Psi)\bN|_{z=0}.
\eeq

\subsection{Elliptic estimates}
In this section, we establish some useful elliptic estimates in the conormal setting.
We first consider the problem:
\beq\label{elliptic0}
\left\{
\begin{array}{l}
     -\Delta^{\vp}\vr=-\div^{\vp}\tilde{F} \\
     \vr|_{z=0}=0\\[5pt]
     \p_{z}^{\vp}\vr|_{z=-1}=\tilde{F}\cdot e_3+g
\end{array}
\right.
\eeq
where $e_3=(0,0,1)^{t},$ $\tilde{F}, g$ are given source terms.
To perform elliptic estimates, it would be convenient to write it in a more explicit way. By a straightforward calculation, one finds that:
$$\div^{\vp}(\cdot)=\f{1}{\p_z\vp}\div(P\cdot),\quad \nabla^{\vp}=\f{1}{\p_z\vp} P^{*}\nabla^{\vp},\quad\Delta^{\vp}=\f{1}{\p_z{\vp}}\div(E\nabla)$$
where 
\beq\label{defPE}
P=\left(\begin{array}{ccc}
    \p_z\vp &0&0  \\
     0&\p_z\vp&0\\
     -\p_1\vp&-\p_2\vp&1
\end{array}
\right),\quad
E=\f{1}{\p_z{\vp}}PP^{\star}
\eeq
%Note that $E$ is uniformly elliptic. Indeed, if $\il \nabla\vp\il_{\infty,t}\leq 1/c_0,{\p_z\vp}\geq c_0,$ then there exists $\delta(1/c_0)$ such that for any vectors $X\in\mR^3,$ $E X\cdot X \geq \delta |X|^2.$
Denote $F=P\tilde{F}$, 
the equation \eqref{elliptic0} is then equivalent to the following elliptic problem:
\beq\label{elliptic1}
\left\{
\begin{array}{l}
     -\div(E\nabla\vr)=-\div F \\
     \vr|_{z=0}=0\\[5pt]
     (E\nabla\vr\cdot e_3)|_{z=-1}=F_3^{b,2}+g
\end{array}
\right.
\eeq
where $F_3^{b,2}=F^{b,2}\cdot e_3.$
In this paragragh, we study the elliptic equations for a given time $t.$
%One should keep in mind that the above elliptic equation is satisfied for each time $t.$ The following lemma is dedicated to various elliptic estimates.
 
\begin{lem}[Elliptic estimates]\label{lemelliptic}
  Suppose that $\il \nabla\vp\il_{\infty,t}\leq 1/c_0,{\p_z\vp}\geq c_0,$ we have the following  estimates:
  for any $k\geq 0,$
\beq\label{elliptic1.5}
   \|\nabla\vr(t)\|_{H_{co}^{k+1}}+
   \|\nabla^2\vr(t)\|_{H_{co}^{k}}\lesssim \Lambda\big(\f{1}{c_0},|h|_{k+2,\infty,t}\big)\big(\|\div F(t)\|_{H_{co}^{k}}+|(F_3^{b,2}+g)(t)|_{\tilde{H}^{k+\f{1}{2}}
   }\big),
  \eeq
 %The following $L_t^{\infty}L^2(\mS)$ estimates hold 
 and for $j+l=k,l\geq 1,j\geq 0,$
  \beq \label{elliptic2}
  \begin{aligned}
    \|\nabla\vr(t)\|_{\cH^{j,l}}&\lesssim\Lambda\big(\f{1}{c_0},\il\nabla\vr\il_{[\f{k}{2}]-1,\infty,t}+|h |_{[\f{k+3}{2}],\infty,t}\big)|h(t)|_{\tilde{H}^{k+\f{1}{2}}}\\
    &
    \qquad+\Lambda\big(\f{1}{c_0},|h |_{[\f{k+3}{2}],\infty,t}\big)
    \big(\|F(t)\|_{\cH^{j,l}}+|g(t)|_{\tilde{H}^{k-\f{1}{2}}}\big),
\end{aligned}
  \eeq
  \beq\label{elliptic2.5}
  \begin{aligned}
      \|\nabla\vr(t)\|_{\cH^{j,l}}&\lesssim
     \Lambda\big(\f{1}{c_0},\il\nabla\vr\il_{[\f{k}{2}]-1,\infty,t}+|h |_{[\f{k+3}{2}],\infty,t}\big)|h(t)|_{\tilde{H}^{k+\f{1}{2}}}\\
      &+\Lambda\big(\f{1}{c_0},|h |_{[\f{k+3}{2}],\infty,t}\big)\big(\|\div F(t)\|_{\cH^{j,l-1}}+|(F_3^{b,2},g)(t)|_{\tilde{H}^{k-\f{1}{2}}}\big),
 \end{aligned}
  \eeq
 \beq\label{elliptic3}
 \begin{aligned}
  \|\nabla^2\vr(t)\|_{\cH^{j,l}}&\lesssim 
 \Lambda\big(\f{1}{c_0}, |h |_{[\f{k+5}{2}],\infty,t}\big)\big(\|\div F(t)\|_{\cH^{j,l}}+|(F_3^{b,2}, g)(t)|_{\tilde{H}^{k+\f{1}{2}}}\big)\\
 &+ \Lambda\big(\f{1}{c_0}\il\nabla\vr\il_{[\f{k-1}{2}],\infty,t}+|h |_{[\f{k+5}{2}],\infty,t}\big)\big(\il\nabla\vr\il_{0,\infty,t}|h(t)|_{\tilde{H}^{k+\f{3}{2}}}+|h(t)|_{\tilde{H}^{k+\f{1}{2}}}\big),
\end{aligned} 
 \eeq
  %\beq\label{elliptic3.1}\begin{aligned} \|{\ep^{-\f{1}{2}}}\nabla\vr(t)\|_{\cH^{j+1,l}}&\lesssim\Lambda\big(\f{1}{c_0},\il\ep^{-\f{1}{2}}\nabla\vr\il_{[\f{k}{2}]-1,\infty,t}+|h |_{[\f{k+3}{2}],\infty,t}\big)|(h,\ep\pt h)(t)|_{\tilde{H}^{k+\f{1}{2}}}\\ &\qquad+\ep^{-\f{1}{2}}\Lambda\big(\f{1}{c_0},|h |_{[\f{k+3}{2}],\infty,t}\big)\big(\| \div F(t)\|_{\cH^{j+1,l-1}}+|(F_3^{b,2},g)(t)|_{\tilde{H}^{k+\f{1}{2}}}\big),\end{aligned} \eeq
   \beq\label{elliptic3.25}
 \begin{aligned}
&\ep^{\f{1}{2}} \|\pt \nabla \vr(t)\|_{\cH^{j,l}}
\lesssim \Lambda\big(\f{1}{c_0},|h|_{k+1,\infty,t}\big)\big(\|\ep^{\f{1}{2}}\pt F(t)\|_{\cH^{j,l}}+\ep^{\f{1}{2}}|\pt g(t)|_{\tilde{H}^{k-\f{1}{2}}})\\
&\qquad\quad+\ep^{\f{1}{2}}\Lambda\big(\f{1}{c_0},\il\nabla\vr\il_{1,\infty,t}+|\pt h|_{k-1,\infty,t}+|h|_{k,\infty,t}\big)(|\pt h(t)|_{\tilde{H}^{k+\f{1}{2}}}+\|\nabla\vr(t)\|_{H_{co}^k}),
 \end{aligned}
 \eeq
 \beq\label{elliptic3.3}
  \begin{aligned}
      &\|{\ep^{\f{1}{2}}}\pt\nabla\vr(t)\|_{\cH^{j,l}}\lesssim \Lambda\big(\f{1}{c_0},|h |_{[\f{k+3}{2}],\infty,t}\big)\big(\|\ep^{\f{1}{2}}\pt\div F(t)\|_{\cH^{j,l-1}}+|\ep^{\f{1}{2}} \pt(F_3^{b,2}, g)
      (t)|_{\tilde{H}^{k-\f{1}{2}}}\big)
    \\
      &\qquad  +\Lambda\big(\f{1}{c_0}, \|\ep^{-\f{1}{2}}\nabla\vr\|_{[\f{k}{2}],\infty,t}+|(h,\pt h,)|_{[\f{k+3}{2}],\infty,t}\big)(|(\ep\pt h,h)(t)|_{\tilde{H}^{k+\f{1}{2}}}+\ep^{\f{1}{2}}\|\nabla\vr(t)\|_{H_{co}^k}).
 \end{aligned}
  \eeq
\end{lem}
\begin{rmk}
We shall use \eqref{elliptic3.25} when $k\leq m-3$ since as will be seen later, $|h|_{m-2,\infty,t}$ can be uniformly controlled. The inequality \eqref{elliptic3.3} will be used when $m-3\leq k\leq m-1.$ 
\end{rmk}
\begin{proof}
We first notice that by using assumptions: $\il \nabla\vp\il_{\infty,t}\leq 1/c_0,{\p_z\vp}\geq c_0,$ $E$ is uniformly elliptic, that is, one can find $\iota(1/c_0)$ such that for any vectors $X\in\mR^3,$ $E X\cdot X \geq \iota |X|^2.$
The inequality \eqref{elliptic1.5} can be proved easily by the variational arguments
and the use of  Poincar\'e inequality:
$$\|\vr(t)\|_{L^2(\mS)}\leq C\|\nabla\vr(t)\|_{L^2(\mS)}.$$
Note that the generic constant $C$ is independent of $t$ and $\ep.$
More precisely, by testing \eqref{elliptic1} by $\vr(t),$ we easily get that:
\beqs
\begin{aligned}
\delta\|\nabla\vr(t)\|_{L^2(\mS)}\leq\int_{\mS}E\nabla\vr(t)\cdot\nabla \vr(t)\,\d x&=-\int_{\mS}\vr(t) \div F(t)\,\d x+\int_{z=-1}(F_3^{b,2}+g)(t)\vr(t)\,\d y\\
&\leq \f{\delta}{2}\|\nabla\vr(t)\|_{L^2(\mS)}+C_{\delta}(\|\div F(t)\|_{L^2(\mS)}+|(F_3^{b,2},g)(t)|_{H^{-\f{1}{2}}}).
\end{aligned}
\eeqs
The estimates of the higher-order norms $\|\nabla\vr(t)\|_{H^{k+1}}$ can be obtained again from variational arguments and commutator estimates. We skip them since they are essentially included in the proof of other inequalities (for instance \eqref{elliptic2} and \eqref{elliptic3}).

We now begin to prove \eqref{elliptic2}.
Let $\alpha=(j,\alpha'),Z^{\alpha}=(\ep\pt)^j Z_1^{\alpha_1}Z_2^{\alpha_2}Z_3^{\alpha_3}.$ 
If $\alpha_3\neq 0,$ taking $Z^{\alpha}$ derivatives on the equation shall destroy the divergence form. The trick to avoid this problem is to use another vector field $\tilde{Z}_3=Z_3+\p_z\phi \text{Id}$, such that:
$\tilde{Z}_3\p_z=\p_z Z_3.$ By induction, we have for any $\alpha_3\geq 1,$ $\tilde{Z}_3^{\alpha_3}\p_z=\p_z Z_3^{\alpha_3},$ which yields
$$\tilde{Z}^{\alpha}\p_z=\colon (\ep\pt)^j Z_1^{\alpha_1}Z_2^{\alpha_2}\tilde{Z_3}^{\alpha_3}\p_z=\p_z Z^{\alpha}.$$
It is useful to notice further that for any $f,$
\beq\label{diffofvectors}
\|(\tilde{Z}^{\alpha}-Z^{\alpha})f(t)\|_{L^2(\mS)}\lesssim \|f(t)\|_{\cH^{j,l-1}}.
\eeq
Taking $\tilde{Z}^{\alpha}$ derivative on the equation \eqref{elliptic1}, we find that:
\beq\label{highelliptic}
\left\{
\begin{array}{l}
     -\div\big(E (Z^{\alpha}\nabla\vr)\big)=\div ([Z^{\alpha},E]\nabla\vr-Z^{\alpha}F)+\div\big(\tilde{Z}^{\alpha}-Z^{\alpha})[(E\nabla\vr)_{\tau}-F_{\tau}]\big),\\
     Z^{\alpha}\vr|_{z=0}=0,\\[5pt]
   Z^{\alpha} (E\nabla\vr)\cdot e_3|_{z=-1}=\mathbb{I}_{\{\alpha_3= 0\}} Z^{\alpha}(F_3^{b,2}+g).
\end{array}
\right.
\eeq
Note that we denote by $X_{\tau}=(X_1,X_2, 0)^{t}$ the horizontal components of a three dimensional vector $X.$
Testing equation \eqref{highelliptic} by $Z^{\alpha}\vr,$ we obtain: %(we drop the $t$ dependence temporarily):
\beq\label{sec3:eq0}
\begin{aligned}
&\delta\|Z^{\alpha}\nabla\vr\|_{L^2}^2\leq \int_{\mS} E Z^{\alpha}\nabla\vr Z^{\alpha}\nabla\vr\,\d x\\
=&\int_{\mS} E Z^{\alpha}\nabla\vr\cdot [Z^{\alpha},\nabla]\vr\,\d x-\int_{\mS}[Z^{\alpha},E]\nabla\vr\cdot\nabla Z^{\alpha}\vr\,\d x\\
-&\int_{\mS}
(\tilde{Z}^{\alpha}-Z^{\alpha})\big((E\nabla\vr)_{\tau}-F_{\tau}\big)\cdot \nabla Z^{\alpha}\vr\,\d x+\int_{\mS}Z^{\alpha}F\cdot\nabla Z^{\alpha}\vr\,\d x-\int_{z=-1}\mathbb{I}_{\{\alpha_3= 0\}}Z^{\alpha}g Z^{\alpha}\vr\,\d y.
\end{aligned}
\eeq
Combined with Young's inequality, property \eqref{diffofvectors} and the trace inequality \eqref{trace}, this yields
\beq\label{ineqelliptic}
\|Z^{\alpha}\nabla\vr(t)\|_{L^2(\mS)}^2\lesssim \|F(t)\|_{\cH^{j,l}}^2+|g(t)|_{\tilde{H}^{k-\f{1}{2}}}^2+
\|(\nabla\vr,E\nabla\vr)(t)\|_{\cH^{j,l-1}}^2+\|[Z^{\alpha},E]\nabla\vr(t)\|_{L^2(\mS)}^2.
\eeq
It follows from the product and commutator estimates \eqref{crudepro}, \eqref{crudecom} that:
\begin{equation}\label{sec3:eq1} 
\begin{aligned}
\|\nabla\vr(t)\|_{\tilde{H}^{j,l}}&\leq \Lambda(1/c_0)
\big(\|F(t)\|_{\cH^{j,l}}+|g(t)|_{\tilde{H}^{k-\f{1}{2}}}+\|\nabla\vr(t)\|_{\cH^{j,l-1}}\\
&+\|\nabla\vr(t)\|_{\cH^{j,l-1}\cap\cH^{j-1,l}}
\|E\|_{[\f{k+1}{2}],\infty,t}+\|E(t)\|_{\tilde{H}^{j,l}}\il\nabla\vr\il_{[\f{k}{2}]-1,\infty,t}
\big).
\end{aligned}
\end{equation}
By Lemma \ref{exth} and the expression of $E$ in \eqref{defPE}, we get
\beq\label{esofE}
\il E\il_{n,\infty,t}\lesssim\Lambda \big(\f{1}{c_0}, |h|_{n+1,\infty,t}\big),\quad \|E(t)\|_{\cH^{j,l}}\lesssim \Lambda \big(\f{1}{c_0}, |h|_{[\f{k}{2}]+1,\infty,t}\big)|h(t)|_{\tilde{H}^{k+\f{1}{2}}}.
\eeq
Inserting \eqref{esofE} into \eqref{sec3:eq1}, we arrive at:
\beq
\begin{aligned}
\|\nabla\vr(t)\|_{\tilde{H}^{j,l}}&\leq \Lambda\big(\f{1}{c_0},\il\nabla\vr\il_{[\f{k}{2}]-1,\infty,t}+|h |_{[\f{k+3}{2}],\infty,t}\big)|h(t)|_{\tilde{H}^{k+\f{1}{2}}}\\
&+\Lambda\big(\f{1}{c_0},|h |_{[\f{k+3}{2}],\infty,t}\big) \big(\|F(t)\|_{\cH^{j,l}}+|g(t)|_{\tilde{H}^{k-\f{1}{2}}}+\|\nabla\vr(t)\|_{\cH^{j,l-1}\cap\cH^{j-1,l}}\big).
\end{aligned}
\eeq
The inequality \eqref{elliptic2} then follows by induction on $j$ and $l.$

To get
\eqref{elliptic2.5}, it suffices to observe that the last three terms in \eqref{sec3:eq0} can indeed be replaced by:
\beqs
\begin{aligned}
&\int_{\mS}Z^{\tilde{\alpha}}\div F\p_y Z^{\alpha}\vr\,\d x-\int_{z=-1}Z^{\alpha}(F_3+g)Z^{\alpha}\vr\,\d y,\quad \text{if}\quad \alpha_3=0,Z^{\alpha}=\p_y Z^{\tilde{\alpha}}.\\
&-\int_{\mS}
(\tilde{Z}^{\alpha}-Z^{\alpha})(E\nabla\vr)_h\cdot \nabla Z^{\alpha}\vr\,\d x\int_{\mS}Z^{\tilde{\alpha}}\div F (Z_3+\p_z\phi) (Z^{\alpha}\vr)\,\d x, \quad \text{if}\quad \alpha_3\neq 0,Z^{\alpha}=Z_3 Z^{\tilde{\alpha}}.
\end{aligned}
\eeqs

To prove \eqref{elliptic3}, we first estimate $\|\p_y\nabla\vr(t)\|_{\cH^{j,l}}$ and then use the equation itself to recover 
$\|\p_z^2\vr(t)\|_{\cH^{j,l}}.$
The estimate of $\|\p_y\nabla\vr(t)\|_{\cH^{j,l}}$
is almost identical %but slightly different 
to that of \eqref{elliptic2.5}. For this one, we only need to 
distinguish the highest derivatives hitting on $E$ (or finally on $h$). 
Hence, when estimating the term $[Z^{\alpha}\p_y, E]\nabla\vr,$
we write 
\beqs
[Z^{\alpha}\p_y, E]\nabla\vr=(Z^{\alpha}\p_y E)\nabla\vr+\text{ other terms}
\eeqs
and control the first term as
\beqs
\|(Z^{\alpha}\p_y E)\nabla\vr(t)\|_{L^2(\mS)}
\lesssim \il\nabla\vr\il_{0,\infty,t}\Lambda(\f{1}{c_0},|h|_{[\f{k}{2}]+2,\infty,t})|h(t)|_{\tilde{H}^{k+\f{3}{2}}}.
\eeqs
%The inequality \eqref{elliptic3.1} is also a  small variation on \eqref{elliptic2.5}, we thus omit the proof.

We now sketch the proof of \eqref{elliptic3.25} and \eqref{elliptic3.3}. For \eqref{elliptic3.25}, 
we first have the following inequality analogues to \eqref{ineqelliptic}.
\beqs\label{ineqelliptic-1}
\begin{aligned}
\|\ep^{\f{1}{2}}Z^{\alpha}\pt\nabla\vr(t)\|_{L^2(\mS)}^2&\lesssim \|\ep^{\f{1}{2}} \pt F(t)\|_{\cH^{j,l}}^2+|\ep^{\f{1}{2}} \pt g(t)|_{\tilde{H}^{k-\f{1}{2}}}^2\\
&+\|\ep^{\f{1}{2}} \pt(\nabla\vr,E\nabla\vr)(t)\|_{\cH^{j,l-1}}^2+\|[\ep^{\f{1}{2}} \pt Z^{\alpha},E]\nabla\vr(t)\|_{L^2(\mS)}^2,
\end{aligned}
\eeqs
%The only difference is that, when dealing with the product and commutator terms, we distinguish the term involving the highest, second highest derivatives of $E$ and the others. For instance, we write
%$$\ep^{\f{1}{2}}[Z^{\alpha}\pt,E]\nabla\vr=(\ep^{\f{1}{2}}Z^{\alpha}\pt E)\nabla\vr+\ep^{\f{1}{2}}Z^{\alpha}E (\pt\nabla\vr)+\sum_{|\beta|=|\alpha|-1}C_{\alpha,\beta}Z^{\beta}\ep^{\f{1}{2}}\pt E Z\nabla\vr+\text{other terms}$$ and control the first three terms as:
where the last two terms can be bounded in a rather rough way:
\beqs
\begin{aligned}
&\|\ep^{\f{1}{2}}\pt (E\nabla\vr)(t)\|_{\cH^{j,l}}\lesssim \| \ep^{\f{1}{2}}\pt\nabla\vr(t)\|_{\cH^{j,l-1}}\Lambda(\f{1}{c_0},|h|_{k,\infty,t})\\
&\qquad\qquad
+\ep^{\f{1}{2}}\Lambda\big(\f{1}{c_0}, \il\nabla\vr\il_{0,\infty,t}+|\pt h|_{k-1,\infty,t}\big)(|\pt h(t)|_{\tilde{H}^{k-\f{1}{2}}}+\|\nabla\vr(t)\|_{H_{co}^{k-1}}),
\end{aligned}
\eeqs
\beqs
\begin{aligned}
\ep^{\f{1}{2}}\|[\pt Z^{\alpha},E]\nabla\vr(t)\|_{L^2(\mS)}&\leq 
\ep^{\f{1}{2}}\|Z^{\alpha}(\pt E\nabla\vr)(t)\|_{L^2(\mS)}+\ep^{\f{1}{2}}\|[Z^{\alpha},E]\pt\nabla\vr(t)\|_{L^2(\mS)}\\
&\lesssim  \il\ep^{\f{1}{2}}\pt \nabla\vr\il_{\cH^{j,l-1}\cap \cH^{j-1,l}}\Lambda\big(\f{1}{c_0},|h|_{[k,\infty,t}\big)\\
&\quad+\ep^{\f{1}{2}}\Lambda\big(\f{1}{c_0},\il\nabla\vr\il_{1,\infty,t}+|\pt h|_{k-1,\infty,t}\big)(|\pt h(t)|_{\tilde{H}^{k+\f{1}{2}}}+\|\nabla\vr(t)\|_{H_{co}^k}),
\end{aligned}
\eeqs
The inequality \eqref{elliptic3.25} then follows from induction on $j,l.$
For \eqref{elliptic3.3}, similar to \eqref{elliptic2.5}, we have:
\beqs\label{ineqelliptic-2}
\begin{aligned}
\|\ep^{\f{1}{2}}Z^{\alpha}\pt\nabla\vr(t)\|_{L^2(\mS)}^2&\lesssim \|\ep^{\f{1}{2}} \pt \div F(t)\|_{\cH^{j,l}}^2+|\ep^{\f{1}{2}}(F_3^{b,2},\pt g)(t)|_{\tilde{H}^{k-\f{1}{2}}}^2\\
&\quad +\|\ep^{\f{1}{2}} \pt(\nabla\vr,E\nabla\vr)(t)\|_{\cH^{j,l-1}}^2+\|[\ep^{\f{1}{2}} \pt Z^{\alpha},E]\nabla\vr(t)\|_{L^2(\mS)}^2.
\end{aligned}
\eeqs
The last two terms are bounded as 
\beqs
\begin{aligned}
\|\ep^{\f{1}{2}} &\pt(\nabla\vr,E\nabla\vr)(t)\|_{\cH^{j,l-1}}^2+\|[\ep^{\f{1}{2}} \pt Z^{\alpha},E]\nabla\vr(t)\|_{L^2(\mS)}^2\\
&\lesssim \ep^{\f{1}{2}}\|\pt\nabla\vr(t)\|_{\cH^{j,l-1}\cap \cH^{j-1,l}}\Lambda\big(\f{1}{c_0},|h|_{[\f{k+3}{2}],\infty,t}\big)\\
&+\Lambda\big(\f{1}{c_0}, \|\ep^{-\f{1}{2}}\nabla\vr\|_{[\f{k}{2}],\infty,t}+|(\pt h, h)|_{[\f{k+3}{2}],\infty,t})\big(|(\ep\pt h,h)(t)|_{\tilde{H}^{k+\f{1}{2}}}+\ep^{\f{1}{2}}\|\nabla\vr(t)\|_{H_{co}^k}\big).
\end{aligned}
\eeqs
%Equipped with the above two inequalities,
We obtain \eqref{elliptic3.3} again by induction on $j$ and $l.$

\end{proof}
\begin{rmk}
Similar to 
\eqref{elliptic2}, \eqref{elliptic3.3} the following estimate also hold, for $j+l=k\geq 3,$
\beq\label{elliptic-useful}
\begin{aligned}
 \|\nabla\vr(t)\|_{H_{co}^k}&\lesssim 
 \Lambda\big(\f{1}{c_0},|h|_{k-n,\infty,t} \big)\big(\|F(t)\|_{H_{co}^{k}}+|g(t)|_{\tilde{H}^{k-\f{1}{2}}}\big)\\
 &+\il\nabla \vr\il_{n,\infty,t}\Lambda\big(
 \f{1}{c_0},|h|_{[\f{k}{2}]+1,\infty,t}\big)|h(t)|_{\tilde{H}^{k+\f{1}{2}}}\, (n=0,1),
 \end{aligned}
 \eeq
\beq\label{elliptic2.5-1}
  \begin{aligned}
    & \ep^{\f{1}{2}} \|\pt\nabla\vr(t)\|_{\cH^{j,l}}
    +\ep^{\f{1}{2}}\|\pt \nabla^2\vr(t)\|_{\cH^{j,l-1}}\\
  &  \lesssim \Lambda\big(\f{1}{c_0},|h |_{k,\infty,t}\big)\big(\|\ep^{\f{1}{2}}\pt F(t)\|_{\cH^{j}} + \|\ep^{\f{1}{2}}\pt\div F(t)\|_{\cH^{j,l-1}}\mathbb{I}_{\{l\geq 1\}}+|\ep^{\f{1}{2}}\pt(F_3^{b,2},g)(t)|_{\tilde{H}^{k-\f{1}{2}}}\big)\\
  &+ \Lambda\big(\f{1}{c_0},|h |_{k,\infty,t}\big)\il\ep^{\f{1}{2}}\pt\nabla\vr\il_{0,\infty,t}|h(t)|_{\tilde{H}^{k+\f{1}{2}}}\\
     & +\ep^{\f{1}{2}}
     \Lambda\big(\f{1}{c_0},\il \ep^{-\f{1}{2}}
     \nabla\vr\il_{[\f{k+1}{2}],\infty,t}+|(h,\pt h) |_{[\f{k+3}{2}],\infty,t}\big)\big(|\pt h(t)|_{\tilde{H}^{k+\f{1}{2}}}+\|
     \nabla\vr(t)\|_{H_{co}^k}\big).
 \end{aligned}
  \eeq
\end{rmk}
\begin{cor}
Let $\nabla^{\vp}\Psi=\bbq u$ be the compressible part of the velocity, we have the following two estimates:
\beq\label{sec-normal-Psi}
\begin{aligned}
\|\nabla\nabla^{\vp}\Psi\|_{L_t^2H_{co}^{m-1}}+\|\nabla^{\vp}\Psi\|_{L_t^2H_{co}^{m}}&\lesssim  (T+\ep)^{\f{1}{2}}\lae,\\
\end{aligned}
\eeq
\beq\label{ptpsiL2}
\|\ep^{\f{1}{2}}\pt\nabla^{\vp}\Psi\|_{L_t^2H_{co}^{m-1}}\lesssim \Lambda\big(\f{1}{c_0},|h|_{L_t^{\infty}\tilde{H}^{m-\f{1}{2}}}\big)\|\ep^{\f{1}{2}}\pt\div^{\vp} u\|_{L_t^2H_{co}^{m-2}}+(T+\ep)^{\f{1}{2}}\lae,
\eeq
\beq\label{psiLinfty}
\begin{aligned}
&\ep^{\f{1}{2}}\|\pt\nabla^{\vp}\Psi\|_{L_t^{\infty}H_{co}^{m-2}}+\ep^{\f{1}{2}}\|\pt\nabla\nabla^{\vp}\Psi\|_{L_t^{\infty}H_{co}^{m-3}}\\
&\lesssim \Lambda\big(\f{1}{c_0},|h|_{L_t^{\infty}\tilde{H}^{m-\f{1}{2}}}\big)\big(\|\ep^{\f{1}{2}}\pt\div^{\vp}u\|_{L_t^{\infty}H_{co}^{m-3}}+\|\ep^{\f{1}{2}}\pt u\|_{L_t^{\infty}\cH^{m-2}}\big)
+(T+\ep)^{\f{1}{2}}\lae.
\end{aligned}
\eeq
\end{cor}
\begin{proof}
We begin with the proof of \eqref{sec-normal-Psi}. Let us detail the estimate of $\|\nabla\nabla^{\vp}\Psi\|_{\hco^{m-1}}$, the other term can be obtained by similar arguments.
It suffices to show that:
\beq\label{sec-nor-psi-1}
\begin{aligned}
\|\nabla\nabla^{\vp}\Psi\|_{L_t^2H_{co}^{m-1}}&\lesssim \Lambda\big(\f{1}{c_0},|h|_{[\f{m}{2}]+2,\infty,t}\big)\|\div^{\vp}u\|_{L_t^2H_{co}^{m-1}}\\
&+\lae(|h|_{\htlde^{m-\f{1}{2}}}+|\ep^{\f{1}{2}}h|_{L_t^2\tilde{H}^{m+\f{1}{2}}}).
\end{aligned}
\eeq
which leads to \eqref{sec-normal-Psi}.
By definition, $\Psi$ solves the elliptic equation:
 \beq\label{eq-comp}
 \left\{
\begin{array}{l}
 \div(E\nabla\Psi)=\div(P u),\\
 \Psi|_{z=0}=0,\\
 \p_{\bn}\Psi|_{z=-1}=0.
 \end{array}
 \right.
 \eeq
We apply \eqref{elliptic3} for $F=P u, \, \div F=\p_z\vp \,\div^{\vp}u,\, F_3^{b,2}=g=0$ to get:
\begin{align*}
\|\nabla^2\Psi\|_{L_t^2H_{co}^{m-1}}&\lesssim \Lambda\big(\f{1}{c_0},|h|_{[\f{m}{2}]+2,\infty,t}\big)\|\p_z\vp\,\div^{\vp}u\|_{L_t^2H_{co}^{m-1}}\\
&\quad+\Lambda\big(\f{1}{c_0},|h|_{[\f{m}{2}]+2,\infty,t}+\il\ep^{-\f{1}{2}}\nabla\Psi\il_{[\f{m}{2}]-1,\infty,t}\big)
(|h|_{\htlde^{m-\f{1}{2}}}+|\ep^{\f{1}{2}}h|_{L_t^2\tilde{H}^{m+\f{1}{2}}}).
\end{align*}
By the product estimate \eqref{crudepro}, we find
\beq\label{sec3:eq2}
\begin{aligned}
&\|\nabla\nabla^{\vp}\Psi\|_{L_t^2H_{co}^{m-1}}\lesssim \|\nabla^2\Psi\|_{L_t^2H_{co}^{m-1}}+\|\nabla\big(\f{\bN}{\p_z\vp}\p_z\Psi\big)\|_{{L_t^2H_{co}^{m-1}}}\\
&\lesssim \Lambda\big(\f{1}{c_0},|h|_{[\f{m}{2}]+2,\infty,t}\big)\|\div^{\vp}u\|_{L_t^2H_{co}^{m-1}}+\Lambda\big(\f{1}{c_0},|h|_{[\f{m}{2}]+2,\infty,t}+\il\ep^{-\f{1}{2}}(\nabla\Psi,\div^{\vp}u)\il_{[\f{m}{2}]-1,\infty,t}%+\il\ep^{-\f{1}{2}}\div^{\vp}u\il_{[\f{m}{2}]-1,\infty,t}
\big)\cdot\\
&\qquad\qquad\qquad\qquad\qquad\qquad\qquad\qquad\qquad(|h|_{L_t^2\tilde{H}^{m-\f{1}{2}}}+\ep^{\f{1}{2}}|h|_{L_t^2\tilde{H}^{m+\f{1}{2}}}).
\end{aligned}
\eeq
Moreover, the Sobolev embedding \eqref{soblev-embed} combined with the inequality \eqref{elliptic1.5} gives for $k\geq 0,$
\beq\label{Linftynablapsi}
\begin{aligned}
\ep^{-\f{1}{2}}\il\nabla\Psi\il_{[\f{m}{2}]-1,\infty,t}\lesssim \ep^{-\f{1}{2}}
(\|\nabla^2\Psi\|_{L_t^{\infty}H_{co}^{[\f{m}{2}]}}+\|\nabla\Psi\|_{L_t^{\infty}H_{co}^{[\f{m}{2}]+1}})\\
\lesssim \Lambda\big(\f{1}{c_0},|h|_{[\f{m}{2}]+2,\infty,t}\big)\|\ep^{-\f{1}{2}}\div^{\vp} u\|_{L_t^{\infty}H_{co}^{[\f{m}{2}]}}.
\end{aligned}
\eeq
%Therefore, 
%\beqs
%\il\nabla\Psi/\ep\il_{0,\infty,t}\lesssim\Lambda\big(\f{1}{c_0},|h|_{3,\infty,t}\big)\|\div^{\vp} u/\ep\|_{L_t^{\infty}H_{co}^{1}}\lesssim \lae.\eeqs
Plugging this inequality into  \eqref{sec3:eq2}, we arrive at \eqref{sec-nor-psi-1}. 

Moreover, by applying \eqref{elliptic3.3},
\eqref{elliptic2.5-1}, \eqref{Linftynablapsi} to the solution of \eqref{eq-comp}, we get \eqref{ptpsiL2} and \eqref{psiLinfty}.
\end{proof}

\begin{cor}\label{corelliptic}
Consider the elliptic system with nontrivial Dirichlet upper boundary condition:
\beq
\left\{
\begin{array}{l}
     -\div(E\nabla\vr)=-\div F, \\
     \vr|_{z=0}=b,\\
     (E\nabla\vr)\cdot e_3|_{z=-1}=F_3^{b,2}+g.
\end{array}
\right.
\eeq
The following estimates hold:
\beq\label{elliptic4.9}
\begin{aligned}
\il\nabla\vr\il_{\infty,t}%&\lesssim \|\nabla\vr\|_{L_t^{\infty}H_{tan}^2}+\|\nabla^2\vr\|_{L_t^{\infty}H_{tan}^1}\\
&\lesssim \Lambda(\f{1}{c_0}, |h|_{3,\infty,t})\big(\|\div F\|_{L_t^{\infty}H_{co}^1}+|b|_{L_t^{\infty}H^{\f{5}{2}}}+|g|_{L_t^{\infty}H^{\f{3}{2}}}\big),
\end{aligned}
\eeq
\beq\label{elliptic5}
\begin{aligned}
\ep^{-\f{1}{2}}\|\nabla\vr(t)\|_{\cH^{j,l}}&\lesssim   \Lambda\big(\f{1}{c_0},\il\ep^{-\f{1}{2}}\nabla\vr\il_{[\f{k}{2}]-1,\infty,t}+|h |_{[\f{k+3}{2}],\infty,t}+\ep^{-\f{1}{2}}|b|_{L_t^{\infty}\tilde{H}^{[\f{k}{2}]+1^{+}}}\big)|h(t)|_{\tilde{H}^{k+\f{1}{2}}}
\\
 &\quad+\ep^{-\f{1}{2}}\Lambda\big(\f{1}{c_0}, |h |_{[\f{k+3}{2}],\infty,t})\big(\|F(t)\|_{\cH^{j,l}}+|b(t)|_{\tilde{H}^{k+\f{1}{2}}}+|g(t)|_{\tilde{H}^{k-\f{1}{2}}}
 \big),
  \end{aligned}
  \eeq
  \beq\label{ellipticuseful-1}
  \begin{aligned}
\|\nabla\vr(t)\|_{H_{co}^k}&\lesssim \Lambda\big(\f{1}{c_0}, |h|_{k-j,\infty,t}\big)\big(\|F(t)\|_{H_{co}^k}+|b(t)|_{H^{k+\f{1}{2}}}+|g(t)|_{L_t^{\infty}H^{k-\f{1}{2}}}\big)\\
&\qquad +\Lambda\big(
 \f{1}{c_0},\il\nabla \vr\il_{j,\infty,t}+|h|_{[\f{k}{2}]+1,\infty,t}\big)|h(t)|_{\tilde{H}^{k+\f{1}{2}}}, \, k\geq 2, j=0 \text{ or } 1,
 \end{aligned}
\eeq
\beq\label{elliptic6}
 \begin{aligned}
&\ep^{\f{1}{2}} \|\pt \nabla \vr(t)\|_{H_{co}^k}
\lesssim \Lambda\big(\f{1}{c_0},|h|_{k+1,\infty,t}\big)\big(\|\ep^{\f{1}{2}}\pt F(t)\|_{H_{co}^k}+|\ep^{\f{1}{2}}\pt b(t)|_{\tilde{H}^{k+\f{1}{2}}}+\ep^{\f{1}{2}}|\pt g(t)|_{\tilde{H}^{k-\f{1}{2}}})\\
&\qquad\quad+\ep^{\f{1}{2}}\Lambda\big(\f{1}{c_0},\il\nabla\vr\il_{1,\infty,t}+|\pt h|_{k-1,\infty,t}+|h|_{k,\infty,t}\big)(|\pt h(t)|_{\tilde{H}^{k+\f{1}{2}}}+\|\nabla\vr(t)\|_{H_{co}^k}),
 \end{aligned}
 \eeq
\end{cor}
\begin{proof}
We introduce the lifting:
$$\vr^H(t,y,z)=\cF^{-1}_{\xi\rightarrow y}(e^{-z^2\langle \xi\rangle^2}\hat{b}(t,\xi))(1+z),$$
and reformulate the problem as:
\beqs
\left\{
\begin{array}{l}
     -\div(E\nabla\vr^L)=-\div (F-E\nabla\vr^H) \\
     \vr^L|_{z=0}=0\\
     \p_{z}\vr^L|_{z=-1}=(F-E\nabla\vr^H)\cdot e_3+g.
\end{array}
\right.
\eeqs
We apply
Lemma \ref{lemelliptic} with %replacing $F$
$F-E\nabla\vr^H.$ 
Note that we use again the product estimate \eqref{crudepro} to bound $E\nabla\vr^H.$ Moreover, Young's inequality and the definition of $\vr^H$ give:
\beqs
\|\nabla\vr^H(t)\|_{\cH^{j,l}}\lesssim |b(t)|_{\tilde{H}^{j+l+\f{1}{2}}},\quad \il\nabla\vr^H\il_{[\f{k}{2}]-1,\infty,t}\lesssim |b|_{[\f{k}{2}],\infty,t}\lesssim |b|_{L_t^{\infty}\tilde{H}^{[\f{k}{2}]+1^{+}}}.
\eeqs
\end{proof}

\section{Regularity of the surface}
In this section, we prove some regularity properties for the surface $h.$ Here and in the sequel, we will denote $m \geq 7$ 
an integer. We also recall that $\cN_{m,T},\cE_{m,T},\cA_{m,T}$ are defined in \eqref{defcN}. 

%and denote $\Lambda$ a polynomial with its arguments and may differ from line to line.

 \begin{lem}\label{lemsurface}
The following regularity estimates hold: %for some $T>0$ and for any
$0<t\leq T,$ %($T$ is independent of $\ep$ and will be fixed later)
   \beq\label{surface2}
  |\pt h|_{L_t^{\infty}\tilde{H}^{m-\f{3}{2}}}+\ep^{\f{1}{2}} |\pt h|_{L_t^{\infty}\tilde{H}^{m-\f{1}{2}}}\lesssim  
 \cE_{m,T}+\cE_{m,T}^2,
  \eeq
  \beq\label{surface3}
\ep^{\f{1}{2}}|\pt^2 h|_{\htlde^{m-\f{3}{2}}}+|\ep^{\f{1}{2}}\pt^2 h|_{L_t^{\infty}\tilde{H}^{m-\f{5}{2}}}
+\sum_{k\leq m-1}|\ep^{\f{1}{2}}(\ep\pt)^k\pt^2 h|_{L_t^2H^{-\f{1}{2}}}
 \lesssim \lae,
 \eeq
 \beq\label{surface1}
  \begin{aligned}
  &|h|_{L_t^{\infty}\tilde{H}^{m-\f{1}{2}}}^2+ \ep|h|_{L_t^{\infty}\tilde{H}^{m+\f{1}{2}}}^2
 % +\ep^{\f{1}{2}}|(\ep\pt)^{m-1}\pt h|_{L_t^{\infty}H^{\f{1}{2}}}^2
 \lesssim
  Y^2_m(0)+T^{\f{1}{2}}\lae.
  \end{aligned}
  \eeq
 where $\Lambda$ denotes a polynomial that may change according to the contexts.
 \end{lem}
 \begin{proof}
 \underline{Proof of \eqref{surface2}:}
We have by using the equation \eqref{surfaceeq2}, the product estimate \eqref{rough-product-bdry}
 the trace inequality 
\eqref{trace} and the definition of $\cE_{m,T}$
that:
\beqs
\begin{aligned}
&\ep^{\f{1}{2}}|\pt h|_{L_t^{\infty}\tilde{H}^{m-\f{1}{2}}}=\ep|(u\cdot \bN)|_{L_t^{\infty}\tilde{H}^{m-\f{1}{2}}}\\
&\lesssim \big(1+|u|_{L_t^{\infty}\tilde{H}^{[\f{m-1}{2}]+\f{1}{2}}}+|h|_{L_t^{\infty}\tilde{H}^{[\f{m}{2}]+\f{3}{2}}}\big)|\ep^{\f{1}{2}}(u,\nabla_y h)|_{L_t^{\infty}\tilde{H}^{m-\f{1}{2}}}\\
&\lesssim (1+\cE_{m,T})(\|\ep^{\f{1}{2}}(u,\nabla u)\|_{L_t^{\infty}H_{co}^{m-1}}+\ep^{\f{1}{2}}|h|_{L_t^{\infty}\tilde{H}^{m+\f{1}{2}}})\lesssim 
\cE_{m,T}+\cE_{m,T}^2.
\end{aligned}
\eeqs
Note that we have $[\f{m-1}{2}]+1\leq m-2, [\f{m}{2}]+\f{3}{2}\leq m-\f{1}{2}$ for $m\geq 5.$
The quantity $ |\pt h|_{L_t^{\infty}\tilde{H}^{m-\f{3}{2}}}$ can be dealt with in the same way, we thus omit the proof.\\[5pt]
\underline{Proof of \eqref{surface3}:}
Let us detail the estimates of the first two terms, the last one can  be controlled by similar calculations.
  Again, we use the equation \eqref{surfaceeq2} for $h$, the product estimate \eqref{rough-product-bdry}, the trace inequality
  \eqref{trace} to obtain that
  \begin{align*}
  \ep^{\f{1}{2}}|\pt h|_{\htlde^{m-\f{1}{2}}}
  &\lesssim 
  |(\ep^{\f{1}{2}}\pt u\cdot\bN, u\cdot \ep^{\f{1}{2}}\pt \bN)|_{\htlde^{m-\f{3}{2}}}\\
 &\lesssim 
 |\ep^{\f{1}{2}}\pt u|_{\htlde^{[\f{m-1}{2}]+\f{1}{2}}}|h|_{L_t^{\infty}\tilde{H}^{m-\f{1}{2}}}
 +(1+|h|_{L_t^{\infty}\tilde{H}^{[\f{m}{2}]+\f{3}{2}}})|\ep^{\f{1}{2}}\pt u|_{\htlde^{m-\f{3}{2}}}\\
 &+ |\ep^{\f{1}{2}}\pt h|_{L_t^{\infty}\tilde{H}^{m-\f{1}{2}}}|u|_{\htlde^{[\f{m}{2}]+\f{3}{2}}}
 +|\ep^{\f{1}{2}}\pt h|_{L_t^{\infty}\tilde{H}^{[\f{m}{2}]+\f{1}{2}}}|u|_{\htlde^{m-\f{3}{2}}}
 \lesssim \lae.
  \end{align*}
  For the second term, we use Equation \eqref{surfaceeq2} and the trace inequality to get:
  \begin{align*}
    |\ep^{\f{1}{2}}\pt^2 h|_{L_t^{\infty}\tilde{H}^{m-\f{5}{2}}}\lesssim   \|\ep^{\f{1}{2}}\pt\p_z (u\cdot\bN)\|_{L_t^{\infty}H_{co}^{m-3}}+\|\ep^{\f{1}{2}}\pt (u\cdot\bN)\|_{L_t^{\infty}H_{co}^{m-2}}.
  \end{align*}
  With the aid of identity \eqref{nor-nor} and the product estimate \eqref{crudepro}, we then find that:
  \begin{align*}
  &|\ep^{\f{1}{2}}\pt^2 h|_{L_t^{\infty}\tilde{H}^{m-\f{5}{2}}}\\
  &\lesssim \lca\big(\|\ep^{\f{1}{2}}\pt\div^{\vp}u\|_{L_t^{\infty}H_{co}^{m-3}}+\| (u,\ep^{\f{1}{2}}\pt u) \|_{L_t^{\infty}H_{co}^{m-2}}+|(h,\ep^{\f{1}{2}}\pt h)|_{L_t^{\infty}\tilde{H}^{m-\f{3}{2}}}\big)\\
  &\lesssim \lae.
  \end{align*}
\underline{Proof of \eqref{surface1}.}
We explain the estimate of 
$ |h|_{L_t^{\infty}H^{m-\f{1}{2}}},$
the control of
$\ep^{\f{1}{2}} |h|_{L_t^{\infty}H^{m+\f{1}{2}}}$
being similar.
Acting $Z^{\alpha}\Lambda_y^{\f{1}{2}}  (|\alpha|\leq m-1, \alpha_3=0)$  
 on \eqref{surfaceeq2}, one obtains:
 \beqs
(\pt+u_y\p_y) (Z^{\alpha}\Lambda_y^{\f{1}{2}}h)-Z^{\alpha}\Lambda_y^{\f{1}{2}}u_3=f=\colon [\Lambda_y^{\f{1}{2}},u_y]Z^{\alpha}\p_y h-\Lambda_y^{\f{1}{2}}\big([Z^{\alpha},u_y]\p_y h\big).
 \eeqs
Multiplying this equation by $ Z^{\alpha}\Lambda_y^{\f{1}{2}}h$ and integrating in space and time, we get that:
 \beq\label{EI-h}
 \begin{aligned}
& |Z^{\alpha}\Lambda_y^{\f{1}{2}}h(t)|_{L_y^2}^2\lesssim |Z^{\alpha} h(0)|_{H^{\f{1}{2}}}^2\\
 &+T^{\f{1}{2}}\Lambda(\il u\il_{1,\infty,t})\big(| u_3|_{L_t^2\tilde{H}^{m-\f{1}{2}}}^2+|f|_{L_t^2L_y^2}^2+| h|_{L_t^{\infty}\tilde{H}^{m-\f{1}{2}}}^2\big)
 \end{aligned}
 \eeq
By the trace inequality \eqref{trace},
\beq\label{h-1}
\big| u_3\big|_{L_t^2\tilde{H}^{m-\f{1}{2}}}^2\lesssim \|(u,\nabla u)\|_{L_t^2H_{co}^{m-1}}^2.
\eeq
To estimate the first term in 
$f,$ we apply the
commutator estimate \eqref{commutator-R2}
to get that:
\beq\label{h-2}
\begin{aligned}
|[\Lambda_y^{\f{1}{2}},u_y]Z^{\alpha} \p_y h|_{L_t^2L_y^2}&\lesssim |Z^{\alpha}\p_y h|_{L_t^2H^{-\f{1}{2}}}|u_y|_{L_t^{\infty}H^{2.5}}\\
&\lesssim |h|_{L_t^2\tilde{H}^{m-\f{1}{2}}} \|(u,\nabla u) \|_{L_t^{\infty}H^2(\mS)}\lesssim T^{\f{1}{2}}\cE_{m,T}^2.
\end{aligned}
\eeq
For the second term in $f,$ we have by the commutator estimate 
\eqref{rough-com-bdry}
and the trace inequality \eqref{trace}
that:
 \beq\label{h-3}
 \begin{aligned}
|[Z^{\alpha},u_y]\p_y h|_{L_t^2H^{\f{1}{2}}} 
&\lesssim |u|_{L_t^{2}\tilde{H}^{[\f{m}{2}]+\f{1}{2}}}| h|_{L_t^{\infty}\tilde{H}^{m-\f{1}{2}}}+
| h|_{L_t^{\infty}\tilde{H}^{[\f{m}{2}]+\f{5}{2}}} | u|_{L_t^2\tilde{H}^{m-\f{1}{2}}}\lesssim \cE_{m,T}^2.%\big(|\ep\pt h|_{L_t^{2}\tilde{H}^{m-\f{1}{2}}}+\|(u,\nabla u)\|_{L_t^{2}\cH^m}\big)\lesssim\lae.
\end{aligned}
\eeq
Inserting \eqref{h-1}-\eqref{h-3} into \eqref{EI-h}, we achieve \eqref{surface1}.
%\beqs|(\ep\pt)^m h|_{L_t^{\infty}H^{\f{1}{2}}}\lesssim Y_m(0)+ T^{\f{1}{4}}\lae.\eeqs
\end{proof}
\section {High order energy estimates}
In this section, we prove two kinds of energy estimates, namely the $\ep-$dependent  high order conormal energy estimates involving at least one spatial derivative, and the higher order estimates when only the time derivatives are involved. %{\color{red}
These quantities we are going to bound appears in the definition of energy norms ${\cE}_{high,m,T}$  in \eqref{defenergy-high} and are  necessary to prove the uniform estimates shown in Sections 10-12.

\subsection{Energy estimate I: Highest order energy estimates.}
\begin{lem}\label{lemhighest}
 Suppose that \eqref{preassumption} holds  for some $T>0$ then for any $0<t\leq T,$ then we have the following energy estimates:
 \beq\label{EI-1}
 \begin{aligned}
&\qquad\ep\|(\sigma,u)\|_{L_t^{\infty}H_{co}^m}^2+\ep\|\nabla u\|_{L_t^2H_{co}^{m}}^2\lesssim\ep \|(\sigma,u)(0)\|_{H_{co}^m}^2+(T+\ep)^{\f{1}{2}}\lae.%\lca\cE_{m,T}^2.
 \end{aligned}
 \eeq
\end{lem}
\begin{proof}
Let us start with \eqref{EI-1} for $m=0$ which is standard.  
Performing direct energy estimates for
$\eqref{FCNS2}$ we get by 
identities \eqref{ipp-1}-\eqref{ipp-3} that:
\beq\label{zeroenergy}
\begin{aligned}
 &\quad\f{1}{2}\int_{\mS}(g_1|\sigma|^2+g_2 |u|^2) (t)\,\d\cV_t+\int_0^t\int_{\mS}2\mu|S^{\vp}u|^2+\lambda|\div^{\vp}u|^2\,\d\cV_s\d s\\
&=\f{1}{2}\int_{\mS}(g_1|\sigma|^2+g_2|u|^2)(0)\,\d\cV_0+\f{1}{2}\int_{\mS}(\pt^{\vp}g_1+\div^{\vp}(g_1u))|\sigma|^2 \,\d\cV_s\d s-a\int_0^t\int_{z=-1}|u_{\tau}|^2\,\d y\d s
\end{aligned}
\eeq
where $u_{\tau}=(u_1,u_2,0)^{t}.$ Thanks to \eqref{preassumption1} and assumption \eqref{preassumption}, we have:
\beqs
\begin{aligned}
\il\pt^{\vp}g_1+\div^{\vp}(g_1u)\il_{0,\infty,t}&\leq \Lambda\big(\f{1}{c_0},\il(\sigma,u)\il_{1,\infty,t}+\il\nabla(\sigma,u)\il_{0,\infty,t}+|h|_{1,\infty,t})\\
&\lesssim \lca.
\end{aligned}
\eeqs
In view of the Korn inequality \eqref{korn}, 
the trace inequality \eqref{trace}, one gets  by using Young's inequality that:
\beq
\begin{aligned}
\|(\sigma,u)\|_{L_t^{\infty}L^2}^2+\|\nabla u\|_{L_t^2L^2}^2&\lesssim \|(\sigma_0,u_0)\|_{L^2(\mS)}^2+\lca \|(\sigma,u)\|_{L_t^{2}L^2}
^2\\
&\lesssim \|(\sigma_0,u_0)\|_{L^2(\mS)}^2+ T\lca \|(\sigma,u)\|_{L_t^{\infty}L^2}^2.
\end{aligned}
\eeq
We now detail the high order estimates in \eqref{EI-1}. Let $\alpha$ be a multi-index with $1\leq |\alpha|\leq m,$ applying $Z^{\alpha}$ on the equation \eqref{FCNS2}, and denoting $(\sigma^{\alpha},u^{\alpha})=Z^{\alpha}(\sigma,u),$ one obtains the system:
\beq
\left\{
\begin{array}{l}
     g_1(\pt^{\vp}+u\cdot\nabla^{\vp})\sigma^{\alpha}+\f{\div^{\vp}u^{\alpha}}{\ep}=\mathcal{C}_{\sigma}^{\alpha}-\f{1}{\ep}[Z^{\alpha},\div^{\vp}]u, \\[5pt]
    g_2(\pt^{\vp}+u\cdot\nabla^{\vp})u^{\alpha}-\div^{\vp}Z^{\alpha}\cL^{\vp}u+\f{\nabla^{\vp}\sigma}{\ep}=\mathcal{C}_u^{\alpha}-\f{1}{\ep}[Z^{\alpha},\nabla^{\vp}]\sigma+
    [Z^{\alpha},\div^{\vp}]\cL^{\vp}u.
\end{array}
\right.
\eeq
where the commutators are given by:
\beq\label{defcsiu}
\begin{aligned}
\cC_{\sigma}^{\alpha}&=\big[Z^{\alpha},\f{g_1}{\ep}\big]\ep\pt\sigma+[Z^{\alpha},g_1u_y]\nabla_y \sigma+[Z^{\alpha},g_1U_z\p_z]\sigma,\\
\cC_{u}^{\alpha}&=\big[Z^{\alpha},\f{g_2}{\ep}\big]\ep\pt u+[Z^{\alpha},g_2u_y]\nabla_y u+[Z^{\alpha},g_2U_z\p_z]u,
\end{aligned}
\eeq
 with 
 \beq\label{defofUz}
 U_z=\f{u\cdot\bN-\pt\vp}{\p_z\vp}
 \eeq
 Note that we have from \eqref{newder} that
 \beq\label{useful-identity}
\p_t^{\vp}+u\cdot\nabla^{\vp}=\pt+u_y\nabla_y+U_z\p_z.
\eeq
The energy equality then reads:
\beq\label{energyineq-first}
\begin{aligned}
\quad &\f{1}{2}\int_{\mS}(g_1|\sigma^{\alpha}|^2+g_2 |u^{\alpha}|^2) (t)\,\d\cV_t+\int_0^t\int_{\mS}2\mu|Z^{\alpha}S^{\vp}u|^2+\lambda|Z^{\alpha}\div^{\vp}u|^2\,\d\cV_s\d s\\
&=F_0^{\alpha}+F_1^{\alpha}+\cdots +F_7^{\alpha}.
\end{aligned}
\eeq
where 
\begin{align*}
  &F_0^{\alpha}=\f{1}{2}\int_{\mS}\big(g_1|\sigma^{\alpha}|^2+g_2|u^{\alpha}|^2\big)\,\d\cV_0, \quad F_1^{\alpha}=\f{1}{2}\int_0^t\int_{\mS}\big(\pt^{\vp}g_1+\div^{\vp}(g_1u)\big)|\sigma^{\alpha}|^2\ \d\cV_s\d s,\\
  & F_2^{\alpha}=-\int_0^t\int_{z=0}[Z^{\alpha},\bN](\mathcal{L}^{\vp}u-(\sigma/\ep)\text{Id})\cdot u^{\alpha}\ \d y\d s \ \mathbb{I}_{\{\alpha_3=0\}},\\
  & F_3^{\alpha}=\int_0^t\int_{\mS} Z^{\alpha}\mathcal{L}^{\vp}u\cdot[Z^{\alpha},\nabla^{\vp}] u \ \d\cV_s \d s, \quad F_4^{\alpha}=-\int_0^t\int_{\mS}[Z^{\alpha},\div^{\vp}]\cL^{\vp}u\cdot u^{\alpha}\ \d\cV_s \d s,\\
  & F_5^{\alpha}=\int_0^t\int_{\mS} \cC_{\sigma}^{\alpha}\sigma^{\alpha}+\cC _{u}^{\alpha}\cdot u^{\alpha}\ \d\cV_s \d s, \quad F_6^{\alpha}=-\f{1}{\ep}\int_0^t\int_{\mS} \sigma^{\alpha}[Z^{\alpha},\div^{\vp}]u+u^{\alpha}\cdot[Z^{\alpha},\nabla^{\vp}]\sigma\ \d\cV_s \d s,\\
  &F_7^{\alpha}=-a\int_0^t\int_{z=-1}|Z^{\alpha}u_{\tau}|^2\ \d y\d s.
\end{align*}
The first two terms can be controlled directly by:
\beq\label{sec5:eq1}
\ep(|F_0^{\alpha}|+|F_1^{\alpha}|)\lesssim \ep\|Z^{\alpha}(\sigma,u)(0)\|_{L^2(\mS)}^2+T\Lambda\big(\f{1}{c_0},\cA_{m,t}\big)\ep\| Z^{\alpha}\sigma\|_{L_t^{\infty}L^2(\mS)}^2.
\eeq
For the boundary term $F_2^{\alpha},$ which vanishes identically if $\alpha_3=0,$ we split it as:
\beqs
F_2^{\alpha}=-\int_0^t\int_{z=0}(\mathcal{L}^{\vp}u-(\sigma/\ep)\text{Id})Z^{\alpha}\bN\cdot u^{\alpha}+[Z^{\alpha},(\mathcal{L}^{\vp}u-(\sigma/\ep)\text{Id}),\bN]u^{\alpha}\,\d y\d s=:F_{21}^{\alpha}+F_{22}^{\alpha}.
\eeqs
By duality and \eqref{product-R2},  $F_{21}^{\alpha}$ can be bounded as:
\begin{equation*}
\begin{aligned}
    |F_{21}^{\alpha}|\lesssim |(\mathcal{L}^{\vp} u-(\sigma/\ep)\text{Id})^{b,1}|_{L_t^{\infty}W_y^{1,\infty}}|(u^{\alpha})^{b,1}|_{L_t^2H_y^{\f{1}{2}}}|Z^{\alpha}\bN|_{L_t^2H_y^{-\f{1}{2}}}.
\end{aligned}
\end{equation*}
By the identities \eqref{sigmabdry}, \eqref{tan-nor}, \eqref{nor-nor} and the definition \eqref{defcA-free}, we have:
\beq\label{inftyupbdry}
\begin{aligned}
|(\mathcal{L}^{\vp} u-(\sigma/\ep)\text{Id})^{b,1}|_{[\f{m}{2}]-1,\infty,t}&\lesssim \Lambda\big(\f{1}{c_0}, |h|_{[\f{m}{2}],\infty,t}+\il\div^{\vp}u\il_{[\f{m}{2}]-1,\infty,t}+\il u\il_{[\f{m}{2}],\infty,t}\big)\\
&\lesssim \Lambda\big(\f{1}{c_0},\cA_{m,t}\big).
\end{aligned}
\eeq
Hence, by the trace inequality and Young's inequality, we get that:
\beqs
\ep|F_{21}^{\alpha}| \leq\delta\ep\|\nabla u\|_{L_t^2H_{co}^m}^2+\ep\big(|Z^{\alpha} h|_{L_t^2H^{\f{1}{2}}}^2+\|u^{\alpha}\|_{\ltx}^2
\big)\Lambda\big(\f{1}{c_0},\cA_{m,t}\big).
\eeqs
For $F_{22}^{\alpha},$ we use successively the Cauchy-Schwarz inequality, the estimate \eqref{inftyupbdry} and the trace inequality
\eqref{trace} to get:
\begin{equation*}
\begin{aligned}
    |F_{22}^{\alpha}|&\lesssim 
    |(u^{\alpha})^{b,1}|_{L_t^2L_y^2}\big|[Z^{\alpha},\cL^{\vp}u-({\sigma}/{\ep})\text{Id},\bN]\big|_{L_t^2L_y^2}\\
    &\lesssim  |(u^{\alpha})^{b,1}|_{L_t^2L_y^2}\big(|(\cL^{\vp}u,\sigma/\ep)|_{[\f{m}{2}]-1,\infty,t}|h|_{L_t^2\tilde{H}^{m}}+|(\cL^{\vp}u,\sigma/\ep)|_{L_t^2\tilde{H}^{m-1}}
    |\bN|_{[\f{m+1}{2}]+1,\infty,t} \big)\\
    &\leq \delta \|\nabla u\|_{L_t^2H_{co}^m}^2+C_{\delta}\lca\big(
    \|u\|_{E^m,t}^2+\|\nabla\div^{\vp} u\|_{L_t^2H_{co}^{m-1}}\|\div^{\vp}u\|_{L_t^2H_{co}^{m-1}}+|h|_{L_t^2\tilde{H}^{m}}^2\big).
\end{aligned}
\end{equation*}
To summarize, we can control $\ep F_2^{\alpha}$ as:
\beq\label{sec5:eq2}
\ep|F_{2}^{\alpha}|\leq 2\delta \ep\|\nabla u\|_{L_t^2H_{co}^m}^2+C_{\delta}\lca\big(T\ep| h|_{L_t^{\infty}\tilde{H}^{m+\f{1}{2}}}^2
+\ep^{\f{1}{2}}(\| u\|_{E^m,t}^2+\ep\|\nabla\div^{\vp} u\|_{L_t^2H_{co}^{m-1}}^2)\big).
\eeq
Let us detail the estimate of $F_3^{\alpha}.$ We use the estimate  
\eqref{comgrad} for $n=2$ and Young's inequality to get that:
\beq\label{sec5:eq3}
\begin{aligned}
|\ep F_3^{\alpha}|&\leq \ep \|Z^{\alpha}\mathcal{L}^{\vp}u\|_{\ltx}\big(\|\nabla u\|_{L_t^2H_{co}^{m-1}}+|h|_{L_t^2\tilde{H}^{m+\f{1}{2}}}\big)\Lambda\big(\f{1}{c_0},|h|_{m-2,\infty,t}+\ep^{\f{1}{2}}\il\nabla u\il_{2,\infty,t}\big)\\
&\leq \delta  \ep \|\nabla u\|_{L_t^2H_{co}^m}^2+\lca\big(\ep\|\nabla u\|_{L_t^2H_{co}^{m-1}}^2+T\ep|h|_{L_t^{\infty}\tilde{H}^{m+\f{1}{2}}}^2\big).
\end{aligned}
\eeq
Similarly, for $F_4,$ by H\"older's inequality, the commutator estimate \eqref{comgrad} and the definition \eqref{defcA-free}, we find
\beq\label{sec5:eq4}
\begin{aligned}
|\ep F_4^{\alpha}|&\leq \ep \|u^{\alpha}\|_{\ltx}\|[Z^{\alpha}, \div^{\vp}]\cL^{\vp} u\|_{\ltx}\\
&\lesssim \ep^{\f{1}{2}} \|u^{\alpha}\|_{\ltx}\Lambda\big(\f{1}{c_0},|h|_{m-2,\infty,t}+\ep^{\f{1}{2}}\il\nabla\cL^{\vp} u\il_{2,\infty,t}\big)\big(|h|_{L_t^2\tilde{H}^{m+\f{1}{2}}}+\|\ep^{\f{1}{2}}\nabla\cL^{\vp} u\|_{L_t^2H_{co}^{m-1}}\big)\\
&\lesssim (T+\ep)^{\f{1}{2}}\lca \cE_{m,t}^2. %T^{\f{1}{2}}\Lambda_{2,\infty,t}\big(\ep|h|_{L_t^{\infty}\tilde{H}^{m+\f{1}{2}}}^2+\|u\|_{L_t^2H_{co}^{m}}^2\big)+\ep\|\nabla^2 u\|_{L_t^2H_{co}^{m-1}}\|u\|_{L_t^2H_{co}^m}\Lambda\big(\f{1}{c_0},|h|_{2,\infty,t}\big),
\end{aligned}
\eeq
%which, combined with the definition of the $\cE_{m,t}$ \eqref{defcN}, yields:
%\beqs|\ep F_4^{\alpha}|\lesssim (T+\ep)^{\f{1}{4}}\Lambda_{2,\infty,t}\cE_{m,t}^2.\eeqs
Next, we control $F_5$ as:
$$\ep |F_5^{\alpha}|\leq T^{\f{1}{2}}\|\ep^{\f{1}{2}}(\sigma^{\alpha},u^{\alpha})\|_{L_t^{\infty}L^2}\|\ep^{\f{1}{2}}(\cC_{\sigma}^{\alpha},\cC_u^{\alpha})\|_{\ltx}.$$
It thus remains to estimate 
$(\cC_{\sigma}^{\alpha},\cC_u^{\alpha})$ defined in \eqref{defcsiu}. Taking benefits of
the commutator estimate \eqref{crudecom} and the estimate \eqref{esofg12-2} for 
$g_1,g_2,$ we obtain:
\beqs
\|\ep^{\f{1}{2}}(\cC_{\sigma}^{\alpha},\cC_{u}^{\alpha})\|_{L_t^2L^2}\lesssim \lca(\|(\sigma,u)\|_{E^m,t}+\ep^{\f{1}{2}}|h|_{L_t^2{H}^{m+\f{1}{2}}} ).
\eeqs
Therefore, we obtain:
\beq\label{sec5:eq5}
|\ep F_5^{\alpha}|\lesssim T^{\f{1}{2}}\lca \cE_{m,t}^2.
\eeq
Let us split $F_6^{\alpha}$ as:
$
F_6^{\alpha}=F_{6,1}^{\alpha}+F_{6,2}^{\alpha}$
with 
\beqs
F_{6,1}^{\alpha}=-\f{1}{\ep}\int_0^t\int_{\mS}\sigma^{\alpha}[Z^{\alpha},\div^{\vp}]u\,\d \cV_s\d s, \qquad F_{6,2}^{\alpha}=-\f{1}{\ep}\int_0^t\int_{\mS}u^{\alpha}\cdot[Z^{\alpha},\nabla^{\vp}]\sigma\,\d \cV_s\d s.
\eeqs
For $F_{6,1}^{\alpha},$ thanks to the commutator estimate \eqref{comgrad},
\beq\label{F61}
\begin{aligned}
|\ep F_{6,1}^{\alpha}|&\lesssim \|\ep^{-\f{1}{2}}\sigma^{\alpha}\|_{L_t^2L^2}\ep^{\f{1}{2}}\|[Z^{\alpha},\div^{\vp}]u\|_{L_t^2L^2}\\
&\lesssim %(T+\ep)^{\f{1}{2}}
(\|\ep^{\f{1}{2}}\pt \sigma\|_{L_t^2\cH^{m-1}}+\|\ep^{-\f{1}{2}}\nabla\sigma\|_{\hco^{m-1}})(\ep^{\f{1}{2}}|h|_{L_t^{2}\tilde{H}^{m+\f{1}{2}}}%+|h|_{{L_t^{2}\tilde{H}^{m-\f{1}{2}}}}
+\ep^{\f{1}{2}}\|\nabla u\|_{\hco^{m-1}})\lca\\
&\lesssim (T+\ep)^{\f{1}{2}}\lca\cE_{m,t}^2.
\end{aligned}
\eeq
Similarly, by using the fact that (recall $m\geq 7$),
$$\ep^{-\f{1}{2}}\il\nabla\sigma\il_{2,\infty,t}\lesssim \lca,$$
we finally find:
\beq\label{F62}
\begin{aligned}
|\ep F_{6,2}^{\alpha}|&\lesssim \|u\|_{\hco^m}\big(\|\nabla\sigma\|_{\hco^{m-1}}+|h|_{\htlde^{m+\f{1}{2}}}\il\nabla\sigma\il_{2,\infty,t}\big)\Lambda\big(\f{1}{c_0},|h|_{m-2,\infty,t}\big)\\
&\lesssim (T+\ep)^{\f{1}{2}}\lca\cE_{m,t}^2.
\end{aligned}
\eeq
Gathering \eqref{F61} and \eqref{F62}, we find that:
\beq\label{sec5:eq6}
|\ep F_6^{\alpha}|\lesssim (T+\ep)^{\f{1}{2}}\lca\cE_{m,t}^2.
\eeq
Finally, for the boundary term $F_7^{\alpha},$ we apply the trace inequality
\eqref{trace} and Young's inequality to get that:
\beq\label{lowerbdry}
\ep |F_7^{\alpha}|\lesssim \delta\ep\|\nabla Z^{\alpha} u_{\tau}\|_{L_t^2L^2}^2+C_{\delta}T\ep\|Z^{\alpha}u_{\tau}\|_{L_t^{\infty}L^2(\mS)}^2.
\eeq
Collecting \eqref{sec5:eq1}, \eqref{sec5:eq2}, \eqref{sec5:eq3}, \eqref{sec5:eq4}, \eqref{sec5:eq5},
\eqref{sec5:eq6}, \eqref{lowerbdry} and  summing up for $|\alpha|\leq m,$ we find by Korn's inequality \eqref{korn} and by choosing $\delta$ small enough,
\beqs
\begin{aligned}
&\quad \ep\|(\sigma,u)\|_{L_t^{\infty}H_{co}^m}^2+\ep \|\nabla u\|_{L_t^{2}H_{co}^m}^2
%&
\lesssim\ep \|(\sigma,u)(0)\|_{H_{co}^m}^2+(T+\ep)^{\f{1}{2}}\lca\cE_{m,t}^2.%+\|u\|_{\hco^m}^2\Lambda\big(\f{1}{c_0},|h|_{2,\infty,t}\big).
\end{aligned}
\eeqs
\end{proof}
\begin{thm}[Estimates for High-order time derivatives]
Under the same assumption as in Lemma \ref{lemhighest}, we have the following estimates: for any $0<t\leq T,$
\beq\label{EI-T1}
 \begin{aligned}
\ep\|\pt(\sigma,u)\|_{L_t^{\infty}\cH^{m-1}}^2+\ep\|\pt \nabla u\|_{L_t^2\cH^{m-1}}^2\lesssim \ep\|\pt(\sigma,u)(0)\|_{\cH^{m-1}}^2+(T+\ep)^{\f{1}{2}}\lae.
 \end{aligned}
 \eeq
\end{thm}
\begin{proof}
Due to the singular terms in the system
\eqref{FCNS2},
we need to deal with the zero order and the higher order estimates for 
$\ep^{\f{1}{2}}\pt(\sigma,u)$ differently.
We will prove in \eqref{lemloworder} the zero order estimate:
\beqs
\ep\|\pt(\sigma,u)\|_{L_t^{\infty}L^2}^2+\ep\|\pt \nabla u\|_{L_t^2L^2}^2\lesssim
\ep\|\pt(\sigma,u)(0)\|_{L^2}^2+(T+\ep)^{\f{1}{2}}\lae.
 % \Lambda_{2,\infty,T} \mathcal{E}_{m,T}^2.
\eeqs
 Let us stress that this estimate does not depend on the higher order estimate to be shown here and vice versa.

We now focus on the higher order estimates. Substituting $Z^{\alpha}$ by
$\ep^{\f{1}{2}}Z_0^k\pt$ ($1\leq k\leq m-1$) in \eqref{energyineq-first}, we find that:
\beq\label{energyineq-Thigh}
\begin{aligned}
\quad &\f{\ep}{2}\int_{\mS}(g_1|Z_0^k \pt\sigma|^2+g_2 |Z_0^k \pt u|^2) (t)\,\d\cV_t+\ep\int_0^t\int_{\mS}2\mu|Z_0^k \pt S^{\vp}u|^2+\lambda|Z_0^k \pt\div^{\vp}u|^2\,\d\cV_s\d s\\
&=F_0^{k}+F_1^{k}+\cdots +F_7^{k}.
\end{aligned}
\eeq
where $F_0^k-F_7^k$ are defined in the same way as $F_0^{\alpha}-F_7^{\alpha}$
(defined in $\eqref{energyineq-first}$)
by changing $Z^{\alpha}$ into $\ep^{\f{1}{2}}Z_0^k\pt.$
Our following task is to control $F_0^k-F_7^k$ one by one. 
The first two terms can be controlled by:
\beq\label{eq1-T}
|F_0^k+F_1^k|\lesssim \ep\|\pt (\sigma,u)\|_{L_t^2\cH^{k}}^2+T\lca\ep\|\pt\sigma\|_{L_t^{\infty}\cH^{k}}^2.
\eeq
Now, for the term $$F_2^k=-\ep\int_0^t\int_{z=0}[Z_0^k\pt,\bN](\cL^{\vp}u-\f{\sigma}{\ep}\text{Id})Z_0^k\pt u\,\d\cV_s\d s,$$ we first use the duality $\langle\cdot\rangle_{H^{\f{1}{2}}\times H^{-\f{1}{2}}}$, Cauchy-Schwarz inequality and the estimate \eqref{inftyupbdry} to control it as:
\beqs
|F_2^k|\lesssim |\ep^{\f{1}{2}}Z_0^k\pt u|_{L_t^2H^{\f{1}{2}}}|\ep^{\f{1}{2}}Z_0^k\pt h|_{L_t^2{H}^{\f{1}{2}}}\lca+|\ep^{\f{1}{2}}Z_0^k\pt u|_{L_t^2L_y^2}
\big|\ep^{\f{1}{2}}[Z_0^k\pt,\bN,(\cL^{\vp}u-\f{\sigma}{\ep}\text{Id})]\big|_{L_t^2L_y^2}
\eeqs
By \eqref{surface2}, the trace inequality 
\eqref{trace} and Young's inequality,
the first term in the right hand side of the above inequality is bounded by:
\beqs
\delta\|\ep^{\f{1}{2}}Z_0^k\nabla^{\vp}\pt u\|_{L_t^2L^2}^2+(T+\ep)^{\f{1}{2}}\lca\cE_{m,t}^2.
\eeqs
Moreover, we use the expansion \eqref{expansion-T}, the estimates \eqref{grad-upbdry}, \eqref{surface2}, the trace inequality \eqref{trace} successively to control the second one as:
\beqs
\begin{aligned}
&C |\ep^{\f{1}{2}}Z_0^k\pt u|_{L_t^2L_y^2}(\ep^{\f{1}{2}}|\pt h|_{L_t^2\tilde{H}^{k}}|(\cL^{\vp}u,\sigma/\ep)^{b,1}|_{[\f{k-1}{2}],\infty,t}+\ep^{\f{1}{2}}|(\cL^{\vp}u,\sigma/\ep)^{b,1}|_{L_t^2\tilde{H}^{k}}|\pt h|_{[\f{k}{2}],\infty,t})\\
&\leq \delta \|\ep^{\f{1}{2}}Z_0^k\nabla^{\vp}\pt u\|_{L_t^2L^2}^2+(T+\ep)^{\f{1}{2}}\Lambda\big(\f{1}{c_0},\cA_{m,t}\big)\cE_{m,t}^2.
\end{aligned}
\eeqs
Note that by \eqref{sigmabdry},  \eqref{tan-nor}, \eqref{nor-nor}, one has that:
\beqs
\begin{aligned}
&|\ep^{\f{1}{2}}(\cL^{\vp}u,\sigma)^{b,1}|_{\htlde^{m-1}}\lesssim
(|\ep^{\f{1}{2}}(\p_y u,\div^{\vp}u)^{b,1}|_{\htlde^{m-1}}+\ep^{\f{1}{2}}|h|_{L_t^2\tilde{H}^{m}})\lca\\
&\lesssim \ep^{\f{1}{4}}\big(\|\nabla u\|_{L_t^{2}H_{co}^{m-1}}+\ep^{\f{1}{2}}\|\nabla u\|_{L_t^2H_{co}^{m}}+\ep^{\f{1
}{2}}\|\nabla\div u\|_{L_t^2H_{co}^{m-1}}\big)\lca+T^{\f{1}{2}}|\ep^{\f{1}{2}}h|_{L_t^{\infty}\tilde{H}^{m}}\lca\\
&\lesssim (\ep^{\f{1}{4}}+T^{\f{1}{2}})\lca\cE_{m,t}.
\end{aligned}
\eeqs
We thus find that:
\beq
\begin{aligned}
|F_2^k|\lesssim  2\delta \|\ep^{\f{1}{2}}Z_0^k\nabla^{\vp}\pt u\|_{L_t^2L^2}^2+ (T+\ep)^{\f{1}{2}}\Lambda\big(\f{1}{c_0},\cA_{m,t}\big)\cE_{m,t}^2.
\end{aligned}
\eeq
Next, with the aid of the commutator estimate $\eqref{comgradT}$ and the estimate \eqref{surface3}, we can control the commutator $[Z_0^{k}\pt,\nabla^{\vp}] u $ as:
\beqs 
\begin{aligned}
\|[Z_0^{k}\pt,\nabla^{\vp}] u \|_{\ltx}\lesssim&\big(|\ep\p_t^2 h|_{L_t^2\tilde{H}^{m-\f{3}{2}}}+\| \ep\pt\p_z u\|_{L_t^2\cH^{m-2}\cap L_t^{\infty}\cH^1}\big)\cdot
\\
&\Lambda\big(\f{1}{c_0},\il \p_z u \il_{1,\infty,t}+|(h,\ep^{\f{1}{2}}\pt h)|_{m-2,\infty,t}+(\int_0^t |\ep^{\f{1}{2}}\pt ^2 h(s)|_{m-2,\infty}\d s)^{\f{1}{2}}\big)\\
&%\lesssim \Lambda\big(\f{1}{c_0},\cA_{m,t}\big) \big(|\ep\p_t^2 h|_{L_t^2\tilde{H}^{m-\f{3}{2}}}^2+\|\ep^{\f{1}{2}}\pt \p_z u\|_{L_t^2\cH^{m-2}}^2\big)
\lesssim \lae.
\end{aligned}
\eeqs
Therefore, we bound the term
$$F_3^k=\ep\int_0^t\int_{\mS} Z_0^{k}\pt\mathcal{L}^{\vp}u\cdot[Z_0^{k}\pt,\nabla^{\vp}] u \ \d\cV_s \d s$$
by using Young's inequality and the assumption $k\leq m-1,$
\beq
\begin{aligned}
|F_3^k|
& \leq \delta \ep\|Z_0^k\pt \nabla^{\vp}u\|_{L_t^2L^2}^2+\ep\lae.
\end{aligned}
\eeq
We proceed to estimate $$F_4^k=-\ep \int_0^t\int_{\mS}[Z_0^k\pt,\div^{\vp}]\cL^{\vp}u\cdot Z_0^k\pt u\,\d \cV_s\d s.$$ By the expansion  \eqref{expansion-2-free}, the estimate \eqref{surface3},  %\eqref{intemediate1.5}
and the assumption $k\leq m-1,$ we obtain:
\begin{align*}
   \|\ep^{\f{1}{2}}[Z_0^k\pt, \div^{\vp}]\cL^{\vp}u\|_{L_t^2L^2}&\lesssim (|\ep\pt^2 h|_{\htlde^{m-\f{3}{2}}}+\ep^{\f{1}{2}}\|\p_z\cL^{\vp}u\|_{L_t^2\cH^{m-1}})\cdot\\
   &\Lambda\big(\f{1}{c_0}, \ep^{\f{1}{2}}\il\p_z\cL^{\vp} u\il_{[\f{m}{2}]-1,\infty,t}+|\pt h|_{[\f{m-1}{2}],\infty,t}+|h|_{[\f{m+1}{2}],\infty,t}\big)\\
   &\lesssim \lae. 
\end{align*}
We thus control $F_4^k$ by the Cauchy-Schwarz inequality:
\beq
\begin{aligned}
|F_4^k|&\leq  T^{\f{1}{2}} \|\ep^{\f{1}{2}}\pt u\|_{L_t^{\infty}\cH^k}\|\ep^{\f{1}{2}}[Z_0^k\pt, \div^{\vp}]\cL^{\vp}u\|_{L_t^2L^2}\\
&\lesssim T^{\f{1}{2}}\lae. %\Lambda\big(\f{1}{c_0},\cA_{m,t}\big)\cE_{m,t}^2.
\end{aligned}
\eeq
The next term $F_5^k$ is defined by
$$F_5^k=\ep \int_0^t\int_{\mS} C_{\sigma}^k Z_0^k\pt\sigma +C_{u}^k\cdot Z_0^k\pt u \,\d \cV_s \d s.$$
To continue, we need the following proposition to control the commutators 
$\ep^{\f{1}{2}}(\cC_{\sigma}^k,\cC_{u}^k):$
\begin{prop}\label{procsiu}
For commutators 
\beq\label{defcom-tem}
\begin{aligned}
\cC_{\sigma}^{k}&=\big[Z_0^{k}\pt,\f{g_1}{\ep}\big]\ep\pt\sigma+[Z_0^{k}\pt,g_1u_y]\nabla_y \sigma+[Z_0^{k}\pt,g_1U_z\p_z]\sigma,\\
\cC_{u}^{k}&=\big[Z_0^{k}\pt,\f{g_1}{\ep}\big]\ep\pt u+[Z_0^{k}\pt,g_1u_y]\nabla_y u+[Z_0^{k}\pt,g_1U_z\p_z]u.
\end{aligned}
\eeq
we have the estimate: for $k\leq m-1$
\beqs 
\|\ep^{\f{1}{2}}(\cC_{\sigma}^k,\cC_{u}^k)\|_{L_t^2L^2}
\lesssim \lae.
\eeqs
\end{prop}
We will postpone the proof of this proposition and continue to estimate the remaining terms $F_5^k-F_7^k.$ By using Proposition \ref{procsiu}, $F_5^k$ can be estimated as:
\begin{align}
|F_5^k|&\lesssim T^{\f{1}{2}}\|\ep^{\f{1}{2}}\pt(\sigma,u)\|_{L_t^{\infty}\cH^{m-1}}\|\ep^{\f{1}{2}}(\cC_{\sigma}^k,\cC_{u}^k)\|_{L_t^2L^2}\lesssim T^{\f{1}{2}}\lae.
\end{align}
For the term $$F_6^k=-\int_0^t\int_{\mS}Z_0^k\pt\sigma \cdot [Z_0^k\pt,\div^{\vp}]u+Z_0^k\pt u\cdot[Z_0^k\pt,\nabla^{\vp}]\sigma\,\d\cV_s\d s,$$ we can apply commutator estimate 
\eqref{comgradT} to obtain:
\beqs
\begin{aligned}
&\ep^{-\f{1}{2}}(\|[Z_0^k\pt,\div^{\vp}]u\|_{L_t^2L^2}+\|[Z_0^k\pt,\nabla^{\vp}]\sigma\|_{L_t^2L^2})\\
&\lesssim \Lambda\big(\f{1}{c_0},\cN_{m,t}\big)\big(
|\ep^{\f{1}{2}}\pt^2  h|_{L_t^2\tilde{H}^{m-\f{3}{2}}}
+\|\ep^{\f{1}{2}}\pt\nabla (\sigma,u)\|_{L_t^2\cH^{m-2}\cap L_t^{\infty}\cH^1}\big)\lesssim\lae. 
\end{aligned}
\eeqs
This estimate, combined with the Cauchy-Schwarz inequality, yields:
\beq
|F_6^k|\lesssim T^{\f{1}{2}}\lae. %\Lambda\big(\f{1}{c_0},\cA_{m,t}\big)\cE_{m,t}^2.
\eeq
Finally, we control the last term $$F_7^k=-a \ep \int_0^t\int_{z=-1}|Z_0^k\pt u_{\tau}|^2 \,\d y\d s$$ by the trace inequality \eqref{trace} and Young's inequality:
\beq\label{eq2-T}
F_7^k \leq \delta \ep\int_0^t 
\int_{\mS} |Z_0^k\pt \nabla^{\vp}u|^2\d\cV_s\d s+(T+\ep)\Lambda\big(\f{1}{c_0},\cA_{m,t}\big)\cE_{m,t}^2.
\eeq
Collecting \eqref{eq1-T}-\eqref{eq2-T}, summing up for $k\leq m-1$ and choosing $\delta$ small enough, we find \eqref{EI-T1}.
\end{proof}
We now give the proof of Proposition \ref{procsiu}.
\begin{proof}[Proof of Proposition \ref{procsiu}]
We use the following two expansions
\beq\label{expansion-2-free}
\begin{aligned}
\ep^{\f{1}{2}}[Z_0^{m-1}\pt,f]g=\sum_{0\leq l\leq [\f{m}{2}]-1}(C_{m}^l Z_0^l g Z_0^{m-1-l}\ep^{\f{1}{2}} \pt f)
+\sum_{ [\f{m}{2}]\leq l\leq m-1}(C_{m}^l Z_0^{l-1}\ep^{\f{1}{2}}\pt g Z_0^{m-l} f)
\end{aligned}
\eeq
\beqs 
\begin{aligned}
\ep^{\f{1}{2}}[Z_0^{m-1}\pt,f]g&=
\sum_{0\leq l\leq 1}(C_{m}^l Z_0^l g Z_0^{m-1-l}\ep^{\f{1}{2}} \pt f)
+C_m^2 Z_0^1 \ep^{\f{1}{2}} \pt g 
Z_0^{m-2} f\\
&\quad+\sum_{ 3\leq l\leq m-1}(C_{m}^l Z_0^{l-1}\ep^{\f{1}{2}}\pt g Z_0^{m-l} f).
\end{aligned}
\eeqs
In light of the second expansion, we  control the last term in
$\cC_{u}^{m-1}$ as follows:
\beqs 
\begin{aligned}
&\ep^{\f{1}{2}}\|[Z_0^{m-1}\pt,g_1U_z]\p_z u\|_{L_t^2L^2}\lesssim 
\il\p_z u\il_{1,\infty,t}\|\ep^{\f{1}{2}}\pt (g_1 U_z)\|_{L_t^2\cH^{m-1}}\\
&\qquad+\|\ep^{\f{1}{2}}\pt \p_z u\|_{L_t^{\infty}\cH^1}\|Z_0^{m-2}(g_1 U_z)\|_{L_t^{\infty}L^2}+\|\ep^{\f{1}{2}}\pt \p_z u\|_{L_t^{2}\cH^{m-1}}\il g_1 U_z\il_{m-3,\infty,t}\\
&\lesssim \lae.
\end{aligned}
\eeqs
The remaining terms appearing in 
$\cC_{\sigma}^{m-1},\cC_u^{m-1}$
can be estimated by using the first expansion:
\beqs
\begin{aligned}
&\|\ep^{\f{1}{2}}(\cC_{\sigma}^{m-1},\cC_{u}^{m-1}-[Z_0^{m-1}\pt,g_1U_z]\p_z u)\|_{L_t^2L^2}\\
&\lesssim \sum_{j=1}^2\big[\|\ep^{\f{1}{2}}\pt (g_j/\ep,g_j u_y,g_j U_z)\|_{L_t^2\cH^{m-1
}} (\il (\sigma,u)\il_{[\f{m}{2}],\infty,t}+\il\nabla\sigma\il_{[\f{m}{2}]-1,\infty,t})\\
&  \quad \qquad + \il \ep\pt (g_j/\ep,g_j u_y,g_j U_z)\il_{[\f{m-1}{2}],\infty,t} \|\ep^{\f{1}{2}}\pt(Z_0,\nabla)(\sigma,u)\|_{L_t^2\cH^{m-2}}\big]\\
&\lesssim \lae. %\Lambda\big(\f{1}{c_0},\cA_{m,t}\big)\cE_{m,t}.
\end{aligned}
\eeqs
\end{proof}

\subsection{Energy estimates II: High-order energy estimate for the compressible part of the system.}
 In this step, we 
 estimate the compressible part $(\nabla^{\vp}\sigma,\div^{\vp}u):$ 
\begin{lem}\label{lemnablasigma}
Under the same assumption as in
Lemma \ref{lemhighest}, the following estimates hold:
\beq\label{EI-2}
\begin{aligned}
&
\ep(\|(\nabla^{\vp}\sigma,\div^{\vp}u)\|_{L_t^{\infty}H_{co}^{m-1}}^2+\|\nabla^{\vp} \div ^{\vp}u\|_{L_t^2H_{co}^{m-1}}^2)\\
&\lesssim\Lambda\big(\f{1}{c_0},|h|_{2,\infty,t}\big)Y_m^2(0) +(T+\ep)^{\f{1}{2}}\lae.
\end{aligned}
\eeq
%\beq\label{EI-T2}\|(\nabla^{\vp}\sigma,\div^{\vp}u)\|_{L_t^{\infty}\cH^{m-2}}^2+\|\nabla^{\vp} \div ^{\vp}u\|_{L_t^2\cH^{m-2}}^2\lesssim \|(\nabla^{\vp}\sigma,\div^{\vp}u)(0)\|_{\cH^{m-2}}^2+(T+\ep)^{\f{1}{6}}\Lambda_{2,\infty,T}\cE_{m,T}^2.\eeq
\end{lem}
\begin{proof}
Let $\beta$ be a multi-index satisfying $|\beta|\leq m-1.$
Applying $Z^{\beta}\nabla^{\vp}$ (resp.\ $Z^{\beta}$) to the equation for $\sigma$ (resp. $u$), we find that:
\beq\label{sec5:eq7}
\left\{
\begin{array}{l}
     g_1(\p_t^\vp+u\cdot\nabla^{\vp})Z^{\beta}\nabla^{\vp}\sigma+\f{1}{\ep}Z^{\beta}\nabla^{\vp}\div^{\vp}u = \cR_{\sigma}^{\beta}\\[5pt]
     g_2(\p_t^\vp+u\cdot\nabla^{\vp})Z^{\beta}u+\mu\curl^{\vp}Z^{\beta}\omega-(2\mu+\lambda)\nabla^{\vp}Z^{\beta}\div^{\vp}u+\f{1}{\ep}Z^{\beta}\nabla^{\vp}\sigma=\cR_{u}^{\beta}
\end{array}
\right.
\eeq
where 
\beq\label{defofRsigmau}
\cR_{\sigma}^{\beta}=\cR_{\sigma,1}^{\beta}+\cR_{\sigma,2}^{\beta}+\cR_{\sigma,3}^{\beta}, \quad \cR_{u}^{\beta}=\cR_{u,1}^{\beta}+\cdots \cR_{u,3}^{\beta},
\eeq
with
\begin{alignat*}{2}
\cR_{\sigma,1}^{\beta}&=Z^{\beta}(\nabla^{\vp}g_1\pt^{\vp}\sigma+\nabla^{\vp}(g_1 u)\cdot\nabla^{\vp}\sigma), &\quad  \cR_{u,1}^{\beta}&=[Z^{\beta},g_2/\ep]\ep\pt u+[Z^{\beta},g_1u_y]\nabla_y u,\\
\cR_{\sigma,2}^{\beta}&=[Z^{\beta},g_1/\ep]\ep\pt \nabla^{\vp}\sigma+[Z^{\beta},g_1u_y]\nabla_y \nabla^{\vp}\sigma, &\quad  \cR_{u,2}^{\beta}&=[Z^{\beta},g_1U_z\p_z]u,\\
\cR_{\sigma,3}^{\beta}&=[Z^{\beta},g_1U_z\p_z]\nabla^{\vp}\sigma, &\quad  \cR_{u,3}^{\beta}&=-\mu[Z^{\beta},\curl^{\vp}]\omega+(2\mu+\lambda)[Z^{\beta},\nabla^{\vp}]\div^{\vp}u
\end{alignat*}
and $U_z$ is defined in \eqref{defofUz}.
Taking the scalar product of \eqref{sec5:eq7} by $(Z^{\beta}\nabla^{\vp}\sigma, -\nabla^{\vp}Z^{\beta}\div^{\vp}u)^{t}$ and by integrating in space and time,
one gets the following energy identity:
\begin{align}\label{defofJ}
\nonumber
&\f{1}{2}\int_{\mS}\big(g_1|Z^{\beta}\nabla^{\vp}\sigma|^2+g_2|Z^{\beta}\div^{\vp}u|^2\big)(t)\,\d\cV_t+(2\mu+\lambda)\|\nabla^{\vp}Z^{\beta}\div^{\vp}u\|_{\ltx}^2\\
&=J_0^{\beta}+J_1^{\beta}+\cdots J_7^{\beta}
\end{align}
with:
\begin{align*}
   &J_0^{\beta}=\f{1}{2}\int_{\mS} \big(g_1|Z^{\beta}\nabla^{\vp}\sigma|^2+g_2|Z^{\beta}\div^{\vp}u|^2\big)(0)\,\d\cV_0, \\
   &J_1^{\beta}=\f{1}{2}\int_0^t\int_{\mS}(\p_t^{\vp}g_1+\div^{\vp}(g_1 u))|Z^{\beta}\nabla^{\vp}\sigma|^2\,\d\cV_s\d s,\\
   &J_2^{\beta}=\int_0^t\int_{\mS}\big(\nabla^{\vp}g_2\cdot\pt^{\vp}Z^{\beta} u+\nabla^{\vp}(g_2u)
   \otimes \nabla^{\vp}Z^{\beta}u\big)Z^{\beta}\div^{\vp}u\,\d\cV_s\d s,\\
   &J_3^{\beta}=\int_0^t\int_{\mS}g_2(\pt^{\vp}+u\cdot\nabla^{\vp}
   )([Z^{\beta},\div^{\vp}]u)Z^{\beta}\div^{\vp}u\,\d\cV_s\d s,\\
&J_4^{\beta}=\int_0^t\int_{z=0}g_2(\pt+u_y\p_y)Z^{\beta}u\cdot\bN Z^{\beta}\div^{\vp}u\,\d y\d s\mathbb{I}_{\{\beta_3=0\}},\\
  & J_5^{\beta}=-\f{1}{\ep}\int_0^t\int_{\mS}
   Z^{\beta}\nabla^{\vp}\sigma[Z^{\beta},\nabla^{\vp}]\div^{\vp} u\,\d\cV_s\d s,\\
   & J_6^{\beta}=\mu\int_0^t\int_{\mS}\curl^{\vp}Z^{\beta}\omega\cdot\nabla^{\vp}Z^{\beta}\div^{\vp} u\,\d\cV_{s} \d s,\\
   &J_7^{\beta}=\int_0^t\int_{\mS}\cR_{\sigma}^{\beta}\cdot Z^{\beta}\nabla^{\vp}\sigma+\cR_{u}^{\beta}\cdot\nabla^{\vp}Z^{\beta}\div^{\vp}u\,\d\cV_{s}\d s.
\end{align*}
The first three terms can be controlled directly:
\beq\label{J0}
\ep J_0^{\beta}\leq \ep \|(\nabla^{\vp}\sigma,\div^{\vp}u)(0)\|_{H_{co}^{m-1}}^2,
\eeq
\beq
\ep (J_1^{\beta}+J_2^{\beta})
\lesssim \ep (\|(\sigma,u)\|_{E^{m},t}^2+|h|_{L_t^2\tilde{H}^{m-\f{1}{2}}}^2).
\eeq
In order to bound $J_3^{\beta},$ we need to control $(\pt^{\vp}+
u\cdot\nabla^{\vp})[Z^{\beta},\div^{\vp}]u.$ By the identity \eqref{useful-identity}, we can write
$$\pt^{\vp}+u\cdot\nabla^{\vp}=\pt+u_1\p_1+u_2\p_2+\f{U_z}{\phi}Z_3.$$ 
Since $U_z|_{\p\mS}=\f{u\cdot\bN-\pt\vp}{\p_z\vp}\big|_{\p S}=0,$ we have by the fundamental theorem of calculus and \eqref{exthLinfty} that:
\beq\label{UzLinfty}
\il U_z/\phi\il_{0,\infty,t}
\lesssim \il(U_z,\p_z U_z)\il_{0,\infty,t}\lesssim\Lambda\big(\f{1}{c_0}, \il (u, \nabla u)\il_{0,\infty,t}+|h|_{2,\infty,t}\big)\lesssim\lca.
\eeq
Therefore, we see that:
\beq\label{J3pre}
\|(\pt^{\vp}+\underline{u}\cdot\nabla^{\vp})[Z^{\beta},\div^{\vp}]u\|_{\ltx}\lesssim \f{1}{\ep} \lca
\|(\ep\pt,\ep Z)[Z^{\beta},\div^{\vp}]u\|_{\ltx}.
\eeq
Let us first consider:
$$\ep\pt[Z^{\beta},\div^{\vp}]u=\ep\pt\big(\f{\bN}{\p_z\vp}[Z^{\beta},\p_z]u\big)+\big[Z^{\beta},\ep\pt\big(\f{\bN}{\p_z\vp}\big)\big]\p_z u+\big[Z^{\beta},\f{\bN}{\p_z\vp}\big]\ep\pt\p_z u.$$
In view of Lemma \ref{lemfpzphi}, the identity \eqref{identity-com-nor} and the commutator estimate \eqref{crudecom}, the first two terms in the right hand side of the above identity can be bounded by:
\beqs
\begin{aligned}
%\|\ep\pt[Z^{\beta},\div^{\vp}]u\|_{\ltx}&\lesssim 
\big(\|\nabla u\|_{\hco^{m-1}}+|(h,\ep\pt h)|_{\htlde^{m-\f{1}{2}}}\big)\Lambda\big(\f{1}{c_0}, \il\nabla u\il_{1,\infty,t}+|h|_{[\f{m}{2}]+2,\infty,t}\big).
%+\big\\&\lesssim(\|\nabla u\|_{\hco^{m-1}}+|(\ep\pt h,h)|_{\htlde^{m-\f{1}{2}}})\Lambda_{2,\infty,t}.
\end{aligned}
\eeqs
For the third one, we control it as:
\beqs 
\begin{aligned}
\big\|\big[Z^{\beta},\f{\bN}{\p_z\vp}\big]\ep\pt\p_z u\big\|_{L_t^2L^2}&\lesssim \big\|\f{\bN}{\p_z\vp}\big\|_{\hco^{m-1}}\il \nabla u\il_{1,\infty,t}+\il\ep^{\f{1}{2}}(\f{\bN}{\p_z\vp})\il_{m-2,\infty,t}\|\ep^{\f{1}{2}}\pt\p_z u\|_{\hco^{m-2}}\\
&\lesssim \lca \cE_{m,t}.
\end{aligned}
\eeqs
Gathering the previous two estimates, we find that:
\beqs 
\|\ep\pt[Z^{\beta},\div^{\vp}]u\|_{\ltx}\lesssim \lca \cE_{m,t}.
\eeqs
In a similar way,  
we have:
\beqs
\|\ep Z[Z^{\beta},\div^{\vp}]u\|_{\ltx}\lesssim \lca \cE_{m,t}.%(\ep\|\nabla u\|_{\hco^{m-1}}+|\ep h|_{\htlde^{m+\f{1}{2}}})\Lambda_{2,\infty,t}.
\eeqs
Plugging the above two estimates into \eqref{J3pre}, we can then control $J_3^{\beta}$ as:
\beq%\label{J3}
\begin{aligned}
\ep J_3^{\beta}&\lesssim \ep^{\f{1}{2}}\|\div^{\vp}u/{\ep^{\f{1}{2}}}\|_{\hco^{m-1}} \|\ep(\pt^{\vp}+\underline{u}\cdot\nabla^{\vp})[Z^{\beta},\div^{\vp}]u\|_{\ltx}\\
&\lesssim 
\ep^{\f{1}{2}}\lca\cE_{m,t}^2.
\end{aligned}
\eeq
We now switch to estimate $J_4^{\beta}.$ 
On the one hand, if $Z^{\beta}=Z_0^k, k\leq m-1,$
we have by the trace inequality \eqref{trace} that:
\beqs 
\begin{aligned}
&\ep^{\f{1}{2}} |(\pt+u_y\p_y)Z_0^k u|_{L_t^2L_y^2}\\
&\lesssim ( \|\ep^{\f{1}{2}}\pt (u,\nabla u)\|_{L_t^2\cH^{m-1}}+\|\ep^{\f{1}{2}}\nabla u\|_{L_t^2H_{co}^m})\Lambda(\il u \il_{0,\infty,t})\lesssim \Lambda(\il u \il_{0,\infty,t})\cE_{m,t}.
\end{aligned}
\eeqs
Therefore, by the trace inequality \eqref{trace}, we get that in this case:
\beq\label{J4-pre1}
\begin{aligned}
\ep J_4^{\beta}&\lesssim \ep^{\f{1}{2}} |Z_0^k \div^{\vp}u|_{L_t^2L_y^2}  |\ep^{\f{1}{2}}(\pt+u_y\p_y)Z_0^k u|_{L_t^2L_y^2}|\bN|_{0,\infty,t}\\
&\lesssim \ep^{\f{1}{2}}(\|\div^{\vp}u/\ep^{\f{1}{2}}\|_{L_t^2\cH^{m-1}}^2+\|\ep^{\f{1}{2}}\nabla \div^{\vp}u\|_{L_t^2\cH^{m-1}}^2+\cE_{m,t}^2)\Lambda(\il u \il_{0,\infty,t}+|h|_{1,\infty,t})\\
&\lesssim \ep^{\f{1}{2}}\Lambda(\il u \il_{0,\infty,t}+|h|_{1,\infty,t})\cE_{m,t}^2.
\end{aligned}
\eeq
On the other hand, if $Z^{\beta}$ contains at least one spatial tangential derivatives 
$\p_{y_1},\p_{y_2},$ we control $\ep J_3^{\beta}$ as follows.
By the equation $\eqref{FCNS2}_2$ and the identity \eqref{sigmabdry}, we can express $(\div^{\vp} u)$ on the boundary $\{z=0\}$ as:
\beqs
{\div^{\vp}u}=\ep g_1(\pt+u_y\p_y)\big(\ep{\div^{\vp}u}+2\mu\ep (\p_1u_1+\p_2 u_2)-\mu\ep(\omega\times\bN)_3\big) \text{ on } \{z=0\}.
\eeqs
This, together with the product estimate \eqref{product}, the identity \eqref{omegatimesn} and the trace inequality \eqref{trace}    yields that:
\begin{align*}
&|(Z^{\beta}\div^{\vp}u)^{b,1}|_{L_t^2H^{-\f{1}{2}}}
\\
&\lesssim |(\div^{\vp}u)^{b,1}|_{\htlde^{m-\f{3}{2}}}\lesssim \ep \big|\big((\div^{\vp}u)^{b,1},\p_y u^{b,1},(\omega\times\bN)_3\big)\big|_{\htlde^{m-\f{1}{2}}}\\%\lca\\
&\lesssim \ep^{\f{1}{2}}\big(\ep^{\f{1}{2}}(|h|_{\htlde^{m+\f{1}{2}}}+\|\nabla u\|_{\hco^{m-1}})\lca+\ep^{\f{1}{2}}\|\nabla\div u\|_{\hco^{m-1}}\\
&\qquad+\ep^{\f{1}{2}}\|\nabla u\|_{\hco^m}\Lambda\big(\f{1}{c_0},|h|_{2,\infty,t}\big)\big),
\end{align*}
 which, combined with the Young's inequality, allows us to control $\ep J_4^{\beta}$  as:
 \beq\label{J4-pre2}
 \begin{aligned}
 |\ep J_4^{\beta}|&\lesssim \Lambda\big(\f{1}{c_0},|h|_{2,\infty,t}\big)|\ep^{\f{1}{2}}(\ep\pt,\ep Z)Z^{\beta} u|_{L_t^2H^{\f{1}{2}}} 
 |\ep^{-\f{1}{2}}Z^{\beta}\div^{\vp}u|_{L_t^2H^{-\f{1}{2}}}\\
 &\leq \delta \|\ep^{\f{1}{2}}\nabla\div^{\vp}u\|_{\hco^{m-1}}^2+C_{\delta}\|\ep^{\f{1}{2}}\nabla u\|_{\hco^m}^2\Lambda\big(\f{1}{c_0},|h|_{2,\infty,t}\big)+ T^{\f{1}{2}}\lca\cE_{m,t}^2.
  \end{aligned}
\eeq
In view of \eqref{J4-pre1} and \eqref{J4-pre2}, we find that:
\beq\label{J4}
|\ep J_4^{\beta}|\leq \delta \|\ep^{\f{1}{2}}\nabla\div^{\vp}u\|_{\hco^{m-1}}^2+C_{\delta}\|\ep^{\f{1}{2}}\nabla u\|_{\hco^m}^2\Lambda\big(\f{1}{c_0},|h|_{2,\infty,t}\big)+(T+\ep)^{\f{1}{2}}\lca\cE_{m,t}^2.
\eeq
 
Next, thanks to \eqref{comgrad},
$J_5^{\beta}$ can be bounded by:
\beq\label{J5}
\begin{aligned}
\ep J_5^{\beta}&\lesssim \ep^{\f{1}{2}} \|\nabla\sigma/\ep^{\f{1}{2}} \|_{\hco^{m-1}}\big(\|\nabla\div^{\vp}u\|_{\hco^{m-2}}+|h|_{\htlde^{m-\f{1}{2}}}\big)\Lambda\big(\f{1}{c_0},|h|_{m-2,\infty,t}+\il\p_z\div^{\vp}u\il_{1,\infty,t}\big)\\
&\lesssim \ep^{\f{1}{2}} \big(\|\nabla\sigma/\ep^{\f{1}{2}} \|_{\hco^{m-1}}^2+\|\nabla\div^{\vp}u\|_{\hco^{m-2}}^2+|h|_{\htlde^{m-\f{1}{2}}}^2\big)\lca\lesssim \ep^{\f{1}{2}}\lca\cE^2_{m,t}.
\end{aligned}
\eeq
Note that by the equation $\eqref{FCNS2}_1,$ we have $\p_z\div^{\vp}u=\p_z(g_1\ep\pt+\ep u_y\p_y+\ep U_{z}%{\phi}Z_3
\p_z)\sigma,$ we thus get that
\beqs
\il\p_z\div^{\vp}u\il_{1,\infty,t}\lesssim\Lambda\big({1}/{c_0}, \il(\sigma,\nabla\sigma)\il_{2,\infty,t}+\il(u,\nabla u)\il_{1,\infty,t}+|h|_{3,\infty,t}\big)\lesssim \lca.
\eeqs
For the next term $J_6^{\beta},$
we assume $\beta_3=0,$ since otherwise it vanishes identically. 
It follows from integration by parts that:
\beqs
\begin{aligned}
 J_6^{\beta}&=\mu \int_0^t\int_{z=0}(Z^{\beta}\omega\times \bn)\Pi
 \nabla^{\vp}Z^{\beta}\div^{\vp}u\,\d y\d s+\mu\int_0^t\int_{z=-1}Z^{\beta}
 (\omega_2,-\omega_1,0)^t\cdot (\p_y,0)Z^{\beta}\div^{\vp}u\,\d y\d s\\
 &=J_{6,1}^{\beta}+J_{6,2}^{\beta}
\end{aligned}
 \eeqs
 where  $\omega=\nabla^{\vp}\times u=(\omega_1,\omega_2,\omega_3)^t.$ 
 In light of the boundary condition \eqref{bebdry2}, we have by integration by parts along the boundary and the trace inequality \eqref{trace} that:
 \beq\label{J62}
 \begin{aligned}
 \ep 
  J_{6,2}^{\beta}\lesssim\ep |u^{b,2}|_{\htlde^{m}}|Z^{\beta}(\div^{\vp}u)^{b,2}|_{L_t^2L^2}&\lesssim \ep^{\f{1}{2}}(\|u\|_{\hco^m}^{\f{1}{2}}\|\ep^{\f{1}{2}}\nabla u\|_{\hco^m}^{\f{1}{2}}+\|u\|_{\hco^m})\cdot\\
&(\|\div^{\vp}u\|_{\hco^m}^{\f{1}{2}}\|\ep^{\f{1}{2}}\nabla\div^{\vp}u\|_{\hco^{m-1}}^{\f{1}{2}}+\|\div^{\vp}u\|_{\hco^{m-1}})\\
 %\leq \delta \ep^2 \|\nabla Z^{\beta}\div^{\vp}u\|_{\ltx}^2+\ep (\|\ep\nabla u\|_{L_t^2H_{co}^{m}}^2+\|u\|_{E^{m},t}^2).
 &\lesssim \ep^{\f{1}{2}}\cE_{m,t}^2.
 \end{aligned}
 \eeq
For $J_{6,1}^{\beta},$ since $\Pi\nabla^{\vp}=\Pi(\p_1,\p_2,0)^{t},$ we also integrate by parts along the boundary to get:
\beqs
\begin{aligned}
\ep J_{6,1}^{\beta} &\lesssim \ep \lca\bigg(\big|\big(Z^{\beta}(\omega^{b,1}\times \bn),Z^{\beta}\bn\big)\big|_{L_t^2H^{\f{1}{2}}}|Z^{\beta}(\div^{\vp}u)^{b,1}|_{L_t^2H^{\f{1}{2}}}\\
&\qquad\qquad\quad +\big|\big(Z^{\beta}\omega^{b,1},[\p_y Z^{\beta},\bn,\omega^{b,1}]\big)\big|_{L_t^2L_y^2}|Z^{\beta}(\div^{\vp}u)^{b,1}|_{L_t^2L_y^2}\bigg).\\
\end{aligned}
\eeqs
Thanks to the boundary condition \eqref{omegatimesn}, we have that
\beqs
|Z^{\beta}(\omega^{b,1}\times \bn)|_{L_t^2H^{\f{1}{2}}}\lesssim |u^{b,1}|_{L_t^2\tilde{H}^{m+\f{1}{2}}}\Lambda\big(\f{1}{c_0},|h|_{2,\infty,t}\big)+(|u^{b,1}|_{L_t^2\tilde{H}^{m-\f{1}{2}}}+|h|_{L_t^2\tilde{H}^{m+\f{1}{2}}})\lca.
\eeqs
Moreover, by \eqref{norpz} \eqref{tanpz}, we have:
\beqs
\begin{aligned}
|Z^{\beta}\omega^{b,1}|_{L_t^2L_y^2}&\lesssim \big(|(\div^{\vp}u)^{b,1}|_{\htlde^{m-1}}+|(u^{b,1},h)|_{\htlde^{m}}\big)\lca,\\
\big|[\p_y Z^{\beta},\bn,\omega^{b,1}]\big|_{L_t^2L_y^2}&\lesssim (|\omega^{b,1}|_{\htlde^{m-1}}+|h|_{\htlde^{m}})\lca\\
&\lesssim 
\big(|(\div^{\vp}u)^{b,1}|_{\htlde^{m-1}}+|(u^{b,1},h)|_{\htlde^{m}}\big)
\lca.
\end{aligned}
\eeqs
Hence, by the trace inequality and Young's inequality, we end up with:
\beq\label{J61}
\begin{aligned}
\ep J_{6,1}^{\beta}&\leq \delta\ep\|\nabla^{\vp} \div^{\vp}u\|_{\hco^{m-1}}^2+\ep\|\nabla u\|_{\hco^{m}}^2\Lambda\big(\f{1}{c_0},|h|_{2,\infty,t}\big)\\
&\quad+\lca(\ep\|\nabla\div^{\vp}u\|_{\hco^{m-2}}^2+\ep | h|_{\htlde^{m+\f{1}{2}}}^2+\ep\| u\|_{E^m,t}).
\end{aligned}
\eeq
Summing up \eqref{J62} and \eqref{J61}, and using \eqref{na-div}, we obtain:
\beq\label{J6}
\ep J_{6}^{\beta}\leq 2\delta\ep^2\|\nabla^{\vp} \div^{\vp}u\|_{\hco^{m-1}}^2+C_{\delta}\Lambda\big(\f{1}{c_0},|h|_{2,\infty,t}\big)\|\ep^{\f{1}{2}} \nabla u\|_{\hco^{m}}^2)+(T+\ep)^{\f{1}{2}}\lca\cE_{m,t}^2.
\eeq
Finally, for $J_7^{\beta},$ by Young's inequality, 
\beq\label{J7-0}
\ep J_7^{\beta}\leq \delta\ep \|\nabla^{\vp}Z^{\beta}\div^{\vp}u\|_{\ltx}^2+C_{\delta}\ep\|\cR_{u}^{\beta}\|_{\ltx}^2+
\ep^{\f{1}{2}}(\|\nabla^{\vp}\sigma/\ep^{\f{1}{2}}\|_{\hco^{m-1}}^2+\ep\|\cR_{\sigma}^{\beta}\|_{\ltx}^2).
\eeq
Hence, it suffices to control 
$\ep^{\f{1}{2}}\|(\cR_{\sigma}^{\beta},\cR_{u}^{\beta})\|_{\ltx}.$ 
Let us first see the estimate of $\ep\cR_{\sigma}^{\beta}.$ 
In view of the definition \eqref{defofRsigmau}, we have by the product estimate \eqref{crudepro} and 
Corollary \ref{corg12} that
\beq\label{sig1}
\begin{aligned}
  \ep^{\f{1}{2}}\|\cR_{\sigma,1}^{\beta}\|_{\ltx}&
  \lesssim \ep^{\f{1}{2}}\lca\big(\|u\|_{E^{m},t}+\|(\ep^{\f{1}{2}}\pt\sigma,\ep^{-\f{1}{2}}\nabla\sigma)\|_{\hco^{m-1}}+|h|_{\htlde^{m-\f{1}{2}}}\big).
\end{aligned}
\eeq
Similarly, by the commutator estimate \eqref{crudecom} and 
Corollary \ref{corg12}, we have that:
\beq
 \ep^{\f{1}{2}}\|\cR_{\sigma,2}^{\beta}\|_{\ltx}\lesssim
 \ep^{\f{1}{2}}\lca\big(\|u\|_{E^{m},t}+\|\ep^{-\f{1}{2}}\sigma\|_{E^m,t}+|h|_{\htlde^{m-\f{1}{2}}}\big).
 %\ep^{\f{1}{2}} \Lambda_{2,\infty,t}(\|(\sigma,u)\|_{E^{m},t}+|h|_{\htlde^{m-\f{1}{2}}}).
\eeq
For $\cR_{\sigma,3}^{\beta},$ we split it as:
\beq\label{Rsi3}
\cR_{\sigma,3}^{\beta}=%Z^{\beta}(\nabla^{\vp}(g_1U_z)\p_z\sigma)-
[Z^{\beta},g_1 U_z/\phi]Z_3\nabla^{\vp}\sigma+(g_1 U_z/\phi)[Z^{\beta},\phi]\p_z\nabla^{\vp}\sigma+g_1 U_z [Z^{\beta},\p_z]\nabla^{\vp}\sigma=\colon (1)+(2)+(3).
\eeq
%By \eqref{product},\eqref{commutator},the first two terms can be controlled directly by:
%\beqs\label{sig12}
%\ep\|(1)+(2)\|\lesssim  \Lambda_{2,\infty,t}(\ep\|(\sigma,u)\|_{E^{k+1},t}+|\ep \pt h|_{\htlde^{k+\f{1}{2}}}+|\ep h|_{\htlde^{k+\f{3}{2}}})\eeqs
%To avoid losing normal derivative on $\sigma,$ we spilt $(3)$ further as:
%\beqs(3)=[Z^{\beta},(g_1U_z)/\phi]Z_3\nabla^{\vp}\sigma+(g_1U_z/\phi)[Z^{\beta},\phi]\p_z\nabla^{\vp}\sigma=\colon (3)_1+(3)_2.\eeqs
Thanks to the commutator estimate \eqref{crudecom}, we have:
\begin{align*}
\ep^{\f{1}{2}}\|(1)\|_{\ltx}&\lesssim \ep^{\f{1}{2}}\|\nabla^{\vp}\sigma\|_{\hco^{m-1}}\il g_1U_z/\phi\il_{[\f{m+1}{2}],\infty,t}+\ep^{\f{1}{2}}\|g_1U_z/\phi\|_{\hco^{m-1}}\il \nabla^{\vp}\sigma\il_{[\f{m}{2}]-1,\infty,t}.
\end{align*}
Note that as $U_z$ vanishes on the boundary, we have by Hardy's inequality,
\beqs
\begin{aligned}
\ep^{\f{1}{2}}\|g_1U_z/\phi\|_{\hco^{m-1}}&\lesssim \ep^{\f{1}{2}} \|\p_z(g_1U_z)\|_{\hco^{m-1}}+\ep^{\f{1}{2}}\|g_1U_z\|_{\hco^{m-1}}\\
&\lesssim \lca ( \|(\sigma,u, \nabla\sigma,\div u)\|_{\hco^{m-1}}+|\ep^{\f{1}{2}} h|_{\htlde^{m+\f{1}{2}}}+|\ep^{\f{1}{2}}\pt h|_{\htlde^{m-\f{1}{2}}}).
\end{aligned}
\eeqs
%We shall detail the $L_{t,x}^{\infty}$ estimate of $(g_1U_z/\phi).$  In the interior domain, since $\phi(z)$ is bounded from below,$$\il g_1U_z/\phi\il_{1,\infty,t}\lesssim \il g_1U_z\il_{2,\infty,t}\lesssim \Lambda_{2,\infty,t}.$$
%In the neighborhood of boundary, since $U_z|_{z=0}=U|_{z=-1}=0,$ the fundamental theorem of calculus yields:
Moreover, as for \eqref{UzLinfty}, the fundamental theorem of calculus leads to:
\begin{align*}
\ep^{\f{1}{2}}\il g_1U_z/\phi\il_{[\f{m+1}{2}],\infty,t} &\lesssim\ep^{\f{1}{2}} \il (U_z,\p_z U_z)\il_{[\f{m+1}{2}],\infty,t}(1+\il Z g_1\il_{[\f{m-1}{2}],\infty,t})\\ &\lesssim \Lambda\big( \f{1}{c_0},\ep^{\f{1}{2}}\il (\sigma,u)\il_{[\f{m+3}{2}],\infty,t}+\ep^{\f{1}{2}}\il \div^{\vp} u)\il_{[\f{m+1}{2}],\infty,t}\\
&\quad+|\ep^{\f{1}{2}} h|_{[\f{m+5}{2}],\infty,t}+|\ep^{\f{1}{2}}\p_t h|_{[\f{m+3}{2}],\infty,t}+\il(\sigma,u)\il_{[\f{m}{2}],\infty,t}+|(h,\pt h)|_{[\f{m}{2}]+1}\big).  \end{align*}
In view of Equation \eqref{surfaceeq2} and the definition \eqref{defcA-free}, we conclude:
\beq\label{na-Uz-Linfty}
\ep^{\f{1}{2}}\il g_1U_z/\phi\il_{[\f{m+1}{2}],\infty,t}\lesssim\lae.
\eeq
We thus obtain that:
\beq
\ep^{\f{1}{2}}\|(1)\|_{\ltx}\lesssim \ep^{\f{1}{2}}\lae.%\Lambda_{2,\infty,t} (\ep^{\f{1}{2}} \|\nabla(\sigma,u)\|_{\hco^{m-1}}+|\ep^{\f{1}{2}} h|_{\htlde^{m+\f{1}{2}}}+|\ep^{\f{1}{2}}\pt h|_{\htlde^{m-\f{1}{2}}}).
\eeq
It remains to estimate $(2),(3)$ in \eqref{Rsi3}. By induction, one has up to some smooth function which depends only on $\phi$ and its derivatives,
\beqs
[Z^{\beta},\phi]=\sum_{\gamma< \beta,|\gamma|\leq |\beta|-1}*_{\beta,\gamma}
Z^{\gamma}\phi,  
\eeqs
The above identity, combined with \eqref{identity-com-nor}, \eqref{na-Uz-Linfty} yields:
\beqs
\ep^{\f{1}{2}}\|(2)+(3)\|_{\ltx}\lesssim\ep^{\f{1}{2}}\|\nabla^{\vp}\sigma/{\ep^{\f{1}{2}}}\|_{\hco^{m-1}} \lca.
\eeqs
To summarize, we have obtained:
\beq\label{sig3}
\ep^{\f{1}{2}} \|\cR_{\sigma,3}^{\beta}\|_{\ltx}\lesssim\lca\ep^{\f{1}{2}}\cE_{m,t}.
\eeq
Collecting \eqref{sig1}-\eqref{sig3}, we thus arrive at:
\beq\label{Rsig}
\ep^{\f{1}{2}}\|\cR_{\sigma}^{\beta}\|_{\ltx}\lesssim \ep^{\f{1}{2}}\lae.
\eeq
To finish the estimates of the right hand side of \eqref{J7-0}, it remains to control $\cR_{u}^{\beta}$ which is defined in \eqref{defofRsigmau}. We first find, in a similar way as for the control of $\cR_{\sigma}^{\beta},$ that:
\beq\label{Ru12}
\ep^{\f{1}{2}}\|(\cR_{u,1}^{\beta}+\cR_{u,2}^{\beta})\|_{\ltx}\lesssim \ep^{\f{1}{2}}\lca\cE_{m,t}.
\eeq
From the identities:
\beqs
[Z^{\beta},\curl^{\vp}]\omega=[Z^{\beta},\f{\bN}{\p_z\vp}\p_z]\times\omega, \quad [Z^{\beta},\nabla^{\vp}]\div^{\vp}u=[Z^{\beta},\f{\bN}{\p_z\vp}\p_z]\div^{\vp}u,
\eeqs
 $\cR_{u,3}^{\beta}$ can be treated thanks to \eqref{comgrad} as:
 \beq\label{Ru3}
 \begin{aligned}
\ep^{\f{1}{2}} \|\cR_{u,3}^{\beta}\|_{\ltx}&\lesssim \Lambda\big(\f{1}{c_0},|h|_{m-2,\infty,t}+\ep^{\f{1}{2}}\il\p_z(\omega,\div^{\vp}u)\il_{1,\infty,t}\big)\big(\ep^{\f{1}{2}}\|\p_z(\omega,\div^{\vp}u)\|_{\hco^{m-2}}+|h|_{\htlde^{m-\f{1}{2}}}\big) \\
&\lesssim \lca\big(\ep^{\f{1}{2}}\|\nabla^2 u\|_{\hco^{m-2}}+\ep^{\f{1}{2}} \|u\|_{E^{m},t}+|h|_{\htlde^{m-\f{1}{2}}}\big).
\end{aligned}
\eeq
Combining \eqref{Ru12} and \eqref{Ru3}, one finds that:
\beq\label{Ru}
\ep^{\f{1}{2}} \|\cR_{u}^{\beta}\|_{\ltx}\lesssim %\lca\|\ep^{\f{1}{2}}\nabla^2 u\|_{\hco^{m-2}}+
(T+\ep)^{\f{1}{2}}\lca\cE_{m,t}.
\eeq
Plugging \eqref{Rsig} and \eqref{Ru} into \eqref{J7-0}, we finally get that:
\beq\label{J7}
\begin{aligned}
\ep|J_7^{\beta}|&\leq
 \delta\ep^2\|\nabla^{\vp}Z^{\beta}\div^{\vp}u\|_{\ltx}^2+(T+\ep) ^{\f{1}{2}}\lca\cE_{m,t}^2.
 %\|(\sigma,u)\|_{E^{m},t}^2 +| h|_{\htlde^{m-\f{1}{2}}}^2+|\ep h|_{\htlde^{m+\f{1}{2}}}^2\big).\\ &
 \end{aligned}
\eeq
Collecting \eqref{J0}-\eqref{J5}, \eqref{J6}, \eqref{J7}, 
and summing up for $k\leq m-1,$
we find that by choosing $\delta$ small enough,
\begin{align*}
&\ep\|(\nabla^{\vp}\sigma,\div^{\vp}u)\|_{L_t^{\infty}H_{co}^{m-1}}^2+\ep\|\nabla^{\vp}\div^{\vp}u\|_{\hco^{m-1}}^2\\
&\lesssim \ep\|(\nabla^{\vp}\sigma,\div^{\vp}u)(0)\|_{H_{co}^{m-1}}^2+\Lambda\big(\f{1}{c_0},|h|_{2,\infty,t}\big)\|\ep^{\f{1}{2}}\nabla^{\vp} u\|_{L_t^2H_{co}^m}^2
+(T+\ep)^{\f{1}{2}}\lae.
\end{align*}
This inequality, combined with \eqref{EI-1} %and \eqref{intemediate1}, 
leads to \eqref{EI-2}.

\end{proof}
\section{Control of the low-order energy norms}
This section is devoted to the control of the lower order term $\cE_{low,T}.$
 and 
 \beq\label{deflowerordernorm2}
{\cE}_{low,T}=\ep^{\f{1}{2}}\|\pt(\sigma,u)\|_{L_t^{\infty}L^2}%+\ep^{-\f{1}{2}}(\|\nabla^2 \sigma\|_{L_t^{\infty}L^2}+\|\nabla^3\sigma\|_{\ltx})
+\ep^{\f{1}{2}}\|(\sigma,u)\|_{L_t^{\infty}H^3}%+\|\nabla^2\sigma\|_{L_t^{\infty}H_{co}^1})
+\ep^{\f{3}{2}} \|\nabla^4 u\|_{\ltx}.
 \eeq
 Except the first norm,
 the other norms appearing in ${\cE}_{low, T}$ are indeed not crucial to get an estimate uniformly in $\ep.$ Nevertheless, their presence allows us to take benefit of the known local existence results %in Sobolev-Slobodeskii spaces 
 \cite{MR697305, Tanaka-Tani, MR1316494} (see Theorem \ref{thm-localexis-nonuniform} in Section 13).

 \begin{lem}\label{lemlow-full}
Under the assumption \eqref{preassumption},
the following estimate holds:
 \begin{equation}\label{loworder-full}
 \begin{aligned}
  &\cE_{low,T}^2\leq \Lambda\big(\f{1}{c_0}, |h|_{3,\infty,T}^2\big)
(Y_m^2(0)+{\cE}_{high,m,T}^2)+ (T+\ep)^{\f{1}{2}}\lae. 
 \end{aligned}
 \end{equation}
 \end{lem}
 \begin{proof}
 This lemma is the consequence of the following three lemmas.
\end{proof}
%Let us begin with the control of 
The first term in $\cE_{low,T}$ is estimated in the next lemma.
 Before stating the result, it is convenient to introduce the notation:
\beq\label{defLinfty2}
\Lambda_{2,\infty,t}=\Lambda\big(\f{1}{c_0}, \il(\sigma,u)\il_{2,\infty,t}+\ep^{-\f{1}{2}}\il(\nabla^{\vp}\sigma, \div^{\vp}u)\il_{1,\infty,t}+\ep^{\f{1}{2}}\il\nabla^2 u\il_{0,\infty,t}+|h|_{3,\infty,t}\big),\eeq
where 
$\Lambda$ denotes a polynomial that may differ from line to line. Note that by the equation for $h$ \eqref{surfaceeq2}, we have:
\beq\label{factpth}|\pt h|_{2,\infty,t}\lesssim \Lambda_{2,\infty,t}.\eeq

\begin{lem}\label{lemloworder}
Assuming that \eqref{preassumption} holds true, then for every $0<t\leq T,$  we have the following estimate,
  \begin{align}\label{lowerder}
  &\ep\|\pt(\sigma,u)\|_{L_t^{\infty}L^2}^2+\ep\|\nabla\pt u\|_{L_t^2L^2}^2
  \lesssim \ep\|\pt(\sigma,u)(0)\|_{L^2(\mS)}^2+ (T+\ep)^{\f{1}{2}}
  \Lambda_{2,\infty,T}
 \mathcal{E}_{m,T}^2.
  \end{align}
 \end{lem}
  \begin{proof}
  Denote $Z_0=\ep\pt.$
 Applying
  $\p_t^{\vp}$ 
  (resp. $\p_t$)
on $\eqref{FCNS2}_1$
  (resp.$\eqref{FCNS2}_2$), one gets that:
  \beq\label{neweq}
  \left\{
  \begin{array}{l}
   \displaystyle g_1(\p_t^{\vp}+\underline{u}\cdot \nabla)(\pt\sigma) +\f{1}{\ep}\pt^{\vp}\div^{\vp}u=\cT_{\sigma}\\[6pt]
    \displaystyle  g_2(\p_t+\underline{u}\cdot \nabla)(\pt^{\vp} u) +\f{1}{\ep}\pt\nabla^{\vp}\sigma-\div^{\vp}(\pt\cL^{\vp}u)=\cT_{u}
  \end{array}
  \right.
  \eeq
  where 
  \beq\label{deftausiu}
  \cT_{\sigma}=\cT_{\sigma}^{1}+\cT_{\sigma}^{2}+\cT_{\sigma}^{3},\,\quad
  \cT_{u}=\cT_{u}^{1}+\cT_{u}^{2}+\cT_{u}^{3}+\cT_{u}^{4}
  \eeq
  with the following definitions:
  \beqs 
  \cT_{\sigma}^{1}=\big(\f{\pt^{\vp} g_1}{\ep}\big)(\ep\pt+\ep\underline{u}\cdot \nabla)\sigma,\quad \cT_{\sigma}^{2}=g_1[\pt,\underline{u}\cdot\nabla]\sigma,\quad \cT_{\sigma}^{3}=-\f{\pt{\vp}}{\p_z\vp}\p_z (\underline{u}\cdot\nabla \sigma),
  \eeqs
    \beqs
  \begin{aligned}
 &\cT_{u}^{1}=\big(\f{\pt g_2}{\ep}\big)(\pt+\underline{u}\cdot \nabla)u, \quad \cT_{u}^{2}=g_2\pt\underline{u}\cdot \nabla u,\\
 &\cT_{u}^{3}=[\pt,\div^{\vp}]\cL^{\vp}u,\quad \cT_{u}^{4}=-g_2(\pt+\underline{u}\cdot \nabla)\big(\f{\p_t\vp}{\p_z\vp}\p_z  u\big).
  \end{aligned}
\eeqs
where $\underline{u}=(u_1,u_2,U_z)$ and $U_z$ is defined in \eqref{defofUz}. Taking the scalar product of \eqref{neweq} and $\ep\big(\p_t\sigma,\p_t^{\vp} u\big)^t,$ integrating in space and time, we get by using Lemma \ref{lemipp} that
   \begin{equation}\label{sec8:eq0}
   \begin{aligned}
 &\f{\ep}{2}\int_{\mS} g_1|\pt\sigma|^2(t)+g_2|\pt^{\vp} u|^2(t)\,\d \cV_t-\ep\int_0^t\int_{\mS}\div^{\vp}(\pt\cL^{\vp}u)\p_t^{\vp}u(s)\,\d\cV_s\d s \\
 &=I_0+I_1+\cdots I_4
   \end{aligned}
   \end{equation}
  where 
  \beqs
  \begin{aligned}
  & I_0=\f{\ep}{2}\int_{\mS} g_1|\pt\sigma|^2(0)+g_2|\pt^{\vp} u|^2(0)\,\d \cV_0,\quad
  I_1=\int_0^t\int_{\mS}\pt\sigma \pt^{\vp}\div^{\vp}u+\pt\nabla^{\vp}\sigma\cdot\pt^{\vp} u\, \d \cV_s\d s,
\\
&  I_2=\f{\ep}{2}\int_0^t\int_{z=0}g_1\pt h|\pt\sigma|^2\,\d y\d s,\quad
 I_3=\f{\ep}{2}\int_0^t\int_{\mS}\big(\pt^{\vp}g_1+\f{1}{\p_z{\vp}}\div(g_1\underline{u}\p_z\vp)\big)|\pt\sigma|^2(s)\,\d\cV_s\d s,\\
   & I_4=\ep\int_0^t\int_{\mS}\pt\sigma \cT_{\sigma}
   +\pt^{\vp} u\cdot\cT_u \,\d\cV_s\d s.
  \end{aligned}
  \eeqs
 We focus on the control of $I_1-I_4$ in the following. Let us with $I_1,$ which is the most involved one and explains why we need to perform energy estimate in this non-standard way.
Let us integrate by parts in space to get:
\begin{align*}
  I_1&=\int_0^t\int_{\mS}\pt^{\vp} u\cdot [\pt,\nabla^{\vp}]\sigma\, \d \cV_s\d s+\int_0^t\int_{z=0}\pt\sigma\pt^{\vp} u\cdot\bN\, \d y\d s=\colon I_{11}+B_1.
\end{align*}
Since $[\pt,\nabla^{\vp}]\sigma=[\pt,\f{\bN}{\p_z\vp}]\p_z\sigma,$ it follows from the Cauchy-Schwarz inequality that:
\beq\label{sec4:eq2.25}
\begin{aligned}
   |I_{11}|&\lesssim \|\pt^{\vp} u\|_{\ltx} \big\|\p_z\sigma\big\|_{L_t^2L^2}
   \il\pt\big(\f{\bN}{\p_z \vp}\big)  \il_{0,\infty,t}\\
   &\lesssim T^{\f{1}{2}}\Lambda_{2,\infty,t}\ep^{\f{1}{2}}\|(\pt u,\nabla u)\|_{L_t^{\infty}L^2}\big\|\ep^{-\f{1}{2}}\nabla\sigma\big\|_{L_t^2L^2}.
\end{aligned}
\eeq
Note that $\Lambda_{2,\infty,t}$
is defined in \eqref{defLinfty2}.
 The boundary term $B_1$ combined with the boundary term
 arising from the integration by parts of the viscous term (in the right hand-side of \eqref{sec8:eq0}), lead to some cancellations, we thus first rewrite the viscous term:
 \beq\label{sec4:eq3}
 \begin{aligned}
 &-\ep\int_0^t\int_{\mS}\div^{\vp}(\pt\cL^{\vp}u)\cdot\p_t^{\vp} u(s)\,\d\cV_s\d s=\ep\int_0^t\int_{\mS}\pt\cL^{\vp}u\cdot\nabla^{\vp}\p_t^{\vp} u\,\d\cV_s\d s\\
 &\quad +\ep a\int_0^t\int_{z=-1} |\pt u_{\tau}|^2\,\d y\d s-\ep
 \underbrace{\int_0^t\int_{z=0}
  \pt\cL^{\vp}u\,\bN \cdot \pt^{\vp} u\,\d y\d s}_{=\colon B_2}.
 \end{aligned}
 \eeq
 In view of the boundary condition \eqref{upbdry2}, the identities  \eqref{grad-upbdry},
\eqref{sigmabdry} as well as the trace inequality \eqref{trace}, we have:
 \beqs
 \begin{aligned}
 B_1+B_2&=-\ep\int_0^t\int_{z=0}\pt^{\vp} u\cdot \big(\cL^{\vp}u-\f{\sigma}{\ep}\text{Id}_3\big)\pt\bN\,\d y\d s\\
 &\lesssim \ep \big|\big(\cL^{\vp}u-\f{\sigma}{\ep}\text{Id}_3\big)\pt\bN\big|_{L_t^2L_y^2}|\pt^{\vp} u|_{L_t^2L_y^2}\\
 &\lesssim \ep^{\f{1}{2}}\big(\|\ep^{\f{1}{2}}\pt u\|_{L_t^2H^1}^2+\|\ep^{\f{1}{2}} u\|_{L_t^2H^2}^2+\|\nabla u\|_{L_t^2H_{co}^1}^2+\|\nabla\div u\|_{L_t^2L^2}^2\big)\Lambda\big(\f{1}{c_0},|\pt h|_{0,\infty,t}+|h|_{1,\infty,t}\big)\\
 &\lesssim \ep^{\f{1}{2}}\Lambda_{2,\infty,t}\cE_{m,t}^2.
 \end{aligned}
 \eeqs
We can also estimate the first two terms in the right hand side of \eqref{sec4:eq3}. By using Young's inequality and the fact $[\nabla^{\vp},\p_t^{\vp}]=0,$
 \beq\label{sec4:eq4}
 \begin{aligned}
 &\ep\int_0^t\int_{\mS}\pt\cL^{\vp}u\cdot\nabla^{\vp}\p_t^{\vp} u\,\d\cV_s\d s\\
 =&\ep\int_0^t\int_{\mS}\pt\cL^{\vp}u\cdot\big(\pt \nabla^{\vp}u-\f{\pt\vp}{\p_z\vp}\p_z \nabla^{\vp}u) \d\cV_s\d s\\
 \geq &\ep\int_0^t\int_{\mS}2\mu|\pt S^{\vp}u|^2+\lambda|\pt\div^{\vp}u|^2\,\d\cV_s\d s-\Lambda_{2,\infty,t}\|\ep^{\f{1}{2}}\pt\cL^{\vp}u\|_{\ltx}\|\ep^{\f{1}{2}}\p_z\nabla^{\vp}u\|_{\ltx}\\
 \geq&\ep \int_0^t\int_{\mS}\mu|\pt S^{\vp}u|^2+\f{\lambda}{2}|\pt\div^{\vp}u|^2\d\cV_s\d s-C_{\mu,\lambda}T\Lambda_{2,\infty,t}\|\ep^{\f{1}{2}}\nabla^{\vp}u\|_{L_t^{\infty}H^1}^2
 \end{aligned}
 \eeq
 Moreover, by the trace inequality, we have
 \beq\label{sec4:eq4.55}
 \ep\int_0^t\int_{z=-1} |\pt u_{\tau}|^2\,\d y\d s\leq \delta \ep\|\pt \nabla^{\vp} u\|_{L_t^2L^2}^2+TC_{\delta}\ep\|(\pt u_{\tau},\nabla u_{\tau})\|_{L_t^{\infty}L^2}^2\Lambda_{2,\infty,t}.
 \eeq
Therefore, we get by collecting \eqref{sec4:eq2.25}-\eqref{sec4:eq4.55} that:
\begin{align}\label{sec4:I3}
  & I_1+\ep\int_0^t\int_{\mS}\div^{\vp}(\pt\cL^{\vp}u)\cdot\p_t^{\vp} u(s)\,\d\cV_s\d s\\
  &\leq -\ep\int_0^t\int_{\mS}\mu|\pt S^{\vp}u|^2+\f{\lambda}{2}|\pt\div^{\vp}u|^2\d\cV_s\d s+\delta\ep \|\pt\nabla^{\vp} u\|_{L_t^2L^2}^2+(T+\ep)^{\f{1}{2}}\Lambda_{2,\infty,t}\cE_{m,t}^2.\notag
\end{align}
  We are now left to control $I_2-I_4.$ The estimates of 
 $I_2,I_3$ are direct, we write
 \beqs
 |I_2|\lesssim \ep|\p_t h|_{\infty,t}|\pt\sigma|_{z=0}|_{L_t^2L_y^2}^2,
 \eeqs
  \beq\label{sec4:eq1}
 |I_3|\lesssim  \Lambda(\il\nabla(\sigma,u)\il_{\infty,t}+\il(\sigma,u)\il_{1,\infty,t}+|h|_{2,\infty,t})\|\ep^{\f{1}{2}}\pt\sigma\|_{\ltx}^2.
  \eeq
We remark that in view of the boundary condition  \eqref{sigmabdry},  one has
\beqs
\pt\sigma|_{z=0}=\pt(\sigma|_{z=0})=\ep\pt\big( (2\mu+\lambda)\div^{\vp}u-2\mu(\p_1u_1+\p_2u_2)+\mu(\omega\times\bN)_3|_{z=0}\big).
\eeqs
Therefore, by the trace inequality \eqref{trace}, we have:
\beq\label{sec4:eq2}
\begin{aligned}
|I_2|\lesssim \ep|\pt h|_{\infty,t}|\pt\sigma|_{L_t^2L_y^2}^2&\lesssim \ep \big(\|\nabla\div^{\vp}u\|_{L_t^2H_{co}^{1}}+\|(u,\nabla u)\|_{L_t^2H_{co}^{2}}
+|h|_{L_t^2\tilde{H}^{2}}\big)\Lambda_{2,\infty,t}\\
&\lesssim \ep \Lambda_{2,\infty,t}\cE_{m,t}^2. 
\end{aligned}
\eeq
As for the term $I_4,$ it can be bounded directly by 
\beq\label{sec4:eq5}
|I_4| \lesssim T^{\f{1}{2}}\big(\|\ep^{\f{1}{2}}\cT_{\sigma}\|_{\ltx}\|\ep^{\f{1}{2}}\pt \sigma\|_{L_t^{\infty}L^2}+\|\ep^{\f{1}{2}}\cT_{u}\|_{\ltx}\|\ep^{\f{1}{2}}\pt  u\|_{L_t^{\infty}L^2}\big).
\eeq
 It thus remains to control the commutators $\cT_{\sigma},\cT_{u}$ defined in \eqref{deftausiu}. 
By the explicit expression of $\cT_{\sigma},\cT_{u},$ we can obtain that:
\beq\label{sec4:eq6}
\ep^{\f{1}{2}}\|(\cT_{\sigma},\cT_{u})\|_{L_t^2L^2}\lesssim \Lambda_{2,\infty,t}(\|\ep^{\f{1}{2}}\pt (\sigma,u)\|_{L_t^2L^2}+\|\nabla(\sigma,u)\|_{L_t^2H_{co}^1})\lesssim \Lambda_{2,\infty,t}\cE_{m,t}.
\eeq
For instance, since we have: $$\ep^{\f{1}{2}}\cT_{\sigma}^1=\ep^{\f{1}{2}}\p_t^{\vp}(g_1/\ep) (\ep\pt+\ep\underline{u}\cdot\nabla)\sigma, \quad \ep^{\f{1}{2}}\cT_{u}^1=\ep^{\f{1}{2}}\p_t(g_2/\ep) (\ep\pt+\ep\underline{u}\cdot\nabla)u$$
by:
 \beqs
  \|\ep^{\f{1}{2}}(\cT_{\sigma}^{1}+\cT_{u}^1)\|\lesssim \Lambda_{1,\infty,t}(\|\ep^{\f{1}{2}}\pt (\sigma,u)\|_{L_t^2L^2}+\|\nabla(\sigma,u)\|_{L_t^2L^2})\lesssim \Lambda_{2,\infty,t}\cE_{m,t}. 
  \eeqs
Collecting \eqref{sec4:eq5}-\eqref{sec4:eq6}, we obtain that
\begin{align}\label{sec4:eq6.5}
 |I_4| &\leq T^{\f{1}{2}}\Lambda_{2,\infty,t}\cE_{m,t}^2.
\end{align}
Now, in view of the estimates: \eqref{sec4:eq1}-\eqref{sec4:eq2}, \eqref{sec4:I3} \eqref{sec4:eq6.5}, we get by choosing $\delta$ small enough,
that
\beq
\begin{aligned}
   & \f{1}{2}\ep\int_{\mS} g_1|\pt\sigma|^2(t)+g_2|\pt^{\vp} u|^2(t)\,\d \cV_t+\ep \int_0^t\int_{\mS}\mu|\pt S^{\vp}u|^2+\f{\lambda}{2}|\pt\div^{\vp}u|^2\,\d\cV_s\d s\\
    &\leq \f{1}{2}\ep\int g_1|\pt\sigma|^2(0)+g_2|\pt^{\vp} u|^2(0)\,\d \cV_0+\delta\ep\|\nabla^{\vp}\pt u\|_{L_t^2L^2}^2+(T+\ep)^{\f{1}{2}}\Lambda_{2,\infty,t}\cE_{m,t}^2. 
% (\sigma,u)\|_{L_t^{\infty}\cH^{r}}^2+\|\nabla(\sigma,u)\|_{L_t^2H_{co}^{r+2}}^2+\|\nabla\sigma/\ep\|_{L_t^{2}\cH^{r}}^2+|(\pt h,h)|_{L_t^2\tilde{H}^{r+2}}^2\big).
\end{aligned}
\eeq
%By the commutator estimates, 
From an explicit commutator, We can write that:
\beqs
\begin{aligned}
&\int_0^t\int_{\mS}\mu|\pt S^{\vp}u|^2+\f{\lambda}{2}|\pt\div^{\vp}u|^2\,\d\cV_s\d s\\
&\geq \int_0^t\int_{\mS}\mu| S^{\vp}\pt u|^2+\f{\lambda}{2}|\div^{\vp}\pt u|^2\,\d\cV_s\d s-\Lambda_{2,\infty,t}T^{\f{1}{2}}\|\nabla u\|_{L_t^{\infty}L^2}^2.
\end{aligned}
\eeqs
Hence, by using Korn's inequality \eqref{korn} and by choosing $\delta$ small enough, we finally obtain \eqref{lowerder}.
\end{proof}
  %\subsection{Estimate of $\overline{\cE}_{low,T}$}
 The following two lemmas are devoted to the estimates of the other norms appearing in $\cE_{low,T},$ for the proof of Lemma \ref{lemlow-full}.
 \begin{lem}
 Suppose that \eqref{preassumption} are holds, then we have for any $0< t\leq T,$
 \beq\label{sec9.2-1}
 \ep\|\nabla^3\sigma\|_{L_t^{\infty}L^2}^2+\ep^{-1}\|\nabla^3\sigma\|_{\ltx}^2 
 +\ep\|\nabla^2\sigma\|_{L_t^{\infty}H_{co}^1}^2+\ep^{-1}\|\nabla^2\sigma\|_{L_t^2H_{co}^1}^2
 \lesssim Y_m^2(0)+ (T+\ep)%^{\f{1}{2}}
 \Lambda_{2,\infty,t}\cE_{m,t}^2. 
 \eeq
 \end{lem}
 \begin{proof}
By applying $\ep^2\nabla^{\vp}$ to the equation $\eqref{FCNS2}_1$ and 
expressing the term $\ep\nabla^{\vp}\div^{\vp}u$ by using the velocity equations $\eqref{FCNS2}_2,$ we find that $\nabla^{\vp}\sigma$ solves
\beq\label{gradsigma}
\ep^2 g_1(\pt+\underline{u}\cdot\nabla)\nabla^{\vp}\sigma+\f{1}{2\mu+\lambda}\nabla^{\vp}\sigma= \cQ_1%+\cQ_2+\cQ_3.
\eeq
where 
$$\cQ_1=
-\ep^2g_1'\nabla^{\vp}\sigma(\ep\pt+\ep\underline{u}\cdot\nabla)\sigma-\ep^2 g_1\nabla^{\vp}u\cdot\nabla^{\vp}\sigma
-\f{\mu \ep}{2\mu+\lambda}\curl^{\vp}\omega-\f{1}{2\mu+\lambda}g_2(\ep\pt+\ep \underline{u}\cdot\nabla)u.$$

Next, by taking $\div^{\vp}$ of the equation 
\eqref{gradsigma}, we find that 
$\Delta^{\vp}\sigma$ solves:
 \beq\label{eqDeltasigma}
 \begin{aligned}
 \ep^2 g_1(\pt+\underline{u}\cdot\nabla)\Delta^{\vp}\sigma+\f{1}{2\mu+\lambda}\Delta^{\vp}\sigma&=\div^{\vp}\cQ_1-\ep^2g_1'\nabla^{\vp}\sigma\cdot\ep\pt\nabla^{\vp}\sigma-\ep^2\nabla^{\vp}(g_1\underline{u})\cdot\nabla\nabla^{\vp}\sigma\\
 &=\colon \cH
 \end{aligned}
 \eeq
Standard energy estimates for
 \eqref{eqDeltasigma} yield:
 \beqs
 \begin{aligned}
 &\quad \ep\|\Delta^{\vp}\sigma\|_{L_t^{\infty}H_{co}^1}^2+\ep^{-1}\|\Delta^{\vp}\sigma\|_{L_t^2H_{co}^1}^2\\
 &\lesssim \ep\|\Delta^{\vp}\sigma(0)\|_{H_{co}^1}^2+T\Lambda_{1,\infty,t}\ep\|\Delta^{\vp}\sigma\|_{L_t^{\infty}H_{co}^1}^2+T^{\f{1}{2}}\|\ep^{-\f{1}{2}}\Delta^{\vp}\sigma\|_{L_t^2H_{co}^1}(\|\ep^{-\f{1}{2}}\cH\|_{L_t^{\infty}H_{co}^1}+\ep^{\f{1}{2}}\Lambda_{2,\infty,t}\cE_{m,t})\\
 &\lesssim  T\Lambda_{1,\infty,t}\cE_{low,t}^2+T^{\f{1}{2}}(\|\ep^{\f{1}{2}}\pt\div^{\vp}u\|_{L_t^{\infty}H_{co}^1}+\ep^{\f{1}{2}}\Lambda_{2,\infty,t}\cE_{m,t})\|\ep^{-\f{1}{2}}\Delta^{\vp}\sigma\|_{L_t^2H_{co}^1}%\lesssim (T+\ep)^{\f{1}{2}}\Lambda_{1,\infty,t}\cE_{m,t}^2.
 \end{aligned}
 \eeqs
 It thus follows from Young's inequality that 
 \beqs
\ep \|\Delta^{\vp}\sigma\|_{L_t^{\infty}H_{co}^1}^2+\ep^{-1}\|\Delta^{\vp}\sigma\|_{L_t^2H_{co}^1}^2\lesssim Y_m^2(0) +T\Lambda_{2,\infty,t}\cE^2_{m,t}.
 \eeqs
 Moreover, we can get also that:
 \beqs
 \begin{aligned}
 &\quad\ep\|\p_z\Delta^{\vp}\sigma\|_{L_t^{\infty}L^2}^2+\ep^{-1}\|\p_z\Delta^{\vp}\sigma\|_{\ltx}^2\\
 &\lesssim \ep\|\p_z\Delta^{\vp}\sigma(0)\|_{L^2}^2+ T \Lambda_{2,\infty,t}\big(\|\ep\nabla^3\sigma\|_{L_t^{\infty}L^2}^2+\|\ep^{\f{1}{2}}\pt\nabla\div u\|_{L_t^{\infty}L^2}^2+\ep \cE_{m,t}^2\big)\\
 &\lesssim Y_m^2(0)+T\Lambda_{2,\infty,t}\cE_{m,t}^2.
 \end{aligned}
 \eeqs
Next, we see that:
 \beqs
 \ep\| \nabla^2\sigma\|_{L_t^{\infty}H_{co}^1}^2\lesssim \ep\|\nabla\sigma\|_{L_t^{\infty}H_{co}^2}^2+\ep\|\p_z^2\sigma\|_{L_t^{\infty}L^2}^2
 \eeqs
  By the expressions of $\Delta^{\vp}\sigma,$
  \beq\label{laplacesigma}
 \Delta^{\vp}\sigma=\f{|\bN|^2}{\p_z\vp}\p_z^2 \sigma+\Delta_y \sigma+\p_1(\bN_1\p_z^{\vp}\sigma)+\p_2(\bN_2\p_z^{\vp}\sigma)+\bN_1\p_z^{\vp}\p_1 \sigma+\bN_2\p_z^{\vp}\p_2\sigma+\f{1}{2}\p_z \sigma\p_z\big|\f{\bN}{\p_z\vp}\big|^2,
 \eeq
 Therefore,
 \beqs
  \begin{aligned}
 \ep\| \nabla^2\sigma\|_{L_t^{\infty}H_{co}^1}^2&\lesssim \ep\Lambda(1/c_0,|h|_{3,\infty,t})\|\nabla\sigma\|_{L_t^{\infty}H_{co}^2}^2+\ep
 \|\Delta^{\vp}\sigma\|_{L_t^{\infty}H_{co}^1}^2\\
 &\lesssim Y_m^2(0)+(T+\ep)\Lambda_{2,\infty,t}\cE_{m,t}^2.
 \end{aligned}
 \eeqs
 Note that $|h|_{3,\infty,t}$ is included in the definition of $\Lambda_{2,\infty,t}$ \eqref{defLinfty2}.
 We have further that:
\beqs
\ep\|\nabla^3\sigma\|_{L_t^{\infty}L^2}^2
\lesssim \ep\|\p_z\Delta^{\vp}\sigma\|_{L_t^{\infty}L^2}^2+\ep\Lambda_{2,\infty,t}
\|\nabla^2\sigma\|_{L_t^{\infty}H_{co}^1}^2\lesssim Y_m^2(0)+ (T+\ep)\Lambda_{2,\infty,t}\cE_{m,t}^2.
\eeqs
In a similar way, the following estimate holds also:
\beqs
\ep^{-1}\|\nabla^3\sigma\|_{\ltx}^2
+\ep^{-1}\|\nabla^2\sigma\|_{L_t^2H_{co}^1}^2
\lesssim Y_m^2(0)+ T\Lambda_{2,\infty,t}\cE_{m,t}^2.
\eeqs
The proof of \eqref{sec9.2-1} is now finished.
 \end{proof}
 \begin{rmk}
In a similar way, one can also show that:
\beq\label{na3sigma-1}
\|\nabla^3\sigma\|_{L_t^2H_{co}^1}
\lesssim Y_{m}(0)+(T+\ep)^{\f{1}{2}}\cE_{m,t}.
\eeq
 \end{rmk}
\begin{lem}
Assume that \eqref{preassumption} holds,
then we have for any $0< t\leq T:$ 
\beq\label{sec9.2-5}
\begin{aligned}
&\quad \ep^{-1}\|\nabla^2\sigma\|_{L_t^{\infty}L^2}^2+ \ep\|\nabla^3 u\|_{L_t^{\infty}L^2}^2+\ep^3\|\nabla^4 u\|_{\ltx}^2
\\
&\lesssim \Lambda\big(\f{1}{c_0}, |h|_{3,\infty,t}^2\big) (\ep\|\nabla^2\sigma\|_{L_t^{\infty}H_{co}^1}^2+{\cE}_{high,m,t}^2)%+\ep^{-1}\|\nabla\sigma\|_{L_t^2H^1}^2\big)
+(T+\ep)\Lambda_{2,\infty,t}{\cE}_{m,t}^2.
\end{aligned}
\eeq
\end{lem}

\begin{proof}
 By taking $\div^{\vp}$ on $\eqref{FCNS2}_2$, we see that $\sigma$ solves the following elliptic problem:
 \beq\label{sigmaelliptic}
 \left\{
 \begin{array}{l}
      -\Delta^{\vp}(\sigma/\ep)=\div^{\vp}G,  \\[5pt]
      \sigma/\ep=(2\mu+\lambda)\div^{\vp}u-2\mu(\p_1u_1+\p_2u_2)-\mu(\omega\times \bN)_3 \quad\text{on}\quad \{z=0\},\\[5pt]
      \p_z^{\vp}\sigma/\ep=-G\cdot e_3+\mu\curl^{\vp}\omega\cdot e_3 \quad\text{on}\quad \{z=-1\},
 \end{array}
 \right.
 \eeq
where 
\beq\label{defofG}
G=\bar{\rho}\pt^{\vp}u+g_2u\cdot\nabla^{\vp}u+\f{g_2-\bar{\rho}}{\ep}\ep\p_t^{\vp}u-(2\mu+\lambda)\nabla^{\vp}\div^{\vp}u.
\eeq 
Note that on the upper boundary we have boundary identity \eqref{omegatimesn} for $\omega\times\bN$ and
on the bottom,  we have
\beq\label{curlwdotn}
\mu\curl^{\vp}\omega\times e_3=\mu(\p_1^{\vp}\omega_2-\p_2^{\vp}\omega_1)=a(\p_1 u_1+\p_2 u_2).
\eeq
Applying the elliptic estimate \eqref{elliptic1.5}, we find that:
\beqs
\begin{aligned}
\ep^{-\f{1}{2}}\|\nabla^2\sigma\|_{L_t^{\infty}L^2}
&\lesssim \Lambda\big(\f{1}{c_0}, |h|_{3,\infty,t}\big) \big(\ep^{\f{1}{2}}\|(\div^{\vp}G,G)\|_{L^{\infty}L^2}
+|\ep^{-\f{1}{2}}\sigma^{b,1}|_{L_t^{\infty}H^{\f{3}{2}}}+|\ep^{-\f{1}{2}}(\p_{\bn} \sigma)^{b,2}|_{L_t^{\infty}H^{\f{1}{2}}}\big)\\
&\lesssim \Lambda\big(\f{1}{c_0}, |h|_{3,\infty,t}\big)(\ep^{\f{1}{2}} \|\div^{\vp}u\|_{L_t^{\infty}H^2}+\|\ep^{\f{1}{2}} \pt u\|_{L_t^{\infty}H_{co}^1}+\|\ep^{\f{1}{2}} \pt\div^{\vp} u\|_{L_t^{\infty}L^2})+\ep^{\f{1}{2}}\Lambda_{2,\infty,t}\cE_{m,t}\\
&\lesssim \Lambda\big(\f{1}{c_0}, |h|_{3,\infty,t}\big)(\ep^{\f{1}{2}}\|\nabla^2\sigma\|_{L_t^{\infty}H_{co}^1}+\tilde{\cE}_{m,t})+\ep^{\f{1}{2}}\Lambda_{2,\infty,t}\cE_{m,t},
\end{aligned}
\eeqs
where $G$ is defined %respectively  \eqref{defPE},
in \eqref{defofG}.
Note that by $\eqref{FCNS2}_1$ and the definition of $\cE_{m,t},$
\beqs
\ep^{\f{1}{2}}\|\div^{\vp}u\|_{L_t^{\infty}H^2}\lesssim \ep^{\f{1}{2}}\|\nabla^2\sigma\|_{L_t^{\infty}H_{co}^1}+\ep^{\f{1}{2}}\Lambda_{2,\infty,t}\cE_{m,t}.
\eeqs
Next, we get by the equation of velocity \eqref{FCNS2} that:
$$\ep%\div^{\vp}\cL^{\vp}
\mu\Delta^{\vp} u=g_2(\ep\pt+\underline{u}\cdot\nabla)u-(\mu+\lambda)\nabla^{\vp}\div^{\vp}u+\nabla\sigma$$
%Moreover, by straightforward computation, \beqs\mu\Delta^{\vp} u=\mu\f{|\bN|^2}{\p_z\vp}\p_z^2 u+\eeqs
 Moreover, a direct computation shows that:
 \beq\label{laplaceu}
 \Delta^{\vp}u=\f{|\bN|^2}{|\p_z\vp|^2}\p_z^2 u+\Delta_y u+\p_1(\bN_1\p_z^{\vp}u)+\p_2(\bN_2\p_z^{\vp}u)+\bN_1\p_z^{\vp}\p_1 u+\bN_2\p_z^{\vp}\p_2u+\f{1}{2}\p_z u\p_z\big|\f{\bN}{\p_z\vp}\big|^2.
 \eeq
By using the previous two identities successively, we find the following two estimates
\beqs
\begin{aligned}
\ep^{\f{1}{2}}\|\nabla^3u\|_{L_t^{\infty}L^2}
&\lesssim \ep^{\f{1}{2}}\|\p_z\Delta^{\vp}u%, \nabla^{\vp}\div^{\vp}u)
\|_{L_t^{\infty}L^2}+\ep^{\f{1}{2}}\|\nabla^2 u\|_{L_t^{\infty}H_{co}^1}\Lambda_{2,\infty,t}\\
&\lesssim\ep^{-\f{1}{2}}\|\nabla\sigma\|_{L_t^{\infty}H^1}+\ep^{\f{1}{2}}
\|\nabla^2\sigma\|_{L_t^{\infty}H_{co}^1}+|h|_{L_t^{\infty}\tilde{H}^{\f{3}{2}}}+\|\ep^{\f{1}{2}}\pt u\|_{L_t^{\infty}H^1}+\ep^{\f{1}{2}}
\Lambda_{2,\infty,t}\cE_{m,t}.
\end{aligned}
\eeqs
and 
\beqs
\begin{aligned}
\ep^{\f{3}{2}}\|\nabla^4 u\|_{L_t^2L^2}
&\lesssim \ep^{\f{3}{2}}\|\nabla^2
\Delta^{\vp}u\|_{L_t^2L^2}
+\ep^{\f{3}{2}}\big(\|\nabla^3 u\|_{L_t^2H_{co}^1}+|h|_{\htlde^{\f{7}{2}}}\big)\Lambda_{2,\infty,t}\\
&\lesssim \ep^{\f{1}{2}} \|\nabla^3\sigma\|
_{L_t^2L^2}+\ep^{\f{1}{2}}\|\nabla^2 u\|_{L_t^2\cH^1}
+\ep^{\f{3}{2}}\|\nabla^3 (\sigma,u)\|_{L_t^2H_{co}^1}\Lambda_{2,\infty,t}+\ep\Lambda_{2,\infty,t}\cE_{m,t}\\
&\lesssim %\|\nabla^3\sigma\|_{L_t^2L^2}+\ep\|\nabla^2 u\|_{L_t^2H_{co}^2}+
\ep^{\f{1}{2}}\|\pt u\|_{L_t^2\cH^1}+
\ep^{\f{3}{2}}\|\nabla^3\sigma\|_{L_t^2H_{co}^1}+
(T^{\f{1}{2}}+\ep)\Lambda_{2,\infty,t}\cE_{m,t}\\
&\lesssim Y_{m}(0)+(T+\ep)^{\f{1}{2}}\Lambda_{2,\infty,t}\cE_{m,t}. 
\end{aligned}
\eeqs
Note that in the second estimate, \eqref{na3sigma-1} has been used in the derivation of the last inequality.
\end{proof}
As stated in the beginning, we can now finish the proof of Lemma \ref{lemlow-full} since
gathering \eqref{lowerder}, \eqref{sec9.2-1} and \eqref{sec9.2-5} we finally obtain
\eqref{loworder-full}.

In the following several sections
(Sections 9-11),
we aim to show the estimate of high order norms $\cE_{high,m,T}$ defined in \eqref{defenergy-high}.
\section{Uniform control of high order energy norms-I}
 In this section, we focus on the uniform $L_t^2H_{co}^{m-1}$ estimates for $\nabla^{\vp}(\sigma,u).$ We first bound the higher order norms for 
 $(\nabla^{\vp}\sigma, \div^{\vp}u)$ 
 by using elliptic estimates for $\sigma$ and the equations to recover spatial derivatives from time derivatives  iteratively. Then, we perform direct energy estimates for the incompressible part $v$ ($v=\bbp u$ solves \eqref{eqofv1})
  to get the uniform control for $\|\nabla^{\vp}v\|_{L_t^2H_{co}^{m-1}}$ (and also $\|v\|_{L_t^{\infty}H_{co}^{m-1}}$ as a by-product).
\subsection{Uniform estimates for the compressible part}\label{sec-compressible}
In this subsection, we focus on the uniform estimates of the compressible part of the solution. More precisely, we shall establish the estimate of $\|(\nabla^{\vp}\sigma,\div^{\vp}u)\|_{
L_t^2H_{co}^{m-1}}.$
\begin{lem}\label{sigmainduction}
Suppose that \eqref{preassumption} is true, we can find some polynomial $\Lambda$, such that, for any $0<t\leq T,$
\begin{equation}\label{nor-compressible} 
\begin{aligned}
\ep^{-1}\|(\nabla^{\vp}\sigma,\div^{\vp}u)\|_{ L_t^2H_{co}^{m-1}}^2+\ep^{-1}\|\nabla\div^{\vp}u\|_{\hco^{m-2}}^2\\
\lesssim \Lambda\big(\f{1}{c_0},|h|^2_{L_T^{\infty}\tilde{H}^{m-\f{1}{2}}}\big)Y^2_m(0)+(T+\ep)^{\f{1}{2}}\lae.
\end{aligned}
\end{equation}
More precisely, we have for any $j,l$ with $j+l\leq m-1,$
\beq\label{induction1}
\begin{aligned}
&\ep^{-\f{1}{2}}\|(\nabla^{\vp}\sigma,\div^{\vp}u)\|_{L_t^2\cH^{j,l}}\lesssim (T+\ep)^{\f{1}{2}}\Lambda\big(\f{1}{c_0},\cN_{m,T}\big)\\
&\qquad+\big(\ep^{\f{1}{2}}\| \nabla\div^{\vp}u\|_{\hco^{m-1}}+\ep^{\f{1}{2}}\|\nabla^{\vp}u\|_{\hco^{m}}+
\ep^{\f{1}{2}}\|\pt (\sigma,u)\|_{L_t^2\cH^{m-1}}\big)\Lambda\big(\f{1}{c_0},|h|_{L_T^{\infty}\tilde{H}^{m-\f{1}{2}}}\big).
\end{aligned}
\eeq
\end{lem}
\begin{proof}
By using the equation $\eqref{FCNS2}_1$ for $\sigma,$ we have:
\beq\label{eq-graddiv}
\nabla\div^{\vp} u=g_1(0)\ep\pt\nabla\sigma+\ep \nabla\big((\f{g_1-g_1(0)}{\ep}\ep\pt\sigma)+g_1\underline{u}\cdot \nabla \sigma\big),
\eeq
combined with the product estimate \eqref{crudepro}, this yields:
\beq\label{na-div}
\ep^{-\f{1}{2}}\|\nabla\div^{\vp}u\|_{L_t^2H_{co}^{m-2}}\lesssim \ep^{-\f{1}{2}}\|\nabla\sigma\|_{L_t^2H_{co}^{m-1}}+\ep^{\f{1}{2}}\Lambda\big(\f{1}{c_0},\cA_{m,t}\big)\cE_{m,t}.
\eeq
By  \eqref{EI-1}, \eqref{EI-T1}, \eqref{EI-2}, \eqref{na-div}, we can derive \eqref{nor-compressible} from \eqref{induction1}.
 In what follows, we shall establish \eqref{induction1}
 by induction on the number of conormal spatial derivatives.
 Firstly, let us rewrite the equation $\eqref{FCNS2}_1$ as:
 \beq\label{re-sigma}
 \div^{\vp}u=g_1(0)\ep\pt \sigma+ \ep\big(\f{g_1-g_1(0)}{\ep}\ep\pt\sigma+g_1u\cdot\nabla\sigma\big),
 \eeq
 By the product estimate \eqref{crudepro}, we obtain:
 \beqs
\ep^{-\f{1}{2}} \|\div^{\vp}u\|_{L_t^2\cH^{m-1}}\lesssim \|\ep^{\f{1}{2}}\pt\sigma\|_{L_t^2\cH^{m-1}}+\ep ^{\f{1}{2}}
 \Lambda\big(\f{1}{c_0},\cA_{m,t}\big)\cE_{m,t}.
 \eeqs
Moreover, as $\sigma$ solves by the elliptic problem \eqref{sigmaelliptic},
we can apply the elliptic estimate \eqref{elliptic5} with
\beqs
b=\sigma^{b,1},\,g=(\ep\mu\curl^{\vp}\omega\cdot e_3)^{b,2}\, F=\ep P {G} \, (\text{the vector } G \text{ is defined in } \eqref{defofG}, \text{ the matrix } P \text{ is defined in } \eqref{defofP} )
\eeqs
and the identity \eqref{curlwdotn}
 to get:
 \beqs
 \begin{aligned}
& \ep^{-\f{1}{2}}\|\nabla^{\vp}\sigma\|_{L_t^2\cH^{m-1}}\lesssim \Lambda\big(\f{1}{c_0},|h|_{[\f{m}{2}]+1,\infty,t}\big)\big(\|\ep^{\f{1}{2}}G\|_{L_t^2\cH^{m-1}}+|\ep^{-\f{1}{2}}\sigma^{b,1}|_{\htlde^{m-\f{1}{2}}}+\ep^{\f{1}{2}}|\p_y u^{b,2}|_{\htlde^{m-\f{3}{2}}}\big)\\
&\qquad\qquad\qquad\qquad\quad+\Lambda\big(\f{1}{c_0},|h|_{[\f{m}{2}]+1,\infty,t}+\il (\ep^{-\f{1}{2}}\nabla\sigma,\ep^{\f{1}{2}}G) \il_{[\f{m}{2}]-1,\infty,t}\big)|h|_{L_t^2\tilde{H}^{m-\f{1}{2}}}.
 %&\lesssim \Lambda\big(\f{1}{c_0},|h|_{[\f{m}{2}]+1,\infty,t}) \ep^{\f{1}{2}}\big(\|\pt u\|_{L_t^2\cH^{m-1}}+\|\nabla\div^{\vp}u\|_{\hco^{m-1}}+\|\nabla^{\vp}u\|_{\hco^m}\big)+(T+\ep)^{\f{1}{2}} \Lambda\big(\f{1}{c_0},\cA_{m,t}\big)\cE_{m,t}.
  \end{aligned}
 \eeqs
By the definition \eqref{defofG} of $G$ and the product estimate \eqref{crudepro},
\beqs
\il\ep^{\f{1}{2}}G\il_{[\f{m}{2}]-1,\infty,t}\lesssim \Lambda\big(\f{1}{c_0},\cA_{m,t}\big),
\eeqs
\beqs
\|\ep^{\f{1}{2}}G\|_{L_t^2\cH^{m-1}}\lesssim \ep^{\f{1}{2}}\big(\|\pt u\|_{L_t^2\cH^{m-1}}+\|\nabla\div^{\vp}u\|_{\hco^{m-1}})+\ep^{\f{1}{2}}\Lambda\big(\f{1}{c_0},\cA_{m,t}\big)\cE_{m,t}.
\eeqs
Moreover, thanks to the identity \eqref{sigmabdry} and the trace inequality \eqref{trace}, we have that:
\beqs
\begin{aligned}
&|\ep^{-\f{1}{2}}\sigma^{b,1}|_{\htlde^{m-\f{1}{2}}}+\ep^{\f{1}{2}}|\p_y u^{b,2}|_{\htlde^{m-\f{3}{2}}}\\
&\lesssim 
\Lambda\big(\f{1}{c_0},|h|_{L_t^{\infty}\tilde{H}^{m-\f{1}{2}}}\big)
\ep^{\f{1}{2}}(\|\nabla u\|_{\hco^{m}}+\|\nabla\div^{\vp}u\|_{\hco^{m-1}})+(T+\ep)^{\f{1}{2}}\lae.
\end{aligned}
\eeqs
Gathering the previous four inequalities,
we get \eqref{induction1} for $j\leq m-1,l=0.$
 For a given integer $l$ $(1\leq l\leq m-1),$ assuming now that \eqref{induction1} holds for $(j,l-1)$ with
 $j+l\leq m-1$ we then prove that it is also true for $(j,l)$ with $j+l\leq m-1.$
 By equation $\eqref{re-sigma}$ and the product estimate \eqref{crudepro}, we get:
 \beqs
 \begin{aligned}
\ep^{-\f{1}{2}} \|\div^{\vp}u\|_{L_t^2\cH^{j,l}}&\lesssim \|\ep^{\f{1}{2}}\pt \sigma\|_{L_t^2\cH^{j,l}}+
 (T+\ep)^{\f{1}{2}}\lae\\%\Lambda_{1,\infty,t}\cE_{m,t}\\
 &\lesssim \|\ep^{-\f{1}{2}}\nabla^{\vp} \sigma\|_{L_t^2\cH^{j+1,l-1}}+
(T+\ep)^{\f{1}{2}}\lae\lesssim \text{R.H.S of } \eqref{induction1}.
 \end{aligned}
 \eeqs
 For the estimate of $\nabla^{\vp}\sigma,$ we first remark that in the elliptic equation \eqref{sigmaelliptic}, $G$ (defined in \eqref{defofG}) can be simplified slightly by changing $\p_t^{\vp} u$ into $\p_t^{\vp}\nabla\Psi,$  since $\div^{\vp}v=0,\p_t^{\vp}v_3|_{z=-1}=0.$ Denote thus
 $$\tilde{G}=\bar{\rho}\pt^{\vp}\nabla^{\vp}\Psi+g_2u\cdot\nabla^{\vp}u+\f{g_2-\bar{\rho}}{\ep}\ep\p_t^{\vp}u-(2\mu+\lambda)\nabla^{\vp}\div^{\vp}u.$$
 We can use again the elliptic estimate \eqref{elliptic5} %and the product estimate \eqref{product} 
 to get that:
 \begin{multline*}
\ep^{-\f{1}{2}} \|\nabla^{\vp}\sigma\|_{L_t^2\cH^{j,l}}\lesssim (T+\ep)^{\f{1}{2}}\lae\\
+ \Lambda\big(\f{1}{c_0}, |h|_{L_t^{\infty}\tilde{H}^{m-\f{1}{2}}}\big)\big(\|\ep^{\f{1}{2}} \tilde{G}\|_{L_t^2\cH^{j,l}}
 +\ep^{\f{1}{2}}(\|\nabla u\|_{\hco^{m}}+\|\nabla\div^{\vp}u\|_{\hco^{m-1}})\big)\\
 \lesssim \Lambda\big(\f{1}{c_0}, |h|_{L_t^{\infty}\tilde{H}^{m-\f{1}{2}}}\big)
 \ep^{\f{1}{2}}\big(\|\pt \nabla^{\vp}\Psi\|_{L_t^2\cH^{j,l}}+\|\nabla^{\vp}\div^{\vp}u\|_{\hco^{m-1}}+\|\nabla^{\vp}u\|_{\hco^{m}}\big)+(T+\ep)^{\f{1}{2}}\lae.
 \end{multline*}
% Note that we control the term $g_2 u\cdot\nabla^{\vp}u$ is controlled thanks to the product estimate \eqref{crudepro} as follows:
 %\beqs \begin{aligned}\ep^{\f{1}{2}}\|u\cdot\nabla^{\vp}u\|_{\hco^{m-1}}&\lesssim \Lambda\big(\il (\sigma, u)\il_{[\f{m}{2}],\infty,t}+\ep^{\f{1}{2}}\il \nabla^{\vp}u\il_{[\f{m}{2}]-1,\infty,t}\big)(\ep^{\f{1}{2}}\|\nabla^{\vp}u\|_{\hco^{m-1}}+\|(\sigma,u)\|_{\hco^{m-1}})\\& \lesssim (T+\ep)^{\f{1}{2}}\lca \cE_{m,T}.\end{aligned}\eeqs
 
Since $\Psi$ solves the elliptic problem \eqref{eq-comp}, we can apply the elliptic estimate \eqref{elliptic3.3} and the estimate \eqref{Linftynablapsi} to get that:
\beqs
\begin{aligned}
\|\ep^{\f{1}{2}}\pt\nabla^{\vp}\Psi\|_{L_t^2\cH^{j,l}}&\lesssim
\Lambda\big(\f{1}{c_0},|h|_{[\f{m}{2}]+1,\infty,t}\big)
\ep^{-\f{1}{2}}\|\div^{\vp}u\|_{L_t^2\cH^{j+1,l-1}}+(T+\ep)^{\f{1}{2}}\lae.
\end{aligned}
\eeqs
Combining the two previous inequalities and using the induction assumption to estimate $\|\div^{\vp}u\|_{L_t^2\cH^{j+1,l-1}}$, one finds:
\beqs
\ep^{-\f{1}{2}}\|\nabla^{\vp}\sigma\|_{L_t^2\cH^{j,l}}\lesssim \text{R.H.S of } \eqref{induction1}.
\eeqs
\end{proof}
\subsection{ Energy estimates: Incompressible part}
In this subsection, we focus on the analysis of the incompressible part of the velocity $v=\bbp u$ whose estimates can be obtained from direct energy estimates.
By \eqref{eqofv1}-\eqref{v-bdry-up}, $v$ solves the following system:
\beq\label{eqofv2}
\left\{
\begin{array}{l}
  \bar{\rho}  \pt^{\vp}v-\mu\Delta^{\vp} v+\nabla^{\vp}\pi=-(f+\nabla^{\vp}q+\bar{\rho}[\bbp,\pt^{\vp}]u), \\[5pt]
(2\mu S^{\vp}v-\pi Id)\bN|_{z=0}=2\mu(\div^{\vp}u \text{Id}-(\nabla^{\vp})^2\Psi)\bN|_{z=0},\\[5pt]
v_3|_{z=-1}=0, \quad \mu\p_z^{\vp}v_{j}|_{z=-1}=a u_{j}|_{z=-1}, \quad j=1,2.\\
\end{array}
\right.
\eeq
where 
 \beq \label{defpi}
 \nabla^{\vp}\pi=\bbp\nabla^{\vp}(\sigma/{\ep}-2(\mu+\lambda)\div^{\vp}u)=\colon\bbp\nabla^{\vp}\theta,
 \eeq
 \beq\label{def-f-free}
 f=\f{g_2-\bar{\rho}}{\ep}(\ep\pt^{\vp}u+\ep u\cdot\nabla^{\vp}u)+\bar{\rho}u\cdot\nabla^{\vp}u,\qquad \nabla^{\vp}q=-\bbq (f-\mu\Delta^{\vp} v).
\eeq
Before stating the main result for $v,$ it is useful to establish some auxiliary estimates for $\nabla^{\vp}\pi,f,\nabla^{\vp}q.$
\begin{prop}
Under the assumption \eqref{preassumption},
the following $L_t^2L^2(\mS)$ type estimates hold: for any $m\geq 7,$
\begin{align}
\|f\|_{\hco^{m-1}}+\|\div^{\vp}f\|_{\hco^{m-2}}+\ep^{\f{1}{2}}\|\pt f\|_{\hco^{m-2}}+\ep^{\f{1}{2}}\|f\|_{L_t^{\infty}H_{co}^{m-2}}\lesssim \lae \label{es-f},  
\end{align}
\beq\label{q}
\begin{aligned}
\|\nabla q\|_{\hco^{m-1}}+\ep^{\f{1}{2}} \|\nabla^{\vp}q\|_{L_t^{\infty}H_{co}^{m-2}}+\ep^{\f{1}{2}}\|\pt\nabla^{\vp}q\|_{\hco^{m-2}}\lesssim
\Lambda\big(\f{1}{c_0},\cN_{m,T}\big),
\end{aligned}
\eeq
\beq\label{pi1}
\begin{aligned}
\il\nabla \pi\il_{1,\infty,t}\lesssim 
\Lambda\big(\f{1}{c_0},|h|_{5,\infty,t}
\big)\cE_{m,T},
\end{aligned}
\eeq
\beq\label{pi}
\begin{aligned}
\|\nabla \pi\|_{\hco^{m-2}}\lesssim 
\Lambda\big(\f{1}{c_0},|h|_{m-2,\infty,t}\big) \|\nabla u\|_{\hco^{m-1}}+
T^{\f{1}{2}}
\lae,
\end{aligned}
\eeq
%\beq\label{pi3}\ep^{\f{3}{4}}\|\nabla^{\vp}\pi\|_{L_t^{\infty}H_{co}^{m-2}}\lesssim \Lambda\big(\f{1}{c_0},|h|_{m-2,\infty,t}\big)\| \ep^{\f{3}{4}}\nabla u\|_{L_t^{\infty}H_{co}^{m-1}}+\ep^{\f{3}{4}}\lae,\eeq
\beq\label{pi4}
\ep^{\f{1}{2}}\|\nabla^{\vp}\pi\|_{L_t^{\infty}H_{co}^{m-2}}\lesssim \Lambda\big(\f{1}{c_0},|h|_{m-2,\infty,t}\big)\| \ep^{\f{1}{2}}\nabla u\|_{L_t^{\infty}H_{co}^{m-1}}+\ep^{\f{1}{2}}\lae,
\eeq
\beq\label{pi2}
\ep^{\f{1}{2}}\|\pt\nabla\pi\|_{L_t^2H_{co}^{m-3}}\lesssim \Lambda\big(\f{1}{c_0},|h|_{m-2,\infty,t}
\big)\|\ep^{\f{1}{2}}\pt(u,\nabla u)\|_{L_t^2H_{co}^{m-2}}+
 (T+\ep)^{\f{1}{2}}\lae,
\eeq
\beq\label{comtime}
    \|[\bbp, \p_t^{\vp}]u\|_{L_t^2H_{co}^{m-1}}+\|[\bbp, \p_t^{\vp}]u\|_{L_t^{\infty}H_{co}^{m-2}}
    + \|\ep^{\f{1}{2}}\pt [\bbp, \p_t^{\vp}]u\|_{L_t^2H_{co}^{m-2}}
    \lesssim \lae.
\eeq
%\beq%\label{comm-bbqpt}
%\|[\bbq,\pt^{\vp}]f\|_{L_t^{\infty}H_{co}^k}\lesssim \Lambda\big(\f{1}{c_0},|h|_{L_t^{\infty}\tilde{H}^{k+\f{3}{2}}}+|\pt h|_{L_t^{\infty}\tilde{H}^{k+\f{1}{2}}}+\|\div^{\vp}f\|_{L_t^{\infty}H_{co}^{k}}\big).\eeq
\end{prop}

\begin{proof}
%The first inequality are easy consequences of the product estimate \eqref{product}, we omit the proof.
\underline{Proof of \eqref{es-f}}.
In view of definition of $f$ in \eqref{def-f-free}, we give details for the estimate of $u\cdot\nabla^{\vp}u$ and $\div^{\vp}(u\cdot\nabla^{\vp}u),$ the other terms can be controlled in a similar way. 
First, for the $L_t^{\infty}H_{co}^{m-2}$ norm, 
we have thanks to the product estimate  \eqref{crudepro} that:
\beqs
\begin{aligned}
\ep^{\f{1}{2}}\|u\cdot\nabla^{\vp}u\|_{L_t^{\infty}H_{co}^{m-2}}&\lesssim \Lambda\big(\f{1}{c_0},\il u\il_{[\f{m}{2}],\infty,t}+\ep^{\f{1}{2}}\il\nabla u\il_{[\f{m}{2}]-1,\infty,t}\big)\|(u,\ep^{\f{1}{2}}\nabla^{\vp}u\|_{L_t^{\infty}H_{co}^{m-2}}\\
&\lesssim \Lambda\big(\f{1}{c_0},\cA_{m,T}\big)\cE_{m,T}.
\end{aligned}
\eeqs
For the first three norms in the left-hand side of \eqref{es-f}, we first have
by the product estimate \eqref{product},
\beqs
\begin{aligned}
&\|u\cdot\nabla^{\vp}u\|_{L_t^2\cH^{0,m-1}}+\|\div^{\vp}(u\cdot\nabla u)\|_{L_t^2\cH^{0,m-2}}\\
&\lesssim \Lambda\big(\f{1}{c_0}, \il (u,\nabla^{\vp} u)\il_{0,\infty,t}+\il\nabla\div^{\vp} u
\il_{1,\infty,t}\big)(\|(u,\nabla^{\vp}u)\|_{L_t^2\cH^{0,m-1}}+\|\nabla^{\vp}\div^{\vp} u\|_{L_t^2\cH^{0,m-2}}).
\end{aligned}
\eeqs
It remains to control  $\|\ep\pt\div^{\vp}(u\cdot\nabla^{\vp}u)\|_{\hco^{m-3}}$ and  $\ep^{\f{1}{2}}\|\pt(u\cdot\nabla^{\vp}u)\|_{\hco^{m-2}}.$
We can estimate them in a rather rough way:
\beqs
\begin{aligned}
&\|\ep\pt\div^{\vp}(u\cdot\nabla^{\vp}u)\|_{\hco^{m-3}}\lesssim 
\|(\ep\pt\nabla^{\vp} u\cdot\nabla^{\vp}u, \ep\pt (u\cdot\nabla^{\vp}\div^{\vp}u))\|_{\hco^{m-3}}\\
&\lesssim \il\nabla^{\vp}u\il_{0,\infty,t}\|\nabla^{\vp}u\|_{\hco^{m-2}}+\il \ep\pt\nabla^{\vp} u\il_{0,\infty,t}\|\nabla^{\vp} u\|_{\hco^{m-3}}\\
&\qquad+\ep^{\f{1}{2}}\|\pt\nabla^{\vp}u\|_{\hco^{m-4}}\il\ep^{\f{1}{2}}\nabla^{\vp}u\il_{m-4,\infty,t}\\
&\quad +\il\nabla^{\vp}\div^{\vp}u\il_{[\f{m}{2}]-2,\infty,t}\|u\|_{\hco^{m-2}}+\il u\il_{[\f{m-1}{2}],\infty,t}\|\nabla^{\vp}\div^{\vp}u\|_{\hco^{m-2}}\\
&\lesssim \lca \cE_{m,T},
\end{aligned}
\eeqs
\beqs
\begin{aligned}
&\ep^{\f{1}{2}}\|\pt(u\cdot\nabla^{\vp}u)\|_{\hco^{m-2}}\lesssim \|( u\cdot\ep^{\f{1}{2}}\pt\nabla^{\vp}u, \ep^{\f{1}{2}}\pt u\cdot \nabla^{\vp}u)\|_{\hco^{m-2}}\\
&\lesssim \il u\il_{1,\infty,t}\|\ep^{\f{1}{2}}\pt\nabla^{\vp}u\|_{\hco^{m-2}}+\|\ep^{\f{1}{2}}\pt  \nabla^{\vp}u\|_{L_t^{\infty}H_{co}^{m-4}}\big(\int_0^t\|u(s)\|_{m-2,\infty}^2\d s\big)^{\f{1}{2}}\\
&\quad+\|\ep^{\f{1}{2}}\pt u\|_{\hco^{m-2}}\il\nabla^{\vp}u\il_{0,\infty,t}+\|\nabla^{\vp}u\|_{\hco^{m-2}}\big(\int_0^t\| \ep^{\f{1}{2}}\pt u(s)\|_{m-3,\infty}\d s\big)^{\f{1}{2}}\\
&\lesssim \lae.
\end{aligned}
\eeqs
%\beqs\begin{aligned}&\|\ep\pt(u\cdot\nabla^{\vp}u)\|_{L_t^2H_{co}^{m-2}}\lesssim \|(\ep\pt u\cdot\nabla^{\vp}u, u\cdot \ep\pt\nabla^{\vp}u)\|_{\hco^{m-2}}\\&\lesssim \ep^{\f{1}{2}}\il\nabla^{\vp}u\il_{[\f{m}{2}]-1,\infty,t}\|\ep^{\f{1}{2}}\pt u\|_{\hco^{m-2}}+\il u\il_{[\f{m}{2}],\infty,t}\|\nabla^{\vp}u\|_{\hco^{m-2}}\\&\quad +\il\ep\pt\nabla^{\vp}u\il_{0,\infty,t}\|u\|_{\hco^{m-2}}+\il \ep^{\f{1}{2}} u\il_{m-3,\infty,t}\|\ep^{\f{1}{2}}\pt\nabla u\|_{\hco^{m-3}}\\&\lesssim \lca \cE_{m,T},\end{aligned}\eeqs
\underline{Proof of \eqref{q}}
Let us now show the estimate  \eqref{q} for $q.$  
By the definition of $\bbq$ in \eqref{defofQ} and the fact that $\div^{\vp}\Delta^{\vp}v=0,$ $q$  solves the elliptic problem:
\beqs
\left\{
\begin{array}{l}
     \div(E\nabla q)=-\div (Pf),  \\
      q|_{z=0}=0,\\[5pt]
      \p_z^{\vp}q|_{z=-1}=-f\cdot e_3|_{z=-1}+g
\end{array}
\right.
\eeqs
where $P$ and $E$ are defined in \eqref{defPE} and $g=(\Delta^{\vp}v_3)^{b,2}=\Delta^{\vp}v_3|_{z=-1}.$ Applying the
elliptic estimate \eqref{ellipticuseful-1}, \eqref{elliptic1.5} for $F=f,$ %g=(\Delta^{\vp}v_3)^{b,2},$
we find:
\beq\label{q-1}
\begin{aligned}
\|\nabla q\|_{\hco^{m-1}}&\lesssim \Lambda\big(\f{1}{c_0}, |h|_{m-2,\infty,t}+ \|\div^{\vp}f\|_{L_t^{\infty}H_{tan}^2}+
\big|(\Delta^{\vp}v_3)^{b,2}\big|_{L_t^{\infty}H_{tan}^{\f{5}{2}}}
\big)\\
&\qquad(\|f\|_{\hco^{m-1}}+|h|_{\htlde^{m-\f{1}{2}}}+ \big|(\Delta^{\vp}v_3)^{b,2}\big|_{\htlde^{m-\f{3}{2}}}),
\end{aligned}
\eeq
\beq\label{q2}
\begin{aligned}
\ep^{\f{1}{2}}\|\nabla q\|_{L_t^{\infty}H_{co}^{m-2}}&\lesssim \Lambda\big(\f{1}{c_0},|h|_{m-2,\infty,t}+\|\div^{\vp}f\|_{L_t^{\infty}H_{tan}^1}+
\big|(\Delta^{\vp}v_3)^{b,2}\big|_{L_t^{\infty}H_{tan}^{\f{3}{2}}}\big)\\
&\qquad (\ep^{\f{1}{2}}\|f\|_{L_t^{\infty}H_{co}^{m-2}}+\ep^{\f{1}{2}}|
(\Delta^{\vp}v)^{b,2}|_{L_t^{\infty}\tilde{H}^{m-\f{5}{2}}}+\ep^{\f{1}{2}}|h|_{L_t^{\infty}\tilde{H}^{m-\f{3}{2}}}),
\end{aligned}
\eeq
\beq\label{q3}
\begin{aligned}
&\ep^{\f{1}{2}}\|\pt\nabla q\|_{L_t^{2}H_{co}^{m-2}}\\
&\lesssim \Lambda\big(\f{1}{c_0},|(h,\ep^{\f{1}{2}}\pt h)|_{m-2,\infty,t}+\|\ep^{-\f{1}{2}}\div^{\vp}f\|_{L_t^{\infty}H_{co}^2}+
\big|(\text{Id},\ep^{\f{1}{2}}\pt)(\Delta^{\vp}v_3)^{b,2}\big|_{L_t^{\infty}H_{tan}^{\f{5}{2}}}\big)\cdot\\
& \quad(\ep^{\f{1}{2}}\|\pt f\|_{L_t^{2}H_{co}^{m-2}}+\ep^{\f{1}{2}}|\pt(\Delta^{\vp}v)^{b,2}|_{L_t^{\infty}\tilde{H}^{m-\f{3}{2}}}+|(h,\ep^{\f{1}{2}}\pt h)|_{L_t^{2}\tilde{H}^{m-\f{3}{2}}}+\|\nabla q\|_{\hco^{m-2}}).
\end{aligned}
\eeq
It follows from direct computations that:
\beqs
\Delta^{\vp}v_3=\Delta^{\vp}u_3-\p_z^{\vp}\div^{\vp} u=(\p_1^{\vp})^2 u_3+(\p_2^{\vp})^2u_3-(\p_1^{\vp}\p_z^{\vp}u_1+\p_2^{\vp}\p_z^{\vp}u_2).
\eeqs
This, combined with the identities $$\p_1^{\vp}|_{z=-1}=\p_1,\quad \p_2^{\vp}|_{z=-1}=\p_2$$
as well as the boundary condition \eqref{bebdry2}, yields:
\beq\label{norofLap}
(\Delta^{\vp}v_3)^{b,2}=-\f{a}{\mu}
(\p_1 u_1+\p_2 u_2)^{b,2}.
\eeq
In light of  \eqref{es-f}, \eqref{q-1}-\eqref{q3}, \eqref{norofLap}, we find 
\eqref{q} by the trace inequality \eqref{trace}. \\[5pt]
\underline{Proof of \eqref{pi}-\eqref{pi2}}.
Let us switch to the estimate of $\pi.$ 
By definition, $\pi$ satisfies the following elliptic problem:
\beqs
\left\{
\begin{array}{l}
     \div(E\nabla \pi)=0,  \\
      \pi|_{z=0}=\theta^{b,1},\\
      \p_z^{\vp}\pi|_{z=-1}=0.
\end{array}
\right.
\eeqs
%\beq\label{thetabdry}\theta=-2\mu(\p_1u_1+\p_2u_2)-2\mu(\Pi(\p_1u\cdot\bN,\p_2u\cdot\bN,0)^{t})_3,\eeq
where $\theta^{b,1}=\theta|_{z=0}.$
Therefore, to prove \eqref{pi1}, we apply 
\eqref{elliptic4.9} to get that:
\beqs
\begin{aligned}
\il\nabla\pi\il_{1,\infty,t}&\lesssim \|\nabla^2\pi\|_{L_t^{\infty}H_{tan}^2}+\|\nabla\pi\|_{L_t^{\infty}H_{tan}^3}\\
&\lesssim \Lambda(\f{1}{c_0},|h|_{4,\infty,t})|\theta^{b,1}|_{L_t^{\infty}H^{\f{7}{2}}}.
\end{aligned}
\eeqs
By using the boundary conditions
\eqref{sigmabdry} \eqref{omegatimesn}, we have that on the upper boundary,
\beq\label{thetabdry}
\theta=-2\mu(\p_1u_1+\p_2u_2)-2\mu(\Pi(\p_1u\cdot\bN,\p_2u\cdot\bN,0)^{t})_3,
\eeq
hence, by the product estimate \eqref{product-R2} and the trace inequality \eqref{trace}, we get:
\beqs
|\theta^{b,1}|_{L_t^{\infty}H^{\f{7}{2}}}
\lesssim (\|\nabla u\|_{L_t^{\infty}H_{co}^4}+\|u\|_{L_t^{\infty}H_{co}^5})\Lambda(\f{1}{c_0},|h|_{5,\infty,t}).
\eeqs
This ends the proof of \eqref{pi1}.

Now, we can apply \eqref{ellipticuseful-1} and
\eqref{pi1} to get that for $p=2,+\infty,$
\beq\label{pi-3}
\begin{aligned}
\|\nabla^{\vp} \pi\|_{L_t^pH_{co}^{m-2}}\lesssim
\Lambda\big(\f{1}{c_0},|h|_{m-2,\infty,t}\big) |\theta^{b,1}|_{L_t^p\tilde{H}^{m-\f{3}{2}}}+
|h|_{L_t^{p}\tilde{H}^{m-\f{3}{2}}}\lae, \,
\end{aligned}
\eeq
In view of \eqref{thetabdry}, one has by the product estimate \eqref{product-R2} and the trace inequality \eqref{trace} that
\beq
|\theta^{b,1}|_{L_t^{p}\tilde{H}^{m-k+\f{1}{2}}}\lesssim \Lambda\big(\f{1}{c_0},|h|_{[\f{m}{2}]+1,\infty,t}\big)\|\nabla u\|_{L_t^{p}H_{co}^{m-1}}+\Lambda\big(\f{1}{c_0},\cA_{m,T}\big)|h|_{L_t^{p}\tilde{H}^{m-\f{1}{2}}},
\eeq
which, combined with \eqref{pi-3}, yields
\eqref{pi}-\eqref{pi4}.
Finally, for the estimate of \eqref{pi2}, we use the elliptic estimate \eqref{elliptic6} to obtain that:
\beqs
\ep^{\f{1}{2}}\|\pt \nabla^{\vp} \pi\|_{\hco^{m-3}}\lesssim \Lambda\big(\f{1}{c_0},|h|_{m-2,\infty,t}\big) (|\ep^{\f{1}{2}}\pt \theta^{b,1}|_{L_t^2\tilde{H}^{m-\f{5}{2}}}+|\ep^{\f{1}{2}}(\Delta^{\vp}v)^{b,2}|_{L_t^2\tilde{H}^{m-\f{7}{2}}})+\ep^{\f{1}{2}}\lae,
\eeqs
we thus obtain \eqref{pi2} by observing that:
\beqs
|\ep^{\f{1}{2}}\pt\theta^{b,1}|_{L_t^2\tilde{H}^{m-\f{7}{2}}}\lesssim\Lambda\big(\f{1}{c_0},|h|_{m-2,\infty,t}\big) \ep^{\f{1}{2}}\|\pt (u,\nabla^{\vp} u)\|_{\hco^{m-2}}+(T+\ep)^{\f{1}{2}}\lae.
\eeqs
\underline{Proof of \eqref{comtime}}. 
Finally, we estimate the commutator between 
the projection and the time derivative.
Set $\nabla^{\vp}\Psi_1=\bbq \pt^{\vp}u,$
then $$[\bbp,\pt^{\vp}]=-[\bbq,\pt^{\vp}]=\nabla^{\vp}(\Psi_1-\Psi).$$ By definition, $\Psi_1-\Psi$ solves the elliptic problem:
\beqs
\Delta^{\vp}(\Psi_1-\p_t^{\vp}\Psi)=0,\quad (\Psi_1-\p_t^{\vp}\Psi)|_{z=0}=\f{\p_t h}{\p_z\vp}\p_z \Psi,
\quad \p_z^{\vp}(\Psi_1-\p_t^{\vp}\Psi)|_{z=-1}=0.
\eeqs
It follows from %\eqref{elliptic4}
\eqref{elliptic-useful} and the product estimate \eqref{product}
that:
\beq
\begin{aligned}
&\|\nabla^{\vp}(\Psi_1-\pt^{\vp}\Psi)\|_{L_t^2H_{co}^{m-1}}\\
&\lesssim  \Lambda\bigg(\f{1}{c_0},|h|_{m-2,\infty,t}+\bigg|\f{\p_t h}{\p_z\vp}\p_z\pi_1\bigg|_{L_t^{\infty}\tilde{H}^{\f{5}{2}}}\bigg)\bigg|(h, \f{\pt h}{\p_z\vp}\p_z\Psi)\bigg|_{L_t^2\tilde{H}^{m-\f{1}{2}}}\\ 
&\lesssim \Lambda\big(\f{1}{c_0},|h|_{m-2,\infty,t}+|\pt h|_{3,\infty,t}+\|\div^{\vp}u\|_{L_t^{\infty}H_{co}^2}\big)\\
&\qquad
\big(|\pt h|_{L_t^2\tilde{H}^{m-\f{1}{2}}}\il\nabla\Psi\il_{3,\infty,t}+\|(\nabla\Psi,\nabla^2\Psi)\|_{L_t^2H_{co}^{m-1}}|\pt h|_{m-3,\infty,t}\big).
\end{aligned}
\eeq
Combined with \eqref{sec-normal-Psi}, \eqref{psiLinfty}, \eqref{Linftynablapsi}, this yields the control of the first quantity in \eqref{comtime}. The second quantity can be controlled in a similar way, we omit the proof.
%we get obtain 
%\Lambda\big(\f{1}{c_0},|h|_{k-1,\infty,t}+|\pt h|_{k-2,\infty,t}+\|\div^{\vp}f\|_{L_t^{\infty}H_{co}^2}\big)\\ \cdot\big ((|h|_{L_t^2\tilde{H}^{k+\f{3}{2}}}+|\pt h|_{L_t^2\tilde{H}^{k+\f{1}{2}}})\|\div^{\vp}f\|_{L_t^{\infty}H_{co}^4}+|h|_{L_t^2\tilde{H}^{k+\f{1}{2}}}+\|\div^{\vp}f\|_{L_t^2H_{co}^k}\big).
\end{proof}

\begin{lem}\label{lemhigh-v}
Suppose that $m\geq 7$ and \eqref{preassumption} holds, then we have the following high order energy estimate for $v$: for every $0<t\leq T,$
\beq\label{v-0}
\| v\|_{L_t^{\infty}H_{co}^{m-1}}^2+\|\nabla^{\vp}v\|_{\hco^{m-1}}^2
\leq \Lambda\big(\f{1}{c_0},|h|_{L_t^{\infty}\tilde{H}^{m-\f{1}{2}}}^2+Y_m^2(0)\big)Y^2_m(0)+(T+\ep)^{\f{1}{2}}\Lambda\big(\f{1}{c_0},\cN_{m,T}\big).
\eeq
%where $\cZ=Z_1,Z_2,Z_3$ denote the spatial conormal derivatives. 
\end{lem}

\begin{rmk}\label{rmkinitial}
By using the elliptic estimates \eqref{elliptic2} and \eqref{Linftynablapsi}, we have:
\beqs
\|\nabla^{\vp}\Psi(0)\|_{H_{co}^{m-1}}\lesssim \Lambda\big(\f{1}{c_0}, \tilde{Y}_{[\f{m}{2}]}(0)\big)(\|u(0)\|_{H_{co}^{m-1}}+|h(0)|_{\tilde{H}^{m-\f{1}{2}}})
\eeqs
where $\tilde{Y}_{[\f{m}{2}]}(0)=%\sum_{|\alpha|\leq [\f{m}{2}]}
\|(\div^{\vp}u)(0)\|_{H_{co}^{[\f{m}{2}]}(\mS)}+\sum_{|\alpha|\leq  [\f{m}{2}]+1}|(Z^{\alpha}h)(0)|_{L^{\infty}(\mathbb{R}^2)}\lesssim Y_m(0).$\\[5pt]
Since $v=u-\nabla^{\vp}\Psi,$ we thus get:
\beqs
\|(v,\nabla^{\vp}\Psi)(0)\|_{H_{co}^{m-1}}\lesssim \Lambda\big(\f{1}{c_0}, Y_{m}(0)\big)Y_{m}(0).
\eeqs
%that is the reason why $\|v(0)\|_{H_{co}^{m-1}}$ does not appear in the right hand side of \eqref{v-0}.
\end{rmk}
\begin{rmk}
By the control of normal derivative of the compressible part
\eqref{sec-normal-Psi}, \eqref{sigmainduction} and of the incompressible part \eqref{v-0},  one deduces that:
\beq\label{nor-L2}
\|\nabla^{\vp}u\|_{\hco^{m-1}}^2\lesssim \Lambda\big(\f{1}{c_0},|h|_{L_t^{\infty}\tilde{H}^{m-\f{1}{2}}}^2\big)Y_m^2(0)+(T+\ep)^{\f{1}{2}}\Lambda\big(\f{1}{c_0},\cN_{m,T}\big).
\eeq
\end{rmk}
\begin{proof}
Let $\alpha=(\alpha_0,\alpha'),|\alpha|=k\leq m-1.$ We can assume that
$Z^{\alpha}$ contains at least one spatial vector field (ie. $|\alpha'|\neq 0$), since $\|v\|_{L_t^{\infty}\cH^{m-1}}$ and $\|\nabla^{\vp}v\|_{L_t^2\cH^{m-1}}$ can be derived directly from the norms that have been bounded. Indeed, one has by elliptic estimates \eqref{elliptic-useful} and \eqref{elliptic3} that 
\beqs
\|v\|_{L_t^{\infty}\cH^{m-1}}
\lesssim \|(u,\nabla^{\vp}\Psi)\|_{L_t^{\infty}\cH^{m-1}}
\lesssim \|u\|_{L_t^{\infty}\cH^{m-1}}\Lambda\big(\f{1}{c_0},|h|_{m-2,\infty,t}\big)+(T+\ep)^{\f{1}{2}}\lae.
\eeqs
\beqs\begin{aligned}
\|\nabla v\|_{L_t^2\cH^{m-1}}&\lesssim \|(u,\nabla^{\vp}\Psi)\|_{L_t^{2}\cH^{m-1}}\lesim \Lambda(\f{1}{c_0},|h|_{m-2,\infty,t})\|\nabla u\|_{L_t^2\cH^{m-1}}+(T+\ep)^{\f{1}{2}}\lae.
\end{aligned}
\eeqs
Applying $Z^{\alpha}$ to $\eqref{eqofv2}_1,$ we obtain:
\begin{align*}
&\quad\bar{\rho}\pt^{\vp}Z^{\alpha}v-2\mu\div^{\vp}Z^{\alpha}S^{\vp}v+\nabla^{\vp}Z^{\alpha}\pi\\
&=-Z^{\alpha}(f+\nabla^{\vp}q+\bar{\rho}[\bbp,\pt^{\vp}]u)-[Z^{\alpha},\nabla^{\vp}]\pi+2\mu[Z^{\alpha},\div^{\vp}]S^{\vp}u-\bar{\rho}[Z^{\alpha},\pt^{\vp}]v.
\end{align*}
Performing standard energy estimates,
we obtain the energy identity:
\beq
\begin{aligned}
&\f{1}{2}\bar{\rho}\int_{\mS}|Z^{\alpha}v|^2(t)\d\cV_t+2\mu \int_{0}^t\int_{\mS}|Z^{\alpha}S^{\vp}v|^2\d\cV_s\d s+ a \int_0^t\int_{z=-1}|Z^{\alpha}v_{\tau}|^2\d y\d s\label{defck}\\
&=\colon\cK_0+\cK_1+\cdots \cK_8,
\end{aligned}
\eeq
where
%\begin{alignat*}{2}
\begin{align*}
    &\cK_0=\f{1}{2}\bar{\rho}\int_{\mS}|Z^{\alpha}v|^2(0)\,\d\cV_0,  \qquad\qquad\quad\qquad \cK_1=\f{1}{2}\bar{\rho}\int_0^t\int_{z=0}\pt h|Z^{\alpha}v|^2\,\d y\d s,\\
    & \cK_2=2\mu\int_{0}^t\int_{\mS}
    Z^{\alpha}S^{\vp}v\cdot[Z^{\alpha},\nabla^{\vp}]v\,\d\cV_s\d s, \quad \cK_3=\int_0^t\int_{z=0}Z^{\alpha}(2\mu S^{\vp}v-\pi\text{Id})\bN\cdot Z^{\alpha}v\,\d y\d s,\\
  % \end{align*}
%\begin{align*}
 &\cK_4=\int_{0}^t\int_{\mS} Z^{\alpha}\pi [\div^{\vp},Z^{\alpha}]v \,\d\cV_s\d s, \qquad\qquad
    \cK_5=-\int_{0}^t\int_{\mS}Z^{\alpha}v \cdot [Z^{\alpha},\nabla^{\vp}]\pi
   \,\d\cV_s\d s ,\\
   & \cK_6=-\bar{\rho}\int_{0}^t\int_{\mS} Z^{\alpha}v\cdot[Z^{\alpha},\pt^{\vp}]v \,\d\cV_s\d s, \qquad \quad \cK_7=2\mu\int_{0}^t\int_{\mS} [Z^{\alpha},\div^{\vp}]S^{\vp}v\cdot Z^{\alpha}v\,\d\cV_s\d s,\\
   &\cK_8=-\int_{0}^t\int_{\mS}Z^{\alpha}v \cdot\big(Z^{\alpha}(f+\nabla^{\vp}q+\bar{\rho}[\bbp,\pt^{\vp}]u)\big)\,\d\cV_s\d s. 
\end{align*}
By the trace inequality, 
\beq\label{lowbdryterm}
a \int_0^t\int_{z=-1}|Z^{\alpha}v_{\tau}|^2\d y\d s\geq-\delta \|\nabla v\|_{\hco^{k}}^2-C_{\delta}(\|\nabla v\|_{\hco^{k-1}}^2+\| v\|_{\hco^{k}}^2),
\eeq
we will choose $\delta$ sufficiently small in the end.
Our following task is to estimate $\cK_0-\cK_8$ one by one. By Remark \ref{rmkinitial}, we get that:
\beq\label{k0}
\cK_0\lesssim \Lambda\big(\f{1}{c_0}, Y_{m}^2(0)\big)Y_{m}^2(0). %\|u(0)\|_{H_{co}^{m-1}}^2+|h(0)|_{\tilde{H}^{m-{1}/{2}}}^2.
\eeq
Thanks to the trace inequality and Young's inequality, $\cK_1$ can be treated as:
\beq
\begin{aligned}
\cK_1 &\lesssim |\pt h|_{0,\infty,t}(\|\nabla Z^{\alpha}v\|_{\ltx}\|Z^{\alpha}v\|_{\ltx}+\|Z^{\alpha}v\|_{\ltx}^2)\\
&\leq  \delta \|\nabla v\|_{L_t^2H_{co}^k}^2+C_{\delta}\|\nabla v\|_{\hco^{k-1}}^2+ 
\Lambda\big(\f{1}{c_0},|\pt h|_{0,\infty,t}\big)\|v\|_{\hco^k}^2.
\end{aligned}
\eeq
%Before going further, we introduce the new quantity:
%\beq\label{tl2infty}\tilde{\Lambda}_{2,\infty,t}=\Lambda(\f{1}{c_0},|h|_{2,\infty,t}+\il\nabla v\il_{0,\infty,t}+\il (v,\nabla^{\vp}\Psi)\il_{2,\infty,t}+\il\nabla\pi\il_{0,\infty,t}),\eeq where $\tilde{\Lambda}$ is a polynomial that will change from line to line.
%We continue to estimate $\cK_2-\cK_8$ defined in \eqref{defck}. 
For the term $\cK_2,$
to deal with the commutator term 
$[Z^{\alpha},\nabla^{\vp}]v,$ we apply \eqref{comgrad1} if 
$\alpha_0=0$ and \eqref{comgrad} if $\alpha_0\geq1$ and find that:
\beq\label{com-realuse}
\begin{aligned}
\|[Z^{\alpha},\nabla^{\vp}]v\|_{L^2}&\lesssim \Lambda\big(\f{1}{c_0},\il\nabla v\il_{1,\infty,t}+|(h,\ep\pt h)|_{m-2,\infty,t}\big)(|h|_{\htlde^{m-\f{1}{2}}}+|\ep h|_{\htlde^{m-\f{3}{2}}})\\
&+\Lambda\big(\f{1}{c_0},|(h,\ep\pt h)|_{m-2,\infty,t}\big)\|\nabla v\|_{\hco^{m-2}}\\
&\lesssim T^{\f{1}{2}}\lae+\Lambda\big(\f{1}{c_0},|(h,\ep\pt h)|_{m-2,\infty,t}\big)\|\nabla v\|_{\hco^{m-2}}.
\end{aligned}
\eeq
Note that by the estimate \eqref{Linftynablapsi}, we have:
%\beq\label{control-tl2}\tilde{\Lambda}_{2,\infty,t}\eeq
\beqs
\begin{aligned}
\il\nabla v\il_{1,\infty,t}&\lesssim \il\nabla (u, \nabla^{\vp}\Psi)\il_{1\infty,t}\\
&\lesssim \Lambda\big(\f{1}{c_0}, \il\nabla u\il_{1\infty,t}+|h|_{4,\infty,t}+
\|\div^{\vp}u\|_{L_t^{\infty}H_{co}^2}\big)\lesssim \lae.
\end{aligned}
\eeqs
Therefore, by Young's inequality, one can control $\cK_2$ by:
\beq\label{k2}
\cK_2\leq \delta \|\nabla v\|_{\hco^k}^2+
\Lambda\big(\f{1}{c_0},|(h,\ep\pt h)|_{m-2,\infty,t}\big)
\|\nabla v\|_{\hco^{k-1}}^2+
T^{\f{1}{2}}\lae.
\eeq
For the boundary term $\cK_3,$ we use the boundary condition $\eqref{eqofv2}_2$ to split it into two terms:
\beqs
\begin{aligned}
\cK_3&=\int_0^t\int_{z=0}Z^{\alpha}\big(2\mu (\div^{\vp}u\text{Id}-\nabla^{\vp}\nabla^{\vp}\Psi)\bN\big)\cdot Z^{\alpha}v-[Z^{\alpha},\bN](2\mu S^{\vp}v-\pi\text{Id})\cdot Z^{\alpha}v \,\d y\d s\\
&=:\cK_{31}+\cK_{32}.
\end{aligned}
\eeqs
Since %$Z^{\alpha}$
$\kappa_3$ vanishes if $\alpha_3\neq 0,$ we may assume that $Z^{\alpha}=\p_y Z^{\tilde{\alpha}}.$ It then follows by
duality that:
\begin{align*}
 \cK_{31}&\lesssim  |Z^{\alpha}v|_{L_t^2H^{\f{1}{2}}}|Z^{\tilde{\alpha}}\big(2\mu (\div^{\vp}u\text{Id}-\nabla^{\vp}\nabla^{\vp}\Psi)\bN\big)|_{L_t^2H^{\f{1}{2}}}
\end{align*}
 Thanks to 
product estimate \eqref{rough-product-bdry}, we obtain for $k\leq m-1,$
\begin{align*}
|Z^{\tilde{\alpha}}\big(2\mu (\div^{\vp}u\text{Id}-(\nabla^{\vp})^2)\bN\big)|_{L_t^2H^{\f{1}{2}}}&\lesssim |(\div^{\vp}u,(\nabla^{\vp})^2\Psi)|_{\htlde^{k-\f{1}{2}}}|h|_{L_t^{\infty}\tilde{H}^{[\f{k-1}{2}]+2^{+}}}\\
&\quad+
|h|_{\htlde^{k+\f{1}{2}}}
|(\div^{\vp}u,(\nabla^{\vp})^2\Psi)|_{L_t^{\infty}\tilde{H}^{[\f{k-1}{2}]+1^{+}}}\\
&\lesssim 
(\|\nabla\div^{\vp}u\|_{\hco^{m-2}}+\|\div^{\vp}u\|_{\hco^{m-1}})\Lambda\big(\f{1}{c_0},|h|_{L_t^{\infty}\tilde{H}^{m-\f{1}{2}}}+\|\nabla\Psi\|_{2,\infty,t}\big)\\
&\quad+T^{\f{1}{2}}\lae \big(|h|_{L_t^{\infty}\tilde{H}^{m-\f{1}{2}}}+\ep^{\f{1}{2}}|h|_{L_t^{\infty}\tilde{H}^{m+\f{1}{2}}}\big)\\
&\lesssim (T+\ep)^{\f{1}{2}}\lae.
\end{align*}
We remark that by the estimate \eqref{Linftynablapsi}, 
one has that
for $l\leq [\f{k-1}{2}]+1^{+}\leq [\f{m}{2}]^{+}\leq m-3$ (since $k\leq m-1, m\geq 7$), 
\beqs
\begin{aligned}
|(\nabla^{\vp})^2\Psi|_{L_t^{\infty}\tilde{H}^l}&\lesssim \|\nabla (\nabla^{\vp})^2\Psi\|_{L_t^{\infty}\tilde{H}^{l}}+\|(\nabla^{\vp})^2\Psi\|_{L_t^{\infty}\tilde{H}^{l}}\\
&\lesssim \big(
\|(\nabla\div^{\vp}u,\div^{\vp}u)\|_{L_t^{\infty}\tilde{H}^l}+|h|_{L_t^{\infty}\tilde{H}^{l+{5}/{2}}}\big)\Lambda\big(\f{1}{c_0},%\cN_{m,T}
\il\nabla^{\vp}\Psi\il_{[\f{m}{2}]-1,\infty,t}+|h|_{[\f{m}{2}]+2,\infty,t}\big)\\
&\lesssim \Lambda(\f{1}{c_0},\cN_{m,T}).
\end{aligned}
\eeqs
Therefore, by the trace inequality and Young's inequality, we get:
\beq\label{esofk31}
\cK_{31}\leq \delta \|\nabla v\|_{\hco^{k}}^2+ C_{\delta}\|\nabla v\|_{\hco^{k-1}}^2+%\cY_{m-1,T}^2\Lambda\big(\f{1}{c_0},|h|_{L_t^{\infty}\tilde{H}^{m-2}}^2\big)+
(T+\ep)^{\f{1}{2}}\lae. 
\eeq
For $\cK_{32},$ in order not to involve too many derivatives on the surface, 
we write it further as:
\beqs
\begin{aligned}
\cK_{32}&=-\int_0^t\int_{z=0}(2\mu S^{\vp}v-\pi\text{Id})Z^{\alpha}\bN \cdot Z^{\alpha}v +[Z^{\alpha},(S^{\vp}v-\pi\text{Id}),\bN] Z^{\alpha}v\,\d y\d s\\
&=:\cK_{321}+\cK_{322}.
\end{aligned}
\eeqs
By the definition \eqref{defpi} for 
$\pi$ we have that on the upper boundary,
\beq\label{pibdry}
\pi=\theta=-2\mu(\p_1u_1+\p_2u_2)-2\mu(\Pi(\p_1u\cdot\bN,\p_2u\cdot\bN,0)^{t})_3.
\eeq
 Moreover, %we it is helpful to look at 
 thanks to the boundary condition 
 \eqref{tanpz}, we can indeed express
 $\p_z^{\vp}v$ on the upper boundary.
On the one hand, we have the identity: %(which indeed holds in the whole domain)
\beq\label{nor-pzv}
\p_z^{\vp}v\cdot\bN=\div^{\vp}v-\p_1 v_1-\p_2 v_2=-(\p_1 v_1+\p_2 v_2).
\eeq
On the other hand, by the identity \eqref{tanpz}, one deduces:
\beq\label{tanpzv}
\begin{aligned}
|\bN|\Pi\p_z^{\vp} v&=|\bN|\Pi\p_z^{\vp} u-|\bN|\Pi\nabla^{\vp}\p_z^{\vp} \Psi\\
&=\Pi (\p_1u\cdot\bn,\p_2u\cdot\bn,0)^{t}-\Pi(\bn_1\p_1u+\bn_2\p_2 u)-|\bN|\Pi(\p_1,\p_2,0)^{t}\p_z^{\vp}\Psi,
\end{aligned}
\eeq
One thus has %by the definition of $\tilde{\Lambda}_{2,\infty,t}$ in \eqref{tl2infty} 
that:
\beqs
|(S^{\vp}v,\pi)^{b,1}|_{1,\infty,t}
\lesssim \Lambda\big(\f{1}{c_0},\il (v,\nabla^{\vp}\Psi )\il_{2,\infty,t}+|h|_{2,\infty,t}\big)
\lesssim \lae.
\eeqs
Therefore, by duality and the trace inequality \eqref{trace}, we obtain
\beq\label{ka321}
\begin{aligned}
\cK_{321}&\leq |2\mu S^{\vp}v-\pi\text{Id}|_{\infty,t}|Z^{\alpha}\bN|_{L_t^2H^{-\f{1}{2}}}|Z^{\alpha}v|_{L_t^2H^{\f{1}{2}}}\\
&\leq \delta\|\nabla v\|_{\hco^{k}}^2+ C_{\delta}\|\nabla v\|_{\hco^{k-1}}^2+(\|v\|_{\hco^{m-1}}^2+T|h|_{L_t^{\infty}\tilde{H}^{m-\f{1}{2}}}^2)\lae.
\end{aligned}
\eeq
Next, we can control $\cK_{322},$ in the following way:
%trace inequality  and Young's inequality
\beqs
\begin{aligned}
\cK_{322}&\lesssim |Z^{\alpha}v|_{L_t^2L_y^2}\big(|h|_{L_t^2\tilde{H}^{m-1}}|(S^{\vp}v,\pi)|_{1,\infty,t}+ |(S^{\vp}v,\pi)|_{\htlde^{k-1}}
|h|_{m-2,\infty,t}\big).
\end{aligned}
\eeqs
By virtue of the boundary conditions 
\eqref{pibdry}-\eqref{tanpzv}, we obtain that:
\beqs
|(S^{\vp}v,\pi)|_{\htlde^{k-1}}\lesssim 
\Lambda\big(|h|_{m-2,\infty,t}+\il(v,\nabla^{\vp}\Psi)\il_{2,\infty,t}\big)\big(|(v,\nabla^{\vp}\Psi)|_{\htlde^{k}}+|h|_{\htlde^{k}}\big).
\eeqs
Combined with the trace inequality \eqref{trace}, Young's inequality and the elliptic estimate \eqref{sec-normal-Psi}, we find:
\begin{align*}
\cK_{322} %&\lesssim   |Z^{\alpha}v|_{L_t^2L_y^2}\tilde{\Lambda}_{2,\infty,t}(|h|_{\htlde^{k}}+|(v,\nabla\Psi)|_{\htlde^{k}})\\
&\leq \delta \|\nabla v\|_{\hco^k}^2+ C_{\delta} \|\nabla v\|_{\hco^{k-1}}^2
+\big(\|v\|_{\hco^{m-1}}^2+(T+\ep)^{\f{1}{2}}\big)\lae.
\end{align*}
This estimate, together with \eqref{ka321}, \eqref{esofk31}, gives (with possibly another $C_{\delta}$)
\beq\label{k3}
\begin{aligned}
\cK_{3}&\leq 3\delta \|\nabla v\|_{\hco^k}^2
+C_{\delta} \|\nabla v\|_{\hco^{k-1}}^2+\big(\|v\|_{\hco^{k}}^2
+(T+\ep)^{\f{1}{2}}\big)\lae.
\end{aligned}
\eeq
For the term $\cK_4$, since $Z^{\alpha}$ contains at least one spatial derivative,
we can estimate it as:
\beqs
\cK_4\lesssim \|\nabla \pi\|_{\hco^{k-1}}\bigg(\|\nabla v\|_{\hco^{k-1}}\Lambda\big(\f{1}{c_0},|h|_{m-2,\infty,t}\big)+|h|_{\htlde^{k+\f{1}{2}}}\Lambda\big(\f{1}{c_0},\il\nabla v\il_{1,\infty,t}+|h|_{m-2,\infty,t}\big)\bigg).
\eeqs
We then apply \eqref{pi} and the elliptic estimate \eqref{sec-normal-Psi} to estimate $\nabla^{\vp}\pi$ as:
\beqs
\begin{aligned}
\|\nabla^{\vp}\pi\|_{\hco^{k-1}}&\lesssim \Lambda\big(\f{1}{c_0}, |h|_{m-2,\infty,t}\big)\|\nabla u\|_{\hco^k}
+T^{\f{1}{2}}\lae\\
&\lesssim \Lambda\big(\f{1}{c_0}, |h|_{m-2,\infty,t}\big)\|\nabla v\|_{\hco^k}+(T+\ep)^{\f{1}{2}}\lae.
\end{aligned}
\eeqs
Therefore, by Young's inequality, we get:
\beqs
\cK_4\leq \delta \|\nabla v\|_{\hco^k}^2
+\Lambda\big(\f{1}{c_0}, |h|_{m-2,\infty,t}^2\big)
\|\nabla v\|_{\hco^{k-1}}^2+(T+\ep)^{\f{1}{2}}\lae. 
\eeqs
Similarly, for $\cK_5,$ by applying \eqref{fpzphi1}, \eqref{pi}, \eqref{pi1}, we obtain:
\beqs
\begin{aligned}
\cK_5&\lesssim \|v\|_{\hco^{k}}\bigg(\Lambda\big(\f{1}{c_0},|h|_{m-2,\infty,t}\big)\|\nabla \pi\|_{\hco^{k}}+\Lambda\big(\f{1}{c_0},\il \nabla \pi\il_{1,\infty,t}+|h|_{m-2,\infty,t}\big)|h|_{\htlde^{m-\f{1}{2}}}\bigg)\\
&\lesssim \|v\|_{\hco^{k}}
\bigg(\Lambda\big(\f{1}{c_0},|h|_{m-2,\infty,t}\big)\|\nabla v\|_{\hco^{k}}+(T+\ep)^{\f{1}{2}}\lae\bigg).
\end{aligned}
\eeqs
Combined with the Young's inequality, this yields:
\beq
\cK_5\leq \delta \|\nabla v\|_{\hco^k}^2 +C_{\delta}
\Lambda\big(\f{1}{c_0},|h|_{m-2,\infty,t}^2\big)\|v\|_{\hco^k}^2+(T+\ep)\lae.
\eeq
For the term $\cK_6,$ we use similar arguments as in \eqref{com-realuse} to deal with the commutator term:
\beqs
\big\|\big[Z^{\alpha},\f{\pt \vp}{\p_z\vp}\p_z\big] v\big\|_{\ltx}\lesssim \big(\|\nabla v\|_{\hco^{k-1}}+ (T+\ep)^{\f{1}{2}}\big)\lae.
\eeqs
Therefore, we control $\cK_6$ by the Cauchy-Schwarz inequality to get:
\beq\label{k6}
\begin{aligned}
\cK_6&\leq \|Z^{\alpha}v\|_{\ltx}\big\|\big[Z^{\alpha},\f{\pt \vp}{\p_z\vp}\p_z \big] v\big\|_{\ltx}\\
&\lesssim \|\nabla v\|_{\hco^{k-1}}^2+\big(\| v\|_{\hco^{m-1}}^2+(T+\ep)^{\f{1}{2}}\big)\lae.
\end{aligned}
\eeq
We are now ready to estimate $\cK_7.$ In order not to lose normal derivative, we split it into three terms:
\beqs
\cK_7=\cK_{71}+\cK_{72}+\cK_{73}.
\eeqs
with 
\begin{align*}
  &\cK_{71}= 2\mu \int_0^t\int_{\mS}[Z^{\alpha},\p_z]\big(\f{1}{\p_z\vp}S^{\vp}v\bN\big)\cdot Z^{\alpha}v \,\d\cV_s\d s,\\
  &\cK_{72}= 2\mu \int_0^t\int_{\mS}
  \bigg(\p_z Z^{\alpha}\big(\f{1}{\p_z\vp}S^{\vp}v\bN\big)-\big(\p_z Z^{\alpha}S^{\vp}v\big)\f{\bN}{\p_z\vp}\bigg) \cdot Z^{\alpha}v\,
  \d\cV_s\d s,\\
 &\cK_{73}=-2\mu\int_0^t\int_{\mS} Z^{\alpha}\bigg (S^{\vp}v\p_z\big(\f{\bN}{\p_z\vp}\big)\bigg)\cdot Z^{\alpha}v \,\d\cV_s\d s.
\end{align*}
To deal with $\cK_{71}$, we can use the identity \eqref{identity-com-nor} to integrate by parts in space. By doing so, we are led to control the following type of terms (up to some smooth functions that depends only on $\phi$ and its derivatives)
\beqs
\int_0^t\int_{\mS}Z^{\gamma}\big(\f{1}{\p_z\vp}S^{\vp}v\bN\big)\p_z (Z^{\alpha}v \p_z\vp)\, \d x\d s,\quad \int_0^t\int_{\p\mS} Z^{\gamma}\big(\f{1}{\p_z\vp}S^{\vp}v\bN\big)Z^{\alpha}v \p_z\vp\, \d y\d s,\quad |\gamma|\leq k-1.
\eeqs
The first type of term can be controlled easily by: $$\delta \|\nabla v\|_{\hco^{k}}^2+C_{\delta}\Lambda\big(\f{1}{c_0}, |h|_{m-2,\infty,t}^2\big)\|\nabla v\|_{\hco^{k-1}}^2
+T\lae,$$
while the second type of terms can be bounded by: 
\beqs
\begin{aligned}
&|v|_{\htlde^k}\big(|S^{\vp}v|_{\htlde^{k-1}}+T^{\f{1}{2}}\big)\lae\\
&\lesssim
|v|_{\htlde^k}\big(|(v,\nabla^{\vp}\Psi)|_{\htlde^k}+
T^{\f{1}{2}}\big)\lae\\
&\leq \delta \|\nabla v\|_{\hco^k}^2+ 
\big(\|v\|_{\hco^k}^2+(T+\ep)^{\f{1}{2}}\big)\lae.
\end{aligned}
\eeqs
Hence, we get that:
\beq\label{k71}
\begin{aligned}
\cK_{71}&\leq 2\delta \|\nabla v\|_{\hco^k}^2+C_{\delta}\Lambda\big(\f{1}{c_0},|h|_{m-2,\infty,t}^2\big)\|\nabla v\|_{\hco^{k-1}}^2\\
&\qquad +\big(\|v\|_{\hco^k}^2+(T+\ep)^{\f{1}{2}}\big)\lae.
\end{aligned}
\eeq
For $\cK_{72},$ we use again integration by parts to split it into three terms:
$\cK_{72}=\cK_{721}+\cK_{722}+\cK_{723},$ with
\begin{align*}
 &\cK_{721}=-2\mu\int_0^t\int_{\mS}\big[Z^{\alpha},\f{\bN}{\p_z\vp}\big]S^{\vp}v\cdot \p_z(Z^{\alpha}v\p_z\vp)\,\d x\d s,\\
 &\cK_{722}=2\mu\int_0^t\int_{\mS} Z^{\alpha}S^{\vp}v\cdot \p_z\big(
 \f{\bN}{\p_z\vp}\big) Z^{\alpha}v \,\d x \d s,\\
 &\cK_{723}=2\mu \int_0^t\int_{\p \mS}\big[Z^{\alpha}, \f{\bN}{\p_z\vp}\big]S^{\vp}v\cdot Z^{\alpha}v\p_z\vp\, \d y\d s.
\end{align*}
In view of the expressions of these three terms, one can show by the commutator estimate \eqref{crudecom} that 
\beq\label{k72}
\begin{aligned}
\cK_{72}&\leq \delta \|\nabla v\|_{\hco^k}^2
+C_{\delta}\Lambda\big(\f{1}{c_0}, |h|_{m-2,\infty,t}^2\big) \|\nabla v\|_{\hco^{k-1}}^2\\
&\quad+\big(\|v\|_{\hco^{k}}^2+(T+\ep)^{\f{1}{2}}\big)\lae.
\end{aligned}
\eeq
Note that the boundary term $\cK_{723}$ can be controlled in 
a similar way as $\cK_{32}.$ We thus skip the details.

For $\cK_{73},$ to avoid losing regularity on the surface, we use the assumption that $|\alpha'|\geq 1$ to integrate by parts in space.
By doing so, we find that it can be bounded as:
\beq\label{k73}
\cK_{73}\leq  \delta \|\nabla v\|_{\hco^k}^2+C_{\delta}\Lambda\big(\f{1}{c_0},|h|_{m-2,\infty,t}^2\big)\|\nabla v\|_{\hco^{k-1}}^2 +(T+\ep)^{\f{1}{2}}\lae.
\eeq
We remark that there is no boundary contribution in the process of integration by parts since 
the spatial vector fields are tangent to the boundary.
Collecting \eqref{k71}-\eqref{k73}, we finally find that:
\beq\label{k7}
\begin{aligned}
\cK_{7}&\leq 4\delta\|\nabla v\|_{\hco^k}^2
+C_{\delta}\Lambda\big(\f{1}{c_0}, |h|_{m-2,\infty,t}^2\big) \|\nabla v\|_{\hco^{k-1}}^2\\
&\quad+\big( \|v\|_{\hco^{m-1}}^2+(T+\ep)^{\f{1}{2}}\big)\lae.
\end{aligned}
\eeq
It remains to treat the last term $\cK_{8}.$  By \eqref{es-f}
\eqref{q}, \eqref{comtime}, we have:
\beq\label{k8}
\begin{aligned}
\cK_8&\lesssim \|v\|_{\hco^k} (\|(f,\nabla^{\vp}q)\|_{\hco^k}+\|[\bbp,\p_t^{\vp}]u\|_{\hco^k})\\
&\lesssim \|v\|_{\hco^{m-1}} \Lambda\big(\f{1}{c_0},\cN_{m,T}\big).
\end{aligned}
\eeq
Gathering \eqref{lowbdryterm}-\eqref{k2}, \eqref{k3}-\eqref{k6}, \eqref{k7}, \eqref{k8}, we find by using Korn's inequality \eqref{korn} and by choosing $\delta$ small enough that for any $0\leq |\alpha|=k\leq m-1,$
\beqs
\begin{aligned}
\|v\|_{L_t^{\infty}H_{co}^k}^2+\|\nabla^{\vp}v\|_{L_t^2H_{co}^k}^2&\lesssim Y_{m}^2(0)+ \Lambda\big(\f{1}{c_0},|h|_{m-2,\infty,t}^2\big)\|\nabla^{\vp}v\|_{L_t^2H_{co}^{k-1}}^2\\
&+(\|v\|_{\hco^{m-1}}+(T+\ep)^{\f{1}{2}})\lae.
\end{aligned}
\eeqs
Therefore, by induction (on $k$), we get (up to changing possibly the polynomial)
\beq\label{sec10-42}
\begin{aligned}
\|v\|_{L_t^{\infty}H_{co}^{m-1}}^2+\|\nabla^{\vp}v\|_{L_t^2H_{co}^{m-1}}^2
&\lesssim (Y^2_m(0)
+\|\nabla v\|_{\ltx}^2)\Lambda\big(\f{1}{c_0},|h|_{L_t^{\infty}\tilde{H}^{m-\f{1}{2}}}^2\big)\\
&+(\|v\|_{\hco^{m-1}}+(T+\ep)^{\f{1}{2}})\lae.
\end{aligned}
\eeq
By \eqref{psiLinfty}, we can extract an extra $T^{\f{1}{2}}$ from $\|v\|_{\hco^{m-1}}.$ More precisely, we obtain: 
\beqs
\begin{aligned}
\|v\|_{\hco^{m-1}}\lesssim \|(u,\nabla^{\vp}\Psi)\|_{\hco^{m-1}}&\lesssim T^{\f{1}{2}}\|(u,\nabla^{\vp}\Psi)\|_{L_t^{\infty}H_{co}^{m-1}}\lesssim T^{\f{1}{2}}\lae.%(\|u\|_{L_t^{\infty}H_{co}^{m-1}}+|h|_{L_t^{\infty}\tilde{H}^{m-\f{1}{2}}})\lae.
\end{aligned}
\eeqs
Moreover, thanks to the elliptic estimate \eqref{elliptic1.5} and the definition $v=\bbp u=u-\nabla^{\vp}\Psi,$ we also have:
\beqs
\|\nabla v\|_{\ltx}\leq \|\nabla u\|_{\ltx}\Lambda\big(\f{1}{c_0},|h|_{3,\infty,t}\big).
\eeqs
Inserting the above two estimates and  \eqref{EI-T1} into \eqref{sec10-42}, we finally arrive at \eqref{v-0}. 
\end{proof}

In the following lemma, we prove some estimates for $\ep^{\f{1}{2}}\pt v,$ which is useful to the estimate for $\ep^{\f{1}{2}}\pt u$ later.
\begin{lem}\label{lemloworder-v}
Under the assumption \eqref{preassumption}, 
the following estimate for $v$ holds:
\beq\label{lower-v}
\begin{aligned}
&\|\ep^{\f{1}{2}} \pt v\|_{L_t^{\infty}H_{co}^{m-2}}^2+\|\ep^{\f{1}{2}} \pt \nabla v\|_{\hco^{m-2}}^2\\
&\lesssim %\Lambda\big(\f{1}{c_0},|h|_{m-2,\infty,t}^2+Y_m^2(0)\big)Y_m^2(0)
\Lambda\big(\f{1}{c_0},|h|_{L_t^{\infty}\tilde{H}^{m-\f{1}{2}}}^2+Y_m^2(0)\big)Y^2_m(0)+(T+\ep)^{\f{1}{2}}\Lambda\big(\f{1}{c_0},\cN_{m,T}\big).
\end{aligned}
\eeq

\end{lem}
\begin{proof}
The proof of this Lemma is very similar to the previous one, we thus only sketch its proof. We have by the elliptic estimate \eqref{elliptic3.3} that:
\beqs
\begin{aligned}
&\|\ep^{\f{1}{2}}\pt v%,\ep^{\f{3}{2}}\pt^2 v)
\|_{L_t^{\infty}\cH^{m-2}}+\|\ep^{\f{1}{2}}\pt\nabla v\|_{L_t^2\cH^{m-2}}\\
&\lesssim \|\ep^{\f{1}{2}}\pt (u,\nabla^{\vp}\Psi)\|_{L_t^{\infty}\cH^{m-2}}+\|\ep^{\f{1}{2}}\pt\nabla( u,\nabla^{\vp}\Psi)\|_{L_t^2\cH^{m-2}}\\
&\lesssim \Lambda\big(\f{1}{c_0},|h|_{m-2,\infty,t}\big) \big(\|\ep^{\f{1}{2}}\pt u\|_{L_t^{\infty}\cH^{m-2}}+\|\ep^{\f{1}{2}}\pt\nabla u\|_{L_t^2\cH^{m-2}}+\|\ep^{\f{1}{2}}\pt\div^{\vp}u\|_{L_t^{2}\cH^{m-2}}\big)\\
&\qquad +(T+\ep)^{\f{1}{2}} \Lambda\big(\f{1}{c_0},\cA_{m,t}\big)\big(|\pt h|_{L_t^{\infty}\tilde{H}^{m-\f{3}{2}}}+|(h,\ep^{\f{1}{2}}\pt h)|_{L_t^{\infty}\tilde{H}^{m-\f{1}{2}}}\big).
\end{aligned}
\eeqs
%We shall focus in the following on the estimate of  $\ep^{\f{1}{2}}\pt v,$ the other quantity can be treated in a similar fashion. 
For any multi-index $\beta$ with $|\beta|=k\leq m-2,$ direct energy estimates for $v$ yield:
\beq\label{EI-v-low}
\begin{aligned}
&\f{1}{2}\bar{\rho}\ep\int_{\mS}|Z^{\beta} \pt v|^2(t)\,\d\cV_t+2\mu\ep \int_{0}^t\int_{\mS}|Z^{\beta} \pt S^{\vp}v|^2\,\d\cV_s\d s+ a\ep \int_0^t\int_{z=-1}|Z^{\beta} \pt v_{\tau}|^2\,\d y\d s\\
&=\colon\tilde{\cK}_0+\tilde{\cK}_1+\cdots \tilde{\cK}_8,
\end{aligned}
\eeq
where $\tilde{\cK}_0-\tilde{\cK}_8$ are terms analogues to ${\cK}_0-{\cK}_8$
defined in \eqref{defck} 
in which $Z^{\alpha}$ is replaced by $\ep^{\f{1}{2}}Z^{\beta}\pt.$ 

At first, thanks to the trace inequality \eqref{trace}, Korn's inequality \eqref{korn} and Young's inequality, we have:
\beq\label{lowbdryterm1}
a\ep \int_0^t\int_{z=-1}|Z^{\beta}\pt v_{\tau}|^2\,\d y\d s\geq-\delta\ep \|Z^{\beta}\pt S^{\vp} v\|_{\ltx}^2-C_{\delta}(\ep\|\pt\nabla v\|_{\hco^{m-3}}^2+\ep\|\pt v\|_{\hco^{m-2}}^2).
\eeq
The remaining task is thus to estimate $\tilde{\cK}_1-\tilde{\cK}_8.$ We assume that $Z^{\beta}$ contains at least one spatial conormal derivative $Z_i (i=1,2,3).$ \\[5pt]
\underline{$\tilde{\cK}_1:$} 
Similar to the proof of \eqref{lowbdryterm1}, we have by the trace inequality \eqref{trace}, Young's inequality and Korn's inequality \eqref{korn} that:
\beq\label{tck1}
\begin{aligned}
\tilde{\cK}_1&=\f{1}{2}\ep\int_0^t\int_{z=0}\pt h|Z^{\beta}\pt v|^2\,\d y\d s\\
&\leq \delta \ep\|Z^{\beta}\pt S^{\vp} v\|_{\ltx}^2 + C_{\delta}\lca (T\ep\|Z^{\beta}\pt v\|_{L_t^{\infty}L^2}^2+\ep\|\pt\nabla v\|_{L_t^2H_{co}^{k-1}}^2).
\end{aligned}
\eeq
\underline{$\tilde{\cK}_2:$} 
By Young's inequality, 
$\cK_2$ can be controlled similarly:
\beqs
\begin{aligned}
\tilde{\cK}_2&=2\mu\ep\int_0^t\int_{\mS} Z^{\beta}\pt S^{\vp}v \cdot [Z^{\beta}\pt,\nabla^{\vp}]v\, \d\cV_s\d s\leq \delta \ep \|Z^{\beta}\pt S^{\vp}v\|_{\ltx}^2+C_{\delta}\ep\|[Z^{\beta}\pt,\nabla^{\vp}]v\|_{\ltx}^2
%\leq \delta \ep \|Z^{\beta}\pt S^{\vp}v\|_{\ltx}^2+C_{\delta}\Lambda\big(\f{1}{c_0},|h|_{m-2,\infty,t}\big)\ep\|\pt\nabla v\|_{L_t^2H_{co}^{m-4}}^2\\&\,
\end{aligned}
\eeqs
Since 
\beqs
[Z^{\beta}\pt,\p_j^{\vp}]f=Z^{\beta}\big(\pt\big(\f{\bN_j}{\p_z\vp}\big)\cdot \p_z f\big)+\big[Z^{\beta}, \f{\bN_j}{\p_z\vp}\big]\pt \p_z f+\f{\bN_j}{\p_z\vp}[Z^{\beta},\p_z]\pt\p_z f,\, j=1,2,3,
\eeqs
we can use the fact that $|\beta|=k\leq m-2$ to get that:
\beq\label{com-format}
\begin{aligned}
&\ep^{\f{1}{2}}\|[Z^{\beta}\pt,\p_j^{\vp}]f\|_{\ltx}\lesssim \Lambda\big(\f{1}{c_0},|h|_{m-2,\infty,t}\big)\ep^{\f{1}{2}}\|\pt\nabla f\|_{L_t^2H_{co}^{k-1}}\\
& +\Lambda\big(\f{1}{c_0}, \il(\text{Id},\ep^{\f{1}{2}}\pt)\p_z f\il_{0,\infty,t}+|(h,\pt h)|_{m-3,\infty,t}\big)(\ep^{\f{1}{2}}\|\p_z f\|_{\hco^{k}}+|(h,\ep^{\f{1}{2}}\pt h)|_{L_t^2\tilde{H}^{k+\f{1}{2}}}).
\end{aligned}
\eeq
We thus obtain that:
\beqs
\tilde{\cK}_2\leq \delta \ep \|Z^{\beta}\pt S^{\vp}v\|_{\ltx}^2+C_{\delta}\Lambda\big(\f{1}{c_0},|h|_{m-2,\infty,t}\big)\|\ep^{\f{1}{2}}\pt\nabla v\|_{L_t^2H_{co}^{k-1}}+(T+\ep)\lae.
\eeqs
\underline{$\tilde{\cK}_3:$} Regarding the estimate of 
$$\tilde{\cK}_3=\ep\int_0^t\int_{z=0}Z^{\beta}\pt(2\mu S^{\vp}v-\pi\text{Id})\bN \cdot Z^{\beta}\pt v\,\d y\d s,$$
as we did for $\cK_3,$ we write:
\beqs
\begin{aligned}
&Z^{\beta}\pt(2\mu S^{\vp}v-\pi\text{Id})\bN=2\mu Z^{\beta}\pt
\big((\div^{\vp}u \text{Id}-(\nabla^{\vp})^2\Psi)\bN\big)+[Z^{\beta}\pt,\bN](2\mu S^{\vp}v-\pi\text{Id})\\
&=2\mu Z^{\beta}\pt(\div^{\vp}u \text{Id}-(\nabla^{\vp})^2\Psi)\bN +\ep^{\f{1}{2}} [Z^{\beta}\pt, \bN](2\mu(\div^{\vp}u \text{Id}-(\nabla^{\vp})^2\Psi)+2\mu S^{\vp}v-\pi\text{Id}).
\end{aligned}
\eeqs
By using the trace inequality \eqref{trace}
and Lemma \ref{lemelliptic}, we get in a  similar way as for \eqref{com-format} that:
\beqs
\begin{aligned}
&\ep^{\f{1}{2}}\big|Z^{\beta}\pt(\div^{\vp}u \text{Id}-(\nabla^{\vp})^2\Psi)\bN\big|_{L_t^2H_y^{-\f{1}{2}}}\\
&\lesssim |h|_{2,\infty,t}\|\ep^{\f{1}{2}}\pt (\div^{\vp}u,(\nabla^{\vp})^2\Psi)\|_{\hco^{m-2}}+\|\ep^{\f{1}{2}}\pt \nabla(\div^{\vp}u,(\nabla^{\vp})^2\Psi)\|_{\hco^{m-3}}\\
&\lesssim \Lambda\big(\f{1}{c_0},|h|_{m-2,\infty,t}\big)\big(\|\ep^{\f{1}{2}}\pt\nabla\div^{\vp}u\|_{\hco^{m-3}}+\|\ep^{\f{1}{2}}\pt\div^{\vp}u\|_{\hco^{m-2}}\big)+\ep^{\f{1}{2}}\lae.
\end{aligned}
\eeqs
%\beqs\begin{aligned}&\lesssim |h|_{m-2,\infty,t}|\ep^{\f{1}{2}}\pt(\div^{\vp}u, (\nabla^{\vp})^2\Psi)|_{L_t^2\tilde{H}^{m-2}}+\ep^{\f{1}{2}}|\pt h|_{\htlde^{m-1}}\il(\div^{\vp}u, (\nabla^{\vp})^2\Psi)\il_{1,\infty,t}\\&\qquad\qquad\qquad\qquad\qquad\qquad\qquad +\ep^{\f{1}{2}}|\pt h|_{m-4,\infty,t}|(\div^{\vp}u, (\nabla^{\vp})^2\Psi)|_{\htlde^{m-3}}\\&\lesssim \Lambda\big(\f{1}{c_0},|h|_{m-2,\infty,t}\big)\|\ep^{\f{1}{2}}\pt(\div^{\vp}u,\nabla\div^{\vp}u)\|_{\hco^{m-3}}+\ep^{\f{1}{2}}\lae.\end{aligned}\eeqs
Moreover, by the boundary conditions \eqref{pibdry}-\eqref{tanpzv}, we have
\beqs
\begin{aligned}
&\ep^{\f{1}{2}}\big|[Z^{\beta}\pt,\bN](2\mu(\div^{\vp}u \text{Id}-(\nabla^{\vp})^2\Psi)+2\mu S^{\vp}v-\pi\text{Id})\big|_{L_t^2L_y^2}\\
&\lesssim |\ep^{\f{1}{2}}\pt (S^{\vp}v, \pi, \div^{\vp}u, (\nabla^{\vp})^2\Psi)^{b,1} |_{\htlde^{k-1}}\Lambda\big(\f{1}{c_0},|h|_{k,\infty,t}\big)\\
&+\Lambda\big(\f{1}{c_0}, |(\text{Id},\ep^{\f{1}{2}}\pt,Z)(S^{\vp}v,\pi,\div^{\vp}u,(\nabla)^2\Psi)^{b,1}|_{0,\infty,t}+|\pt h|_{k-1,\infty,t}\big)\cdot\\
&\qquad\qquad\big(\ep^{\f{1}{2}}|(S^{\vp}v,\pi,\div^{\vp}u,(\nabla)^2\Psi)|_{\htlde^{k}}+|(h,\ep^{\f{1}{2}}\pt h)|_{L_t^2\tilde{H}^{k+1}})\\
&\lesssim \big(|\ep^{\f{1}{2}}\pt v|_{\htlde^{k}}+
\|\ep^{\f{1}{2}}\pt (\div^{\vp} u,\nabla\div^{\vp} u)\|_{L_t^2H_{co}^{k-1}}\big)\Lambda\big(\f{1}{c_0},|h|_{k,\infty,t}\big)+(T+\ep)^{\f{1}{2}}\lae.
\end{aligned}
\eeqs
Therefore, by duality, the Cauchy-Schwarz inequality and Young's inequality \eqref{korn}, we obtain that:
\beq
\begin{aligned}
\tilde{\cK}_3 
&\leq \delta \|\ep^{\f{1}{2}}Z^{\beta}\pt \nabla v\|_{\hco^{k}}^2+(T+\ep)
\lae. \\
&+ C_{\delta}
\Lambda\big(\f{1}{c_0},|h|_{m-2,\infty,t}^2\big)\big(\|\ep^{\f{1}{2}}\pt (v,\div^{\vp}u)\|_{\hco^{m-2}}^2+\|\ep^{\f{1}{2}}\pt\nabla\div^{\vp}u\|_{\hco^{m-3}}^2).
\end{aligned}
\eeq
\underline{$\tilde{\cK}_4:$} 
$\tilde{\cK}_4:$ has the following expression:
\beqs
\tilde{\cK}_4=\ep\int_0^t\int_{\mS}Z^{\beta}\pt \pi [\div^{\vp},Z^{\beta}\pt]v\, \d\cV_s\d s
\eeqs
By H\"older inequality, the estimate  \eqref{pi2} for $\ep^{\f{1}{2}}\nabla\pt \pi,$ the Korn inequality \eqref{korn} and the commutator estimate
\eqref{com-format}
we get:
\beqs
\begin{aligned}
\tilde{\cK}_4&\leq \|\ep^{\f{1}{2}}\nabla\pt\pi\|_{\hco^{k-1}}\|\ep^{\f{1}{2}}[\div^{\vp},Z^{\beta}\pt]v\|_{\ltx} \leq \delta\|\ep^{\f{1}{2}} \pt \nabla v\|_{L_t^2H_{co}^{k}}^2
\\
&+C_{\delta}\Lambda\big(\f{1}{c_0},|h|_{m-2,\infty,t}^2\big)(\|\ep^{\f{1}{2}}\pt \nabla^{\vp}v\|_{\hco^{k-1}}^2+\|\ep^{\f{1}{2}}\pt \div^{\vp}u\|_{\hco^{m-2}}^2)+(T+\ep)^{\f{1}{2}}\lae.
\end{aligned}
\eeqs
\underline{$\tilde{\cK}_5.$}
By the Cauchy-Schwarz inequality and estimates \eqref{pi}, \eqref{pi2}, we obtain:
\beq
\begin{aligned}
\tilde{\cK}_5&\lesssim \|\ep^{\f{1}{2}}Z^{\beta}\pt v\|_{\ltx}\|\ep^{\f{1}{2}}[Z^{\beta}\pt,\nabla^{\vp}]\pi\|_{\ltx}\\
&\lesssim
\|\ep^{\f{1}{2}}\pt v\|_{\hco^{k}}\big(\Lambda\big(\f{1}{c_0},|h|_{m-2,\infty,t}\big)\|\ep^{\f{1}{2}}\pt\nabla\pi\|_{\hco^{k-1}}+(T+\ep)^{\f{1}{2}}\lae\big)\\
&\leq \delta \|\ep^{\f{1}{2}}\pt \nabla v\|_{\hco^{k}}^2+(T+\ep)^{\f{1}{2}}\lae. 
\end{aligned}
\eeq
\underline{$\tilde{\cK}_6,\tilde{\cK}_8:$}
By \eqref{es-f}, \eqref{q}, \eqref{comtime}, we have:
\beqs
\ep^{\f{1}{2}}\|\pt (f+\nabla^{\vp} q+\pt[\bbp,\pt^{\vp}]u)\|_{L_t^2H_{co}^{m-2}}\lesssim \lae. 
\eeqs
In addition, since $|\beta|=k\leq m-2,$ the following estimate holds:
\beqs
\begin{aligned}
\ep^{\f{1}{2}}\|\big[Z^{\beta}\pt,\f{\pt\vp}{\p_z\vp}\p_z\big]v\|_{\ltx}&\lesssim \Lambda\big(\f{1}{c_0},|(h,\pt h)|_{m-3,\infty,t}+\il \nabla v,\ep^{\f{1}{2}}\pt \nabla v\il_{0,\infty,t}
\big) \\
&\big(\|(\ep^{\f{1}{2}}\pt\nabla v,\nabla v)\|_{L_t^{2}H_{co}^{m-3}}\|(\ep^{\f{1}{2}}\pt\nabla v,\nabla v)\|_{L_t^{\infty}H_{co}^{m-4}}|(\pt h,\ep^{\f{1}{2}}\pt^2 h)|_{\htlde^{m-3}}\big)\\
&\lesssim\lae.
\end{aligned}
\eeqs
Therefore, we control $\tilde{\cK}_6+\tilde{\cK}_8$ as:
\beq
\begin{aligned}
\tilde{\cK}_6+\tilde{\cK}_8&\lesssim \|\ep^{\f{1}{2}}Z^{\beta}\pt v\|_{\ltx}\bigg( \ep^{\f{1}{2}}\big\|\big[Z^{\beta}\pt,\f{\pt\vp}{\p_z\vp}\p_z\big]v\big\|_{\ltx}+\ep^{\f{1}{2}}\big\|Z^{\beta}\pt(f+\nabla^{\vp}q+[\bbp,\pt^{\vp}]u)\big\|_{\ltx} \bigg)\\
&\lesssim T^{\f{1}{2}}\lae. 
\end{aligned}
\eeq
\underline{$\tilde{\cK}_7.$} 
For this term, one needs to integrate by parts to avoid losing normal derivatives. By following the same lines
as the control of $\cK_7$ in Lemma \ref{lemhigh-v},  we find that:
\beq\label{tck7}
\begin{aligned}
\tilde{\cK}_7&\leq \delta \|\ep^{\f{1}{2}}\pt \nabla v\|_{\hco^{k}}^2
+\Lambda\big(\f{1}{c_0}, |h|_{m-2,\infty,t}^2
\big)\|\ep^{\f{1}{2}}\pt\nabla^{\vp}v\|_{\hco^{k-1}}^2+(T+\ep)^{\f{1}{2}}\lae.
\end{aligned}
\eeq
Plugging \eqref{tck1}-\eqref{tck7} into 
\eqref{EI-v-low}, we get by choosing $\delta$ small enough and by using Korn inequality \eqref{korn} that for any $0\leq k\leq m-2,$
\beqs
\begin{aligned}
&\quad\|\ep^{\f{1}{2}}\pt v\|_{L_t^{\infty}H_{co}^{k}}^2+\|\ep^{\f{1}{2}}\pt\nabla v\|_{L_t^2H_{co}^{k}}^2\lesssim \|\ep^{\f{1}{2}}\pt v(0)\|_{H_{co}^{m-2}(\mS)}^2+ (T+\ep)^{\f{1}{2}}\lae\\
&+\Lambda\big(\f{1}{c_0},|h|_{m-2,\infty,t}^2\big)\big(\|\ep^{\f{1}{2}}\pt\nabla^{\vp}v\|_{\hco^{k-1}}^2+\|\ep^{\f{1}{2}}\pt \nabla^{\vp}\div^{\vp}u\|_{\hco^{m-3}}^2+\|\ep^{\f{1}{2}}\pt \div^{\vp}u\|_{\hco^{m-2}}^2\big),
\end{aligned}
\eeqs
where we have used the convention that $\|\cdot\|_{H_{co}^l}=0,$ if $l< 0.$ 
This estimate, combined with \eqref{surface1}, \eqref{lowerder}, \eqref{nor-compressible} and the induction on $k$
yields \eqref{lower-v}.
\end{proof}
\section{$\ep-$dependent high order energy estimate-II}
In this subsection, we aim to control $\ep^{\f{1}{2}}\|\nabla u\|_{L_t^{\infty}H_{co}^{m-1}},$
%$\ep^{\f{3}{4}}\|\nabla u\|_{L_t^{\infty}H_{co}^{m-1}}$ 
which is useful for the control of $L^{\infty}$ type norms.
\begin{lem}\label{intermediate-epnablau}
Under the assumption \eqref{preassumption}, we have for any $0<t\leq T,$
\beq\label{epnablauLinfty1}
\begin{aligned}
\ep \|\nabla u\|_{L_t^{\infty}H_{co}^{m-1}}^2
\lesssim \Lambda\big(\f{1}{c_0}, |h|^2_{L_t^{\infty}\tilde{H}^{m-\f{1}{2}}}%+|\ep^{\f{1}{2}}h|^2_{L_t^{\infty}\tilde{H}^{m+\f{1}{2}}}
+Y_m^2(0)\big)Y^2_m(0)+(T+\ep)^{\f{1}{4}}\Lambda\big(\f{1}{c_0},
\cN_{m,T}\big).
\end{aligned}
\eeq
\end{lem}
\begin{proof}
We will prove the following estimates:
\beq\label{epnablauLinfty}
\begin{aligned}
&\|\ep^{\f{1}{2}}\nabla u\|_{L_t^{\infty}H_{co}^{m-1}}^2\lesssim Y^2_m(0)+(T+\ep)^{\f{1}{4}}\lae\\
&+\Lambda\big(\f{1}{c_0}, |h|^2_{m-2,\infty,t}\big)\big(\|\ep^{\f{1}{2}}\nabla u\|_{\hco^{m}}^2+
\|\ep^{\f{1}{2}}\nabla\div^{\vp} u\|_{\hco^{m-1}}^2+\|\ep^{\f{1}{2}}\pt(u,\nabla u)\|_{L_t^2\cH^{m-1}\cap\hco^{m-2} }^2\big).\\
%&\qquad\qquad\qquad\qquad\quad+\|\ep^{\f{1}{2}} \pt(u,\nabla u)\|_{}^2+\|u\|_{E^m,t}+\|\ep^{-\f{1}{2}}\nabla^{\vp}\sigma\|_{\hco^{m-1}}\big).
\end{aligned}
\eeq
By \eqref{EI-1}, \eqref{EI-T1}, \eqref{EI-2},  %\eqref{epnablauLinfty},
%\eqref{nor-compressible}, \eqref{v-0},
\eqref{lower-v}, we can then find a polynomial $\Lambda$, such that \eqref{epnablauLinfty1} holds.

The inequality \eqref{epnablauLinfty} can be obtained by direct energy estimates. Applying $Z^{\alpha}, |\alpha|\leq m-1$
to $\eqref{FCNS2}_2,$ taking the scalar product with $-\ep^2 Z^{\alpha}(\div^{\vp}\cL^{\vp}u)$ and integrating in space and time, we get  by integration by parts that:
 \beq\label{sec8.3:eq0}
 \begin{aligned}
 &\ep \mu\int_{\mS}|Z^{\alpha}S^{\vp}u|^2(t)\,\d\cV_t+\f{1}{2}\ep\lambda \int_{\mS}|Z^{\alpha}\div^{\vp}u|^2(t)\,\d\cV_t+\ep\|Z^{\alpha}\div^{\vp}\cL^{\vp}u\|_{\ltx}^2\\
 &+\f{a}{2}\ep\int_{z=1}|Z^{\alpha}u_{\tau}|^2(t)\,\d y=K_0+K_1+\cdots+K_5,
 \end{aligned}
 \eeq
where 
\beqs
\begin{aligned}
K_0&=\ep%^{\f{3}{2}}
\mu\int_{\mS}|Z^{\alpha}S^{\vp}u|^2(0)\,\d\cV_0+\f{1}{2}\ep\lambda \int_{\mS}|Z^{\alpha}\div^{\vp}u|^2(0)\,\d\cV_0+\f{a}{2}\ep\int_{z=1}|Z^{\alpha}u_{\tau}|^2(0)\,\d y,\\
K_1&=-\ep\int_0^t\int_{\mS}\big(\pt[\nabla^{\vp},Z^{\alpha}]u+[\nabla^{\vp},\pt]Z^{\alpha}u\big)\cdot Z^{\alpha}\cL^{\vp}u\,\d\cV_s\d s,\\
K_2&=\ep\int_0^t\int_{\mS}\pt Z^{\alpha}u\cdot
[Z^{\alpha},\div^{\vp}]\cL^{\vp}u\,\d\cV_s\d s,\quad K_3=\ep\int_0^t\int_{\mS}Z^{\alpha}(\f{g_2-1}{\ep}\ep\pt u)\cdot Z^{\alpha}\div^{\vp}\cL^{\vp}u\,\d\cV_s\d s,\\
K_4&=\ep\int_0^t\int_{\mS}Z^{\alpha}(\nabla^{\vp}\sigma)Z^{\alpha}(\div^{\vp}\cL^{\vp}u)\,\d\cV_s\d s,\quad K_5=-\ep\int_0^t\int_{\p\mS}Z^{\alpha}\cL^{\vp}u \bN \cdot \pt Z^{\alpha}
u\,\d y\d s.\\
\end{aligned}
\eeqs
At first, by the trace inequality \eqref{trace}: 
\beq\label{bdryterm}
\f{a}{2}\ep\int_{z=1}|Z^{\alpha}u_{\tau}|^2(t)\,\d y\geq -\delta\|\ep^{\f{1}{2}} \nabla u(t)\|_{H_{co}^{m-1}}^2-C_{\delta}\ep\| u\|_{L_t^{\infty}H_{co}^{m-1}}^2.
\eeq
Next, for the term $K_1,$ we use \eqref{comgrad} to find that:
\beqs
\begin{aligned}
K_1&\lesssim \|\nabla^{\vp} u\|_{\hco^{m-1}}\bigg(\ep\Lambda\big(\f{1}{c_0},|\pt h|_{0,\infty,t}\big)\|\nabla u\|_{\hco^{m-1}}+\|\ep\pt[\nabla^{\vp},Z^{\alpha}]u\|_{\ltx}\bigg).
\end{aligned}
\eeqs
By using the identity \eqref{comidentity}, we find that:
\beqs 
\ep\pt[Z^{\alpha},\nabla^{\vp}]u=\ep^{\f{1}{2}}\big[Z^{\alpha},\ep^{\f{1}{2}}\pt (\f{\bN}{\p_z\vp})\big]\p_z u+\ep\pt \big(\f{\bN}{\p_z\vp}[Z^{\alpha},\p_z]u\big)+\ep^{\f{1}{2}}\big[Z^{\alpha}, \f{\bN}{\p_z\vp}\big]\ep^{\f{1}{2}}\pt\p_z u.
\eeqs
The $L_t^2L^2$ norm of the first two terms in the right hand side can be controlled by: 
\beqs 
\begin{aligned}
&\ep^{\f{1}{2}}\Lambda\big(\f{1}{c_0},|(h,\ep^{\f{1}{2}}\pt h )|_{m-2,\infty,t}+\il\nabla u\il_{1,\infty,t}\big)\big(\|\nabla u\|_{L_t^2H_{co}^{m-2}}+|(h,\ep^{\f{1}{2}}\pt h)|_{L_t^2\tilde{H}^{m-\f{1}{2}}}\big)\\
&\lesssim \ep^{\f{1}{2}}\lae.
\end{aligned}
\eeqs
Moreover, the third term can be bounded as:
\beqs
\begin{aligned}
\big\|\ep^{\f{1}{2}}\big[Z^{\alpha}, \f{\bN}{\p_z\vp}\big]\ep^{\f{1}{2}}\pt\p_z u\big\|_{\ltx}&\lesssim T^{\f{1}{2}}\il\ep^{\f{1}{2}}\pt\nabla u\il_{L_t^2H_{co}^{1}}\big|\ep^{\f{1}{2}}\big(\f{\bN}{\p_z\vp}\big)\big|_{m-2,\infty,t}\\
&+ \ep^{\f{1}{2}}
\Lambda\big(\f{1}{c_0}\il\ep^{\f{1}{2}}\pt\nabla u\il_{0,\infty,t}+|h|_{m-2,\infty,t}\big)(\|\ep^{\f{1}{2}}\pt\nabla u\|_{L_t^2H_{co}^{m-2}}+|h|_{\htlde^{m-\f{1}{2}}})\\
&\lesssim(T+\ep)^{\f{1}{2}}\lae.
\end{aligned}
\eeqs
The previous two estimates then lead to:
\beqs
\|\ep\pt[Z^{\alpha},\nabla^{\vp}]u\|_{\ltx}\lesssim (T+\ep)^{\f{1}{2}}\lae,
\eeqs
from which we find that:
\beq\label{K1}
\begin{aligned}
K_1\lesssim (T+\ep)^{\f{1}{2}}\lae.
\end{aligned}
\eeq
Thanks to the commutator estimate \eqref{comgrad}, we control the term $\ep^{\f{1}{2}}[Z^{\alpha},\div^{\vp}]\cL^{\vp}u$ in the term $K_2$ as follows:
\beqs
\begin{aligned}
\ep^{\f{1}{2}}\|[Z^{\alpha},\div^{\vp}]\cL^{\vp}u\|_{\ltx}&\lesssim\Lambda\big(\f{1}{c_0},\il\ep^{\f{1}{2}}\p_z\cL^{\vp}u\il_{1,\infty,t}+|h|_{m-2,\infty,t}\big)|h|_{\htlde^{m-\f{1}{2}}} \\%\|u\|_{\hco^m}
&\qquad+\Lambda\big(\f{1}{c_0},|h|_{m-2,\infty,t}\big)\|\ep^{\f{1}{2}}\nabla\cL^{\vp}u\|_{\hco^{m-2}}\lesssim T^{\f{1}{2}}\lae.
\end{aligned}
\eeqs
Therefore, by Cauchy-Schwarz inequality, $K_2$ can be bounded by:
\beq\label{K2}
K_2\lesssim \|\ep^{\f{1}{2}}\pt u\|_{\hco^{m-1}} \|\ep^{\f{1}{2}}[Z^{\alpha},\div^{\vp}]\cL^{\vp}u\|_{\ltx}\lesssim %\Lambda\big(\f{1}{c_0},|h|_{m-2,\infty,t}\big)(\|u\|_{\hco^{m}}^2+\|\ep^{\f{1}{2}}\pt u\|_{\hco^{m-1}}^2+\|\ep^{-\f{1}{2}}\nabla^{\vp}\sigma\|_{\hco^{m-2}}^2)+
T^{\f{1}{2}}\lae.
\eeq
Moreover, by the product estimate \eqref{crudepro}, %\eqref{commutator},
we obtain:
\beq\label{K34}
K_3+K_4\leq \delta \|\ep^{\f{1}{2}}  Z^{\alpha}\div^{\vp}\cL^{\vp}u\|_{\ltx}^2+C_{\delta}\ep\lca\|(\sigma,u)\|_{E^m,t}.
\eeq
For the term $K_5$, we use the boundary condition \eqref{upbdry2} to split it as : 
\beqs
\begin{aligned}
K_5&=-\ep\int_0^t\int_{z=0}Z^{\alpha}(\sigma/\ep)\pt Z^{\alpha}u\cdot\bN+ [Z^{\alpha},\bN]\cL^{\vp}u\cdot \pt Z^{\alpha}u\,\d y\d s=\colon K_{51}+K_{52}.\\
\end{aligned}
\eeqs
Thanks to the trace inequality \eqref{trace} and the boundary conditions \eqref{norpz}, \eqref{tanpz}, $K_{52}$ can be bounded as:
\beqs
\begin{aligned}
K_{52}&\lesssim
 |\ep^{\f{1}{2}}\pt Z^{\alpha} u|_{L_t^2H^{-\f{1}{2}}}|\ep^{\f{1}{2}}[Z^{\alpha},\bN]\cL^{\vp}u|_{L_t^2H^{\f{1}{2}}}%(\il u\il_{2,\infty,t}+|h|_{2,\infty,t})+\ep |u|_{\htlde^m}|[Z^{\alpha},\bN,\cL^{\vp}u]|_{L_t^2L_y^2}\\
 \\
 &\lesssim \big(\|\ep^{\f{1}{2}}\pt(u,\nabla u)\|_{L_t^2\cH^{m-1}}+\|\ep^{\f{1}{2}}\pt\nabla u\|_{\hco^{m-2}}\big)\big(\ep^{\f{1}{2}}|h|_{\htlde^{m+\f{1}{2}}}\Lambda\big(\f{1}{c_0},\cA_{m,t}\big)+\ep^{\f{1}{2}}\Lambda\big(\f{1}{c_0},\cN_{m,t}\big)\big)\\
&\lesssim (T+\ep)^{\f{1}{2}}\Lambda\big(\f{1}{c_0},\cA_{m,t}\big)\cE_{m,t}^2.
\end{aligned}
\eeqs
For $K_{51},$ we take benefits of the boundary condition \eqref{sigmabdry} and the trace inequality \eqref{trace} to find that, if $Z^{\alpha}=(\ep\pt)^{j},j\leq m-1,$
\beqs
\begin{aligned}
K_{51}&\lesssim \ep^{\f{1}{4}} \lca \big(\|(\ep^{\f{1}{2}}\pt(u,\nabla u),\div^{\vp}u,\ep^{\f{1}{2}}\nabla\div^{\vp}u)\|_{L_t^2\cH^{m-1}}^2+\|\nabla u\|_{\hco^{m-1}}^2+ |h|_{\htlde^{m-\f{1}{2}}}^2\big)\\
&\lesssim \ep^{\f{1}{4}}\lca \cE_{m,t}^2,
\end{aligned}
\eeqs
if $Z^{\alpha}=Z_3 Z^{\tilde{\alpha}},$ this term vanishes on the boundary
and if $Z^{\alpha}=\p_y Z^{\tilde{\alpha}},$
\beqs
\begin{aligned}
K_{51}&\lesssim  |\ep^{\f{1}{2}}Z^{{\alpha}}(\sigma/\ep)|_{L_t^2H_y^{\f{1}{2}}}
|\ep^{\f{1}{2}}\pt Z^{\tilde{\alpha}} u|_{L_t^2H_y^{\f{1}{2}}}|h|_{2,\infty,t}\\
&\lesssim \Lambda(\f{1}{c_0},|h|_{m-2,\infty,t}) 
%\lca 
\big(\|\ep^{\f{1}{2}}\nabla\div^{\vp} u\|_{\hco^{m-1}}^2+\|\ep^{\f{1}{2}}\nabla u\|_{\hco^{m}}^2+\|\ep^{\f{1}{2}} \pt(u,\nabla u)\|_{\hco^{m-2}}^2\big)\\
&\quad+T^{\f{1}{2}}\lca |h|_{\htlde^{m-\f{1}{2}}}
.
\end{aligned}
\eeqs
The previous three inequalities
%the definition of $\cD_{m,T}$ 
yield:
\beqs\label{K5}
\begin{aligned}
K_{5}&\lesssim(T+\ep)^{\f{1}{4}}\Lambda\big(\f{1}{c_0},\cA_{m,t}\big)\cE_{m,t}^2+ \Lambda(\f{1}{c_0}, |h|_{m-2,\infty,t})\cdot\\
&\big(\|\ep^{\f{1}{2}}\nabla\div^{\vp} u\|_{\hco^{m-1}}^2+\|\ep^{\f{1}{2}}\nabla u\|_{\hco^{m}}^2+\|\ep^{\f{1}{2}}\pt(u,\nabla u)\|_{\hco^{m-2}\cap L_t^2\cH^{m-1} }^2\big).
\end{aligned}
\eeqs
Inserting this inequality and \eqref{bdryterm}-\eqref{K34} into \eqref{sec8.3:eq0}, using Korn's inequality \eqref{korn1} and choosing $\delta$ small enough, we obtain \eqref{epnablauLinfty}. 
\end{proof}

\section{Uniform control of high order energy norms-II} 
\subsection{$L_t^{\infty}L^2$ type norm for the compressible part}
In this section, we aim to get the a-priori estimates for
$\|(\nabla^{\vp}\sigma,\div^{\vp}u)\|_{L_t^{\infty}H_{co}^{m-2}}.$ This is mainly done by induction arguments.
\begin{lem}\label{sigmainduction1}
Suppose that \eqref{preassumption} is true, we have for any $0<t\leq T, m\geq 7,$
\begin{equation}\label{nor-compressible1}
\begin{aligned}
 &\ep^{-1}\|(\nabla^{\vp}\sigma,\div^{\vp}u)\|_{L_t^{\infty}H_{co}^{m-2}}^2\\
 &\lesssim \Lambda\big(\f{1}{c_0},|h|_{L_t^{\infty}\tilde{H}^{m-\f{1}{2}}}^2+Y_m^2(0)\big)Y_m^2(0)
 +(\ep+T)^{\f{1}{4}}\Lambda\big(\f{1}{c_0},\cN_{m,T}\big).
 \end{aligned}
\end{equation}
\end{lem}
\begin{proof}
We shall prove for for $j+l\leq m-2$ that:
\beq\label{induction2-1}
\begin{aligned}
&\ep^{-\f{1}{2}}\|(\nabla^{\vp}\sigma,\div^{\vp}u)\|_{L_t^{\infty}\cH^{j,l}}\\
&\lesssim
(T+\ep)^{\f{1}{4}}\Lambda\big(\f{1}{c_0},\cN_{m,T}\big)+\Lambda\big(\f{1}{c_0},|h|_{m-2,\infty,t}\big)\|\ep^{\f{1}{2}}\pt \div^{\vp}u\|_{L_t^{\infty}H_{co}^1}|h|_{L_t^{\infty}H_{co}^{m-\f{3}{2}}}\\
&\quad+\Lambda\big(\f{1}{c_0},|h|_{L_t^{\infty}\tilde{H}^{m-\f{1}{2}}}\big)(\|\ep^{\f{1}{2}}\pt(\sigma, u)\|_{L_t^{\infty}\cH^{m-2,0}}+\|\ep^{\f{1}{2}}\nabla(\sigma,u)\|_{L_t^{\infty}H_{co}^{m-1}}),
\end{aligned}
\eeq 
and also:
\beq\label{induction2-2}
\begin{aligned}
\|\ep^{\f{1}{2}}\pt(\div^{\vp}u,\nabla^{\vp}\sigma)\|_{L_t^{\infty}H_{co}^1}\lesssim \|\ep^{\f{1}{2}}\pt (\sigma,u)\|_{L_t^{\infty}\cH^{2}}+(T+\ep)^{\f{1}{2}}\lae.
\end{aligned}
\eeq
These two inequalities, together with \eqref{EI-T1}, \eqref{EI-2} and \eqref{epnablauLinfty1} lead to \eqref{nor-compressible}.
Indeed, thanks to the estimate \eqref{EI-T1}, we derive that:
\beqs 
\|\ep^{\f{1}{2}}\pt(\div^{\vp}u,\nabla^{\vp}\sigma)\|_{L_t^{\infty}H_{co}^1}\lesssim \Lambda\big(\f{1}{c_0},|h|_{m-2,\infty,t}\big)Y_m(0)+(T+\ep)^{\f{1}{2}}\lae.
\eeqs
Inserting this inequality into \eqref{induction2-1}, and using the estimate \eqref{EI-T1}, \eqref{EI-2}, \eqref{epnablauLinfty1}, we find \eqref{nor-compressible1}.

%We will present the proof of \eqref{induction2-1}, the proof of \eqref{induction2-2} is similar.
%Inequality \eqref{induction2-1} holds for $l=0, j\leq m-2$ due to the appearance of term $\|(\nabla^{\vp}{\sigma},\div^{\vp}u)\|_{L_t^{\infty}\cH^{m-2}}$ in the right hand side. Let us assume that \eqref{induction2} is true for $j+1,l-1,$ with $j+l\leq m-2,l\geq 1$ and prove the case for $j,l.$
 We present the proof of \eqref{induction2-1}. First of all, for any non-negative integers $j,l$ such that $j+l\leq m-2,$
 it follows from the equation \eqref{re-sigma} that:
   \beq\label{divu-induction}
   \begin{aligned}
   &\ep^{-\f{1}{2}}\|\div^{\vp}u\|_{L_t^{\infty}\cH^{j,l}} \lesssim \|\ep^{\f{1}{2}}\pt\sigma\|_{L_t^{\infty}\cH^{j,l}}+\ep^{\f{1}{2}} \|\big(\f{g_1-g_1(0)}{\ep} \ep\pt+g_1\underline{u}\cdot\nabla\big)\sigma\|_{L_t^{\infty}\cH^{j,l}}\\
   &\lesssim \|\ep^{\f{1}{2}}\pt\sigma\|_{L_t^{\infty}\cH^{m-2,0}}+ \|\ep^{-\f{1}{2}}\nabla^{\vp}\sigma\|_{L_t^{\infty}\cH^{j+1,l-1}}
   \mathbb{I}_{\{l\geq 1\}}
   +\ep^{\f{1}{2}}
   \lca \cE_{m,t}.
 \end{aligned}
   \eeq
Let us  control $\|\nabla^{\vp}\sigma\|_{L_t^{\infty}\cH^{j,l}}.$ 
As before, we denote $$\theta={\sigma}/{\ep}-2(\mu+\lambda)\div^{\vp}u.$$ 
By the equation of velocity,  $$\nabla^{\vp}\theta=-\pt^{\vp}u-f+\mu\Delta^{\vp}v,$$ where $$f=\f{g_2-\bar{\rho}}{\ep}\ep\pt^{\vp}u+u\cdot\nabla^{\vp}u, v=\bbp u.$$ We thus get that:
\beq\label{sec11:eq0}
\begin{aligned}
\ep^{-\f{1}{2}}\|\nabla^{\vp}\sigma\|_{L_t^{\infty}\cH^{j,l}}&\lesssim \ep^{\f{1}{2}} \|\nabla^{\vp}\div^{\vp}u\|_{L_t^{\infty}H_{co}^{m-2}}+\ep^{\f{1}{2}}\|(\bbp,\bbq)\nabla^{\vp}\theta\|_{L_t^{\infty}\cH^{j,l}}\\
&\lesssim \ep^{\f{1}{2}}\|\pt^{\vp}\nabla^{\vp}\Psi\|_{L_t^{\infty}\cH^{j,l}}+\ep^{\f{1}{2}} \|\nabla^{\vp}\div^{\vp}u\|_{L_t^{\infty}H_{co}^{m-2}}\\
&\quad +\ep^{\f{1}{2}}\big\| \big(\nabla^{\vp}\pi,[\bbq,\pt^{\vp}]u,\nabla^{\vp}q\big) \big\|_{L_t^{\infty}H_{co}^{m-2}}=:
\eqref{sec11:eq0}_1+\eqref{sec11:eq0}_2+\eqref{sec11:eq0}_3.
%+\ep^{\f{1}{2}}\|\|_{L_t^{\infty}H_{co}^{m-2}}+\ep^{\f{1}{2}}\|\|_{L_t^{\infty}H_{co}^{m-2}}.
\end{aligned}
\eeq
where we have used the defintion \eqref{defpi},\eqref{def-f-free}.
%Recall that $\bbp \nabla^{\vp}\theta=\nabla^{\vp}\pi, \nabla^{\vp}q=-\bbq(f-\mu\Delta^{\vp}v).$
By the elliptic estimate \eqref{elliptic2.5-1},
\beq\label{epptnablapsi}
\begin{aligned}
&\|\ep^{\f{1}{2}}\pt \nabla^{\vp}\Psi\|_{L_t^{\infty}\cH^{j,l}}\\
&\lesssim (T+\ep)^{\f{1}{2}}\lae+\Lambda\big(\f{1}{c_0},|h|_{m-2,\infty,t}\big)\|\ep^{\f{1}{2}}\pt \div^{\vp}u\|_{L_t^{\infty}H_{co}^1}|h|_{L_t^{\infty}H_{co}^{m-\f{3}{2}}}\\
&+\Lambda\big(\f{1}{c_0},|h |_{m-2,\infty,t}\big) \big(\|\ep^{\f{1}{2}}\pt u\|_{L_t^{\infty}\cH^{m-2,0}}+\|\ep^{\f{1}{2}}\pt \div^{\vp}u\|_{L_t^{\infty}\cH^{j
,l-1}}  \mathbb{I}_{\{l\geq 1\}}\big).
\end{aligned}
\eeq
Next, by the elliptic estimate \eqref{elliptic3}, we find:
\beqs
\begin{aligned}
\ep^{\f{1}{2}}\big\|\f{\pt \vp}{\p_z\vp} \p_z\nabla^{\vp}\Psi\big\|_{L_t^{\infty}H_{co}^{m-2}}&\lesssim \ep^{\f{1}{2}}\lae \big(\|\div^{\vp}u\|_{L_t^{\infty}H_{co}^{m-2}}+|h|_{L_t^{\infty}\tilde{H}^{m-\f{1}{2}}}+|\pt h|_{L_t^{\infty}\tilde{H}^{m-2}}\big)\\
&\lesssim \ep^{\f{1}{2}} \lae.
\end{aligned}
\eeqs
Together with \eqref{epptnablapsi}, this yields:
\beq\label{sec11:eq1}
\begin{aligned}
&\|\ep^{\f{1}{2}}\pt^{\vp}\nabla^{\vp}\Psi\|_{L_t^{\infty}\cH^{j,l}}\\
&\lesssim (T+\ep)^{\f{1}{2}}\lae+\Lambda\big(\f{1}{c_0},|h|_{m-2,\infty,t}\big)\|\ep^{\f{1}{2}}\pt \div^{\vp}u\|_{L_t^{\infty}H_{co}^1}|h|_{L_t^{\infty}H_{co}^{m-\f{3}{2}}}\\
&+\Lambda\big(\f{1}{c_0},|h |_{m-2,\infty,t}\big) \big(\|\ep^{\f{1}{2}}\pt u\|_{L_t^{\infty}\cH^{m-2,0}}+\|\ep^{\f{1}{2}}\pt \div^{\vp}u\|_{L_t^{\infty}\cH^{j
,l-1}}  \mathbb{I}_{\{l\geq 1\}}\big).
\end{aligned}
%\lesssim\Lambda\big(\f{1}{c_0},|h|_{m-2,\infty,t}\big)\|\div^{\vp} u\|_{L_t^{\infty}\cH^{j+1,l-1}}+(T+\ep)^{\f{1}{2}} \lae.
\eeq
Let us control the terms $\eqref{sec11:eq0}_2,\eqref{sec11:eq0}_3$ appearing in \eqref{sec11:eq0}:\\[6pt]
$\bullet\quad \eqref{sec11:eq0}_2= \ep^{\f{1}{2}}\|\nabla\div^{\vp}u\|_{L_t^{\infty}H_{co}^{m-2}}.$ Thanks to the equation \eqref{eq-graddiv}, we have:
\beq\label{sec11:eq2}
\begin{aligned}
&\quad \ep^{\f{1}{2}}\|\nabla\div^{\vp}u\|_{L_t^{\infty}H_{co}^{m-2}}\\
&\leq \ep^{\f{1}{2}}\|\nabla^{\vp}\sigma\|_{L_t^{\infty}H_{co}^{m-1}}+\ep^{\f{3}{2}}\big\|\big(\f{\pt\vp}{\p_z\vp}\p_z\sigma,\nabla^{\vp}(\f{g_2-1}{\ep}\ep\pt\sigma),\nabla^{\vp}(g_2\underline{u}\cdot\nabla\sigma)\big)\big\|_{L_t^{\infty}H_{co}^{m-2}}\\
&\lesssim \ep^{\f{1}{2}}\|\nabla^{\vp}\sigma\|_{L_t^{\infty}H_{co}^{m-1}}+\ep^{\f{1}{2}}\lae.
\end{aligned}
\eeq
$\bullet\quad \eqref{sec11:eq0}_3=\ep^{\f{1}{2}}\|(\nabla^{\vp}q,\nabla^{\vp}\pi,[\bbq,\pt^{\vp}]u)\|_{L_t^{\infty}H_{co}^{m-2}}.$ By \eqref{q}, \eqref{pi4}, \eqref{comtime}, we have that:
\beq\label{qp}
\ep^{\f{1}{2}}\|(\nabla^{\vp}q,\nabla^{\vp}\pi,[\bbq,\pt^{\vp}]u)\|_{L_t^{\infty}H_{co}^{m-2}}\lesssim \Lambda\big(\f{1}{c_0},|h|_{m-2,\infty,t}\big)\|\ep^{\f{1}{2}}\nabla u\|_{L_t^{\infty}H_{co}^{m-1}}+\ep^{\f{1}{2}}\lae.
\eeq
Inserting \eqref{sec11:eq1}-\eqref{qp}
into \eqref{sec11:eq0}, we achieve that: %(by changing $\Lambda$ accordingly)
\beq
\begin{aligned}
&\|\nabla^{\vp}\sigma\|_{L_t^{\infty}\cH^{j,l}}\lesssim \ep^{\f{1}{2}}\lae+\Lambda\big(\f{1}{c_0},|h|_{m-2,\infty,t}\big)\|\ep^{\f{1}{2}}\pt \div^{\vp}u\|_{L_t^{\infty}H_{co}^1}|h|_{L_t^{\infty}H_{co}^{m-\f{3}{2}}}\\
&+(\|\div^{\vp} u\|_{L_t^{\infty}\cH^{j+1,l-1}}+\|\ep^{\f{1}{2}}\pt u\|_{L_t^{\infty}\cH^{m-2,0}}+\|\ep^{\f{1}{2}}\nabla(\sigma,u)\|_{L_t^{\infty}H_{co}^{m-1}})\Lambda\big(\f{1}{c_0},|h|_{m-2,\infty,t}\big).
\end{aligned}
\eeq
Together with \eqref{divu-induction} and induction arguments, this yields \eqref{induction2-1}.
\end{proof}
\begin{rmk}
By the estimates \eqref{psiLinfty}
\eqref{ptpsiL2} and \eqref{nor-compressible}, \eqref{nor-compressible1}, we find
\begin{align*}
&\|\ep^{\f{1}{2}}\pt\nabla^{\vp}\Psi\|_{L_t^{\infty}H_{co}^{m-2}}+\|\ep^{\f{1}{2}}\pt(\nabla^{\vp})^2\Psi\|_{L_t^{\infty}H_{co}^{m-3}\cap \hco^{m-2}}\\
&\lesssim \Lambda\big(\f{1}{c_0},|h|_{L_t^{\infty}\tilde{H}^{m-\f{1}{2}}}^2+Y_m^2(0)\big)Y_m^2(0)
 +(\ep+T)^{\f{1}{4}}\Lambda\big(\f{1}{c_0},\cN_{m,T}\big).
\end{align*}
which further, together with \eqref{lower-v}, yields that:
\beq\label{ptu-free-Linfty}
\begin{aligned}
&\|\ep^{\f{1}{2}}\pt u\|_{L_t^{\infty}H_{co}^{m-2}}
+\|\ep^{\f{1}{2}}\pt \nabla u\|_{L_t^{2}H_{co}^{m-2}}\\
&\lesssim \Lambda\big(\f{1}{c_0},|h|_{L_t^{\infty}\tilde{H}^{m-\f{1}{2}}}^2+Y_m^2(0)\big)Y_m^2(0)
 +(\ep+T)^{\f{1}{4}}\Lambda\big(\f{1}{c_0},\cN_{m,T}\big).
 \end{aligned}
\eeq
\end{rmk}
\subsection{Uniform control of the gradient of the velocity-II} 
In this subsection, we aim  to 
control the $L_t^{\infty}H_{co}^{m-4}$ norm of 
$(\nabla u,\ep^{\f{1}{2}}\pt\nabla u)$ More precisely, the following lemma will be proved.
\begin{lem}\label{lemnablau-LinftyL2}
Under the assumption \eqref{preassumption}, for any $0<t\leq T,$ we have the following estimate:
\beq\label{nablau-LinftyL2}
\begin{aligned}
&\|\nabla u\|_{L_t^{\infty}H_{co}^{m-4}}^2+\|\ep^{\f{1}{2}}\pt\nabla u\|_{L_t^{\infty}H_{co}^{m-4}}^2\\
&\lesssim  \Lambda\bigg(\f{1}{c_0},|h|^2_{L_t^{\infty}\tilde{H}^{m-\f{1}{2}}}+Y^2_m(0)\bigg)Y^2_m(0)+(T+\ep)^{\f{1}{4}}\lae.
\end{aligned}
\eeq
\end{lem}
\begin{proof}
 By the identities \eqref{tanpz} and
\beqs
\begin{aligned}
|\bN|\Pi(\p_z^{\vp}u)&=\Pi(\p_{\bn}^{\vp}u-\bn_1\p_1u-\bn_2\p_2u)\\
&=\omega\times \bn+\Pi\big((\nabla^{\vp}u)^{t}\cdot\bn-\bn_1\p_1u-\bn_2\p_2u\big)\\
&=\omega\times \bn+\Pi(\p_1u\cdot\bn,\p_2 u\cdot\bn,0)^{t}-\Pi(\bn_1\p_1u+\bn_2\p_2u),
\end{aligned}
\eeqs
we have that:
\beqs
\|\nabla^{\vp}u\|_{L_t^{\infty}H_{co}^{m-4}}\lesssim \Lambda\big(\f{1}{c_0},|h|_{m-2,\infty,t}\big) \|u\|_{L_t^{\infty}H_{co}^{m-3}}+\|(\omega\times \bn,\div^{\vp}u)\|_{L_t^{\infty}H_{co}^{m-4}},
\eeqs
\begin{align*}
\|\ep^{\f{1}{2}}\pt\nabla^{\vp}u\|_{L_t^{\infty}H_{co}^{m-4}}&\lesssim \Lambda\big(\f{1}{c_0},|h|_{m-2,\infty,t}\big) \|\ep^{\f{1}{2}}\pt u\|_{L_t^{\infty}H_{co}^{m-3}}+\|\ep^{\f{1}{2}}\pt(\omega\times \bn,\div^{\vp}u)\|_{L_t^{\infty}H_{co}^{m-4}}\\
&\qquad+(T+\ep)^{\f{1}{2}}\lae.
\end{align*}
Therefore, \eqref{nablau-LinftyL2}
can be derived from the estimate 
\eqref{ptu-free-Linfty},
Lemma \ref{lemnablasigma} for $\div^{\vp}u,$ Lemma \ref{lemhigh-v} for $v,$ Lemma \ref{lemsurface} for $h$
as well as the next lemma for $\omega\times\bn.$ 
\end{proof}
\begin{lem}\label{lemomegatimesbn}
Suppose that Assumption \eqref{preassumption} is true,
then the following estimate holds:
\begin{equation}\label{omegatimesbn}
\begin{aligned}
&\|\omega\times\bn\|_{L_t^{\infty}H_{co}^{m-4}}^2+\|\ep^{\f{1}{2}}\pt(\omega\times\bn)\|_{L_t^{\infty}H_{co}^{m-4}}^2\lesssim Y^2_{m}(0)+(T+\ep)^{\f{1}{4}}\lae\\
&\qquad\qquad\qquad\qquad\qquad+\Lambda\big(\f{1}{c_0},|h|_{m-2,\infty,t}\big) %\big(\|v\|_{L_t^{\infty}H_{co}^{m-2}}+
\|(v,\ep^{\f{1}{2}}\pt v, \ep^{\f{1}{2}}\nabla u)\|_{L_t^{\infty}H_{co}^{m-2}}^2.
\end{aligned}
\end{equation}
\end{lem}
\begin{proof}
As explained in the introduction, although
$\omega\times \bn$ satisfies a transport-diffusion equation without singular terms, one cannot control it by direct energy estimates due to the lack of information of the trace of
$\omega\times \bn$ on the boundary. Since $$(\omega\times\bn)|_{z=0}=2\Pi(\p_1u\cdot \bn,\p_2u\cdot\bn,0)^{t}|_{z=0}.$$
A natural attempt in order to do energy estimates is to introduce the modified vorticity: $\tilde{\omega}=\omega\times \bn-\Pi(\p_1u\cdot \bn,\p_2u\cdot\bn,0)^{t}.$ Nevertheless, if taking this way,  
we are confronted with the original difficulty due to the presence of a singular term in the equation of $\omega\times\bn.$ However, 
since the singular term appears only in the equation of the compressible part of the velocity, it is still useful to introduce the following quantity:
\beq
\omega_{\bn}=\omega\times \bn-2\Pi(\p_1v\cdot \bn,\p_2v\cdot\bn,0)^{t}.
\eeq
where $v$ is the incompressible part of the velocity. As will be seen later, the main advantage to work on $\omega_{\bn}$ rather than $\omega\times \bn$ is that up to remainders, one can reduce the estimate of $\omega_{\bn}$ to that of the compressible part of the velocity and one can extract some extra power of $T$ in the  estimates, %of $L_t^{\infty}H_{co}^{m-3}$ norm of $\omega_{\bn},$ 
which is essential to establish the local existence on a uniform time interval. 

Since away from the boundary, the conormal space is equivalent to the standard Sobolev space, it suffices to estimate $\omega_{\bn}$ near the boundary.
In the following, we shall focus on its control near the surface,
the case near the bottom is similar (and is even simpler, one can refer to \cite{masmoudi2021uniform} for details). To overcome the difficulty resulting from the nontrivial boundary condition, the general strategy to get a uniform estimate for $\omega_{\bn}$ is to split its system into two systems, one  carries on  the nonlinear terms and the initial data but with trivial Dirichlet boundary condition, while the other one is just a free heat equation with zero initial data and nontrivial Dirichlet boundary condition. The first system
can be treated by direct energy estimates because of the homogeneous Dirichlet boundary condition. The analysis of the second system relies on the explicit formulae for the heat equation in the half-space. 

To use the explicit formulae of the heat equation in the half-space, it is convenient to use a coordinate system in which the Laplacian has a good form.
We thus use the following normal geodesic coordinates:
\begin{equation}\label{geodesic}
   \begin{aligned}
   \tilde{\Phi}_t:
\quad \mS_{\kappa}=&\mathbb{R}^2\times [-\kappa,0]\longrightarrow \Omega_t\\
& \qquad (y,z) \rightarrow \left(\begin{array}{l}
    y\\
    h(t,y)
\end{array}\right)+z\bn^{b,1}(y)
 \end{aligned} 
\end{equation}
where $\bn^{b,1}=\f{\bN^{b,1}}{|\bN^{b,1}|}=
(-\p_1h,-\p_2 h,1)/{\sqrt{1+|\nabla h|^2}}$ denotes
the outward normal vector.
Straightforward computations show that:
\beqs
\D\tilde{\Phi}_t= \Bigg(\begin{array}{ccc}
    1 &0& \bn_1^{b,1}  \\
    0 &1& \bn_2^{b,1}  \\
    \p_1 h&\p_2 h& \bn_3^{b,1}
\end{array}\Bigg)
+z\Bigg(\begin{array}{ccc}
    \p_1\bn_1^b& \p_2\bn_1^b & 0  \\
    \p_1\bn_2^b&\p_2\bn_2^b& 0  \\
   \p_1\bn_3^b&\p_2\bn_3^b& 0
\end{array}\Bigg)
\eeqs
Therefore, as long as $|h|_{2,\infty,T}<+\infty,$ and $\kappa$ small enough, we have that: $\det(D\tilde{\Phi}_t)>0$ on $[0,T]\times \mS_{\kappa},$
hence $\tilde{\Phi}_t$ is a diffeomorphism between $\mS_{\kappa}$ and $\tilde{\Phi}_t(\mS_{\kappa}).$
The Riemann metric induced by the pullback of the Euclidean 
metric in $\Omega_t$ through $\tilde{\Phi}_t^{-1}$ has the 
block structure:
\beqs
g(y,z)=\left(
\begin{array}{cc}
   \tilde{g}(y,z) & 0  \\
    0 & 1 
\end{array}
\right)
\eeqs
where $\tilde{g}$ is a matrix that depends on the gradient of $\tilde{\Phi}_t.$
Therefore, the Laplacian in
this metric takes the form:
\beq\label{Riemannmetric}
\Delta_{g}f=\p_z^2 f+\f{1}{2}\p_z(\ln |g|)\p_z f+\Delta_{\tilde{g}}f, 
\eeq
where 
\beqs
\Delta_{\tilde{g}}f=\f{1}{|\tilde{g}|^{\f{1}{2}}}\sum_{1\leq i,j\leq 2}\p_{y^i}(\tilde{g}^{ij}|\tilde{g}|^{\f{1}{2}}\p_{y^j}f)\quad |\tilde{g}|=\det \tilde{g}.
\eeqs
We take a cut off function $\chi=\chi_0(\f{z}{C(\kappa)}),$ where $\chi_0(s): \mathbb{R}_{-}\rightarrow \mathbb{R}$ is a smooth function supported on $[-\f{3}{4},0]$ and equal to $1$ on the interval $[-\f{1}{2},0],$ $C(\kappa)$ is chosen such that $\Phi_t(\mathbb{R}^2\times[-C_{\kappa},0])\subset \tilde{\Phi}_t(\mS_{\kappa}),$
the following task is to  estimate  $\chi\omega_{\bn}.$ Let us begin with the derivation of the equations satisfied by $\chi
\omega_{\bn}.$
First of all, by taking the curl of $\eqref{FCNS2}_2$, we find that $\omega=\curl^{\vp} u$ solves:
\beq\label{eq-omega-free}
(\bar{\rho}\p_t^{\vp}-\mu\Delta^{\vp})\omega=G^{\omega}
\eeq
with 
$$G^{\omega}=-u\cdot\nabla^{\vp}\omega+\omega\cdot\nabla^{\vp}u-\omega\div^{\vp}u-\f{\nabla g_2}{\ep}\times ((\ep\pt+\ep\underline{u}\cdot\nabla)u)+\f{\bar{\rho}-g_2}{\ep}((\ep\pt+\ep\underline{u}\cdot\nabla)\omega).$$
Hence $\chi\omega\times\bn$
is governed by:
\beqs
(\bar{\rho}\p_t^{\vp}-\mu\Delta^{\vp})(\chi\omega\times\bn)=G_{\chi}^{\omega}
\eeqs
with
\beq\label{Gchiomega}
G_{\chi}^{\omega}=
\chi G^{\omega}\times \bn%+(\underline{u}\cdot\nabla(\chi\bn))\times\omega
-\mu\Delta^{\vp}(\chi\bn)\omega-2\mu\nabla^{\vp}\omega \times \nabla^{\vp}(\chi\bn)+\bar{\rho}\omega\times \pt^{\vp}(\chi\bn).
\eeq
By \eqref{eqofv2}, $v$ satisfies the equation:
\beqs
\bar{\rho}\pt^{\vp}v-\mu\Delta^{\vp} v=-(f+\nabla^{\vp}q+\bar{\rho}[\bbp,\pt^{\vp}]u)-\nabla^{\vp}\pi=\colon H,
\eeqs
which gives:
\beqs
(\bar{\rho}\pt^{\vp}-\mu\Delta^{\vp})(\p_j v\cdot\bN)=L_j
\eeqs
with 
$$L_j=[\p_j H +\p_j\big(\f{\pt\vp}{\p_z\vp}\big)\p_z v-\mu[\p_j,\Delta^{\vp}]v]\cdot\bN+\bar{\rho}\p_j v\cdot \p_t^{\vp}\bN-2\mu\nabla^{\vp}\p_j v\cdot\nabla^{\vp}\bN-\mu\Delta^{\vp}\bN\cdot\p_j v.$$
Denote $\varsigma=2(\p_1v\cdot\bN,\p_2v\cdot\bN,0)^{t}, L=(L_1,L_2,0)^{t}.$
Therefore, by recalling the definition of
projection $\Pi=\text{Id}_3-\bn\otimes\bn,$
it holds that:
\beqs
(\bar{\rho}\pt^{\vp}-\mu\Delta^{\vp})(\chi\Pi\varsigma)=G_{\chi}^{\varsigma}
\eeqs
where
\beq\label{Gchitheta}
G_{\chi}^{\varsigma}=2\chi\Pi L+\bar{\rho}\chi[\pt,\Pi]\varsigma
-\bar{\rho}\chi\f{\pt\vp}{\p_z\vp}[\p_z,\Pi]\varsigma+\mu\chi[\Pi,\Delta^{\vp}]\varsigma+\bar{\rho}[\pt^{\vp},\chi]\Pi\varsigma+\mu[\chi,\Delta^{\vp}]\Pi\varsigma.
\eeq
We thus finally find that:
\beq
(\bar{\rho}\pt^{\vp}-\mu\Delta^{\vp})(\chi\omega_n)=G_{\chi}^{\varsigma}+G_{\chi}^{\omega}.
\eeq
For the sake of notational simplicity, we
denote
$\zeta=\chi\omega_{\bn}, G_{\chi}^{\zeta}=G_{\chi}^{\varsigma}+G_{\chi}^{\omega}.$
Consider $$\tilde{\zeta}(t,x)=\zeta(t,\Phi_{t}^{-1}\circ \tilde{\Phi}_t(x)),$$
then $\tilde{\zeta}:[0,T]\times%\mathbb{R}_{-}^3
\mS_{\kappa}\rightarrow \mathbb{R}$ solves the system:
\beqs
\left\{
\begin{array}{l}
    (\bar{\rho}\pt-\mu\Delta_g)\tilde{\zeta}=\widetilde{G_{\chi}^{\zeta}}+
   \bar{\rho} (\D\tilde{\Phi}_t)^{-1}\pt\tilde{\Phi}_t\cdot\nabla\tilde{\zeta}, \\
  \tilde{\zeta}|_{t=0}=\zeta(\Phi_0^{-1}\circ \tilde{\Phi}_0),\\
  %\chi(\omega\times\bn)
 \tilde{\zeta}|_{z=0}=-2\Pi(\p_1 \nabla^{\vp}\Psi\cdot \textbf{n},\p_2 \nabla^{\vp}\Psi\cdot \textbf{n},0)^{t}|_{z=0}.
\end{array}
\right.
\eeqs
where $\Delta_g$ is defined in \eqref{Riemannmetric}. Since $\tz$ vanishes in the vicinity of  $\{z=-\kappa\},$ we can extend it by zero to the whole lower half space $\mathbb{R}_{-}^3.$
Denote 
$$\|f\|_{L_t^{p}H_{co}^{k}(\mathbb{R}_{-}^3)}=\sum_{|\alpha|\leq k}\|Z^{\alpha}f\|_{L_t^p L^2(\mathbb{R}_{-}^3)}.$$
By Proposition \ref{propequiva-norm}, we have:
\beqs
\|\zeta\|_{L_t^{\infty}H_{co}^{m-4}(\mS)}\lesssim \|\zeta\|_{L_t^{\infty}H_{co}^{m-4}(\mathbb{R}_{-}^3)}\lesssim \Lambda(\f{1}{c_0},|h|_{m-2,\infty,t}) \|\tz\|_{L_t^{\infty}H_{co}^{m-4}(\mathbb{R}_{-}^3)},
\eeqs
\beqs
\begin{aligned}
&\|\ep^{\f{1}{2}}\pt\zeta\|_{L_t^{\infty}H_{co}^{m-4}(\mS)}\lesssim \|\ep^{\f{1}{2}}\pt\zeta\|_{L_t^{\infty}H_{co}^{m-4}(\mathbb{R}_{-}^3)}\\
&\lesssim \Lambda(\f{1}{c_0},|h|_{m-2,\infty,t})(\|\ep^{\f{1}{2}}\pt\tz\|_{L_t^{\infty}H_{co}^{m-4}(\mathbb{R}_{-}^3)}+\ep^{\f{1}{2}}\|\tz\|_{L_t^{\infty}H_{co}^{m-3}(\mathbb{R}_{-}^3)})+\ep^{\f{1}{2}}\lae,
\end{aligned}
\eeqs
\begin{align*}
\ep^{\f{1}{2}}\|\tz\|_{L_t^{\infty}H_{co}^{m-3}(\mathbb{R}_{-}^3)})&\lesssim  \Lambda(\f{1}{c_0},|h|_{m-2,\infty,t})\|\ep^{\f{1}{2}}\zeta\|_{L_t^{\infty}H_{co}^{m-3}(\mS)}\\
&\lesssim   \Lambda(\f{1}{c_0},|h|_{m-2,\infty,t})\|\ep^{\f{1}{2}}\nabla u\|_{L_t^{\infty}H_{co}^{m-3}(\mS)}
+(T+\ep)^{\f{1}{2}}\lae.
\end{align*}
Therefore, \eqref{omegatimesbn} follows from the estimate:
\beq\label{esoftz}
%\|\tz\|_{L_t^{\infty}H_{co}^{m-3}(\mathbb{R}_{-}^3)}^2+
\|(\tz,\ep^{\f{1}{2}}\pt\tilde{\zeta})\|_{L_t^{\infty}H_{co}^{m-4}(\mS)}\lesssim Y^2_{m}(0)+(T+\ep)
^{\f{1}{4}}\lae,
\eeq
which is the consequence of Lemma \ref{lemtz1} and Lemma \ref{lemtz2}.
\end{proof}
\begin{prop}\label{propequiva-norm}
Suppose that $\cT_t:\mathbb{R}_{-}^3\rightarrow\mathbb{R}_{-}^3$ is a $C^{m-3}$ diffeomorphism with $\cT_t(y,0)=y, \forall y\in\mR^2.$
For any function $f(t,\cdot)$ which supported on $\mS_{\kappa},$ and for $p=2,+\infty,$ it holds that
\beq\label{Linftyeq-free}
\il f(s,\cT_s
\cdot)\il_{k,\infty,t}\lesssim \Lambda(\il(\cT ,\p_z\cT)\il_{k,\infty,t}) \il f\il_{k,\infty,t},
\eeq
\beq\label{equiva-norm0}
\|f(s,\cT_s
\cdot)\|_{L_t^{p}H_{co}^k(\mathbb{R}_{-}^3)}\lesssim \Lambda(\il(\cT ,\p_z\cT)\il_{k,\infty,t}) \|f\|_{L_t^{p}H_{co}^k(\mathbb{R}_{-}^3)},
\eeq
\beq\label{equiva-norm1}
\begin{aligned}
&\|\ep^{\f{1}{2}}\partial_s  [f(s,\cT_s
\cdot)]\|_{L_t^{p}H_{co}^k(\mathbb{R}_{-}^3)}\lesssim \Lambda(\il (\cT,\p_z\cT)\il_{k,\infty,t}
) \|\ep^{\f{1}{2}}(\pt,\cZ) f\|_{L_t^{p}H_{co}^k(\mathbb{R}_{-}^3)}\\
&+\ep^{\f{1}{2}}\Lambda(\il \pt (\cT,\p_z\cT)\il_{k-1,\infty,t})\|f\|_{L_t^pH_{co}^k(\mathbb{R}_{-}^3)}+\il \cZ \tz\il_{0,\infty,t}\Lambda(\|\pt\p_z\cT\|_{L_t^{\infty}H_{co}^k})
\end{aligned}
\eeq
where we denote $\cZ=(\p_{y_1},\p_{y_2},Z_3)$ the spatial conormal derivatives.
\end{prop}
\begin{rmk}\label{rmk-changevariable}
Since $\Phi_t^{-1}\circ \tilde{\Phi}_t=\Phi_t^{-1}(t, y_1+z\bn_1^{b,1},y_2+z\bn_2^{b,1},h+z\bn_3^{b,1}),$ and $|D\Phi_t^{-1}|\lesssim |h|_{1,\infty,t},$ we have that:
\beqs
\il( \Phi_t^{-1}\circ \tilde{\Phi}_t,\p_z(\Phi_t^{-1}\circ \tilde{\Phi}_t))\il_{k,\infty,t}\lesssim \Lambda(\f{1}{c_0},|h|_{k+1,\infty,t}).
\eeqs
\end{rmk}
\begin{proof}
The proof of this lemma just follows from the Leibniz rule, we thus omit the proof. 
\end{proof}

As explained before, to show \eqref{esoftz}, we write $\tilde{\zeta}=\tilde{\zeta}_1+\tilde{\zeta}_2,$ where $\tilde{\zeta}_1,\tilde{\zeta}_2$ satisfy the following two systems:
\beq\label{tz1}
\left\{
\begin{array}{l}
   (\bar{\rho}\pt-\mu\p_z^2 )\tilde{\zeta}_1=0,\quad (t,x)\in [0,T]\times\mathbb{R}_{-}^3,  \\[5pt]
     \tilde{\zeta}_1|_{t=0}=0,\quad \tilde{\zeta}_1|_{z=0}=\tz|_{z=0}=-2\Pi(\p_1 \nabla^{\vp}\Psi\cdot \textbf{n},\p_2 \nabla^{\vp}\Psi\cdot \textbf{n},0)^{t}|_{z=0}.
\end{array}
\right.
\eeq
\beq\label{tz2}
\left\{
\begin{array}{l}
   (\bar{\rho}\pt-\mu\Delta_g)\tilde{\zeta}_2=\widetilde{G_{\chi}^{\zeta}}+\bar{\rho}\pt\tilde{\Phi}_t(\D\tilde{\Phi}_t)^{-1}\nabla\tilde{\zeta}+\f{1}{2}\mu\p_z(\ln|g|)\p_z\tilde{\zeta}_1-\mu\Delta_{\tilde{g}}\tilde{\zeta}_1,  \\[5pt]
     \tilde{\zeta}_2|_{t=0}= \tilde{\zeta}|_{t=0},\quad \tilde{\zeta}_2|_{z=0}=0.
\end{array}
\right.
\eeq
%For the first equation, we can use the explicit formulae of the heat equation in the half-space.
\begin{lem}\label{lemtz1}
Under the assumption \eqref{preassumption}, it holds that, for any $0<t\leq T,$
\beq\label{tz1-0}
\|(\tz_1,\ep^{\f{1}{2}}\pt \tz_1)\|_{L_t^{\infty}H_{co}^{m-4}(\mathbb{R}_{-}^3)}+\|(\tz_1,\ep^{\f{1}{2}}\pt\tz_1)\|_{L_t^{2}H_{co}^{m-3}(\mathbb{R}_{-}^3)}\lesssim T^{\f{1}{4}}\lae,
\eeq
\beq\label{tz1-00}
\|(\text{Id},\ep^{\f{1}{2}}\pt,\p_y,Z_3)\tz_1\|_{L^{\infty}([0,T]\times\mR_{-}^3)}
\lesssim \Lambda\big(\f{1}{c_0},\cN_{m,T}\big).
\eeq
\end{lem}
\begin{proof}
 We present the estimates for $\ep^{\f{1}{2}}\pt\tz_1$ appearing in the inequality \eqref{tz1-0}, the estimates for $\tz_1$ is similar and easier.
Let $\gamma=(\gamma',\gamma_3)$ a multi-index such that $|\gamma|\leq m-4,$
$Z^{\gamma}=Z_{tan}^{\gamma'}Z_3^{\gamma_3}$ where $ Z_{tan}^{\gamma'}=
Z_0^{\gamma_0}Z_1^{\gamma_1}Z_2^{\gamma_2}.$
Taking $Z_{tan}^{\gamma'}$ on the equation of \eqref{tz1}, we get:
\beqs
\left\{
\begin{array}{l}
   (\bar{\rho}\pt-\mu\p_z^2 )(Z_{tan}^{\gamma'}\pt \tilde{\zeta}_1)=0,\quad (t,x)\in [0,T]\times\mathbb{R}_{-}^3,  \\
Z_{tan}^{\gamma'}\pt \tilde{\zeta}_1|_{t=0}=0,\quad Z_{tan}^{\gamma'}\pt\tilde{\zeta}_1|_{z=0}=Z_{tan}^{\gamma'}\pt\tz|_{z=0}.
\end{array}
\right.
\eeqs
By the explicit formulae of the heat equation on the half-line, we have that:
\beq\label{tz1-1}
\ep^{\f{1}{2}}Z^{\gamma}\pt \tz_1(t,y,z)=2\tilde{\mu}\ep^{\f{1}{2}}\int_0^t \f{1}{(4\pi\tilde{\mu}(t-s))^{\f{1}{2}}} Z_3^{\gamma_3}\p_z\big(e^{-\f{z^2}{4\tilde{\mu}(t-s)}}\big) Z_{tan}^{\gamma'}\pt\tz|_{z=0}(s,y)\,\d s
\eeq
where $\tilde{\mu}=\mu/{\bar{\rho}}.$ 
To continue, we need the following estimate whose proof is elementary and is left for the reader:
%which is essentially proved in Lemma 7.7 of \cite{masmoudi2021uniform}. 
for any $l\geq 0$
\beq\label{claim}
\|Z_3^{l}\p_z(e^{-\f{z^2}{4\tilde{\mu}(t-s)}})\|_{L_z^2(0,\infty)}\lesssim  (t-s)^{-\f{1}{4}}.
\eeq
Now, taking the $L_z^2L_y^2$ norm of \eqref{tz1-1} and applying \eqref{claim},
we find that for any $0<t\leq T,$
\beqs
\ep^{\f{1}{2}}\|Z^{\gamma}\pt\tz_1\|_{L_t^{\infty}L^2(\mathbb{R}_{+}^3)}\lesssim T^{\f{1}{4}}|\ep^{\f{1}{2}}\pt \tz|_{z=0}|_{L_t^{\infty}\tilde{H}^{m-4}}.
\eeqs
By the trace inequality \eqref{trace} and the estimate \eqref{psiLinfty}, we get that:
\begin{equation*}
\begin{aligned}
|\ep^{\f{1}{2}}\pt \tz_1|_{z=0}|_{L_t^{\infty}\tilde{H}^{m-4}}&\lesssim |(\nabla^{\vp}\Psi,\ep^{\f{1}{2}}\pt\nabla^{\vp}\Psi)|_{L_t^{\infty}\tilde{H}^{m-3}}\Lambda\big(\f{1}{c_0}, |(h,\ep^{\f{1}{2}}\pt h)|_{m-3,\infty,t}\big)\\
&\lesssim\|\ep^{\f{1}{2}}\pt(\nabla^{\vp}\Psi, \nabla\nabla^{\vp}\Psi),(\nabla^{\vp}\Psi,\nabla\nabla^{\vp}\Psi)\|_{L_t^{\infty}H_{co}^{m-3}(\mS)}\Lambda\big(\f{1}{c_0}, |(h,\ep^{\f{1}{2}}\pt h)|_{m-3,\infty,t}\big)\\
&\lesssim\lae.
\end{aligned}
\end{equation*}
Combined the previous two inequalities, one finds:
\beqs
\|\ep^{\f{1}{2}}\pt\tz_1\|_{L_t^{\infty}H_{co}^{m-4}(\mR_{-}^3)}\lesssim T^{\f{1}{4}}\lae.
\eeqs
Similarly, by employing Young's inequality %(after extending $\tilde{\zeta}|_{z=0}$ by zero to $t>T$ and $t<0$) 
and the estimate \eqref{sec-normal-Psi}, 
 we obtain that:
 \begin{equation*}
\begin{aligned}
&\|\ep^{\f{1}{2}}\pt\tz_1\|_{L_t^{2}{H}_{co}^{m-3}(\mathbb{R}_{-}^3)}\lesssim T^{\f{1}{4}}|\ep^{\f{1}{2}}\pt \tz|_{z=0}|_{L_t^{2}\tilde{H}^{m-3}}\\
&\lesim T^{\f{1}{4}}\Lambda\big(\f{1}{c_0}, |(h,\ep \pt h|_{m-2,\infty,t}\big)\cdot\|\ep^{\f{1}{2}}\pt(\nabla^{\vp}\Psi, \nabla\nabla^{\vp}\Psi),\ep^{-\f{1}{2}}(\nabla^{\vp}\Psi, \nabla\nabla^{\vp}\Psi)\|_{L_t^{2}H_{co}^{m-2}(\mS)}\\
&\lesssim T^{\f{1}{4}}\lae.
\end{aligned}
 \end{equation*}
The above inequality then
leads to \eqref{tz1-0}. We now show the $L_{t,x}^{\infty}$ estimate \eqref{tz1-00}. It results from \eqref{tz1-1} that:
for any $t>0,z>0, j=1,2, Z^0=\text{Id}, Z^1=(\ep^{\f{1}{2}}\pt,\p_y),$
\beq\label{tz1Linfty1}
\begin{aligned}
\|Z_{tan}^j\tz_1(t,\cdot,z)\|_{L_y^{\infty}}&\leq \big|Z_{tan}^j\tz_1|_{z=0}\big|_{L_t^{\infty}L_y^{\infty}}\int_{0}^t \sqrt{2\pi^{-1}\tilde{\mu}^2} z^{-2}\big(\f{z^2}{2\tilde{\mu}(t-s)}\big)^{\f{3}{2}}
e^{-\f{z^2}{4\tilde{\mu}(t-s)}}\d s\\
&\leq C(\tilde{\mu})\big|Z_{tan}^j\tz_1|_{z=0}\big|_{L_t^{\infty}L_y^{\infty}}\lesssim \Lambda(\ep^{-\f{1}{2}}\il\nabla^{\vp}\Psi\il_{2,\infty,t}+|h|_{2,\infty,t}+|\ep^{\f{1}{2}}\pt h|_{1,\infty,t})
\end{aligned}
\eeq
where $C(\tilde{\mu})$ is a constant that depends only on $\tilde{\mu}.$
In the same fashion, we have
\beq\label{tz1Linfty2}
\begin{aligned}
\|Z_{3}\tz_1(t,\cdot,z)\|_{L_y^{\infty}}&\leq \big(\sqrt{2\pi^{-1}\tilde{\mu}^2} \phi(z)z^{-1}\int_{0}^t z^{-2}
P\big(\f{z}{\sqrt{2\tilde{\mu} s}}\big)\d s\big)
 \big|\tz_1|_{z=0}\big|_{L_t^{\infty}L_y^{\infty}}\\
&\leq C(\tilde{\mu})\big|\tz_1|_{z=0}\big|_{L_t^{\infty}L_y^{\infty}}\lesssim \Lambda(\il\nabla^{\vp}\Psi\il_{1,\infty,t}+|h|_{1,\infty,t}),
\end{aligned}
\eeq
where $P(z)=|(1-z^2)|z^3 e^{-z^2}.$
Note that $\phi(z)z^{-1}=(1+z)/(2-z)^2$ is uniformly bounded for $z>0.$
The proof of \eqref{tz1-00} is now finished.
\end{proof}
\begin{lem}\label{lemtz2}
Suppose that \eqref{preassumption} holds, for any $0<t\leq T,$ we have the following estimates:
\beq\label{EIoftz2}
\|\tz_2\|_{L_t^{\infty}H_{co}^{m-4}(\mR_{-}^3)}^2+\|(\nabla \tz_2,\ep^{\f{1}{2}}\pt\nabla\tz_2)\|_{L_t^{2}H_{co}^{m-4}(\mR_{-}^3)}^2\lesssim Y_{m}^2(0)+
T^{\f{1}{4}}\lae,
\eeq
\beq\label{EIoftz2-1}
\|\ep^{\f{1}{2}}\pt\tz_2\|_{L_t^{\infty}H_{co}^{m-4}(\mR_{-}^3)}^2+\|\ep^{\f{1}{2}}\pt\nabla\tz_2\|_{L_t^{2}H_{co}^{m-4}(\mR_{-}^3)}^2\lesssim \Lambda\big(Y_m^2(0)+\widetilde{\cE}_{m,t}^2\big)Y_{m}^2(0)+T^{\f{1}{4}}\lae.
\eeq
\end{lem}
\begin{proof}
Again, we only give the details for the estimate of $\ep^{\f{1}{2}}\pt\tz_2,$ the one of $\tz_2$ is similar and slightly easier to deal with.
Let $\beta$ be a multi-index such that $|\beta|=k\leq m-4.$ 
Since $$\Delta_{\tilde{g}}=\p_i(\tilde{g}^{ij}\p_j\cdot)-\p_i(|\tilde{g}|^{-\f{1}{2}})\tilde{g}^{ij}|\tilde{g}|^{\f{1}{2}}\p_{i}f,$$ 
to avoid losing derivatives on the surface, it is convenient to rewrite the system \eqref{tz2} as:
\beq\label{tz2-0}
\left\{
\begin{array}{l}
   \big(\bar{\rho}\pt-\mu\p_z^2-\mu\p_i(\tilde{g}^{ij}\p_j \cdot)\big)\tilde{\zeta}_2=F_{\chi}^{\tz},  \\[5pt]
     \tilde{\zeta}_2|_{t=0}= \tilde{\zeta}|_{t=0},\quad \tilde{\zeta}_2|_{z=0}=0,
\end{array}
\right.
\eeq
where 
\begin{align*}
  F_{\chi}^{\tz}&=\widetilde{G_{\chi}^{\zeta}}-\bar{\rho}\pt\tilde{\Phi}_t(\D\tilde{\Phi}_t)^{-1}\nabla\tilde{\zeta}+\f{1}{2}\mu\p_z(\ln|g|)\p_z\tilde{\zeta}+\mu\p_i(\ln|g|)\tilde{g^{ij}}\p_j\tilde{\zeta}+\mu \p_i(\tilde{g}^{ij}\p_{j}\tz_1). 
\end{align*}
Note that we have used the summation convention for $i,j=1,2.$
Applying $Z^{\beta}$ on the equation \eqref{tz2-0}, we get that:
\beqs
\ep^{\f{1}{2}}\big(\bar{\rho}\pt-\mu\p_z^2-\mu\p_i(\tilde{g}^{ij}\p_j)\big)(Z^{\beta}\pt\tilde{\zeta}_2)=Z^{\beta}\ep^{\f{1}{2}}\pt F_{\chi}^{\tz}+\mu[Z^{\beta}\ep^{\f{1}{2}}\pt,\p_z^2]\tz+\mu\p_i[Z^{\beta}\ep^{\f{1}{2}}\pt,\tilde{g}^{ij}]\tz,
\eeqs
from which we get the energy inequality:
\beq\label{tz2-EI}
\begin{aligned}
&\bar{\rho}\ep\|Z^{\beta}\pt\tz_2(t)\|_{L^2(\mR_{-}^3)}^2+\mu\ep\|\p_z Z^{\beta}\pt\tz_2\|_{L_t^2L^2(\mR_{-}^3)}^2+\mu\int_0^t\int_{\mR_{-}^3} \tilde{g}_{ij}\p_iZ^{\beta}\pt\tz_2\cdot \p_j Z^{\beta}\pt\tz_2\,\d x\d s\\
&\leq \bar{\rho}\ep\|Z^{\beta}\pt\tz(0)\|_{L^2(\mR_{-}^3)}^2+\mu\ep \bigg|\int_0^t\int_{\mR_{-}^3} [Z^{\beta}\pt,\p_z^2]\tz_2\cdot Z^{\beta}\pt\tz_2\,\d x\d s\bigg|\\
&+\mu\ep \bigg|\int_0^t\int_{\mR_{-}^3} [Z^{\beta},\tilde{g}^{ij}]\p_j\tz_2 \p_i Z^{\beta}\tz_2\,\d x\d s\bigg|+\ep\bigg|\int_0^t\int_{\mR_{-}^3} Z^{\beta}F_{\chi}^{\tz}\cdot Z^{\beta}\tz_2\,\d x\d s\bigg|.
\end{aligned}
\eeq
As long as $\kappa$ is chosen small enough, the matrix $\left(\begin{array}{cc}
    \tilde{g}_{11} & \tilde{g}_{12}\\
    \tilde{g}_{21}  & \tilde{g}_{22}
\end{array}\right)$
is positive definite, so that the last two terms in the first line of \eqref{tz2-EI} control $C_{\kappa}\|\ep^{\f{1}{2}}\nabla  Z^{\beta}\pt\tz_2\|_{L_t^2L^2(\mR_{-}^3)}^2.$ In the sequel, to lighten the notation load and without much ambiguity, we shall denote $$\|\tilde{f}\|_{L_t^{p}H_{co}^{k}}=\|\tilde{f}\|_{L_t^{p}H_{co}^{k}(
{\mR}_{-}^3)},\, \|f\|_{L_t^{p}H_{co}^{k}}=\|f\|_{L_t^{p}H_{co}^{k}(\mS)},\,\qquad p=2,+\infty.$$
%For simplicity, we would just use $L_t^{p}H_{co}^m$ to denote $L_t^{p}H_{co}^m(\mathbb{R}_{-}^3).$ Since all the functions involved in the following proof are supported in $\mS,$ it is also equivalent to $L_t^pH_{co}^m(\mS).$
We begin now to estimate the last three terms of the right hand side of \eqref{tz2-EI}. 
At first, we have up to some smooth functions depending on $\phi,$
\beqs
[Z^{\beta},\p_z^2]=\sum_{|\tilde{\beta}|\leq |\beta|-1}*_{\beta,\tilde{\beta}}\p_z^2 Z^{\tilde{\beta}}+\sum_{|\gamma|\leq |\beta|-1}*_{\beta,\gamma}\p_z Z^{\gamma},
\eeqs
Therefore, thanks to integration by parts and Young's inequality, we write:
\beq\label{tz2-1}
\mu\ep \big|\int_0^t\int_{\mR_{-}^3} [Z^{\beta},\p_z^2]\pt\tz\cdot Z^{\beta}\pt\tz_2\,\d x\d s\big|\leq \delta\ep \|\p_z Z^{\beta}\pt\tz\|_{L_t^2L^2(\mR_{-}^3)}^2+C_{\delta}\ep(\|\p_z\pt\tz_2\|_{L_t^2H_{co}^{k-1}}^2+\|\pt \tz_2\|_{L_t^2H_{co}^{m-5}}^2).
\eeq
Similarly, by Young's inequality, we have:
\beq\label{tz2-2}
\begin{aligned}
&\mu \ep \big|\int_0^t\int_{\mR_{-}^3} [Z^{\beta}\pt,\tilde{g}^{ij}]\p_j\tz_2 \cdot\p_i Z^{\beta}\pt\tz_2\,\d x\d s\big|\leq\delta\ep\|\nabla Z^{\beta}\pt\tz_2\|_{L_t^2L^2(\mR_{-}^3)}^2\\
&+C_{\delta}\Lambda\big(\f{1}{c_0},|(h,\ep^{\f{1}{2}}\pt h)|_{m-2,\infty,t}+|\pt h|_{2,\infty,t}\big)%+|\pt h|_{m-3,\infty,t})
(\|(\tz_2,\ep^{\f{1}{2}}\pt\tz_2)\|_{\hco^{m-4}}^2+\ep\|\tz_2\|_{\hco^{m-3}}^2).
\end{aligned}
\eeq
%We remark that by the estimate \eqref{tz1-00},
%\beqs\begin{aligned}\il(\tz_2,Z\tz_2)\il_{L^{\infty}([0,T]\times\mR_{-}^3)}&\lesssim\il\zeta\il_{1,\infty,t}+%\il\tz_1\il_{1,\infty,t}\il(\tz_1,Z\tz_1)\il_{L^{\infty}([0,T]\times\mR_{-}^3)}\\&\lesssim\Lambda(|h|_{2,\infty,t}+\il\nabla^{\vp}\Psi\il_{2,\infty,t}+\il \nabla u\il_{1,\infty,t})\\&\lesssim \Lambda\big(\f{1}{c_0},\cN_{m,T}\big).\end{aligned}\eeqs
We are now in position to control the last term in \eqref{tz2-EI}.
We split it into several terms:
\beqs
\ep\int_0^t\int_{\mR_{-}^3} Z^{\beta}\pt F_{\chi}^{\tz}\cdot Z^{\beta}\pt\tz_2\,\d x\d s=\colon \cJ_1+\cJ_2+\cJ_3+\cJ_4.
\eeqs
with 
\beqs
\begin{aligned}
&\cJ_1=\ep\int_0^t\int_{\mR_{-}^3} Z^{\beta}\pt \widetilde{G_{\chi}^{\zeta}}\cdot Z^{\beta}\pt\tz_2\,\d x\d s,\quad \cJ_2=\bar{\rho}\ep\int_0^t\int_{\mR_{-}^3} Z^{\beta}\pt \big((\D\tilde{\Phi}_s)^{-1}\p_s\tilde{\Phi}_s\cdot\nabla\tilde{\zeta}\big)\cdot Z^{\beta}\pt\tz_2\,\d x\d s,\\
&\cJ_3=\mu\ep\int_0^t\int_{\mR_{-}^3} Z^{\beta}\pt \p_i(\tilde{g}^{ij}\p_{j}\tz_1)\cdot Z^{\beta}\pt\tz_2\,\d x\d s,\quad\cJ_4= \f{1}{2}\mu\ep\int_0^t\int_{\mR_{-}^3} Z^{\beta}\pt\big(
\p_z(\ln|g|)\p_z\tilde{\zeta}\big)\cdot Z^{\beta}\pt\tz_2 \,\d x\d s ,\\
&\cJ_5=\f{1}{2}\mu\ep
\int_0^t\int_{\mR_{-}^3} Z^{\beta}\pt\big(
\p_i(\ln|g|){\tilde{g}^{ij}}\p_j\tilde{\zeta}\big)\cdot Z^{\beta}\pt\tz_2\,\d x\d s.\\
\end{aligned}
\eeqs
To estimate $\cJ_2,$ let us split it into two terms $\cJ_2=\cJ_{21}+\cJ_{22}:$
\begin{align*}
&\cJ_{21}=\bar{\rho}\ep\int_0^t\int_{\mR_{-}^3} Z^{\beta}\pt\bigg(\div\big((\D\tilde{\Phi}_s)^{-1}\p_s\tilde{\Phi}_s\big)\tz\bigg)Z^{\beta}\pt\tz_2\,\d x\d s,\\
&\cJ_{22}=\bar{\rho}\ep\int_0^t\int_{\mR_{-}^3}Z^{\beta}\pt\p_l\bigg(\big((\D\tilde{\Phi}_s)^{-1}\p_s\tilde{\Phi}_s\big)_l\tz\bigg)Z^{\beta}\pt\tz_2\,\d x\d s.
\end{align*}
We emphasize that since there is no gain of the regularity of $\tilde{\Phi}$ from that of $h$ (roughly speaking, one needs $k+1$ derivatives of $h$ to control $k$ derivatives of $\tilde{\Phi}$), careful attention needs to be paid to the regularity of the surface in the following computations.
To estimate $\cJ_{21},$ in order not to lose regularity on the surface, we consider two cases. If $Z^{\beta}$ contains at least one spatial conormal derivative, we  integrate by parts in space, and then use Young's inequality to get:
\beqs
\begin{aligned}
\cJ_{21}&\leq \delta\ep\|\nabla Z^{\beta}\pt\tz_2\|_{L_t^2L^2(\mR_{-}^3)}^2+\big(\|(\tz,\ep^{\f{1}{2}}\pt\tz)\|_{\hco^{m-4}}^2+|\ep^{\f{1}{2}}\pt^2 h|_{\htlde^{m-3}}^2\big)\cdot\\
&\qquad\qquad\qquad\qquad \Lambda\big(\il\tz\il_{1,\infty,t}+|(h,\ep^{\f{1}{2}}\pt h)|_{m-2,\infty,t}+|\pt h|_{m-3,\infty,t}+|\ep^{\f{1}{2}}\pt^2 h|_{m-5,\infty,t}\big).
%+\|\tz\|_{\hco^{m-5}}^2).
%&\quad +\Lambda(\f{1}{c_0})(\int_0^t|\ep^{\f{1}{2}}\pt^2 h(s)|_{m-4,\infty}\d s)^{\f{1}{2}}\|\tz\|_{L_t^{\infty}H_{co}^1}^2\\
\end{aligned}
\eeqs
Moreover, we have by Proposition \ref{propequiva-norm} and estimate \eqref{tz1-0} that for $l=3,4$ %$(p,k)=(2,4)$ and 
\beq\label{equiva-norm}
\begin{aligned}
&\|(\tz_2,\ep^{\f{1}{2}}\pt\tz_2)\|_{L_t^2H_{co}^{m-l}}\leq \|(\tz,\ep^{\f{1}{2}}\pt\tz),(\tz_1,\ep^{\f{1}{2}}\pt\tz_1)\|_{L_t^2H_{co}^{m-l}}\\
&\lesim \|(\nabla u,\ep^{\f{1}{2}}\pt\nabla u)\|_{L_t^2H_{co}^{m-l}}\Lambda\big(\f{1}{c_0},|(h,\ep^{\f{1}{2}}\pt h)|_{m-l+1,\infty,t}\big) + T^{\f{1}{4}}\lae
\\&\lesssim \left\{\begin{array}{cc}
    T^{\f{1}{4}}\lae  & \text{ if } l=4, \\
     \lesssim \lae & \text{ if } l=3,
\end{array}
\right.
\end{aligned}
\eeq
and by  \eqref{tz1-00} that:
\beq\label{tz2infty}
\begin{aligned}
 &\|(\text{Id},\ep^{\f{1}{2}}\pt,\p_y,Z_3)\tz\|_{L^{\infty}([0,T]\times\mR_{-}^3)}\\
 &\lesssim
 \il (\text{Id},\ep^{\f{1}{2}}\pt,\p_y,Z_3)\zeta\il_{0,\infty,t}\Lambda\big(|\ep^{\f{1}{2}}\pt h|_{2,\infty,t}+|h|_{3,\infty,t}\big)\lesssim \lca.
 %+\il(\tz_1,Z_0\tz_1)\il_{L^{\infty}([0,T]\times\mR_{-}^3)}\\&\lesssim\Lambda(|h|_{2,\infty,t}+\il\nabla^{\vp}\Psi\il_{2,\infty,t}+\il \nabla u\il_{*,1,\infty,t}+\il\ep^{\f{1}{2}}\p_z^2 u\il_{0,\infty,t})\\ &\lesssim \Lambda\big(\f{1}{c_0},\cN_{m,T}\big).
\end{aligned}
\eeq
Therefore, by combining \eqref{surface3}, we obtain that in this case,
\beq\label{J21-1}
\cJ_{21}\leq \delta\ep\|\nabla Z^{\beta}\tz_2\|_{L_t^2L^2(\mR_{-}^3)}^2 %+\Lambda(\f{1}{c_0})\big(\int_0^t|\ep^{\f{1}{2}}\pt^2 h(s)|_{m-4,\infty}\d s\big)^{\f{1}{2}}\|\tz_2\|_{L_t^{\infty}H_{co}^1}^2
+T^{\f{1}{2}}\lae.
\eeq
%Note that by 
%\beq\label{equiva-norm}
%\|\tz\|_{L_t^{\infty}H_{co}^{m-3}}\lesssim \|\zeta\|_{L_t^{\infty}H_{co}^{m-3}}\Lambda\big(\f{1}{c_0}, |h|_{m-2,\infty,t}\big)\lesssim \|\nabla u\|_{L_t^{\infty}H_{co}^{m-3}}\Lambda\big(\f{1}{c_0}, |h|_{m-2,\infty,t}\big).\\\eeq
If $Z^{\beta}=(\ep\pt)^k,$ $(k\leq m-4),$ thanks to
\eqref{surface3}, \eqref{tz1-0}, \eqref{equiva-norm}, \eqref{tz2infty},
we can control $\cJ_{21}$ as:
\beq\label{J21-2}
\begin{aligned}
\cJ_{21}&\lesssim \|\ep^{\f{1}{2}}\pt\tz_2\|_{\hco^{m-4}}\Lambda\big(\f{1}{c_0},\il(\tz,\ep\pt\tz)\il_{0,\infty,t}+\cG_{\infty,t}(h)\big)\cdot\\
&\qquad (|(\ep^{\f{1}{2}}\pt^2 h,\ep^{\f{3}{2}}\pt^3 h)|_{\htlde^{m-3}}+\|(\tz_2,\ep^{\f{1}{2}}\pt\tz_2)\|_{\hco^{m-4}})
\\
&\lesssim T^{\f{1}{4}}\lae, %\big(\|(\tz_1,\tz)\|_{L_t^{\infty}H_{co}^k}^2+|h|_{L_t^2H_y^3}^3+|\ep\pt h|_{\htlde^{k+2}}^2+|\ep\pt^2 h|_{\htlde^{k+1}}^2\big)\\&\lesssim T^{\f{1}{2}}\lae. 
\end{aligned}
\eeq
where $$\cG_{\infty,t}(h)\colon =|(h,\ep^{\f{1}{2}}\pt h)|_{m-2,\infty,t}+|\pt h|_{m-3,\infty,t}+
|(\ep^{\f{1}{2}}\pt^2 h, \ep^{\f{3}{2}}\pt^3 h)|_{m-5,\infty,t}.$$
Note that by \eqref{surface2}-\eqref{surface3}, and the Sobolev embedding $H^{\f{3}{2}}(\mR^2)\hookrightarrow L^{\infty}(\mR^2),$
$$\cG_{\infty,t}(h)\lesssim\lae.$$
Collecting \eqref{J21-1} and \eqref{J21-2}, we finally get that %(since $k\leq m-3$):
\beq\label{J21}
\cJ_{21}\leq \delta \|\nabla Z^{\beta}\tz_2\|_{L_t^2L^2(\mR_{-}^3)}^2+T^{\f{1}{4}}\lae.
\eeq
For $\cJ_{22},$ we write $Z^{\beta}\p_l=[Z^{\beta},\p_l]+\p_l Z^{\beta},$ we integrate by parts for the second term and follow similar arguments as in the estimate of $\cJ_{21}$ to get that:
\beqs
\begin{aligned}
\cJ_{22}%&\leq \delta \|\nabla Z^{\beta}\tz_2\|_{L_t^2L^2(\mR_{-}^3)}^2+T^{\f{1}{2}}\Lambda_{2,\infty,t} \big(\|\tz\|_{L_t^{\infty}H_{co}^{m-1}}^2+\cE_{m,T}^2\big)\\
&\leq  \delta \|\nabla Z^{\beta}\tz_2\|_{L_t^2L^2(\mR_{-}^3)}^2+ T^{\f{1}{4}}\lae.
\end{aligned}
\eeqs
Combined with \eqref{J21}, this yields:
\beq\label{cJ2}
\cJ_{2}\leq 2\delta \|\nabla Z^{\beta}\tz_2\|_{L_t^2L^2(\mR_{-}^3)}^2+T^{\f{1}{4}}\lae. %\big(\|\tz\|_{L_t^{\infty}H_{co}^k}^2+\cE_{m,T}^2\big).
\eeq
For $\cJ_3,$ we integrate by parts again and use the
Cauchy-Schwarz inequality to get: 
%Proposition \eqref{propequiva-norm} to obtain:
\beqs
\begin{aligned}
\cJ_3&\lesssim \|\ep^{\f{1}{2}}\pt (\tilde{g}^{ij}\p_j\tz_1)\|_{\hco^{m-4}}\|\ep^{\f{1}{2}}\pt\tz_2\|_{\hco^{m-3}}\\
&\lesssim \Lambda\big(\f{1}{c_0},|(h,\ep^{\f{1}{2}}\pt h)|_{m-2,\infty,t}\big)\|(\tz_1,\ep^{\f{1}{2}}\pt\tz_1)\|_{\hco^{m-3}}\|\ep^{\f{1}{2}}\pt\tz_2\|_{\hco^{m-3}}.
\end{aligned}
\eeqs
By estimates \eqref{tz1-0}, \eqref{equiva-norm}, we find that:
\beq
\cJ_{3}%\lesssim \lca \|\ep^{\f{1}{2}}\pt(\tz_1,\tz),(\tz_1,\tz)\|_{\hco^{m-3}}^2
\lesssim T^{\f{1}{4}}\lae.
\eeq
We begin now to estimate $\cJ_4.$ 
By writing $$\p_z(\ln|g|)\p_z\tz=-\p_z^2(\ln|g|)\tz+\p_z\big(\p_z(\ln|g|)\tz\big),$$
we can follow the similar computations as in the estimates of $\cJ_2$ to obtain
(it is indeed easier in the sense that $\p_z^2(\ln|g|),\p_z(\ln |g|)$ involve  only two derivatives of $h$ thanks to Remark \ref{rmk-changevariable})
\beq
\begin{aligned}
\cJ_4&\leq \delta\|\ep^{\f{1}{2}}\pt\nabla Z^{\beta}\tz_2\|_{L_t^2L^2(\mR_{-}^3)}^2+\Lambda\big(\f{1}{c_0},|(h,\ep^{\f{1}{2}}\pt h)|_{m-2,\infty,t}\big) \|(\tz,\ep^{\f{1}{2}}\pt\tz)\|_{\hco^{m-4}}^2\\
&\leq \delta\|\ep^{\f{1}{2}}\pt\nabla Z^{\beta}\tz_2\|_{L_t^2L^2(\mR_{-}^3)}^2+T^{\f{1}{2}}\lae.
\end{aligned}
\eeq
We proceed to estimate $\cJ_5.$ %can be  in a similar way as in the proof of \eqref{J21}.
If $Z^{\beta}=(\ep\pt)^k,$ we control it by %the product estimate \eqref{product} and
inequalities \eqref{surface3}, \eqref{tz1-0}, \eqref{equiva-norm}:
\beqs
\begin{aligned}
\cJ_5&\lesssim \|\ep^{\f{1}{2}}\pt \tz_2\|_{\hco^{m-4}}\Lambda\big(\f{1}{c_0},
 \il\ep^{\f{1}{2}}\tilde{\zeta}\il_{1,\infty,t}+|(h,\ep^{\f{1}{2}}\pt h)|_{m-2,\infty,t}  \big)\big( \|(\tz,\ep^{\f{1}{2}}\pt\tz)\|_{\hco^{m-3}}+|\ep\pt^2 h|_{L_t^2\tilde{H}^{m-2}}\big)
 \\
 &\lesssim T^{\f{1}{2}}\lae.
%&\leq \delta \|\nabla Z^{\beta}\tz_2\|_{L_t^2L^2(\mR_{-}^3)}^2+\Lambda_{2,\infty,t}\big( \|\tz\|_{\hco^{m-3}}^2+\|\tz_1\|_{\hco^{m-2}}^2+|h|_{L_t^2H_y^3}^2+|\ep\pt h |_{\htlde^{m-1}}^2 \big)\\&\leq \delta \|\nabla Z^{\beta}\tz_2\|_{L_t^2L^2(\mR_{-}^3)}^2+T^{\f{1}{2}}\lae.
\end{aligned}
\eeqs
If $Z^{\beta}$ contains at least one spatial conormal derivative, we integrate by parts in space and control it in a similar way as
$\cJ_3:$
\beqs
%\begin{aligned}\cJ_5&\leq \delta \|\ep^\nabla Z^{\beta}\zeta_2\|_{L_t^2L^2(\mR_{-}^3)}^2+%\lae\Lambda_{2,\infty,t}\big( \|\tz\|_{\hco^{m-3}}^2+| h|_{\htlde^{m-1}}^2 \big)\\&\leq \delta \|\nabla Z^{\beta}\zeta_2\|_{L_t^2L^2(\mR_{-}^3)}^2+T\lae.\end{aligned}
\begin{aligned}
\cJ_5&\lesssim \|\ep^{\f{1}{2}}\pt (\p_i(\ln|g|)\tilde{g}^{ij}\p_j\tz))\|_{\hco^{m-5}}\|\ep^{\f{1}{2}}\pt\tz_2\|_{\hco^{m-3}}\\
&\lesssim \Lambda\big(\f{1}{c_0},|(h,\ep^{\f{1}{2}}\pt h)|_{m-2,\infty,t}\big)\|(\tz,\ep^{\f{1}{2}}\pt\tz)\|_{\hco^{m-4}}\|\ep^{\f{1}{2}}\pt\tz_2\|_{\hco^{m-3}}\\
&\lesssim T^{\f{1}{2}}\lae.
\end{aligned}
\eeqs
To summarize, we get that:
\beq\label{cJ5}
\cJ_5\lesssim T^{\f{1}{2}}\lae. %\|\nabla Z^{\beta}\zeta_2\|_{L_t^2L^2(\mR_{-}^3)}^2+
\eeq
We are now left to control the term $\cJ_1.$ After checking every term  of  $G_{\chi}^{\omega}$ and $G_{\chi}^{\varsigma}$ defined in \eqref{Gchiomega} and 
\eqref{Gchitheta}, we find that the problematic terms that may lead to a loss of derivatives are the following:
\beqs
G_{\chi,1}^{\omega}=(u\cdot\nabla^{\vp}\omega)\times\chi\bN,
\quad G_{\chi,2}^{\omega}=\nabla^{\vp}\omega\times\nabla^{\vp}(\chi\bn), \quad G_{\chi,1}^{\varsigma}=\chi\Pi([\p_1,\Delta^{\vp}]v\cdot\bN, [\p_2,\Delta^{\vp}]v\cdot\bN,0)^{t}.
\eeqs
All the other terms can be controlled directly through the Cauchy-Schwarz inequality, the estimate \eqref{equiva-norm} and Proposition \ref{remainderL2}:
\beqs
\begin{aligned}
&\int_0^t\int_{\mR_{-}^3}
\ep^{\f{1}{2}} Z^{\beta}\pt\big(\widetilde{G_{\chi}^{\zeta}}-\widetilde{G_{\chi,1}^{\omega}}-\widetilde{G_{\chi,2}^{\omega}}-\widetilde{G_{\chi,1}^{\varsigma}}\big)\cdot \ep^{\f{1}{2}} Z^{\beta}\pt\tz_2
\,\d x\d s\\
&\lesssim \|\ep^{\f{1}{2}}\pt\tz_2\|_{\hco^{m-4}}\|\ep^{\f{1}{2}}\pt\big(\widetilde{G_{\chi}^{\zeta}}-\widetilde{G_{\chi,1}^{\omega}}-\widetilde{G_{\chi,2}^{\omega}}-\widetilde{G_{\chi,1}^{\varsigma}}\big)\|_{\hco^{m-4}}\\
&\lesssim\|\ep^{\f{1}{2}}\pt\tz_2\|_{\hco^{m-4}}
%\Lambda(|(h,\ep^{\f{1}{2}}\pt h)|_{m-2,\infty,t})(\|\ep^{\f{1}{2}}\pt\tz_1\|_{\hco^{m-3}}+\|\ep^{\f{1}{2}}\pt\zeta\|_{\hco^{m-4}(\mS)})
\|\ep^{\f{1}{2}}\pt(G_{\chi}^{\zeta}-G_{\chi,1}^{\omega}-G_{\chi,2}^{\omega}-G_{\chi,1}^{\varsigma})\|_{\hco^{m-4}(\mS)}\\
&\lesssim T^{\f{1}{4}}\lae. %\cE_{m,T}^2.
\end{aligned}
\eeqs
Note that by Proposition \ref{remainderL2},
\begin{equation*}
  \begin{aligned}
&\quad\|G_{\chi}^{\zeta}-G_{\chi,1}^{\omega}-G_{\chi,2}^{\omega}-G_{\chi,1}^{\varsigma}\|_{\hco^{m-3}(\mS)}\lesssim \lae.
%&\lesssim \Lambda(|h|_{3,\infty,t}+\il\nabla^{\vp}(u,\nabla^{\vp}\Psi)\il_{1,\infty,t})(|h|_{\htlde^{m-\f{1}{2}}}+\|(u,\nabla^{\vp}\Psi)\|_{E^{m-1},t})
  \end{aligned}
\end{equation*}
It remains to control the remaining three terms. We shall explain the estimates of the term
involving $G_{\chi,1}^{\omega}.$ 
Let us first rewrite: 
\beqs
\begin{aligned}
u\cdot\nabla^{\vp}\omega &=u_1\p_{y_1}\omega+u_2\p_{y_2}\omega+(u\cdot\bN)\cdot\p_z\omega=R_1-R_2.
\end{aligned}
\eeqs
where 
\beqs
R_1=\p_{y_1}(u_1\omega)+\p_{y_2}(u_2\omega)+\p_{z}\big((\f{u\cdot\bN}{\p_z\vp}) \omega\big),\quad R_2=\p_{y_1} u_1\cdot \omega+\p_{y_2} u_2\cdot \omega+
\p_z\big(\f{u\cdot\bN}{\p_z\vp}\big)\cdot\omega
\eeqs
Since $$\p_z\big(\f{u\cdot\bN}{\p_z\vp}\big)=\p_z^{\vp}u\cdot\bN+u\cdot\p_z(\f{\bN}{\p_z\vp})=\div^{\vp}u-\p_{y_1}u_1-\p_{y_2}u_2+u\cdot\p_z(\f{\bN}{\p_z\vp}),$$
there is no term like $\p_z u\cdot\p_z u$ appearing in $R_2,$
we thus can show by using similar arguments as in the proof of Proposition \ref{remainderL2} that:
$$\|\ep^{\f{1}{2}}\pt R_2\|_{\hco^{m-4}}\lesssim \lae,$$ which further yields:
\beq\label{TR1}
\int_0^t\int_{\mR_{-}^3}
\ep^{\f{1}{2}}Z^{\beta}\pt\widetilde{R_2}\cdot \ep^{\f{1}{2}}Z^{\beta}\pt\tz_2\,
\d x\d s\lesssim T^{\f{1}{2}}\lae.
\eeq
Next, by the change of variable, we have:
$$\widetilde{R_1}= (\D({\Phi_t}\circ\widetilde{\Phi_t}^{-1})^{-1})_{jl}\p_l\big[\widetilde {I_j(u)\omega}\big], \qquad \text{ where }I(u)=(u_1,u_2,\f{u\cdot\bN}{\p_z\vp}).$$
 Therefore, using a similar strategy as the one employed in the estimate of $\cJ_2,$ we find that:
\beqs
\int_0^t\int_{\mR_{-}^3}
\ep^{\f{1}{2}}\pt Z^{\beta}\widetilde{R_1}\cdot \ep^{\f{1}{2}}Z^{\beta}\pt\tz_2\,
\d x\d s\leq \delta\|\ep^{\f{1}{2}}\nabla Z^{\beta}\pt\tz_2\|_{L_t^2L^2(\mR_{-}^3)}^2+T^{\f{1}{4}}\lae,
\eeqs
which, together with \eqref{TR1}, leads to:
\beqs
\int_0^t\int_{\mR_{-}^3}
\ep^{\f{1}{2}}Z^{\beta}\pt\widetilde{G_{\chi,1}^{\omega}}\ep^{\f{1}{2}}Z^{\beta}\pt\tz_2\,\d x\d s\leq 
\delta\|\ep^{\f{1}{2}}\nabla Z^{\beta}\pt\tz_2\|_{L_t^2L^2(\mR_{-}^3)}^2+T^{\f{1}{4}}\lae.
\eeqs
Following similar arguments, one can also show that:
\beqs
\int_0^t\int_{\mR_{-}^3}\ep^{\f{1}{2}}\pt Z^{\beta}(\widetilde{G_{\chi,2}^{\omega}}+\widetilde{G_{\chi,1}^{\varsigma}})\cdot  \ep^{\f{1}{2}} Z^{\beta} \pt\tz_2\,
\d x\d s\leq \delta\|\nabla Z^{\beta}\ep^{\f{1}{2}}\pt\tz_2\|_{L_t^2L^2(\mR_{-}^3)}^2+T^{\f{1}{4}}\lae.
\eeqs
To summarize, we have obtained that:
\beq\label{cJ1}
\cJ_1\leq 2\delta\|\nabla Z^{\beta}\tz_2\|_{L_t^2L^2(\mR_{-}^3)}^2+T^{\f{1}{4}}\lae. %\cE_{m,T}^2.
\eeq
Gathering \eqref{cJ2}-\eqref{cJ5},\eqref{cJ1}
and using \eqref{equiva-norm},
we obtain:
\beq\label{tz2-3}
\big|\int_0^t\int_{\mR_{-}^3} Z^{\beta}F_{\chi}^{\tz}\cdot Z^{\beta}\tz_2\,\d x\d s\big|\leq 10\delta \|\nabla Z^{\beta}\tz_2\|_{L_t^2L^2(\mR_{-}^3)}^2+T^{\f{1}{4}}\lae.
\eeq
Inserting \eqref{tz2-1}, \eqref{tz2-2} and \eqref{tz2-3} into \eqref{tz2-EI}, we get by choosing $\delta$ small enough that
for any $0\leq k\leq m-4,$
\beq\label{finalomegat}
\begin{aligned}
\|\ep^{\f{1}{2}}\pt\tz_2\|_{L_t^{\infty}H_{co}^{k}}^2+\|\ep^{\f{1}{2}}\pt\nabla\tz_2\|_{\hco^{k}}^2&\lesssim Y_{m}^2(0)+\|\ep^{\f{1}{2}}\pt\nabla\tz_2\|_{\hco^{k-1}}^2+
T^{\f{1}{4}}\lae.
\end{aligned}
\eeq
Note that in the above, we use the convention that 
$\|\cdot\|_{L_t^2H_{co}^l}=0$ if $l< 0.$ 
Moreover, we can show by repeating the procedure to prove \eqref{finalomegat} that: 
\beq\label{finalomega}
\begin{aligned}
\|\tz_2\|_{L_t^{\infty}H_{co}^{k}}^2+\|\nabla\tz_2\|_{\hco^{k}}^2&\lesssim Y_{m}^2(0)+\|\nabla\tz_2\|_{\hco^{k-1}}^2+
T^{\f{1}{4}}\lae.
\end{aligned}
\eeq
The estimate \eqref{EIoftz2} then stems from \eqref{finalomega} and an induction on $k\in [0,m-4],$ 
the estimate \eqref{EIoftz2-1}
can also be derived from \eqref{EIoftz2} and induction arguments. 
\end{proof}
In the following, we show an estimate needed to control $\cJ_1$ in the above lemma.
\begin{prop}\label{remainderL2}
Assume that \eqref{preassumption} holds, then for any $0<t\leq T,$
\beqs
\|(\Id,\ep^{\f{1}{2}}\pt)(G_{\chi}^{\zeta}-G_{\chi,1}^{\omega}-G_{\chi,2}^{\omega}-G_{\chi,1}^{\varsigma})\|_{\hco^{m-4}(\mS)}\lesssim \lae.
\eeqs
\end{prop}
\begin{proof}
One can show this estimate by bounding each term appearing in $G_{\chi}^{\zeta}-G_{\chi,1}^{\omega}-G_{\chi,2}^{\omega}-G_{\chi,1}^{\varsigma}.$ We will give the details for one term, namely $\omega\cdot\nabla^{\vp}u,$ which is 
the most difficult one, the other terms can be controlled easily.
Let us write $$\omega\cdot\nabla^{\vp}u=\omega_1\p_{y_1}u+\omega_2\p_{y_2}u+(\omega\cdot\bN) \p_z^{\vp} u.$$ Furthermore, we have: 
\beqs
\begin{aligned}
\omega\cdot\bN&=\div^{\vp}(u\times\bN)+u\cdot(\nabla^{\vp}\times\bN)\\
&=-(u\times\bN)\cdot\p_z^{\vp}\bN-\p_{y_1}(u\times\bN)_1-\p_{y_2}(u\times\bN)_2+u\cdot(\nabla^{\vp}\times\bN). 
\end{aligned}
\eeqs
We thus see that $\omega\cdot\nabla^{\vp}u=\p_z u\cdot F_1(\p_y u, \nabla^{\vp}\bN, u,\bN, \f{1}{\p_z\vp})+F_2(\p_y u,\p_y u)$
where $F_1, F_2$ are some polynomials with degree $4.$ Let us control 
$\ep^{\f{1}{2}}\pt(\p_z u \p_y u)$ for example, the other ones can be bounded in a similar way (note that we do not lose regularity on the surface the terms involving $\bN$). By counting the derivatives hitting on each term, one finds that:
\begin{align*}
&\|(\ep^{\f{1}{2}}\pt \p_z u \cdot\p_y u, \p_z u \cdot\ep^{\f{1}{2}}\pt\p_y u)
\|_{\hco^{m-4}}\\
&\lesssim
\il\ep^{\f{1}{2}}\pt \p_z u\il_{0,\infty,t}\|u\|_{\hco^{m-3}}
+\|\ep^{\f{1}{2}}\pt \p_z u\|_{\hco^{m-4}}\il u\il_{m-4,\infty,t}\\
&\quad+\il\nabla u\il_{1,\infty,t}\|\ep^{\f{1}{2}}\pt u\|_{\hco^{m-3}}+
\il \ep^{\f{1}{2}}\pt u\il_{m-5,\infty,t}
\|\nabla u\|_{L_t^{\infty}H_{co}^{m-4}}\\
&\lesssim \lae.
\end{align*}
%Note that thanks to the Sobolev embedding \eqref{soblev-embed},
%\beqs\big(\int_0^t \|u(s)\|_{m-3,\infty}\d s\big)^{\f{1}{2}}\lesssim \|u\|_{\hco^{m-1}}+\|\nabla u\|_{\hco^{m-2}}\lesssim \lae,\eeqs
%\beqs \big(\int_0^t \|(\ep^{\f{1}{2}}\pt u)(s)\|_{m-4,\infty}\d s\big)^{\f{1}{2}}\lesssim \|\ep^{\f{1}{2}}\pt u\|_{\hco^{m-2}}+\|\ep^{\f{1}{2}}\pt\nabla u\|_{\hco^{m-3}}\lesssim \lae.\eeqs
\end{proof}
 \subsection{Estimate of the second order  normal derivatives of the velocity}
 To finish the a-priori estimates for the energy norms, we are left to  estimate $\nabla^2 u$ in a non-uniform way  which is the object of the following lemma.
\begin{lem}\label{lemsec-normal-u}
Assume that \eqref{preassumption} holds for some $T>0,$ then for any $0<t\leq T,$ the following estimate holds,
\beq\label{sec-normal-u}
\|\ep^{\f{1}{2}}\nabla^2 u\|_{L_t^{\infty}H_{co}^{m-2}\cap L_t^2\cH_{co}^{m-1}}^2\lesssim  \Lambda\bigg(\f{1}{c_0},|h|_{L_t^{\infty}\tilde{H}^{m-\f{1}{2}}}^2+Y_{m}^2(0)\bigg)Y_m^2(0)+(T+\ep)^{\f{1}{4}}\lae.
\eeq
\end{lem}
\begin{proof}
We will prove the following two inequalities:
\beq\label{intemediate2}
\begin{aligned}
\ep^{\f{1}{2}}\|\nabla^2 u\|_{L_t^{\infty}H_{co}^{m-2}}&\lesssim (T+\ep)^{\f{1}{2}}\lae+\Lambda\big(\f{1}{c_0},|h|_{m-2,\infty,t}\big)\|\ep^{\f{1}{2}}\nabla u\|_{L_t^{\infty}H_{co}^{m-1}}
\\
&\qquad +\ep^{\f{1}{2}}\|\pt u\|_{L_t^{\infty}H_{co}^{m-2}}+\ep^{-\f{1}{2}}\|\nabla^{\vp}\sigma\|_{L_t^{\infty}H_{co}^{m-2}},
\end{aligned}
\eeq
\beq \label{intemediate1.5}
\begin{aligned}
\ep^{\f{1}{2}}\|\nabla^2 u\|_{L_t^2\cH^{m-1}}
&\lesssim (T+\ep)^{\f{1}{2}}\lae\\
&+\Lambda\big(\f{1}{c_0},|h|_{m-2,\infty,t}\big)\|\ep^{\f{1}{2}}\nabla u\|_{L_t^{2}H_{co}^{m}}+\ep^{-\f{1}{2}}\|\nabla^{\vp}\sigma\|_{L_t^{2}\cH^{m-1}} %+\ep^{\f{1}{2}}\|\pt u\|_{L_t^{2}\cH^{m-1}}+,
\end{aligned}
\eeq 
where $\cN_{m,T}$ is defined in \eqref{defcN}.
These two estimates, together with 
\eqref{EI-1},
\eqref{EI-T1},
\eqref{nor-compressible}, \eqref{epnablauLinfty1}, \eqref{nor-compressible1}, yield \eqref{sec-normal-u}.
To prove \eqref{intemediate2} and \eqref{intemediate1.5}, it suffices to control $\ep^{\f{1}{2}}\p_z^2 u.$  Let us rewrite the equations $\eqref{FCNS2}_2$ as
 \beq\label{rewrite-u}
\ep^{\f{1}{2}}\Delta^{\vp}u=\ep^{\f{1}{2}} g_2(\pt+\underline{u}\cdot\nabla)u+\ep^{-\f{1}{2}}\nabla^{\vp}\sigma-\ep^{\f{1}{2}}\nabla^{\vp}\div^{\vp}u.
 \eeq
In view of \eqref{rewrite-u}, \eqref{laplaceu}, we have by the product estimate \eqref{crudepro} and the definition of $\cE_{m,t}$ that:
\begin{align*}
\|\ep^{\f{1}{2}} \p_z^2 u\|_{L_t^{\infty}H_{co}^{m-2}}
&\lesssim \|\ep^{\f{1}{2}}\pt u\|_{L_t^{\infty}H_{co}^{m-2}}+\|\ep^{-\f{1}{2}}\nabla\sigma\|_{L_t^{\infty}H_{co}^{m-2}}\\
&+\Lambda\big(\f{1}{c_0},|h|_{m-2,\infty,t}\big)\big(\ep^{\f{1}{2}}\|(\sigma, u,\nabla\sigma,\nabla u)\|_{L_t^{\infty}H_{co}^{m-1}}+\lca|\ep^{\f{1}{2}}h|_{L_t^{\infty}\tilde{H}^{m-\f{1}{2}}}\big)\\
&\lesssim \|\ep^{\f{1}{2}}\pt u\|_{L_t^{\infty}H_{co}^{m-2}}+\|\ep^{-\f{1}{2}}\nabla\sigma\|_{L_t^{\infty}H_{co}^{m-2}}+\Lambda\big(\f{1}{c_0},|h|_{m-2,\infty,t}\big)\|\ep^{\f{1}{2}}\nabla u\|_{L_t^{\infty}H_{co}^{m-1}}\\
&\quad+(T+\ep)^{\f{1}{2}}\lae.
\end{align*}
%In a similar way, we get:\beqs\begin{aligned}\ep^{\f{1}{2}}\|\p_z^2 u\|_{L_t^{\infty}H_{co}^{m-3}}&\lesssim \|\ep^{\f{1}{2}}\pt u\|_{L_t^{\infty}H_{co}^{m-3}}+ \|\ep^{\f{1}{2}}\nabla u\|_{L_t^{\infty}H_{co}^{m-2}}+\|\ep^{-\f{1}{2}}\nabla^{\vp}\sigma\|_{L_t^{\infty}H_{co}^{m-3}} \\&\quad +\ep^{\f{1}{2}}\Lambda\big(\f{1}{c_0},\il\nabla u\il_{0,\infty,t}+|h|_{m-2,\infty,t}\big)\cE_{m,t} \end{aligned} \eeqs
 We thus finish the proof of  \eqref{intemediate2}. 
 The inequality  \eqref{intemediate1.5} can be shown in a similar way, 
 we thus omit the proof.
\end{proof}

\section{Control of the $L_{t,x}^{\infty}$ norm}
In this section, we prove 
Proposition \ref{prop-Linfty}, the a-priori estimate for
$\cA_{m,T}:$
 \begin{equation}\label{defcAmt-1}
 \begin{aligned}
  \mathcal{A}_{m,T}(\sigma,u)&=
 % \cG_{\infty,t}(h)
 |h|_{m-2,\infty,t}+\il
  \nabla u\il_{1,\infty,T}
 +\il\ep^{-\f{1}{2}}(\nabla ^{\vp}\sigma,\div^{\vp} u)\il_{m-5,\infty,T}+\il\ep^{\f{1}{2}}\pt(\sigma,u)\il_{m-5,\infty,T}
\\&\quad +\il(\text{Id}, \ep\pt)(\sigma,u)\il_{m-4,\infty,T}
  +\il\varepsilon^{\f{1}{2}}\nabla u\il _{m-3,\infty,T}+\il\varepsilon^{\f{1}{2}}(\sigma, u)\il _{m-2,\infty,T}.
 \end{aligned}
 \end{equation}
 \begin{rmk}
 By the identity \eqref{epdeltau} and the equation $\eqref{FCNS2}_2$
 for $u,$
 we have that:
 \beq
 \ep^{\f{1}{2}}\il\p_z^2 u\il_{m-5,\infty,t}\lesssim \lca.
 \eeq
 \end{rmk}
 \begin{rmk}
 As $[\f{m}{2}]\leq m-4$ if $m\geq 7,$ we thus have:
 \beqs
 \begin{aligned}
&%\|\nabla u\|_{*,[\f{m}{2}]-1,\infty,t}+
\il\ep^{-\f{1}{2}}(\nabla ^{\vp}\sigma,\div^{\vp} u)\il_{[\f{m}{2}]-1,\infty,T}+\il\ep^{\f{1}{2}}\pt(\sigma,u)\il_{[\f{m}{2}]-1,\infty,T}+\ep^{\f{1}{2}}\il\p_z^2 u\il_{[\f{m}{2}]-1,\infty,t}\\
&+\il(\text{Id}, \ep\pt)(\sigma,u)\il_{[\f{m}{2}],\infty,T}
  +\il\varepsilon^{\f{1}{2}}\nabla u\il _{[\f{m}{2}]+1,\infty,T}+\il\varepsilon^{\f{1}{2}} u\il _{[\f{m}{2}]+2,\infty,T}\lesssim \cA_{m,T}.
  \end{aligned}
 \eeqs
 \end{rmk}
%The $L_{t,x}^{\infty}$ estimate of $\nabla\sigma$ results from the maximum principle of the damped transport equation \eqref{gradsigma} satisfied by $\nabla\sigma.$ To control $L_{t,x}^{\infty}$ norm of $\nabla u,$ we reduce the matter to the estimate of vorticity $\omega$ which is done by using the Green function.
%and then use the explicit formulae satisfied by the $\omega$
%and the transport diffusion equation \eqref{eq-omega-free} satisfied by the vorticity.
The other terms appearing in $\cA_{m,T}$ can be obtained by the Sobolev embedding \eqref{soblev-embed}.

\begin{proof}[Proof of Proposition \ref{prop-Linfty}]
 By the Sobolev embedding $H^{\f{3}{2}}(\mathbb{R}^2)\hookrightarrow L^{\infty}(\mathbb{R}^2),$ we have directly that:
\begin{equation}\label{hLinfty}
|h|_{m-2,\infty,T}\lesssim |h|_{L_T^{\infty}\tilde{H}^{m-\f{1}{2}}}\lesssim \tilde{\cE}_{m,T}.
\end{equation}

Furthermore, thanks to the Sobolev embedding
 \eqref{soblev-embed}, the last four terms in \eqref{defcAmt-1} can be controlled by the ones appearing in $\cE_{m,T}.$ Indeed,
 \begin{equation}
\begin{aligned}
 &\varepsilon^{\frac{1}{2}}\il\pt(\sigma,u) \il_{m-5,\infty,T}\lesssim  \sup_{0\leq s\leq T} \big( \|\ep^{\f{1}{2}}\pt u(s)\|_{H_{co}^{m-3}}+\|\varepsilon^{\f{1}{2}}\pt \nabla u(s)\|_{H_{co}^{m-4}}\big)\lesssim \tilde{\cE}_{m,T},\label{ptuinfty}
 \end{aligned}
\end{equation}
\begin{equation*}\label{epdeltau}
\begin{aligned}
 &\varepsilon^{\f{1}{2}}\il\nabla u(s)\il_{m-3,\infty,T}\lesssim  \sup_{0\leq s\leq T} \big( \|\varepsilon^{\f{1}{2}}\nabla u(s)\|_{H_{co}^{m-1}}+\|\varepsilon^{\f{1}{2}}\nabla^2 u(s)\|_{H_{co}^{m-2}}\big)\lesssim \tilde{\cE}_{m,T}, 
 \end{aligned}
\end{equation*}
 \begin{equation}\label{siuLinfty}
 \il(\sigma,u)\il_{m-4,\infty,T}\lesssim  \sup_{0\leq s\leq T}\big(\|(\sigma,u)(s)\|_{H_{co}^{m-1}}+\|\nabla(\sigma,u)(s)\|_{H_{co}^{m-4}}\big) \lesssim \tilde{\cE}_{m,T}.
\end{equation}
 \begin{equation*}
\il\ep\pt(\sigma,u)\il_{m-4,\infty,T}\lesssim  \sup_{0\leq s\leq T}\big(\|\ep^{\f{1}{2}}\pt(\sigma,u)(s)\|_{H_{co}^{m-2}}+\ep^{\f{1}{2}}\|\ep\pt\nabla(\sigma,u)(s)\|_{H_{co}^{m-3}}\big) \lesssim \tilde{\cE}_{m,T}.
\end{equation*}
\beq
\ep^{\f{1}{2}}\il(\sigma, u) %\varepsilon^{\f{3}{4}}\nabla u)
\il _{m-2,\infty,T}\lesssim  \sup_{0\leq s\leq T}\big(\|(\sigma,u)(s)\|_{H_{co}^{m}}+\|\ep^{\f{1}{2}}\nabla(\sigma,u)(s)\|_{H_{co}^{m-1}}\big) \lesssim \tilde{\cE}_{m,T}.
\eeq
For the third term in \eqref{defcAmt-1}, we can use the equation for $\sigma$ to get that:
\beq\label{divuinfty}
\begin{aligned}
\ep^{\f{1}{2}}\il\div^{\vp}u\il_{m-5,\infty,T}&\lesssim \il \ep^{\f{1}{2}}\pt\sigma \il_{m-5,\infty,T}+\ep^{\f{1}{2}} \big(\il(u,\ep\pt\sigma,\nabla\sigma\il_{m-5,\infty,T}+|h|_{m-4,\infty,T}\big)^2\\
&\lesssim \tilde{\cE}_{m,T}+\ep^{\f{1}{2}}\Lambda\big(\f{1}{c_0},\cA_{m,T}\big).
\end{aligned}
\eeq
Moreover, in view of \eqref{hLinfty}, \eqref{siuLinfty} and identity $\Pi\nabla^{\vp}=\Pi(\p_1,\p_2,0)^{t},$
\beqs
\begin{aligned}
\ep^{-\f{1}{2}}\il\nabla^{\vp}\sigma\il_{m-5,\infty,T}&\lesssim \ep^{-\f{1}{2}}\il\p_y\sigma\il_{m-5,\infty,T}(1+|h|_{m-4,\infty,T})^2+\ep^{-\f{1}{2}}\il\nabla^{\vp}\sigma\cdot\bn\il_{m-5,\infty,T}|h|_{m-4,\infty,T}\\
&\lesssim 
\tilde{\cE}_{m,T}+\tilde{\cE}_{m,T}^3+\il\nabla^{\vp}\sigma\cdot\bn\il_{[\f{m}{2}]-1,\infty,T}^2.
\end{aligned}
\eeqs
Indeed, we have used the Sobolev embedding \eqref{soblev-embed} to get that: 
\beqs
\ep^{-\f{1}{2}}\il\p_y\sigma\il_{m-5,\infty,T}\lesssim \ep^{-\f{1}{2}}\il\nabla^{\vp}\sigma\il_{L_t^{\infty}H_{co}^{m-3}}\lesssim \tilde{\cE}_{m,T}.
\eeqs
Therefore, it remains to control 
$\ep^{-\f{1}{2}}\il\nabla^{\vp}\sigma\cdot\bn\il_{[\f{m}{2}]-1, \infty, T},$ which is the aim of the following lemma.
\end{proof}
\begin{lem}
Suppose that \eqref{preassumption} holds, then:
\begin{equation}\label{normal-grad-sigma}
\begin{aligned}
\ep^{-\f{1}{2}}\il\nabla^{\vp}\sigma\cdot\bn\il_{[\f{m}{2}]-1,\infty,T}&\lesssim Y_m^2(0)+
\tilde{\cE}_{m,T}^2+\ep^{\f{1}{2}}\Lambda\big(\f{1}{c_0},\cA_{m,T}\big).
\end{aligned}
\end{equation}
\end{lem}
\begin{proof}
By \eqref{gradsigma}, we find that
$\nabla^{\vp}\sigma$ solves
\beq\label{gradsigma1}
\ep^2 g_1(\pt+\underline{u}\cdot\nabla)\nabla^{\vp}\sigma+\f{1}{2\mu+\lambda}\nabla^{\vp}\sigma= \cQ_1%+\cQ_2+\cQ_3.
\eeq
where $\cQ_1=\cQ_{11}+\cQ_{12}+\cQ_{13},$ with
$$\cQ_{11}=-\ep^2g_1'\nabla^{\vp}\sigma(\ep\pt+\ep\underline{u}\cdot\nabla)\sigma-\ep^2 g_1\nabla^{\vp}u\cdot\nabla^{\vp}\sigma,$$ 
$$\cQ_{12}=-\f{\mu \ep}{2\mu+\lambda}\curl^{\vp}\omega,\quad \cQ_{13}=-\f{1}{2\mu+\lambda}g_2(\ep\pt+\ep \underline{u}\cdot\nabla)u.$$
Denote $R=\nabla^{\vp}\sigma\cdot\bn,$ then by \eqref{gradsigma}, $R$ solves:
\beqs
\ep^2 g_1(\pt+\underline{u}\cdot\nabla)R+\f{1}{2\mu+\lambda}R=\ep^2 g_1\nabla^{\vp}\sigma(\pt+\underline{u}\cdot)\bn+\cQ_1\cdot\bn=\colon \cQ_2+\cQ_1\cdot\bn
\eeqs
For any multi-index with $|\beta|\leq m-5,$ denote
$R^{\beta}=Z^{\beta}R,$ then $R^{\beta}$ satisfies:
\beqs
\ep^2 g_1(\pt+\underline{u}\cdot\nabla)R^{\beta}+\f{1}{2\mu+\lambda}R^{\beta}=Z^{\beta}(\cQ_2+\cQ_1\cdot\bn)+\cC_{R,1}^{\beta}+\cC_{R,2}^{\beta}=\colon \cQ^{\beta},
\eeqs
where  $$\cC_{R,1}^{\beta}=-\ep^2[Z^{\beta},g_1/\ep]\ep\pt R, \quad \cC_{R,2}^{\beta}=-\ep^2[Z^{\beta},g_1\underline{u}\cdot \nabla] R.$$
Define $X_t(x)=X(t,x)$ the unique flow associated to $\underline{u}:$
$$\pt X(t,x)=\underline{u}(t,X(t,x)),\quad X(0,x)=x.$$
Note that since $\underline{u}\cdot \bn|_{z=0}=0,$
and $u\in \Lip([0,T]\times\Omega),$ we have for each $t\in[0,T],$ $X_t:\mS\rightarrow\mS$ is a diffeomorphism.  
Denote $f^{X}=f(t,X(t,x)),$ then $R^{\beta,X}$ solves the ODE:
\beqs
\ep^2 (g_1\pt R^{\beta})(t,X_t(x))+\f{1}{2\mu+\lambda}R^{\beta}(t,X_t(x))=Q^{\beta}(t,X_t(x))
\eeqs
from which, we deduce that:
\beqs
R^{\beta}(t,X_t(x))=e^{-\int_0^t \f{1}{\ep^2 g_1(s,X_s(x))}\d s}R^{\beta}(0)+\int_0^t e^{-\int_{\tau}^t \f{1}{\ep^2 g_1(x,X_s(x))}\d s}\f{1}{\ep^2}Q^{\beta} (\tau, X_{\tau}(x))\d\tau.
\eeqs
By assumption \eqref{preassumption}, $ c_0\leq g_1(t, X_t(x))\leq\f{1}{c_0}$ for any $(t,x)\in[0,T]\times\mS.$ Therefore,
\beq\label{Rbeta}
\begin{aligned}
\ep^{-\f{1}{2}}\il R^{\beta}\il_{0,\infty,T}&\lesssim \ep^{-\f{1}{2}} \sup_{(t,x)\in[0,T]\times\mS}| R^{\beta}(t, X_t(x))|\\
&\lesssim \ep^{-\f{1}{2}} \il R^{\beta}(0)\il_{L^{\infty}(\mS)}+\ep^{-\f{1}{2}}\int_0^T e^{-{c_0(t-s)}/{\ep^2}}\f{1}{\ep^2}\d s \il Q^{\beta}\il_{0, \infty,T}\\
&\lesssim Y_m(0)+\ep^{-\f{1}{2}}\il Q^{\beta}\il_{0, \infty, T}.
\end{aligned}
\eeq
It thus suffices to control the term $\ep^{-\f{1}{2}}\il Q^{\beta}\il_{0,\infty,T}.$
First of all, by the property \eqref{preassumption1}, we get that:
\beq\label{cR1}
\ep^{-\f{1}{2}}\il\cC_{R,1}^{\beta}\il_{0,\infty,T}\lesssim \ep^{\f{3}{2}} \big(\il(\sigma,\nabla \sigma)\il_{m-5,\infty, T}+|h|_{m-4,\infty,T}\big)^2\lesssim \ep^{\f{3}{2}}\cA^2_{m,T}.
\eeq
Next, by using that $\underline{u}\cdot \nabla=u_y\p_y+ \f{U_z}{\phi}Z_3 R,$ we can control the second commutator term as:
\beq\label{cR2}
\ep^{\f{3}{2}}\il\cC_{R,2}^{\beta}\il_{0,\infty,T}\lesssim \ep \big(\il(\sigma,u,\nabla \sigma,\ep^{\f{1}{2}}\nabla u)\il_{m-5,\infty, T}+|h|_{m+3,\infty, T}\big)^2\lesssim\ep\cA^2_{m,T}.
\eeq
Similarly, we can find some polynomial $\Lambda$, such that
\beq
\begin{aligned}
\ep^{-\f{1}{2}}\il Z^{\beta}(\cQ_2+\cQ_{11}\cdot\bn)\il_{0,\infty,T}&\lesssim \ep^{\f{1}{2}}\Lambda\big(\f{1}{c_0},  \il(\sigma,u,\nabla \sigma,\ep^{\f{1}{2}}\nabla u)\il_{[\f{m}{2}]-1,\infty,T}+ |h|_{[\f{m}{2}]+1,\infty,T}\big)\\
&\lesssim \ep^{\f{1}{2}} \Lambda\big(\f{1}{c_0}, \cA_{m,T}\big).
\end{aligned}
\eeq
 Moreover, in light of \eqref{hLinfty} and \eqref{siuLinfty}, we have
 \beq
 \begin{aligned}
\ep^{-\f{1}{2}} \il Z^{\beta}(\cQ_{13}\cdot\bn)\il_{0,\infty,T}&\lesssim 
 \il (\ep^{\f{1}{2}}\pt u\cdot\bn, \ep^{\f{1}{2}}\underline{u}\cdot\nabla u) \il_{m-5,\infty, T}+\ep^{\f{1}{2}}\Lambda\big(\f{1}{c_0},\cA_{m,T}\big)\\
% +\ep^{\f{1}{2}}\Lambda( \il(\sigma,u,\ep^{\f{1}{2}}\nabla u)\il_{[\f{m}{2}]-1,\infty,T}+ |h|_{[\f{m}{2}],\infty,T})\\
 &\lesssim 
\tilde{\cE}_{m,T}^2+\ep^{\f{1}{2}}\Lambda\big(\f{1}{c_0},\cA_{m,T}\big).
 \end{aligned}
 \eeq
Finally, since 
\beqs
\begin{aligned}
\curl^{\vp}\omega\cdot\bn&=\div^{\vp}(\omega\times \bn)+\omega\cdot\curl^{\vp} \bn\\
&=-(\omega\times\bn)\cdot\p_z^{\vp}\bN+\p_1(\omega\times\bn)_1+\p_2(\omega\times\bn)_2+\omega\cdot\curl^{\vp} \bn
\end{aligned}
\eeqs
 involves only tangential derivatives of 
$\nabla^{\vp} u,$ one has again by \eqref{hLinfty} and \eqref{siuLinfty} that:
\beq\label{cQ3}
\ep^{-\f{1}{2}}\il Z^{\beta}(\cQ_{12}\cdot\bn)\il_{0,\infty,T}\lesssim
 \big(\il\ep^{\f{1}{2}}\nabla u\il_{m-4,\infty,T}+ |h|_{m-3,\infty,T}\big)^2\lesssim \tilde{\cE}_{m,T}^2.
 \eeq
Collecting \eqref{cR1}-\eqref{cQ3}, we find that:
\beqs
\il \cQ\il_{[\f{m}{2}]-1,\infty, T}\lesssim 
\tilde{\cE}_{m,T}^2+\ep^{\f{1}{2}}\Lambda\big(\f{1}{c_0},\cA_{m,T}\big).
\eeqs
 Inserting this inequality into \eqref{Rbeta}, we eventually get \eqref{normal-grad-sigma}.
\end{proof}
In the following Lemma, we obtain the $L_{t,x}^{\infty}$ estimates of 
$\nabla u,$ namely $\il\ep^{\f{1}{2}}\pt \nabla u\il_{0,\infty,t}, \il \nabla u\il_{1,\infty,t}.$ 
\begin{lem}\label{lem-inftynablau}
Assume that \eqref{preassumption} holds, then we have that for any $0<t\leq T,$
\begin{equation}\label{inftynablau}
    \il\ep^{\f{1}{2}}\pt \nabla u\il_{0,\infty,t}+\il \nabla u\il_{1,\infty,t}\lesssim 
  \Lambda\big(\f{1}{c_0},Y_m(0)\big) +
    \Lambda\big(\f{1}{c_0},|h|_{3,\infty,t}\big)
    \tilde{\cE}_{m,T}+(T+\ep)^{\f{1}{4}}\lae.
\end{equation}
\end{lem}
\begin{proof}
In view of the identities \eqref{norpz} and 
\beqs
\begin{aligned}
\Pi(\p_z^{\vp}u)&=\f{1}{|\bN|}\Pi\big(\omega\times\bN+(\nabla^{\vp}u)^{t}\cdot\bn-\bn_1\p_1 u-\bn_2\p_2 u\big)\\
&=\f{1}{|\bN|}(\omega\times\bN)+\Pi\nabla^{\vp}(u\cdot\bn)-\Pi\big((\nabla^{\vp}\bn)^t u-\bn_1\p_1 u-\bn_2\p_2 u\big),
\end{aligned}
\eeqs
one gets that:
\begin{align*}
&\il \nabla u\il_{1,\infty,t}+\ep^{\f{1}{2}}\il\pt \nabla u\il_{0,\infty,t}\lesssim\ep^{\f{1}{2}}\lca\\
&+ \Lambda\big(\f{1}{c_0},|h|_{3,\infty,t}\big)\big(\il u\il_{2,\infty,t}+\il\ep^{\f{1}{2}}\pt u\il_{1,\infty,t}+\|\ep^{-\f{1}{2}}\div^{\vp}u\|_{1,\infty,t}
+\il\omega\il_{1,\infty,t}+\il\ep^{\f{1}{2}}\pt \omega\il_{0,\infty,t}\big).
\end{align*}
The inequality \eqref{inftynablau} then follows from \eqref{ptuinfty}, \eqref{siuLinfty}, \eqref{divuinfty} and the next lemma for the estimates of $\omega.$
\end{proof}
\begin{lem}
Under the same assumption as in Lemma \eqref{lem-inftynablau},
\beq\label{sec13:eq7}
\il\omega\il_{1,\infty,t}+\il\ep^{\f{1}{2}}\pt \omega\il_{0,\infty,t}\lesssim \Lambda\big(\f{1}{c_0},Y_m(0)\big)+ \Lambda\big(\f{1}{c_0},|h|_{3,\infty,t}\big)\tilde{\cE}_{m,T}+
(T+\ep)^{\f{1}{4}}\lae.
\eeq
\end{lem}
\begin{proof}
Away from the boundary, the conormal spaces are equivalent to the usual Sobolev space, the $L_{t,x}^{\infty}$ estimate for $\omega$ can be obtained directly from the usual Sobolev embedding. It thus suffices to establish the corresponding estimates near the boundaries. In what follows, we will detail their estimates near the upper boundary (which corresponds to the free surface), the one near the bottom being easier and has essentially been performed in \cite{masmoudi2021uniform}. As in the proof of Lemma \ref{lemomegatimesbn}, we will employ the normal geodesic coordinates \eqref{geodesic} to take the benefit of the explicit formula for the heat equation on the half line.
Taking the same cut off function $\chi=\chi_0(\f{z}{C(\kappa)})$ introduced in Lemma \ref{lemomegatimesbn} (which satisfies
$\Phi_t(\Supp \chi)\Subset \tilde{\Phi}_t(\mS_k) $), we use the equation \eqref{eq-omega-free} to obtain that:
\beqs
(\bar{\rho}\pt-\mu\Delta^{\vp})(\chi\omega)=\chi G^{\omega}-\mu \Delta^{\vp}\chi\omega-\mu\p_z\chi(\bN\cdot \nabla^{\vp})\omega=: G^{\chi,\omega}
\eeqs
where 
$$G^{\omega}=-u\cdot\nabla^{\vp}\omega+\omega\cdot\nabla^{\vp}u-\omega\div^{\vp}u-\f{\nabla g_2}{\ep}\times ((\ep\pt+\ep\underline{u}\cdot\nabla)u)+\f{\bar{\rho}-g_2}{\ep}((\ep\pt+\ep\underline{u}\cdot\nabla)\omega).$$
For a function $f(t,\cdot)$ supported on $\mR^2\times [-C(\kappa),0],$ we use the notation $$\tilde{f}(t,x)=f(t,\Phi_t^{-1}\circ \tilde{\Phi}_t(x)).$$
By the change of  variable, we find that $\widetilde{\chi\omega}$ satisfies the system:
\begin{align}\label{chiomega-free1}
%\left\{
%\begin{array}{l}
  (\bar{\rho}\pt-\mu\p_z^2)\widetilde{\chi\omega}&= \widetilde{F^{\chi,\omega}}=\colon \widetilde{G^{\chi,\omega}}+ \bar{\rho}(D\tilde{\Phi}_t)^{-1}\pt \tilde{\Phi}_t\cdot\nabla\widetilde{\chi\omega}\\
&\quad +\mu\big[ \f{1}{2}\p_z(\ln|g|)\p_z+\p_i(\ln|g|)\tilde{g^{ij}}\p_j+ \p_i(\tilde{g}^{ij}\p_{j}\cdot)\big](\widetilde{\chi\omega})\nonumber
\end{align}
supplemented with the initial  and the boundary conditions:
\beq\label{chiomega-free2}
\widetilde{\chi\omega}|_{t=0}=\chi\omega|_{t=0}(\Phi_0^{-1}\circ\tilde{\Phi}_0), \qquad \widetilde{\chi\omega}|_{z=0}=\omega|_{z=0}=\colon \omega^{b,1}.
\eeq
Let
\beqs
E(t,z,z')=\tilde{\mu}\f{1}{(4\pi\tilde{\mu} t)^{\f{1}{2}}}\big(e^{-\f{|z-z'|^2}{4\tilde{\mu}t}}-e^{-\f{|z+z'|^2}{4\tilde{\mu}t}}\big),\qquad \tilde{\mu}=\bar{\rho}/{\mu},
\eeqs
the solution to the system \eqref{chiomega-free1}-\eqref{chiomega-free2} can be expressed as:
\beq\label{explicit-vor1}
\begin{aligned}
\widetilde{\chi\omega}(t,y,z)&=-\int_0^t (\p_{z'} E)(t-s,z,0) \omega^{b,1}(s,y)\, \d s+
\int_{-\infty}^0 E(t,z,z')\widetilde{\chi\omega}|_{t=0}(y,z')\,\d z'\\
&\qquad+\int_0^t \int_{-\infty}^0 E(t-s,z,z')
\widetilde{F^{\chi,\omega}}(s,y,z')\,\d z'\d s=(1)+(2)+(3).
\end{aligned}
\eeq
\underline{Control of the boundary term (1)}.
As in the estimate of \eqref{tz1Linfty1} and \eqref{tz1Linfty2}, we can bound the boundary term as:
\beqs
\begin{aligned}
\il (\text{Id},\ep^{\f{1}{2}}\pt, \p_y,Z_3)(1)\il_{0,\infty,t}\leq C(\mu) \|(\text{Id},\ep^{\f{1}{2}}\pt, \p_y) \omega^{b,1} \|_{L_{t,y}^{\infty}}
\end{aligned}
\eeqs
By the identities \eqref{nor-nor}, \eqref{tan-nor},
one sees that
$$\omega^{b,1}\approx F(u^{b,1},\p_y u^{b,1}, (\div^{\vp} u)^{b,1}, \bn^{b,1}, \nabla\bn^{b,1}),$$
which, together with the previous inequality, yields that:
\beq\label{sec13:eq5}
\begin{aligned}
&\il (\text{Id},\ep^{\f{1}{2}}\pt, \p_y,Z_3)(1)\il_{0,\infty,t}\\
&\lesssim \Lambda\big(\f{1}{c_0},|h|_{3,\infty,t}\big)\big(\il\ep^{\f{1}{2}}\pt u\il_{1,\infty,t}+\il\ep^{-\f{1}{2}} \div ^{\vp}u \il_{1,\infty,t}+\il u\il_{2,\infty,t}\big)+\ep^{\f{1}{2}}\lae.\\
&\lesssim \Lambda\big(\f{1}{c_0},|h|_{3,\infty,t}\big)\widetilde{\cE}_{m,t}+\ep^{\f{1}{2}}\lae.
\end{aligned}
\eeq
\underline{Control of the initial evolution (2).} 
Since $\pt,\p_y$ commute with the 
operator $\bar{\rho}\pt-\mu\p_z^2$, the following identity holds:
\beqs
(\ep^{\f{1}{2}}\pt,\p_y)(2)=\int_{-\infty}^0 E(t,z,z')(\ep^{\f{1}{2}}\pt,\p_y)(\widetilde{\chi\omega})|_{t=0}(y,z')\,\d z',
\eeqs
from which we derive that:
\beq\label{vor-2-1}
\begin{aligned}
&\il(\text{Id},\ep^{\f{1}{2}}\pt,\p_y)(2)\il_{0,\infty,t}\lesssim \bigg\|\int_{-\infty}^0|E(t,z,z')|\d z'\bigg\|_{L_t^{\infty}L_z^{\infty}}\big\|(\text{Id},\ep^{\f{1}{2}}\pt,\p_y)\widetilde{\chi\omega})|_{t=0}\big\|_{L^{\infty}(\mS_{\kappa})}\\
&\lesssim \Lambda\big(\f{1}{c_0},|h_0|_{2,\infty}+|\ep^{\f{1}{2}}\pt h|_{t=0}|_{1,\infty}\big)\big(\|(\ep^{\f{1}{2}}\pt \omega)|_{t=0}\|_{L^{\infty}(\mS)}+\|(\text{Id},\p_y,Z_3)\omega_0\|_{L^{\infty}(\mS)}\big) \\
&\lesssim \Lambda\big(\f{1}{c_0},Y_m(0)\big).
\end{aligned}
\eeq
To control $Z_3(2),$ we denote   $E_{\pm}(t,z,z')= \tilde{\mu}\f{1}{(4\pi\tilde{\mu} t)^{\f{1}{2}}} e^{-\f{|z\pm z'|^2}{4\tilde{\mu}t}}.$ 
By writing $z=z-z'+z'$ or $z=z+z'-z',$ one can split $Z_3(2)$ into two terms:
\beqs 
\begin{aligned}
Z_3(2)=\int_{-\infty}^0 \phi(z)\p_z (E_{-}-E_{+})(t,z,z') (\widetilde{\chi\omega})|_{t=0}\, \d z'=(Z_3(2))_1+(Z_3(2))_2
\end{aligned}
\eeqs
with 
\beno
 (Z_3(2))_1&=&\phi_1(z)\int_{-\infty}^0 \big((z-z')\p_z E_{-}-(z+z')\p_z E_{+}\big)(t,z,z')(\widetilde{\chi\omega})|_{t=0}\,\d z',\\
 (Z_3(2))_2&=&\phi_1(z)\int_{-\infty}^0 E(t,z,z') \p_{z'}(z' (\widetilde{\chi\omega})|_{t=0})\,\d z',
\eeno
where we use the notation $\phi(z)=\f{z(1-z)}{(2-z)^2}=z\phi_1(z).$
By straightforward calculation, we obtain
\beqs 
\big|\phi_1(z)\int_{-\infty}^0 \big((z-z')\p_z E_{-}-(z+z')\p_z E_{+}\big)(t,z,z')\d z'\big|\leq C(\tilde{\mu})
\eeqs
where $C(\tilde{\mu})$ is a constant depending only on $\tilde{\mu}$ (\,in particular, independent of $z$ and $t$).
The first term $(Z_3(2))_1$ can thus be bounded as:
\beqs 
 \il(Z_3(2))_1\il_{0,\infty,t}\lesssim \|(\widetilde{\chi\omega})|_{t=0}\|_{L^{\infty}(\mS_{\kappa})}\lesssim 
 \Lambda\big(\f{1}{c_0},Y_m(0)\big).
\eeqs
Next, by writing 
\beqs
\p_{z'}(z' (\widetilde{\chi\omega})|_{t=0})=(\widetilde{\chi\omega})|_{t=0}+\f{1}{\phi_1(z')}Z_3 (\widetilde{\chi\omega})|_{t=0},
\eeqs
and by observing that $\phi_1(z)$ has a uniform positive lower bound on $[-\kappa,0],$
we control the second term as:
\beqs
\il(Z_3(2))_2\il_{0,\infty,t}\lesssim \|(\text{Id},Z_3)(\widetilde{\chi\omega})|_{t=0}\|_{L^{\infty}(\mS_{\kappa})}\lesssim  \Lambda\big(\f{1}{c_0},Y_m(0)\big).
\eeqs
To summarize, we have obtained that 
\beqs 
\il Z_3(2)\il_{0,\infty,t}\lesssim\Lambda\big(\f{1}{c_0},Y_m(0)\big),
\eeqs
which, together with \eqref{vor-2-1}, yields that:
\beq\label{sec13:eq6}
\il(\text{Id},\ep^{\f{1}{2}}\pt, \p_y, Z_3)(3)\il_{0,\infty,t}\lesssim \Lambda\big(\f{1}{c_0},Y_m(0)\big).
\eeq
\underline{Control of the nonlinear term (3).}
We need to distinguish the terms appearing in $\widetilde{F^{\chi,\omega}}$ that involves one normal derivative of the vorticity and the others. Therefore, let us denote 
\beq\label{sec13:eq0}
\widetilde{F^{\chi,\omega}}=\bar{\rho}\widetilde{\chi u\cdot\nabla^{\vp}\omega}+\bar{\rho}\pt \tilde{\Phi}_t\cdot\nabla (\widetilde{\chi\omega})-\mu\widetilde{ \p_z\chi\bN\cdot\nabla^{\vp}\omega}
+\f{1}{2}\mu \p_z(\ln |g|)\p_z \widetilde{\chi\omega}+R,
\eeq
where the remainder term $R$ satisfies 
the estimate 
\beqs
\begin{aligned}
\|\ep^{\f{1}{2}}\pt R\|_{\hco^{2}}+\|R\|_{\hco^3}&\lesssim \lca(\|\ep^{\f{1}{2}}\pt(\sigma,u ,\nabla u)\|_{\hco^4}+\|(\sigma,u, \nabla u)\|_{\hco^5})\\
&\lesssim \lae.
\end{aligned}
\eeqs
By using the Sobolev embedding 
$H^2(\mR^2)\hookrightarrow L_y^{\infty}(\mR^2)$ 
we can deal with the term
$$\int_0^t \int_{-\infty}^0 E(t-s,z,z')
R(s,y,z')\,\d z'\d s$$ as follows:
\beqs 
\begin{aligned}
&\int_0^t \int_{-\infty}^0 E(t-s,z,z')
(\text{Id},\ep^{\f{1}{2}}\pt, \p_y )R(s,y,z')\,\d z'\d s\\
&\lesssim \big(\int_0^t \int_{-\infty}^0 |E(t-s,z,z')|^2 \d z'\d s\big)^{\f{1}{2}} \|(\text{Id},\ep^{\f{1}{2}}\pt, \p_y )R\|_{L_t^2L_{z'}^2L_y^{\infty}} \\
&\lesssim \big(\int_0^t (t-s)^{-\f{1}{2}}\d s \big)^{\f{1}{2}}
\|(\text{Id},\ep^{\f{1}{2}}\pt, \p_y )R\|_{L_t^2H_{co}^2}\lesssim T^{\f{1}{4}} \lae.
\end{aligned}
\eeqs
Moreover, as in the control of 
$Z_3(2),$ we have that:
\beq\label{sec13:eq1}
\begin{aligned}
&Z_3\int_0^t \int_{-\infty}^0 E(t-s,z,z')R(s,y,z')\,\d z'\d s\lesssim \|(\text{Id},Z_3)R\|_{\hco^{2}}\cdot\\
 &\bigg[ \bigg(\int_0^t \int_{-\infty}^0 |E|^2 \d z'\d s\bigg)^{\f{1}{2}}+\bigg(\int_0^t \int_{-\infty}^0 \big(|(z-z')\p_z E_{-}|^2+ |(z+z')\p_z E_{+}|^2\big)\d z'\d s\bigg)^{\f{1}{2}}\bigg]\\
 &\lesssim T^{\f{1}{4}}\lae.
\end{aligned}
\eeq
We are left to treat the first four terms appearing in \eqref{sec13:eq0}, for which we need to integrate by parts in order not to lose normal derivative. Let us explain the estimate for the term
\beqs 
\bar{\rho}\int_0^t \int_{-\infty}^0 E(t-s,z,z')\widetilde{\chi u\cdot\nabla^{\vp}\omega}\,\d z'\d s
\eeqs
By straightforward calculation, we find that
\begin{align}
\widetilde{\chi u\cdot\nabla^{\vp}\omega}&=\widetilde{\chi u}_k(\D\tilde{\Phi})_{jk}\p_j(\widetilde{\chi_1 \omega})+\widetilde{\chi u}_k(\D\tilde{\Phi})_{3k}\p_z(\widetilde{\chi_1 \omega})\nonumber\\
&=\widetilde{\chi u}_k(\D\tilde{\Phi})_{jk}\p_j(\widetilde{\chi_1 \omega})-\p_z(\widetilde{\chi u}_k(\D\tilde{\Phi})_{3k})\widetilde{\chi_1 \omega}+\p_z(\widetilde{\chi u}_k(\D\tilde{\Phi})_{3k}\widetilde{\chi_1 \omega})\label{change-variableLinfty}
\end{align}
where $\chi_1$ is a cut-off function
supported on $[-C(\kappa),0]$
that satisfies $\chi_1\chi=\chi.$
The Einstein summation convention is used for $j=1,2, k=1,2,3.$ As the first two terms in the right hand side of the above identity does not involve normal derivatives of $(\widetilde{\chi_1 \omega}),$ we have by following the same procedure as in the estimate of $R$ that:
\beqs 
\bar{\rho}\int_0^t \int_{-\infty}^0 E(t-s,z,z') \big(\widetilde{\chi u}_k(\D\tilde{\Phi})_{jk}\p_j(\widetilde{\chi_1 \omega})-\p_z(\widetilde{\chi u}_k(\D\tilde{\Phi})_{3k})\widetilde{\chi_1 \omega}\big) \,\d z'\d s\lesssim T^{\f{1}{4}}\lae.
\eeqs
For the one whose integrand involves the last term of \eqref{change-variableLinfty}, we integrate by parts in $z'$ to get that:
\beqs 
\begin{aligned}
&\bar{\rho}\int_0^t \int_{-\infty}^0 E(t-s,z,z')\p_{z'}(\widetilde{\chi u}_k(\D\tilde{\Phi})_{3k}\widetilde{\chi_1 \omega})\d z'\d s\\
&\lesssim \int_0^t \|\p_{z'}E(t-s,z,\cdot)\|_{L_{z'}^2}\d s \, \|\widetilde{\chi u}_k(\D\tilde{\Phi})_{3k}\widetilde{\chi_1 \omega}\|_{L_t^{\infty}L_{z'}^2L_y^{\infty}}\\
&\lesssim T^{\f{1}{4}} \|\widetilde{\chi u}_k(\D\tilde{\Phi})_{3k}\widetilde{\chi_1 \omega}\|_{L_t^{\infty}H_{co}^2}\lesssim T^{\f{1}{4}}\lae.
\end{aligned}
\eeqs
In addition to the above two inequalities, we have also analogs of \eqref{sec13:eq1}, that is to say:
\beq
\begin{aligned}
&\bar{\rho}(\ep^{\f{1}{2}}\pt, \p_y, Z_3)\int_0^t \int_{-\infty}^0 E(t-s,z,z')\widetilde{\chi u\cdot\nabla^{\vp}\omega}\,\d z'\d s\\
&\lesssim \lae \int_0^t \|E(t-s,z,\cdot),\p_{z'}(E(t-s,z,\cdot),\,(z-\cdot)\p_z E_{-},\,(z+\cdot)\p_z E_{+})\|_{L_{z'}^2}\d s \\
&\lesssim \lae\int_0^t (t-s)^{-\f{3}{4}}\d s\lesssim T^{\f{1}{4}}\lae.
\end{aligned}
\eeq
We have thus finished the estimate of the term
$\int_0^t \int_{-\infty}^0 E(t-s,z,z')\widetilde{\chi u\cdot\nabla^{\vp}\omega}\,\d z'\d s.$
The other three terms in \eqref{sec13:eq0} can be dealt with in the same way. Consequently, we find that for any 
$t\in (0,T], z<0,$
\beq\label{sec13:eq4}
\int_0^t \int_{-\infty}^0 E(t-s,z,z')\widetilde{F^{\chi,\omega}}
\d z'\d s\lesssim T^{\f{1}{4}}\lae.
\eeq
Collecting \eqref{sec13:eq5}, \eqref{sec13:eq6} and \eqref{sec13:eq4}, we find that:
\beqs
\il(\text{Id},\ep^{\f{1}{2}}\pt,\p_y, Z_3)(\widetilde{\chi\omega})\il_{0,\infty,t}\lesssim \Lambda\big(\f{1}{c_0},Y_m(0)\big)+ \Lambda\big(\f{1}{c_0},|h|_{3,\infty,t}\big)\tilde{\cE}_{m,t}+
(T+\ep)^{\f{1}{4}}\lae,
\eeqs
By the property \eqref{Linftyeq-free}, this leads to \eqref{sec13:eq7}.
\end{proof}

\section{Proof of Theorem \ref{thm-localexis}.}
This section is devoted to the proof of Theorem
\ref{thm-localexis} which is based on the known local existence results (non-uniform with respect to $\ep$) and the uniform estimates established in the previous sections. 
The local existence in the Sobolev-Slobodeskii space
$H^{4,2}$ (see the definition \eqref{defs-s})
is established  in \cite{MR697305}  \cite{MR1316494} (see also \cite{Tani-81} for the local existence in H\"older spaces). All these results deal with the case where the reference domain is a smooth bounded domain, nevertheless, by following the same arguments as in these papers, one can easily  obtain a similar result when the reference domain is changed into a strip or half space.
The following theorem corresponds  to Theorem B of \cite{MR697305} %or Theorem 2.1 in \cite{Tanaka-Tani} 
or Theorem 6.2 in \cite{MR1316494} in this framework.

\begin{thm}\label{thm-localexis-nonuniform}
Assume that the compatibility condition $\eqref{Compatibility condition}$ holds up to order 2 and
\beqs
(\sigma_0^{\ep},u_0^{\ep})\in (H^3(\mS))^4,\quad h_0^{\ep}\in H^{\f{7}{2}}(\mR^2),\quad 1+h_0^{\ep}\geq 3c_0>0,
\eeqs
$\delta$ is chosen sufficiently small such that
$$\p_z\vp_0^{\ep}(x)=1+\p_z\eta_0^{\ep}(1+z)+\eta_0^{\ep} \geq 2c_0>0, \forall x\in \mS,$$
where $\eta_0^{\ep}$ is the extension of $h_0^{\ep}$ defined in \eqref{defeta}.
Then %there exists $\ep_1,$ 
for any $\ep\in (0,1],$ we can find $T^{\ep}>0$  such that:
\beqs
(\sigma^{\ep},u^{\ep})\in C([0,T^{\ep}],H^3(\mS)), \quad h^{\ep}\in C([0,T^{\ep}],H^{\f{7}{2}}(\mR^2)).
\eeqs
Moreover, 
\beq\label{defs-s}
u^{\ep}\in H^{4,2}([0,T^{\ep}]\times\mS)=\{u\big| \p_t^j u\in L^2\big([0,T^{\ep}],H^{4-2j}(\mS)\big),j=0,1,2\}
\eeq
and \eqref{preassumption} holds.
\end{thm}
We shall combine this theorem with the uniform regularity estimates established in the previous sections. 
Set 
$$T_{*}^{\ep}=\sup\big\{T|(\sigma^{\ep},u^{\ep})\in C([0,T],H^3(S)), u^{\ep}\in 
W^{4,2}([0,T^{\ep}]\times\mS)
\text{ and } \eqref{preassumption}  \text{ holds}\big\}.$$
Since the initial datum is assumed to belong to $Y_m^{\ep},$ a space with higher regularity, by standard propagation of regularity arguments (for example based on applying finite difference instead of derivatives) and the computations presented in Section 6-Section 12, we can find the following uniform estimates of  Theorem \ref{thm-apriori}:
%By Theorem \ref{thm-apriori}, the following uniform estimate holds: 
\beq\label{uniformestimate1}
\cN_{m,T}^{\ep}\leq P_5\big(\f{1}{c_0}, Y^{\ep}_m(0) \big)+ (T+\ep)^{\vartheta}P_{6}\big(\f{1}{c_0},Y^{\ep}_m(0)+\cN_{m,T}^{\ep}\big).
\eeq
where $0<\vartheta<1$ and $P_{5}, P_{6}$ are two increasing continuous functions that are independent of $\ep.$ 
By the fundamental theorem of calculus
and Lemma \ref{exth}, one finds for $0\leq t\leq T$
\beq
\begin{aligned}
\p_z\vp(t,x)&=\p_z\vp(0,x)+\int_0^t(\p_t\eta+(1+z)\pt\p_z\eta)(s,x)\,\d s\\
&\geq \p_z\vp(0,x)-C_1T|\pt h(t)|_{
L^{\infty}(\mR^2)},
\end{aligned}
\eeq
\beq\label{sec13:eq10}
\|(\nabla\vp,\nabla^2\vp)(t)\|_{L^{\infty}(\mS)}\leq \|(\nabla\vp,\nabla^2\vp)(0)\|_{L^{\infty}(\mS)}+C_2T|h(t)|_{W^{2,\infty}(\mR^2)}.
\eeq
where $C_1,C_2$ are two constants independent of $\ep.$
Moreover, $\ep\sigma^{\ep}$ can be expanded by using the characteristic method:
\beq\label{lagranian}
\ep\sigma^{\ep}(t,x)=\ep\sigma_0^{\ep}(X^{-1}(t,x))
-\int_0^t (\div u^{\ep}/g_1)(X(s,X^{-1}(t,x)))\d s
\eeq
where $X(t,x)$ is the unique flow associated to $\underline{u}.$
Let us define      
\beqs
T^{\ep}_{*}=\sup\{T\geq 0\big| (\sigma^{\ep},u^{\ep})\in C([0,T],H^3), u^{\ep}\in W^{4,2}([0,T]\times\mS)\},
\eeqs
\beqs
\begin{aligned}
T_0^{\ep}=\sup\big\{0\leq T\leq\min\{T^{\ep}_{*},1\}&\big| \cN_{m,T}(\sigma^{\ep}, u^{\ep})\leq 2P_5\big({1}/{c_0}, M\big)\\
%&-2{\bar{c}}\bar{P}\leq \ep\sigma^{\ep}(t,x)\leq 2\bar{P}/\bar{c} \quad
&
\eqref{asstobeshown} \text{ holds
for all } (t,x)\in [0,T]\times\mS \big\}.
\end{aligned}
\eeqs
where $M$ is chosen such that $M \geq \sup_{\ep\in(0,1]}Y_m(\sigma_0^{\ep}, u_0^{\ep}).$

 We now choose successively two  constants $0<\ep_0\leq 1$ and $T_0>0$ (uniform in $\ep\in(0,\ep_0]$) which are small enough, such that:
 \beqs
 ({T_0}+\ep_0)^{\vartheta}%\Lambda_0
 P_6\big(1/c_0, M+2P_5(1/c_0,M)\big)<\f{1}{2}P_5({1}/{c_0},M),
 \eeqs
 \beqs 
 C_1T_0 P_5(1/c_0,M)^2\leq c_0,\quad 
 C_2T_0  P_5(1/c_0,M)\leq 1/(2c_0),\quad
  2P_5(1/c_0,M)T_0/c_0\leq \bar{c}\bar{P}.
 \eeqs
In order to prove Theorem \ref{thm-localexis}, it suffices  to show that $T_0^{\ep}\geq {T_0}$ for every  $0<\ep\leq \ep_0.$ Suppose otherwise $T_0^{\ep}<{T_0}$ for some $0<\ep\leq \ep_0,$ then in view of  inequalities
\eqref{uniformestimate1}-\eqref{sec13:eq10} and the formula \eqref{lagranian}, we have by the definition of $\ep_0$ and $T_0$ that:
\beq\label{sec6:eq3}
\cN_{m,T}(\sigma^{\ep},u^{\ep})\leq \f{3}{2}P_5({1}/{c_0},M) \qquad \forall T\leq \tilde{T}=\min\{T_0, T^{\ep}_{*}\},
\eeq
\beq\label{sec6:eq4}
\p_z\vp^{\ep}(t,x)\geq c_0, \quad 
|(\nabla\vp^{\ep},\nabla^2\vp^{\ep})(t,x)|\leq {1}/{c_0}, \quad -2{\bar{c}}\bar{P}\leq \ep\sigma^{\ep}(t,x)\leq 2\bar{P}/\bar{c}\quad\forall (t,x)\in [0,\tilde{T}]\times\Omega.
\eeq
We intend to prove that $\tilde{T}=T_0\leq T_{*}^{\ep}.$ This fact, combined with the definition of $T_0^{\ep}$ and the estimates \eqref{sec6:eq3}, \eqref{sec6:eq4},
yields $T_0^{\ep}\geq T_0,$ which is a contradiction with the  assumption $T_0^{\ep}<T_0.$ To continue, we shall need the claim stated and proved below.
 Indeed, once the following claim holds, we have by \eqref{sec6:eq3} that 
 $\|(\sigma^{\ep},u^{\ep})(T_0)\|_{H^3(\Omega)}<+\infty.$ Using the local existence result stated in 
 Theorem \ref{thm-localexis-nonuniform}, we obtain that $T_{*}^{\ep}>T_0=\tilde{T}.$

 $\textbf{Claim.}$ 
For all $\ep\in(0,1],$
if $\cN_{m,T}(\sigma^{\ep},u^{\ep})<+\infty,$ then $(\sigma^{\ep},u^{\ep})\in C([0,T], H^3),$
 $u^{\ep}\in H^{4,2}([0,T]\times\mS).$
 \begin{proof}[Proof of claim]
By the definition of $\cN_{m,T},$ we derive that:
$$\ep^{\f{3}{2}} u^{\ep}\in 
L^2([0,T],H^4),\quad \ep^{\f{3}{2}}\pt u^{\ep}\in L^2([0,T],H^2), \quad \ep^{\f{3}{2}}\pt^2 u\in L^2([0,T],L^2) \quad \ep^{\f{1}{2}} \sigma^{\ep} \in L^{\infty}([0,T],H^3), $$
which yields by interpolation that $\ep^{\f{3}{2}} u^{\ep}\in C([0,T],H^3)\cap H^{4,2}([0,T]\times\mS).$
Moreover, carrying out direct energy estimates for $\sigma^{\ep}$ in $H^3(\Omega),$
one gets that:
\beq\label{ineq-gronwall}
|\pt R^{\ep}(t)|
\leq  K^{\ep}%\big(R^{\ep}(t)+\big)
f^{\ep}(t)
\eeq
where  $K^{\ep}=\Lambda(1/c_0, \il(\sigma^{\ep},\nabla\sigma^{\ep},\nabla u^{\ep},\ep^{\f{1}{2}}\nabla^2 u^{\ep})\il_{\infty,t})$ is  uniformly bounded
and $$R^{\ep}(t)=\|\ep^{\f{1}{2}}\sigma^{\ep}(t)\|_{H^3}^2,\quad f^{\ep}(t)=\|\ep^{\f{3}{2}} u^{\ep}(t)\|_{H^4}^2+\|\ep^{\f{1}{2}}u^{\ep}(t)\|_{H^3}^2+\|(\sigma^{\ep},\ep^{-\f{1}{2}}\nabla\sigma^{\ep})(t)\|_{H^2}^2\in L^1([0,T]).$$
Inequality \eqref{ineq-gronwall} and the boundedness of $\|R^{\ep}(\cdot)\|_{L^{\infty}([0,T])}$
leads to the fact that $R^{\ep}(\cdot)\in C([0,T]),$ which further yields that $\ep^{\f{1}{2}}\sigma^{\ep}\in C([0,T],H^3).$ This ends  the proof of the claim.
Note that at this stage we do not require the norm $\|(\sigma^{\ep},u^{\ep})\|_{C([0,T], H^3)}$ to be bounded uniformly in $\ep$.
\end{proof}

\section{Convergence}
%Based on the uniform estimates established in  the convergence of the solutions of compressible free surface Navier-Stokes equations $\eqref{FCNS1}$ to the ones of incompressible equations are quite easy,we write it briefly just for the reader's convenience.
This section aims to show Theorem \ref{thm-convergence}.
%which
%is essentially the consequence of uniform estimates established in Theorem \ref{thm-localexis} and compactness arguments. 
In the following, we denote $Q_{T_0}=[0,T_0]\times \mS,\, \Gamma_{T_0}=[0,T_0]\times\mathbb{R}^2.$ 

First of all, for the surface, since
$\pt h^{\ep}$ is uniformly bounded in 
$L^{\infty}([0,T_0], H^{m-{3}/{2}}(\mR^2)),$ $h^{\ep}$ is uniformly bounded in  $L^{\infty}([0,T_0],H^{m-{1}/{2}}(\mR^2)),$ one has that
$h^{\ep}$ converges (say to $h^0$) in $C([0,T_0],H_{loc}^{s}(\mR^2))$ for any $0\leq s< m-{1}/{2}.$ Further, from the definition of $\vp^{\ep}$ \eqref{defvpep} and Lemma 
\eqref{exth}, we conclude also that
 $\vp^{\ep}\rightarrow \vp_0$ in 
 $C([0,T_0],H_{loc}^{s}(\mS)), 0\leq s< m$ where $\vp^0$ is defined in a similar way as \eqref{defvpep} by replacing $h^{\ep}$ with $h^0.$

Next, since $(\ep^{\f{1}{2}}\pt\sigma^{\ep}, \ep^{\f{1}{2}}\sigma^{\ep})$ is uniformly bounded in $L^{\infty}([0,T_0],H^1(\mS))\times L^{\infty}([0,T_0],H^3(\mS)),$ we have that $\ep^{\f{1}{2}}\sigma^{\ep}$ 
is uniformly bounded %and converges to $0$ 
in $C^{\gamma}(Q_{T_0}), 0<\gamma<\f{1}{2}.$ 
In view of the definition of $\sigma^{\ep}: \sigma^{\ep}=(P(\rho)-P(\bar{\rho}))/{\ep},$ 
we have that $P(\rho^{\ep})\rightarrow P(\bar{\rho})$
in %$C([0,T_0],H^1(\mS)),$
$C^{\gamma}(Q_{T_0}),$ which, combined with the uniform boundedness of $\il\nabla P(\rho^{\ep})\il_{\infty,t},$ yields the convergence of $\rho^{\ep}$ to $\bar{\rho}$ in $C^{\gamma}(Q_{T_0}).$ 

Let us see the convergence of the velocity. 
We write $u^{\ep}=\nabla^{\vp^{\ep}}\Psi^{\ep}+v^{\ep},$ where 
$\nabla^{\vp^{\ep}}\Psi^{\ep}$ and $v^{\ep}$ denote the compressible and incompressible part of the velocity (see definitions \eqref{defofQ}, \eqref{defofP}). On the one hand, since
%$\|\div^{\vp^{\ep}}u^{\ep}\|_{L^{\infty}([0,T_0],H^1(\mS))}=\cO(\ep^{\f{1}{2}}),$
$\ep^{-\f{1}{2}}\div^{\vp^{\ep}}u^{\ep},$
$\ep^{\f{1}{2}}\pt\div^{\vp^{\ep}}u^{\ep}$ are both uniformly bounded in $L^{\infty}([0,T_0],H^1(\mS)),$ we get that
$\div^{\vp^{\ep}}u^{\ep}\rightarrow 0$ in
$C^{\gamma}([0,T_0], H^1(\mS)), 0<\gamma<\f{1}{2}.$ By elliptic estimates
\eqref{elliptic1.5},
$\nabla^{\vp^{\ep}}\Psi^{\ep}\rightarrow 0$ in 
$C^{\gamma}([0,T_0], H^2(\mS)).$ 
On the other hand, due to the uniform boundedness of $\pt v^{\ep}$  in $L^{2}([0,T_0],H^{-1}(\mS)),$  and of $v^{\ep}$
 in $L^{\infty}([0,T_0],H^1(\mS)),$
we obtain by Aubin-Lions lemma that up to extraction of subsequences, 
$v^{\ep}$ converges (say to $u^0$) in 
$C([0,T_0],L_{loc}^2(\mS)).$ 
Since we will prove that $u^0$ is the \textit{unique} solution (in conormal spaces with additional regularity property), to 
the incompressible free-surface Navier-Stokes equations  this convergence holds indeed for the whole family. 
We thus proved that $u^{\ep}$ converges to $u^0$ in $C^{\gamma}([0,T_0],H^1(\mS))+ C([0,T_0],L_{loc}^2(\mS)).$

To conclude, we have achieved that 
\beq\label{convergence}
\sigma^{\ep}\rightarrow 0 %\text{ in}\quad C^{\gamma}(Q_{T_0}), 
\quad  \rho^{\ep}\rightarrow \bar{\rho} \quad \nabla^{\vp^{\ep}}\Psi^{\ep}\rightarrow 0 \,\text{ in}\quad  C^{\gamma}(Q_{T_0}) \quad v^{\ep}\rightarrow u^0 \quad\text{in} \quad C\big([0,T_0],L_{loc}^2\big), 
\eeq
\beq\label{convergence1}
\vp^{\ep}\rightarrow \vp^0 \text{ in }
C\big([0,T_0],H_{loc}^{s}(\mS))\quad
\quad h^{\ep}\rightarrow
h^{0}
\quad\text{in} \quad C\big([0,T_0],H_{loc}^{s}(\mathbb{R}^2)\big), \quad 0\leq s\leq m-\f{1}{2}.
\eeq

We now show that there exists $\pi_0\in L^2([0,T_0],\cH^{0,m-1})$ such that $(u^0,\pi^0,h^0)$ is the (unique) solution to the incompressible free surface system \eqref{FINS1}. Let us rewrite the equations for the incompressible part of the velocity (see \eqref{eqofv2}) as follows:
\beq\label{re-eqv}
\bar{\rho}(\pt^{\vp^{\ep}}v^{\ep}+v^{\ep}\cdot\nabla^{\vp^{\ep}}v^{\ep})%+ v^{\ep}\cdot\nabla^{\vp_{\ep}}\nabla^{\vp^{\ep}}\Psi^{\ep}
-\mu\Delta^{\vp^{\ep}}v^{\ep}
+\nabla^{\vp^{\ep}}\tilde{\pi}^{\ep}=F^{\ep}.
\eeq
where 
\begin{align*}
&\nabla^{\vp^{\ep}}\tilde{\pi}^{\ep}=\nabla^{\vp^{\ep}}({\pi}^{\ep}-q^{\ep})-[\p_t^{\vp^{\ep}},\bbp]u^{\ep},\\
&F^{\ep}=\ep\f{g_2-1}{\ep}(\pt+\underline{u^{\ep}}\cdot\nabla)u^{\ep}-\bar{\rho}(v^{\ep}\cdot(\nabla^{\vp^{\ep}})^2\Psi^{\ep}+\nabla^{\vp^{\ep}}\Psi^{\ep}\cdot\nabla^{\vp^{\ep}} u^{\ep}).
\end{align*}
with $\nabla^{\vp^{\ep}}\pi^{\ep}, \nabla^{\vp^{\ep}}q^{\ep}$  defined in \eqref{defpi}. Note that by the definition \eqref{defofQ}, \eqref{defofP} for $\bbq,\bbp,$ the commutator $-[\p_t^{\vp^{\ep}},\bbp]u^{\ep}$ can be expressed as a gradient:
\beq\label{re-comtime}
-[\p_t^{\vp^{\ep}},\bbp]u^{\ep}=[\p_t^{\vp^{\ep}},\bbq]u^{\ep}=\nabla^{\vp^{\ep}}(\pt^{\vp^{\ep}}\Psi^{\ep}-\tilde{\Psi}^{\ep})
\eeq
where we denote $\nabla^{\vp^{\ep}}\tilde{\Psi}^{\ep}=\bbq (\p_t^{\vp^{\ep}} u^{\ep}).$ 
By estimates established in \eqref{q}, \eqref{pi2} and  \eqref{comtime}, we readily see that $\nabla\tilde{\pi}^{\ep}$ is uniformly bounded in $L^2([0,T_0],\cH^{0,m-2}).$ Therefore, there exists $\pi^0\in L^2([0,T_0],\cH^{0,m-1})$ such that $\nabla \tilde{\pi}^{\ep}$ tends (up to subsequences) to $\nabla \pi^0$ in $L_{w}^2(Q_{T_0})$ and 
$\tilde{\pi}^{\ep}$ converges to $\pi_0$ in $L_{w}^2([0,T_0],L_{loc}^2(\mS)).$
Next, by boundary conditions $\eqref{eqofv2}_2-\eqref{eqofv2}_3$ as well as the fact \eqref{re-comtime}, we have that:
\begin{align}
  (2\mu S^{\vp^{\ep}} u^{\ep} -\tilde{\pi}^{\ep} \text{Id})\bN^{\ep}= 2\mu (\div^{\vp^{\ep}}u \text{Id}-(\nabla^{\vp^{\ep}})^2\Psi^{\ep})\bN^{\ep}+(\f{\pt h^{\ep}}{\p_z\vp^{\ep}}\p_z\Psi^{\ep})\bN^{\ep} \quad \text{ on } z=0, \label{re-bdry-1}\\
  v_3^{\ep}=0,\, \mu\p_z^{\vp^{\ep}} v_{j}^{\ep}= a u_j^{\ep}\quad (j=1,2) \qquad \text{ on } {z=-1}. \label{re-bdry-2}
\end{align}

Let us now choose a smooth vector $\psi =(\psi_1,\psi_2,\psi_3)^t\in \big[C_{c}^{\infty}\big(\overline{Q_{T_0}}\big)\big]^3$ with condition $\psi_3|_{z=-1}=0.$ 
Multiplying the equations \eqref{re-eqv} by $\psi$ and integrating by parts in space and time, we find  by using the boundary conditions \eqref{re-bdry-1}, \eqref{re-bdry-2}
that: 
\beq\label{weakformu-incom}
\begin{aligned}
&\bar{\rho}\int_{\mS} (v^{\ep}\cdot \psi)(t,\cdot) \,\d\cV_t^{\ep} + 2\mu \int_0^t\int_{\mS} S^{\vp^{\ep}}v^{\ep} \cdot \nabla^{\vp^{\ep}}\psi\,\d\cV_s^{\ep}\d s + \bar{\rho}\int_0^t\int_{\mS}(v^{\ep}\cdot\nabla^{\vp^{\ep}} v^{\ep})\cdot \psi\,\d\cV_s^{\ep}\d s\\
&=\bar{\rho}\int_{\mS} (v^{\ep}\cdot \psi)(0,\cdot) \,\d\cV_0^{\ep}+\int_0^t\int_{\mS}F^{\ep}\cdot\psi\,\d\cV_s^{\ep}\d s+\bar{\rho}\int_0^t\int_{\mS}v^{\ep}\cdot\pt^{\vp^{\ep}}\psi\,\d\cV_s^{\ep}\d s+\int_0^t\int_{\mS}\tilde{\pi}^{\ep}\div^{\vp^{\ep}}\psi\,\d\cV_s^{\ep}\d s\\
&\qquad+ a
\int_0^t\int_{z=-1}(u_1^{\ep}\cdot\psi_1+u_2^{\ep}\cdot\psi_2)\,\d y\d s+\int_0^t\int_{z=0} (v^{\ep}\cdot \bN^{\ep}) (v^{\ep}\cdot\psi) \,\d y\d s\\
&\qquad+\int_0^t\int_{z=0} (2\mu \div^{\vp^{\ep}}u^{\ep}+\f{\pt h^{\ep}}{\p_z\vp^{\ep}}\p_z\Psi^{\ep})(\psi\cdot\bN^{\ep})- (\nabla^{\vp^{\ep}})^2\Psi^{\ep}\bN^{\ep}\cdot \psi \,\d y\d s
\end{aligned}
\eeq
where $\d\cV_t^{\ep}=\f{1}{\p_z \vp^{\ep}}(t,\cdot)\,\d y \d z.$
Since $v^{\ep}\rightarrow v^0$ in $C([0,T_0],L_{loc}^2(\mS)),$ $\p_z\vp^{\ep}$ converges to $\p_z \vp^0$ in $C([0,T_0],C_{loc}(\mS)),$ we see that:
\beq\label{timebdry-conv}
\bar{\rho}\int_{\mS} (v^{\ep}\cdot \psi)(t,\cdot) \,\d\cV_t^{\ep}\rightarrow \bar{\rho}\int_{\mS} (u^{0}\cdot \psi)(t,\cdot) \,\d\cV_t^{0},\quad 
\bar{\rho}\int_{\mS} (v^{\ep}\cdot \psi)(0,\cdot) \,\d\cV_0^{\ep}\rightarrow \bar{\rho}\int_{\mS} (u^{0}\cdot \psi)(0,\cdot) \,\d\cV_0^{0}
\eeq
Let us now show the convergence of  the last two terms in the left hand side of the above identity. Since
\beq\label{fact1} v^{\ep}\rightarrow u^0 \text{ in } L^2([0,T_0],L^2_{loc}(\mS)),
\, \nabla v^{\ep} \rightharpoonup  \nabla u^0\,
\text{ in } L^2(Q_{T_0}),\, v^{\ep} \text{ uniformly bounded in  } L^2([0,T_0],H^1(\mS))
\eeq
\beq\label{fact2}
\vp^{\ep}\rightarrow \vp^0 \,\text{ in } C([0,T_0],C^1_{loc}(\mS)), \qquad (\p_z{\vp}^{\ep},\p_z\vp_0)(t,x)\geq c_0>0, \forall (t,x)\in Q_{T_0} 
\eeq
one gets that: 
$\mS^{\vp^{\ep}}v^{\ep}\rightharpoonup \mS^{\vp_0} v^0,$  $\nabla^{\vp^{\ep}}\psi\rightarrow \nabla^{\vp^{0}}\psi $ in  $L^2(Q_{T_0}),$ which leads to the fact:
\beq\label{L-term23}
\begin{aligned}
&2\mu \int_0^t\int_{\mS} S^{\vp^{\ep}}v^{\ep} \cdot \nabla^{\vp^{\ep}}\psi\,\d\cV_s^{\ep}\d s + \bar{\rho}\int_0^t\int_{\mS}(v^{\ep}\cdot\nabla^{\vp^{\ep}} v^{\ep})\cdot \psi\,\d\cV_s^{\ep}\d s\\
  &\rightarrow  2\mu \int_0^t\int_{\mS} S^{\vp^{0}}u^{0} \cdot \nabla^{\vp^{0}}\psi\,\d\cV_s^{0}\d s + \bar{\rho}\int_0^t\int_{\mS}(u^{0}\cdot\nabla^{\vp^{0}} u^{0})\cdot \psi\,\d\cV_s^{0}\d s
\end{aligned}
\eeq 
It suffices to deal with the 
convergence of the the last four terms  in the right hand side of \eqref{weakformu-incom}.
As $\nabla^{\vp^{\ep}}\psi^{\ep}=\cO (\ep^{\f{1}{2}})$ in  $L_t^2H^1$ and $(u^{\ep},\ep^{\f{1}{2}}\pt u^{\ep})$ uniformly bounded in $L^2([0,T_0],H^1(\mS)),$ one readily see that $F^{\ep}\rightarrow 0$ in $L^2(Q_{T_0}),$ which gives that:
\beq
\int_0^t\int_{\mS}F^{\ep}\cdot\psi\,\d\cV_s^{\ep}\d s\rightarrow 0.
\eeq
Next, since $\pt \vp^{\ep}\rightarrow \pt \vp^0$ in $L_{w}^2([0,T_0],L^2(\mS)),$ we have by combining \eqref{fact2} that $\pt^{\vp^{\ep}}\psi\rightharpoonup \pt^{\vp^{0}}\psi$ in $L^2(Q_{T_0})$ This, together with \eqref{fact1} gives that: 
\beq\label{R-term3} 
\bar{\rho}\int_0^t\int_{\mS}v^{\ep}\cdot\pt^{\vp^{\ep}}\psi\,\d\cV_s^{\ep}\d s\rightarrow \bar{\rho}\int_0^t\int_{\mS}u^{0}\cdot\pt^{\vp^{0}}\psi\,\d\cV_s^{0}\d s.
\eeq
As for  \eqref{L-term23}, we have also that:
\beq 
\int_0^t\int_{\mS}\tilde{\pi}^{\ep}\div^{\vp^{\ep}}\psi\,\d\cV_s^{\ep}\d s\rightarrow \int_0^t\int_{\mS}{\pi}^{0}\div^{\vp^{0}}\psi\,\d\cV_s^{0}\d s.
\eeq
To proceed, we prove that $(u^{\ep})^{b,j},(v^{\ep})^{b,j}$ both convergent to $(u^0)^{b,j}$ in $L_{loc}^2([0,T_0]\times\mR^2)$ where $j=1,2.$
Indeed, by the trace inequality  and the fact \eqref{fact1}, one has for any $K\subset \mR^2$ compact,
\beqs 
|(v^{\ep})^{b,j}-(u^0)^{b,j}|_{L^2([0,T_0]\times K)}\lesssim \|v^{\ep}-u^0|_{L^2([0,T_0],L^2(\tilde{K}\times [-1,0])}^{\f{1}{2}}\|v^{\ep}-u^0|_{L^2([0,T_0],H^1(\mS)}^{\f{1}{2}}\rightarrow 0.
\eeqs
where $\tilde{K}\subset \mR^2$ is a compact set such that $K\Subset \tilde{K}.$
The same argument applies also for $u^{\ep}.$ Therefore, one deduces that:
\beq\label{R-term4}
\begin{aligned}
a\int_0^t\int_{z=-1}(u_1^{\ep}\cdot\psi_1+u_2^{\ep}\cdot\psi_2)\,\d y\d s+\int_0^t\int_{z=0} (v^{\ep}\cdot \bN^{\ep}) (v^{\ep}\cdot\psi) \,\d y\d s\\
\rightarrow a\int_0^t\int_{z=-1}(u_1^{0}\cdot\psi_1+u_2^{0}\cdot\psi_2)\,\d y\d s+\int_0^t\int_{z=0} (u^{0}\cdot \bN^{0}) (u^{0}\cdot\psi) \,\d y\d s
\end{aligned}
\eeq
 Finally, by the trace inequality 
 $\div^{\vp^{\ep}} u^{\ep}, \nabla^{\vp^{\ep}}\Psi^{\ep},  (\nabla^{\vp^{\ep}})^2\Psi^{\ep}=\cO(\ep^{\f{1}{2}})$ in $L^2([0,T_0],L^2(\mR^2)),$ which yields that:
\beq\label{R-term5}
\int_0^t\int_{z=0} (2\mu \div^{\vp^{\ep}}u^{\ep}+\f{\pt h^{\ep}}{\p_z\vp^{\ep}}\p_z\Psi^{\ep})(\psi\cdot\bN^{\ep})- (\nabla^{\vp^{\ep}})^2\Psi^{\ep}\bN^{\ep}\cdot \psi \,\d y\d s\rightarrow 0.
\eeq
Plugging \eqref{timebdry-conv} and \eqref{L-term23}-\eqref{R-term5} into \eqref{weakformu-incom}, we find that $(u^0,\pi^0,h^0)$ satisfies \eqref{def-weak}. Finally, it is direct to see that $u^0$ has the  additional regularity \eqref{addi-regu}. In particular, $u^0$ is Lipschitz continuous, which is sufficient to verify the uniqueness. For the reader's convenience, we will sketch the proof in the following subsection.
\subsection{Uniqueness of limit system.}\label{subsection-uniqueness}
Suppose that there are two solutions $(h^1,u^1,\nabla\pi^1)$ and $(h^2,u^2,\nabla\pi^2)$ to the system \eqref{FINS1}-\eqref{FCNS1-bc} on the time interval $[0,T_0]$ with the same initial data ($\vp^1,\vp^2$ are defined through \eqref{defvpep} and \eqref{defeta} associated to $h^1,h^2$). Let $h=h^1-h^2,u=u^1-u^2, \pi=\pi^1-\pi^2.$ 
We prove that $h=0,u=0.$ 
By direct calculation, we find that $(h,u)$ solves the following system:
\begin{align}
    \pt h+(u^1)^{b,1}\cdot\nabla_y h+u^{b,1}\cdot\nabla_y h^2+u_3^{b,1}=0
\end{align}
\begin{align}
    (\pt+\underline{u}^1\cdot\nabla)u+\na^{\vp^1}\pi-\mu\Delta^{\vp^1}u=F
\end{align}
where 
\beq\label{def-F-unique}
F=-(\underline{u}^1-\underline{u}^2)\cdot\nabla u^2+(\nabla^{\vp^2}-\nabla^{\vp^1})\pi^2+\mu(\Delta^{\vp^1}-\Delta^{\vp^2})u^2,\quad \underline{u}^i=(u_1^i,u_2^i,\f{u^i\cdot\bN^i-\pt\vp^i}{\p_z\vp^i}), i=1,2,
\eeq
and with boundary conditions:
\begin{align}\label{bdrycon-diff}
    (S^{\vp^1}u-\pi\text{Id}_3)\bn^1=[(S^{\vp^2}-S^{\vp^1})u^2]\bn^1+(-S^{\vp^2}u^2+\pi^2\text{Id}_3)(\bn^1-\bn^2) \quad \text{ on } \{z=0\},\\
    \mu\p_z u_j=a u_j \,(j=1,2) \quad u_3=0
    \quad \text{ on } \{z=-1\}.\label{bdrycon-diff0}
\end{align}
Define $$E(t)=|h(t)|_{H^{\f{3}{2}}(\mR^2)}^2+\|(u,\p_y u)(t)\|_{L^2(\mS)}^2.$$ 
It suffices to prove that 
\beq\label{unique-ineq}
E(t)+\int_0^t \|\nabla(u,\p_y u)(s)\|_{L^2(\mS)}^2 \d s \leq \Lambda(R)\int_0^t E(s) \d s, \,\forall t\in [0,T_0].
\eeq
where $$R=\sum_{i=1}^2\big(\il(u^i,\nabla u^i,\p_y \nabla u^i)\il_{0,\infty,t}+\il(\pi^i,\nabla\pi^i)\il_{0,\infty,t}+|h^i|_{L_t^{\infty}H^4}\big).$$
Direct energy estimates on $h$ lead to:
\beq\label{unique-h}
\begin{aligned}
|h(t)|_{H^{\f{3}{2}}(\mR^2)}^2&\lesssim \Lambda(R)
\big(|h|_{L_t^2H^{\f{3}{2}}(\mR^2)}^2+|h|_{L_t^2H^{\f{3}{2}}(\mR^2)}|u^{b,1}|_{L_t^2H^{\f{3}{2}}(\mR^2)}\big)\\
&\leq \f{1}{2} %\int_0^t\int_{\mS} |\nabla(u,\p_y u)|^2 \d \cV^1\d s
\int_0^t\|\nabla(u,\p_y u)(s)\|_{L^2(\mS)}^2 \,\d s+\Lambda(R)\int_0^t E(s) \,\d s.
\end{aligned}
\eeq
Thanks to Lemma \ref{lemipp} and boundary condition %\eqref{bdrycon-diff}, 
\eqref{bdrycon-diff0}, we can obtain the energy equality:
\beqs 
\begin{aligned}
&\f{1}{2} \int_{\mS} |u(t)|^2\,\d \cV_t^1
+2\mu\int_0^t\int_{\mS} |S^{\vp^1}u|^2\,
\d \cV_s^1 \d s+a \int_0^t \int_{\mR^2}|u|^2 \,\d y\d s\\
&= \int_0^t\int_{\mS} \pi \, \div^{\vp^1} u\, \d \cV_s^1 \d s + \int_0^t\int_{\mS} F\cdot u \, \d \cV_s^1 \d s+2\mu \int_0^t  (S^{\vp^1}u-\pi\text{Id}_3)\bn^1 \cdot u \,\d y \d s,
\end{aligned}
\eeqs
where $\d \cV_t^1=\p_z\vp^1(t,\cdot)\,\d x.$
In light of the definition \eqref{def-F-unique} for $F,$ boundary condition \eqref{bdrycon-diff} as well as the identity:
\beqs 
\div^{\vp^1} u=(\div^{\vp^2}-\div^{\vp^1})u^2,
\eeqs
 we can obtain, after lengthy but direct computations, that:
 \beqs 
 \int_{\mS} |u(t)|^2\d \cV^1
+\int_0^t\int_{\mS} |\nabla u|^2
\d \cV^1 \d s \leq 
\Lambda(R)\big(\int_0^t E(s) \,\d s+ \|\pi \|_{L_t^2L^2(\mS)}|h|_{L_t^2H^{\f{1}{2}}}\big).
 \eeqs
Following similar arguments, one can also show that:
 \beqs 
 \int_{\mS} |\p_y u(t)|^2\d \cV^1
+\int_0^t\int_{\mS} |\nabla \p_y u|^2
\d \cV^1 \d s \leq 
\Lambda(R)\big(\int_0^t E(s) \,\d s+ \|\pi \|_{L_t^2H^1(\mS)}|h|_{L_t^2H^{\f{3}{2}}}\big).
 \eeqs
 By the elliptic estimates performed in Section 5, we can find that:
 \beqs 
 \|\pi \|_{L_t^2H^1(\mS)}\lesssim \Lambda(R)\big(|h|_{L_t^2H^{\f{3}{2}}(\mR^2)}+\|(u,\p_y u)\|_{L_t^2H^1(\mS)}\big).
 \eeqs
 Combining the previous three inequalities and using Young's inequality, we have:
 \beqs 
 %\int_{\mS} |(u,\p_y u)(t)|^2\,\d \cV^1
 \|(u,\p_y u)(t)\|_{L^2(\mS)}^2+\int_0^t \|\nabla(u,\p_y u)(s)\|_{L^2(\mS)}^2 \d s
\leq 
\Lambda(R)\int_0^t E(s) \,\d s.
 \eeqs
 Together with \eqref{unique-h}, this yields \eqref{unique-ineq}.
 
\section{Remarks for other reference domains.}\label{rmkhalfspace}
In this section, we shall explain how to extend the uniform estimates results established in sections 5-12 to the case when the reference domain is 
a channel with infinite depth or a bounded domain. We will only explain the former case since the latter can be dealt with by using the similar covering as in \cite{MR3741102} and by working in  local coordinates based on the former case.

Assume now that $\Omega^{\ep}_t$ is given by:
$$\Omega^{\ep}_t=\{x=(y,z)|\, y\in\mR^2, z<h^{\ep}(t,y)\}.$$ 
The first step is still to use the so-called harmonic extension transformation to reduce the problem to a fixed domain. Consider the map
\beq
\begin{aligned}
\Phi_t^{\ep}:\mR_{-}^3&\rightarrow\Omega_t^{\ep}\\
(y,z)&\rightarrow \Phi^{\ep}(t,y,z)
=(y,\vp^{\ep}(t,y,z))^t
\end{aligned}
\eeq
where 
\beq\label{defvpep-1}
\vp^{\ep}(t,y,z)=A z+\eta^{\ep}(t,x)
\eeq
Here $\eta$ is given by \eqref{defeta}
and $A$ is a constant which is chosen sufficiently large such that $\p_z\varphi^{\ep}>0.$
We introduce the conormal vector fields  $$ Z_0=\ep\p_t, \quad Z_1=\p_{y_1}, \quad Z_2=\p_{y_2}, \quad Z_3=\phi(z)\p_z.$$
where the weight function $\phi(z)=z/(1-z).$  We can define  conormal spaces analogous  to those in Section 1.2 by using these vector fields. %\eqref{defconormal1}
Furthermore, we can use the quantity $\cN_{m,T}^{\ep}$ defined in \eqref{defcN} (with the conormal norms being changed accordingly in the current definition). The projections $ \bbq, \bbp$ that send a vector field 
in $(L^2(\mR_{-}^3\,\d \cV_t))^3, ( \d \cV_t=\p_z\vp\,\d y\d z) $ to its compressible part and incompressible part are defined as: $\mathbb{P}_t=\text{Id}-\mathbb{Q}_t$ and
\begin{equation}
   \begin{aligned}
   \mathbb{Q}_t: \quad &
L^2(\mR_{-}^3\,\d \cV_t)^3\rightarrow L^2(\mR_{-}^3\,\d \cV_t)^3\\
&\qquad\qquad f \rightarrow \mathbb{Q}_t f=\nabla^{\vp^{\ep}}\vr
   \end{aligned} 
\end{equation}
where $\vr$ satisfies the elliptic equation with trivial Dirichlet boundary condition:
\beq\label{defofQ-1}
\left\{
\begin{array}{l}
     -\Delta^{\vp^{\ep}}\vr=-\div^{\vp^{\ep}}f \quad \text{ in } \mR_{-}^3 \\[5pt]
     \vr|_{z=0}=0
\end{array}
\right.
\eeq
Denote further $v^{\ep}=\bbp u^{\ep}, \nabla^{\vp^{\ep}}\Psi^{\ep}=\bbq u^{\ep}.$

%Almost all the computations done in Section 5-12 are the same ().
Following the similar (and even easier since there is no lower boundary) computations done in Section 5-12, we can prove  uniform estimates analogous to those of   Theorem \ref{thm-apriori}, we thus do not detail them.
We comment that one crucial point that we have used in the computations is that $\il\nabla\Psi^{\ep}\il_{0,\infty,t}$ can be controlled by the $L_t^{\infty}H_{co}^1$ norm of $\div^{\vp^{\ep}} u^{\ep}$ (rather than $u^{\ep}$) which has a size of $\ep^{\f{1}{2}}.$ 
This is achieved by Sobolev embedding and
elliptic estimate similar to \eqref{elliptic1.5}. In the current situation, due to the lack of suitable Poincar\'e  inequality, only $\|\nabla^2\Psi\|_{L_{t}^{\infty}H_{co}^1}$ (but not $\|\nabla\Psi^{\ep}\|_{L_{t}^{\infty}H_{co}^2}$ ) can be controlled 
by $\|\div^{\vp}u^{\ep}\|_{L_{t}^{\infty}H_{co}^1}.$ Nevertheless, %when the reference space is changed by $\mR_{-}^3,$ 
in the current situation, one has the following Sobolev embedding:
\beqs
\|f\|_{L^{\infty}(\mR_{-}^3)}\lesssim \|\nabla f\|_{H_{tan}^1(\mR_{-}^3)}
\eeqs
which leads to:
\beqs
\il\nabla\Psi^{\ep}\il_{0,\infty,t}\lesssim \il\nabla^2\Psi^{\ep}\il_{L_{t}^{\infty}H_{co}^1}\lesssim \Lambda\big(\f{1}{c_0}, |h|_{3,\infty,t}\big)\|\div^{\vp}u\|_{L_{t}^{\infty}H_{co}^1}.
\eeqs

\begin{section}{Appendix}
We give a short proof of  \eqref{product-R2}. The proof of $|f g|_{H^s(\mathbb{R}^2)}\lesssim |f|_{H^s}|g|_{W^{1,\infty}}, (0\leq s\leq 1)$ can be found in Theorem 15.2 of \cite{MR3590375}. The case for $-1<s<0$ is derived by duality. We thus focus on the proof of inequality:
$|f g|_{H^s(\mathbb{R}^2)}\lesssim |f|_{H^s}|g|_{H^{1^{+}}}, (-1<s\leq 1).$
We shall use Bony's decomposition: 
$$fg=T_g f+ \tilde{T}_f g=\sum_{j\geq 0}S_{j-1}g\Delta_j f+\sum_{k\geq -1}S_{k+2} f \Delta_{k} g.$$ One can refer to [p.61, [6]] for the definition of 
nonhomogeneous dyadic block \(\Delta_k\) and nonhomogeneous low-frequency cut-off operator \(S_k\).
For any $s\in \mR,$ one can control $T_g f$ as:
\beqs
|T_{g} f|_{H^s(\mR^2)}\lesssim |g|_{L^{\infty}}|f|_{H^s}\lesssim 
|g|_{H^{1^{+}}}|f|_{H^s}.
\eeqs
As for $\tilde{T}_f g,$ if $s<0,$ we control it with the aid of Bernstein inequality:
\beqs
\begin{aligned}
\big(2^{js}|\Delta_j \tilde{T}_f g|_{L^2}\big)_{l^2}&\lesssim \bigg(2^{j(s+1)}|\Delta_j \big(\sum_{k}S_{k+2}f\Delta_k g\big)|_{L^1}\bigg)_{l_j^2}\\
&\lesssim \big(2^{js}\sum_{k\leq j+5}|\Delta_{k}g|_{L^2}\big)_{l_j^2}\sup_k (2^{ks}|S_{k+2}f|_{L^2})\lesssim |g|_{H^1}|f|_{H^s},
\end{aligned}
\eeqs
and if $s>0,$
\beqs
\begin{aligned}
|\tilde{T}_f g|\lesssim \sup_k \big(2^{k(s-1-\kappa)}|S_{k+2}f|_{L^{\infty}}\big)|g|_{H^{1+\kappa}}\lesssim |f|_{H^s}|g|_{H^{1+\kappa}},
\end{aligned}
\eeqs
where $\kappa>0$ is a number that can be arbitrarily close to 0. The proof is now complete.
\end{section}

\section*{Acknowledgement}
The work of N. Masmoudi is supported by  NSF grant DMS-1716466 and  by Tamkeen under the NYU Abu Dhabi Research Institute grant
of the center SITE. F. Rousset was partially supported by the ANR projects ANR-18-CE40-0027 and ANR-18-CE40-0020-01. C. Sun 
benefits the postdoc fellowship funded by Labex CIMI. This work was conducted during the PhD study of the third author in LMO, he would like to thank the institute for providing great research environment.
\bibliographystyle{plain}
\nocite{*}
\bibliography{ref}
\end{document}